\documentclass[11pt, a4paper,reqno]{amsart}
\usepackage[T1]{fontenc}
\usepackage{amsmath}
\usepackage{amsthm}
\usepackage{amsfonts}
\usepackage{amssymb}
\usepackage{graphicx}
\usepackage{amsbsy}
\usepackage{mathrsfs}
\usepackage[colorlinks=true]{hyperref}
\newtheorem{theorem}{Theorem}[section]
\newtheorem{lemma}[theorem]{Lemma}

\newtheorem{proposition}[theorem]{Proposition}
\newtheorem{corollary}[theorem]{Corollary}
\newtheorem{definition}[theorem]{Definition}
\theoremstyle{remark}
\newtheorem{remark}[theorem]{\it \bf{Remark}\/}

\numberwithin{equation}{section}
\catcode`@=11
\def\section{\@startsection{section}{1}%
  \z@{1.5\linespacing\@plus\linespacing}{.5\linespacing}%
  {\normalfont\bfseries\large\centering}}
\catcode`@=12
\newcommand{\be}{\begin{equation}}
\newcommand{\ee}{\end{equation}}
\newcommand{\bea}{\begin{eqnarray}}
\newcommand{\eea}{\end{eqnarray}}
\newcommand{\bee}{\begin{eqnarray*}}
\newcommand{\eee}{\end{eqnarray*}}

\catcode`@=11
\def\section{\@startsection{section}{1}%
  \z@{1.5\linespacing\@plus\linespacing}{.5\linespacing}%
  {\normalfont\bfseries\large\centering}}
\catcode`@=12

\newcommand{\R}{\mathbb{R}}

\newcommand{\C}{\mathbb{C}}
\newcommand{\N}{\mathbb{N}}
\newcommand{\T}{\mathbb{T}}

\newcommand{\pt}{\partial}
\renewcommand{\leq}{\leqslant}
\renewcommand{\geq}{\geqslant}

\newcommand{\eps}{\varepsilon}

\newcommand{\norm}[1]{ \left| \! \left| #1 \right| \! \right| }
\newcommand{\Or}{\mathcal{O}}

\def\Re{{\rm Re}\,}
\def\Im{{\rm Im}\,}
\def\a{\alpha}
\def\pa{\partial}
\def\ka{\kappa}
\def\fref{\eqref}
\def\k{\kappa}

\def\l{\lambda}
\def\e{\varepsilon}

\def\E{\mathcal E}
\def\matchal{\mathcal}
\def\b{\beta}
\def\L{\matchal L}

\def\lslone{\frac{(\lambda_1)_{s_1}}{\lambda_1}}
\def\xslone{\frac{(x_1)_{s_1}}{\l_1}}
\def\lsltwo{\frac{(\lambda_2)_{s_2}}{\lambda_2}}
\def\xsltwo{\frac{(x_2)_{s_2}}{\l_2}}

\def\tL{\tilde{\Lambda}}
\def\La{\Lambda }

\def\la{\langle}
\def\ra{\rangle}

\def\LQb{\mathcal{L}_{\beta}}
\def\LQbn{\mathcal{L}_{\beta_n}}

\def\pt{\widetilde{\mathcal P}}
\def\et{\widetilde{\e}}
\def\Mod{{\rm Mod}}
\def\Dhalf{|D|^{\frac12}}
\def\build#1_#2^#3{\mathrel{\mathop{\kern 0pt#1}\limits_{#2}^{#3}}}
\def\td_#1,#2{\mathrel{\mathop{\build\longrightarrow_{#1\rightarrow #2}^{}}}}

\newcommand{\jb}[1]
{\langle #1 \rangle}

\begin{document}
\date{November 25, 2016}
\title[ A two-soliton for the cubic half-wave equation]{A two-soliton with transient turbulent regime for the cubic half-wave equation on the real line}
\author{Patrick G\'erard}
\address{ Patrick G\'erard, Laboratoire de Math\'ematiques d'Orsay, Univ. Paris-Sud, CNRS, Universit\'e Paris--Saclay, 91405 Orsay, France} \email{{\tt patrick.gerard@math.u-psud.fr}}
\author{Enno Lenzmann}
\address{ Enno Lenzmann, Mathematisches Institut, Universit\"at Basel, Spiegelgasse 1, CH--4051 Basel}\email{{\tt enno.lenzmann@unibas.ch}}
\author{Oana Pocovnicu}
\address{Oana Pocovnicu, Department of Mathematics, Heriot--Watt University and The Maxwell Institute for the Mathematical Sciences, Edinburgh EH14 4AS, UK}
\email{{\tt o.pocovnicu@hw.ac.uk}}
\author{Pierre Rapha\"el}
\address{Pierre Rapha\"el, Laboratoire Jean--Alexandre Dieudonn\'e, Universit\'e Nice Sophia Antipolis, Campus Valrose, 06130 Nice, France}
\email{{\tt Pierre.RAPHAEL@unice.fr}}

\subjclass{35B40, 35L05, 35Q41, 35Q51, 37K40}

\begin{abstract}
We consider the focusing cubic half-wave equation on the real line
$$
i \partial_t u + |D| u = |u|^2 u, \ \ \widehat{|D|u}(\xi)=|\xi|\hat{u}(\xi), \ \ (t,x)\in \Bbb R_+\times \Bbb R.
$$
We construct an asymptotic global-in-time compact two-soliton solution with arbitrarily small $L^2$-norm which exhibits the following two regimes: (i) a transient turbulent regime characterized by a dramatic and explicit growth of its $H^1$-norm on a finite time interval, followed by (ii) a saturation regime in which the $H^1$-norm remains stationary large forever in time. 
\end{abstract}

\keywords{Multi-soliton, modulation theory, wave turbulence, growth of Sobolev norms, half-wave equation, cubic Szeg\H{o} equation}

\maketitle
\tableofcontents

\section{Introduction}




In this paper we consider the $L^2$-critical focusing half-wave equation on $\R$: 
\begin{align}
\label{halfwave}
\text{(Half-wave)} \quad
\begin{cases}
i\partial_t u + |D| u =|u|^2 u\\
u_{|t=0}=u_0\in H^{\frac 12}(\R)
\end{cases}
, \quad (t,x)\in \Bbb R_+\times \Bbb R, \ \ u(t,x)\in \Bbb C,
\end{align}
where we use the pseudo--differential operators $$D=-i\pa_x, \ \ \widehat{|D|f}(\xi) = |\xi| \widehat{f}(\xi).$$ 
Evolution problems with nonlocal dispersion such as \eqref{halfwave} naturally arise in various physical settings, including continuum limits of lattice systems \cite{KiLeSt2011}, models for wave turbulence \cite{CaMa2001,MaMc1997}, and gravitational collapse \cite{ElSc2007,FrLe2007}. 
The phenomenon 
that we study in this paper is the
growth of high Sobolev norms in infinite dimensional Hamiltonian 
systems,
which has attracted considerable attention over the past twenty years \cite{Bo1996, St1997, MaMc1997, Bo2000, ZGPD2001, CaMa2001, CKSTT2010, GeGr2010, Po2011, H2014, GK2015, HPTV, HP, GHP, GeGr2016} .
The aim of this paper is to develop a robust approach for constructing solutions whose high Sobolev norms grow over time, based on multisolitary wave interactions. In particular, we construct an asymptotic two-soliton solution of \eqref{halfwave}
that exhibits the following two regimes: (i) a transient turbulent regime characterized by a dramatic and explicit growth of its $H^1$-norm on a finite time interval, followed by (ii) a saturation regime in which the $H^1$-norm remains stationary large forever in time. 

\subsection{The focusing cubic half-wave equation}
Let us recall the main qualitative features of the half-wave model \fref{halfwave}. The Cauchy problem is locally well-posed in $H^{\frac 12}$, see \cite{GeGr2012, KrLeRa}, and for all $u_0\in H^{\frac 12}$, there exists a unique solution $u\in \mathcal C([0,T),H^{\frac 12})$ with the blow up alternative 
\be
\label{vneionvenvo}
T<+\infty\ \ \mbox{implies}\ \ \lim_{t\uparrow T}\|u(t)\|_{H^{\frac 12}}=+\infty.
\ee Moreover, additional $H^s$-regularity on the data, $s>\frac 12$, is propagated by the flow. The Hamiltonian model \fref{halfwave} admits three conservation laws:
\bee
&&{\rm Mass}:\ \  \int|u(t,x)|^2dx=\int |u_0(x)|^2dx\\
&&{\rm Momentum}: \ \ \Re\left(\int Du\overline{u}(t,x)dx\right)=\Re\left(\int Du_0\overline{u_0}(x)dx\right)\\
&&{\rm Energy}:\ \ E(u(t)):=\frac 12\int|\Dhalf u|^2(t,x)dx-\frac 14\int |u|^4(t,x)dx=E(u_0).
\eee
The scaling symmetry $$u_\l(t,x)=\l^{\frac 12}u(\l^2 t,\l x)$$ leaves the $L^2$-norm invariant $$\|u_\l(t,\cdot)\|_{L^2}=\|u(\l^2 t,\cdot)\|_{L^2}$$ and hence the problem is $L^2$-critical.

By a standard variational argument, the best constant in the Gagliardo-Nirenberg inequality 
$$ \Vert u\Vert _{L^4}^4\lesssim  \|\Dhalf u\|_{L^2}^2\|u\|_{L^2}^2, \ \ \forall u\in H^{\frac 12}, $$ is attained on the {\it unique} positive even ground state solution to $$|D|Q+Q-Q^3=0.$$
Note that the uniqueness 
of $Q$
is a nontrivial claim, recently obtained in \cite{FrLe2012}. 
This implies the lower bound 
\be
\label{gagnoreneebrharp} E(u)\ge \frac12 \left[1-\frac{\|u\|_{L^2}^2}{\|Q\|_{L^2}^2}\right] \int_\R |\Dhalf u|^2dx, \ \  \forall u\in H^{\frac 12}.
\ee 
Using the conservation of mass and energy, it then follows for $u_0\in H^{\frac 12}$
with $\|u_0\|_{L^2}<\|Q\|_{L^2}$ that
\begin{equation}\label{boundH12}
\|u(t)\|_{H^{\frac 12}}\leq C(\|u_0\|_{L^2},E(u_0)), \quad \forall t\in\R. 
\end{equation}
Combining this with \fref{vneionvenvo}, one obtains the global existence criterion:
\be
\label{globalexistemce}
u_0\in H^{\frac 12}\ \ \mbox{and} \ \ \|u_0\|_{L^2}<\|Q\|_{L^2}\ \ \mbox{imply}\ \ T=+\infty.
\ee
This criterion is sharp as there exist minimal mass finite energy finite time blow up solutions, 
see \cite{KrLeRa}. In this paper we will only consider solutions with $u_0\in H^{1}$ of arbitrarily small mass, which are hence global-in-time $u\in \mathcal C(\R, H^1)$.

\subsection{Growth of high Sobolev norms}

One of the main topics in the study of nonlinear Hamiltonian PDEs is the long time behaviour
of global-in-time solutions. 
A possible type of behavior, that attracted significant attention over the last twenty years, 
is 
the so called {\it forward energy cascade} phenomenon. 
This phenomenon refers to the {\it conserved} energy of global-in-time solutions moving from 
low-frequency concentration zones to high-frequency ones 
over time. 
One way to illustrate it
is the growth of high Sobolev norms:
\[\|u(t)\|_{H^s}=\left(\int \jb{\xi}^{2s} |\hat u(t,\xi)|^2d\xi \right)^{\frac 12}.\]
Indeed, for sufficiently large $s>0$, above the level of regularity of the conserved Hamiltonian,
the growth over time of $\|u(t)\|_{H^s}$ 
indicates that the Fourier transform $\hat u(t,\xi)$ is supported on higher and higher frequencies $\xi$
as the time $t$ increases. 
To the best of the authors' knowledge, 
all the rigorous mathematical analysis that has been done on the forward energy cascade 
focuses on finding infinite dimensional Hamiltonian PDEs that admit examples of solutions 
exhibiting growth of 
high Sobolev norms. 
A lot of the results available are in the context of 
nonlinear Schr\"odinger equations (NLS). 
In particular, for the defocusing cubic nonlinear Schr\"odinger equation
on $\T^2$, Bourgain \cite{Bo2004} asked whether there exist solutions $u$ with initial 
condition $u_0\in H^s(\T^2)$, $s>1$, such that
\[\limsup_{t\to\infty}\|u(t)\|_{H^s}=\infty.\]
Despite attracting considerable attention, this question remains unanswered. 

The forward energy cascade phenomenon also appears in the physical theory of {\it wave {\rm(}weak{\rm )} turbulence}. 
This is a theory in plasma physics and water waves, based on pioneering work 
of Zakharov from the 1960s, with many similarities to Kolmogorov's 
theory of hydrodynamical turbulence. 
It can be loosely defined as the ``out-of-equilibrium statistics of random nonlinear waves" (see \cite{H2014}).
Even though wave turbulence refers to a statistical description of solutions and not to single solutions,
and even though this theory does not yet have a rigorous mathematical justification, 
it is believed that exhibiting examples of solutions whose high Sobolev norms grow over time 
is a first step and a minimal necessary condition for wave turbulence. 
As far as the authors are aware, all mathematically rigorous results that are available are in this spirit,
and so is the main result of this paper. 

In the following, we briefly mention some of the references in the literature 
regarding the growth of high Sobolev norms for nonlinear Hamiltonian PDEs.
First, in the context of NLS,
polynomial-in-time {\it upper} bounds on the growth of $\|u(t)\|_{H^s}$,
\[\|u(t)\|_{H^s}\lesssim \jb{t}^{c(s-1)}, \quad s>1,\]
were obtained; see Bourgain \cite{Bo1996, Bo2004}, Staffilani \cite{St1997}, Sohinger \cite{Soh1, Soh2},
Colliander, Kwon, and Oh \cite{CKO}.

The first examples of Hamiltonian PDEs (nonlinear Schr\"odinger equations and nonlinear wave equations) that 
admit solutions with energy transfer were constructed by Bourgain \cite{Bo1995, Bo1996, Bo1997}. 
However, these examples do not deal with standard NLS or NLW, 
but with modifications of these specifically designed to exhibit infinite growth of high Sobolev norms (these are PDEs involving, instead of the Laplace operator, a perturbation of it, or PDEs with a suitably chosen nonlocal nonlinearity). In \cite{Kuksin}, Kuksin considered small dispersion cubic NLS
and proved that generic solutions grow larger than a negative power of the dispersion. 
A seminal result is that by Colliander, Keel, Staffilani, Takaoka, and Tao \cite{CKSTT2010} 
who proved arbitrarily large growth of high Sobolev norms in finite time for the defocusing cubic NLS on $\T^2$. 
More precisely, given $s>1$, $\eps\ll 1$, and $K\gg 1$,
they constructed a solution $u$ such that
\[\|u(0)\|_{H^2}\leq \eps \quad \text{and} \quad \|u(T)\|_{H^s}\geq K,\]
for some finite time $T>0$. 
The influential result in \cite{CKSTT2010}, especially their intricate combinatorial construction, was refined and generalized to various other settings \cite{H2014, GK2015, Guardia2014, HPTV, GHP, HP}.
In particular, in \cite{HPTV}, 
an example of infinite growth of high Sobolev norms was obtained for the defocusing cubic NLS on $\R\times \T^d$, $d\geq 2$. 
For the cubic NLS on $\T^2$, however,
the fate of the solution $u$ after the growth time $T$ remains unknown. 

For the cubic half-wave equation, due to mass and energy conservation,
the $H^{\frac 12}$-norm of solutions with initial data in $H^{\frac 12}$ is uniformly bounded in time, both for the defocusing equation, as well as for the focusing equation with initial data of sufficiently small mass $\|u(0)\|_{L^2}<\|Q\|_{L^2}$ (see \eqref{boundH12} above).
However, in the spirit of \cite{CKSTT2010},
arbitrarily large growth in finite time of higher Sobolev norms --- $H^s$-norms with $s>1/2$ --- was proved on $\R$ in \cite{Po2013}\footnote{In \cite{Po2013}, only a {\it relative} growth of high Sobolev norms was obtained, $\frac{\|u(T_\eps)\|_{H^s}}{\|u(0)\|_{H^s}}\to \infty$ as $\eps\to 0$ for some $T_\eps\gg 1$. 
However, this readily yields arbitrary large growth in finite time via an $L^2$-invariant scaling argument. 
Secondly, the result in \cite{Po2013} is stated for the defocusing half-wave equation, but essentially the same proof works for the focusing half-wave equation with initial data of small mass.} and on $\T$ in \cite{GeGr2012}. 
As in \cite{CKSTT2010}, the behaviour of the solutions that exhibit growth remains unknown after some finite time,
which is what motivated our work in the present paper. 
The results in \cite{GeGr2012, Po2013} are based on information on the totally resonant model associated with the cubic half-wave equation, namely the Szeg\H{o} equation. 
Infinite growth of high Sobolev norms for solutions of the Szeg\H{o} equation
was obtained on $\R$ in \cite{Po2011} and on $\T$ in \cite{GeGr2016}. 
Moreover, on $\T$, this
was shown \cite{GeGr2016} to be a generic phenomenon, displaying infinitely many {\it forward} and {\it backward} energy cascades. Also notice that long time divergence of high Sobolev norms was also obtained for a perturbation of the cubic Szeg\H{o} equation on $\T $ in \cite{Xu2014}.
We present below the key features of the Szeg\H{o} equation 
and its relation to the cubic half-wave equation.


\subsection{The Szeg\H{o} program}

Applying the Szeg\H{o} projector $\Pi _+$ of $L^2$ onto nonnegative Fourier modes:
$$\widehat{\Pi^+u}(\xi)={\bf 1}_{\xi>0} \hat{u}(\xi),$$
 the half-wave equation \fref{halfwave} becomes 
\bee
\left \{ \begin{array}{ll}   i(\pa _tu_+-\pa _xu_+)=\Pi ^+(\vert u\vert ^2u)\\
i(\pa _tu_- +\pa _xu_-)=(I-\Pi ^+)(\vert u\vert ^2u) \\
u_+:=\Pi ^+u\ ,\ u_-:=(I-\Pi ^+)u\ .
\end{array}\right .  
\eee
For small data in the range of $\Pi _+$ and of norm $\e \ll 1$ in a sufficiently regular Sobolev space one can show \cite{GeGr2012, Po2013} that, for times of order $\e ^{-2}|\log\eps|$, an approximation of the  half-wave flow is given by 
the cubic Szeg\H{o} equation 
\be
\label{eqSzeg\H{o}}
\ \ \left\{\begin{array}{ll}i\pa_tu=\Pi^+(|u|^2u)\\ u_{|t=0}=u_0\in H^{\frac 12}.
\end{array}
\right. 
\ee
The Szeg\H{o} equation can be understood as the {\it totally} resonant model associated to \fref{halfwave}. It is still a nonlinear Hamiltonian model, well-posed in $H^{\frac 12}$, and the conservation of mass and momentum implies that all $H^{\frac 12}$-solutions are global-in-time and
$$\|u(t)\|_{H^{\frac 12}}\simeq \|u(0)\|_{H^{\frac 12}}, \ \ \forall t\in\R.$$ 
A spectacular feature of the cubic Szeg\H{o} equation discovered  in \cite{GeGr2010} is its {\it complete integrability} in the sense of the existence of a Lax pair, which in particular allows for the derivation of explicit families of special solutions of either multisolitary waves or breather-type, both on the line and on the torus, see \cite{Po, Po2011, GeGr2010,GeGrbis, GeGr2015, GeGr2016}. The complete integrability implies the conservation of infinitely many conservation laws which, however, roughly speaking,  all live at the $H^{\frac 12}$-level of regularity only.

In \cite{Po2011}, Pocovnicu exhibits for the flow on the line, one of the very first explicit examples of growth of high Sobolev norms for a nonlinear infinite dimensional Hamiltonian model: 
$$\|u(t)\|_{H^{\frac 12}}\lesssim 1, \ \ \lim_{t\to +\infty}\|u(t)\|_{H^1}=+\infty\ \ \mbox{as}\ \ t\to +\infty.$$ The analysis in \cite{Po2011} is based on the explicit computation of a two-soliton solution for the cubic Szeg\H{o} flow, relying on complete integrability.\footnote{The key property that triggers growth of high Sobolev norms $\|u(t)\|_{H^s}\sim t^{2s-1}$, $s>\frac 12$, is that the Hankel operator $H_{u}$ in the Lax pair of the Szeg\H{o} equation has a multiple (double) eigenvalue.} Indeed, as observed in \cite{Po}, \fref{eqSzeg\H{o}} admits a traveling wave solution 
\be
\label{Szeg\H{o}rpfile}
u(t,x)=Q^+(x-t)e^{-it}\ \ \mbox{with}\ \ Q^+(x):=\frac{2}{2x+i}.
\ee
Using complete integrability formulas, an exact two-soliton can be computed: $$u(t,x)= \alpha_1(t)Q^+\left(\frac{x-x_1(t)}{\k_1(t)}\right)e^{-i\gamma_1(t)}+\alpha_2(t)Q^+\left(\frac{x-x_2(t)}{\k_2(t)}\right)e^{-i\gamma_2(t)}\ ,$$ with the asymptotic behavior on the manifold of solitary waves,
\be
\label{blowupszego}
\left\{\begin{array}{ll}
\alpha_1(t)\sim1, \ \ \k_1(t)\sim 1-\eta, \ \ x_1(t)\sim (1-\eta)t, \ \ 0<\eta\ll1\\
\alpha_2(t)\sim 1, \ \ \k_2(t)\sim \frac{1}{t^2}, \ \ x_2(t)\sim t.
\end{array}\right.
\ee
In particular, this two-soliton exhibits growth of high Sobolev norms over time $\|u(t)\|_{H^s}\sim t^{2s-1}$, $s>\frac 12$, and the mechanism of growth is the concentration of the second 
bubble $k_2(t)\sim \frac{1}{t^2}$. 

The full dynamical system underlying two-solitons for the Szeg\H{o} equation and the associated codimension one set of turbulent initial data is revisited in details in Appendix \ref{systemszego}.

Combining the growth of high Sobolev norms for a two-soliton of the Szeg\H{o} equation on $\R$ 
\cite{Po2011} discussed above, with a long time approximation theorem relating the Szeg\H{o} model and the half-wave equation, 
yields the following arbitrarily large growth in finite time result for the half-wave equation:
\begin{theorem}[\cite{Po2013}]\label{oldthm}  
Let $0<\eps\ll1$. 
There exists a solution of the {\rm(}focusing/defocusing{\rm)} cubic half-wave equation on $\R$ and there exists 
$T\sim e^{\frac{c}{\eps^3}}$ such that
\[\|u(0)\|_{H^1}=\eps \quad \text{and} \quad \|u(T)\|_{H^1}\geq \frac{1}{\eps}\gg 1.\]
\end{theorem}
As in \cite{CKSTT2010}, the behaviour of the turbulent solution in the above theorem after the time $T$ remains unknown. 
In this paper, we construct a turbulent solution of \eqref{halfwave} that we can control {\it for all future times}.
Furthermore, our aim in this paper is to develop a robust approach to compute turbulent regimes based on multisolitary wave interactions, avoiding on purpose complete integrability tools.


\subsection{Mass-subcritical traveling waves}


As observed in \cite{KrLeRa} following \cite{FrLe2012}, the half-wave problem \eqref{halfwave} admits {\it mass-subcritical} small speed traveling waves\footnote{Note that this phenomenon does not exist for the mass-critical focusing nonlinear Schr\"odinger equation on $\R$ due to the degeneracy induced by the Galilean symmetry $u_{\beta}(t,x)=Q_\beta(x-\beta t)e^{i\gamma_\beta(t)}$ with $Q_\beta(x)=Q(x)e^{i\beta x}$ and hence $\|Q_\beta\|_{L^2}=\|Q\|_{L^2}$ for all $\beta\in \Bbb R$, and indeed solutions with mass below that of the ground state scatter \cite{Dodson}.}
\be
\label{travellingwave}
u_{\beta}(t,x)=Q_\beta\left(\frac{x-\beta t}{1-\beta}\right)e^{-it}, \ \ \frac{|D|-\beta D}{1-\beta} Q_\beta + Q_\beta - |Q_\beta|^2 Q_\beta = 0,
\ee
with $$\lim_{\beta\to 0}Q_\beta=Q, \ \ \|Q_\beta\|_{L^2}<\|Q\|_{L^2}.$$ An elementary but spectacular observation is that these traveling waves in fact exist for all $|\beta|<1$ and converge in the {\it singular relativistic limit} $\beta\to 1$ to the soliton of the limiting Szeg\H{o} equation given by \eqref{Szeg\H{o}rpfile}:
$$\lim_{\beta\uparrow 1}\|Q_\beta- Q^+\|_{H^{\frac 12}}=0.$$ 
See Section \ref{sectiontravellingwaves}. Note from \eqref{travellingwave} that this is fundamentally a singular elliptic limit, and the associated almost relativistic traveling waves are arbitrarily small in the critical space: $$\lim_{\beta\uparrow 1}\|u_\beta(t,\cdot)\|_{L^2}=0.$$
Hence, another link is made between the half-wave problem and its totally resonant limit given by the Szeg\H{o} equation through the sole consideration of the full family of {\it nonlinear} traveling waves.


\subsection{Statement of the result}


In Theorem \ref{oldthm}, the turbulent solution of \eqref{halfwave} was constructed as a long time approximation of the turbulent two-soliton of the Szeg\H{o} equation. 
The approximation theorem used is valid for any solution of the Szeg\H{o} equation (respectively of the half-wave equation) with small regular data, not only for two-solitons. 
In this paper, we take a more efficient approach.
Instead of approximating a large class of solutions of \eqref{halfwave} by their Szeg\H{o} counterparts,
we concentrate on constructing a single solution of \eqref{halfwave} that mimics the growth mechanism of the turbulent two-soliton
of the Szeg\H{o} equation. 
Of course, complete integrability is lost, but the analysis initiated by Martel in  \cite{Martelmulti} and revisited in \cite{KMRhartree} for the nonlocal Hartree problem paves the way to the construction of {\it compact} two-bubble elements. 
More precisely,
one can in principle extract from the equation the approximate dynamical system driving each solitary wave 
of an asymptotic two-soliton, 
at least in a regime where the waves are separated in space, and the robust energy method developed in \cite{KMRhartree} allows one to follow the flow all the way to $+\infty$.\\

\begin{theorem}[Solution with transient turbulent regime and saturated growth]
\label{thmmain}
There exists a universal constant $0<\delta ^*\ll 1$ and, for all $\delta \in (0,\delta ^*)$,  there exists $0<\eta ^*(\delta )\ll 1$ such that the following holds. For every $\eta \in (0,\eta^*)$, let the times $$T_{\rm in}=\frac{1}{\eta^{2\delta}}, \ \ T^-=\frac{\delta}\eta,$$  then there exists a solution $u\in \mathcal C([T_{\rm in},+\infty),H^1)$ to \eqref{halfwave} which is $H^{\frac 12}$-compact as $t\to +\infty$ with the following behavior:\\
\noindent{\it 1. Initial data}: the initial data at time $T_{\rm in}$ has size $$\|u(T_{\rm in})\|_{L^2}^2\sim \eta, \ \ \|||D|^{\frac 12}u(T_{\rm in})\|_{L^2}^2\sim 1, \ \ \|Du(T_{\rm in})\|_{L^2}^2\sim \frac{1}{\eta^{1+2\delta}}.$$
\noindent{\it 2. Turbulent regime}: on $[T_{\rm in},T^-]$, the solution experiences  a turbulent interaction with {\it an explicit monotone growth} of the $H^1$-norm
\be
\label{growthexplicit}
\|u(t)\|_{H^1}^2=\frac{t^2}{\eta}(1+O(\sqrt{\delta})).
\ee
\noindent{\it 3. Saturation}: the interaction ceases after $T^-$ and there holds the saturation
$$\|u(t)\|_{H^1}^2=\frac{1}{\eta^3}e^{O(\frac{1}{\delta})} \ \ \mbox{for}\ \ t\geq T^-.$$
\end{theorem}
The turbulent interaction behind \eqref{growthexplicit} is an explicit energy transfer along the singular branch of traveling waves $Q_\beta$, and the solution can more explicitly be described as follows. For all times $t\in[T_{\rm in},+\infty)$, the solution admits a two solitary wave decomposition  
$$u(t,x)=\sum_{j=1}^2\frac{1}{\l_j^{\frac 12}(t)}Q_{\beta_j(t)}\left(\frac{x-x_j(t)}{\l_j(t)(1-\beta _j(t))}\right)e^{-i\gamma_j(t)}+\e(t,x)$$ with the following properties:

\smallskip
\noindent{\em 1. Structure of the first soliton}: the first soliton remains nearly unchanged, i.e. for all $t\geq T_{\rm in}$, $$\l_1(t)\sim 1, \ \ 1-\beta_1(t)\sim \eta, \ \ x_1(t)\sim (1-\eta)t, \ \ \gamma_1(t)\sim t.$$
\noindent{\em 2. Concentration of the second soliton}: the second soliton behaves like a solitary wave $$\l_2(t)\sim1, \ \ x_2(t)\sim \beta_2 t, \ \ \gamma_2(t)\sim t$$ with a concentration of size in the transient turbulent regime: 
$$1-\beta_2(t)=\frac{\eta(1+O(\sqrt{\delta})}{t^2}\ \ \mbox{for}\  \ t\in [T_{\rm in},T^-],$$ 
which saturates after the interaction time $T^-$: 
$$1-\beta_2(t)=\eta^3e^{O(\frac{1}{\delta})}\ \ \mbox{for}\ \ t\ge T^-.$$
\noindent{\em 3. Asymptotic compact behaviour}: this solution is minimal near $+\infty$, i.e.
$$\lim_{t\to +\infty}\|\e(t)\|_{H^1}=0.$$

\subsection{Comments on the result}

Theorem \ref{thmmain} exhibits, for a canonical dispersive model, 
an {\it explicit mechanism} of growth of high Sobolev norms. 
To the best of the authors' knowledge, 
this is one of the first results in which one can control {\it for all times} a turbulent solution of a 
nonlinear Hamiltonian PDE.
 
\smallskip
\noindent
1. {\it The two regimes}. 
The key element behind Theorem \ref{thmmain} is the derivation of the leading order ODEs driving the geometrical parameters as in \cite{KrLeRa}. There are two main new pieces of information. 
First, we can compute explicitly the rate of concentration which is given by the $t$-growth as in \cite{Po2011}. 
This rate is very sensitive to the {\it phase shift} between the waves in the transient regime, and another phase shift would generate another speed. Note that the growth can be computed for any $H^s$-Sobolev norm above the energy, i.e. $s>\frac 12$, and the data can also be taken arbitrarily small in $H^1$ by a fixed rescaling. Secondly and unlike in the case of the Szeg\H{o} equation, there is no infinite growth of the $H^1$-norm for the solution we construct. Here we encountered an essential feature in the structure of the $Q_\beta$ solitary wave. The limiting solitary wave of the Szeg\H{o} equation has according to \fref{Szeg\H{o}rpfile} a far out decay 
$$Q^+(x)\sim \frac{1}{\la x\ra},$$ while for $Q_\beta$ there is a transition regime 
\be
\label{decayvbeibvejbi}
Q_\beta(x)\sim \frac1{\la x\ra(1+(1-\beta)\la x\ra)}, \ \ \beta<1.
\ee 
In particular,
when the waves forming the two-soliton separate and their relative distance becomes large $$|x_2-x_1|\gg \frac{1}{1-\beta},$$ their interaction 
weakens from $\frac{1}{\la x\ra}$ to $\frac{1}{\la x\ra^2}$, and this explains why the concentration mechanism stops in the far out two-soliton dynamics.
\\

\noindent
2. {\it Compact bubbles with energy transfer}. Theorem \ref{thmmain} lies within the construction of compact elements which has attracted a considerable attention for the past ten years both for global problems since the pioneering breakthrough work \cite{Martelmulti} and \cite{MMnls, KMRhartree, MMwave} and blow up problems \cite{merlemulti, Rszeftel, MMR2, KrLeRa}.  It is in particular shown in \cite{KrLeRa} how the presence of polynomially decaying interactions can lead to dramatic deformations of the soliton dynamics, for example from the straight line motion for each wave to the hyperbolic two body problem of gravitation for the two-soliton of the 
gravitational Hartree model on $\R^3$. The energy transfer mechanism between KdV waves  \cite{tetsu, MMinvent} or the recent multibubble infinite time blow up mechanism of \cite{MRmulti}  are deeply connected to Theorem \ref{thmmain}. 
This is the first instance, however, when modulation analysis used in all the above cited works,
is employed
to find solutions that exhibit growth of high Sobolev norms.
Let us insist that the growth \eqref{growthexplicit} does not excite the $L^2$-scaling instability of the problem as in \cite{KrLeRa}, but the $\beta$-instability which according to \eqref{travellingwave} is $\dot{H}^{\frac 12}$-critical and hence compatible with the small data coercive conservation laws. More generally, there is
little understanding of the long time asymptotics of wave equations in small dimensions due to the lack of dispersion, see for example \cite{LT}, and it is essential for the construction to consider {\it compact} 
nondispersive flows.\\

\noindent
3. {\it Specificity of the analysis}. 
The following two problems are simpler than the result in Theorem \ref{thmmain}: 
(i) the construction of an asymptotic two-soliton {\it without} turbulent interaction in the continuation of \cite{KMRhartree}, 
and (ii) exhibiting a growth mechanism of the $H^1$-norm on some sufficiently large time interval as in \cite{Po2013}, using the limiting singular Szeg\H{o} regime (see Theorem \ref{oldthm} above). 
The aim of Theorem \ref{thmmain} is to perform both the above in the same time and, in particular, to capture the associated saturation of the $H^1$-norm which we expect displays some universality, and hence describes the long time dynamical bifurcation of \eqref{halfwave} from the Szeg\H{o} singular regime \eqref{blowupszego} {\it beyond usual Ehrenfest-like times}. We then face two essential difficulties. First, the nonlocal  nature of the problem in the presence of slowly decaying solitary waves makes interactions very large and hard to decouple as in \cite{KMR,MMwave}. 
In particular, we need to control the logarithmic instability of the phase shift between the waves, which is central for the derivation of the growth mechanism. 
This forces us to develop both the complete description of the bifurcation $Q^+\to Q_\beta$ and a new strategy for the derivation of {\it sharp} modulation equations for geometrical parameters, see Proposition \ref{sharpmod}. Secondly, the need for high order approximations of the solution required to capture the leading order mechanism is reminiscent of the pioneering two-soliton interaction computations in \cite{MMannals, MMinvent}. But the main difficulty here is the fact that the traveling wave equation \eqref{travellingwave} is a singular elliptic problem which degenerates as $\beta\to1$. Hence one looses the control of natural energy norms in the concentration process, which a priori should ruin the approach developed in \cite{KMR}. The wave-like structure of the equation is essential to overcome this difficulty. We also need to develop various new 
estimates involving the $\Pi^+$ projection operator onto positive frequencies 
since in the concentration process, this 
projection and the Szeg\H{o}-like regimes are essential for the analysis.\\ 

\noindent
4. {\it Regularity shift in the growth of Sobolev norms.} Compared to previous results on the growth of high Sobolev norms for nonlinear Schr\"odinger equations, see \cite{CKSTT2010, GK2015, H2014, HPTV, HP, GHP}, it is interesting to notice that  
Theorem \ref{thmmain} implies the existence of small data in $H^1$ such that the $H^s$-norm of the solution becomes large, not only for $s=1$, but also for $s<1$ close to $1$. 
Notice that this regularity shift also holds --- with unbounded solutions at infinity --- for the cubic Szeg\H{o} equation, see \cite{Po2013, GeGr2016}, where in \cite{GeGr2016} this phenomenon is established to be generic.

\smallskip
Having completed this work, let us mention a number of related open problems.
\begin{itemize}
\item The main one is probably the existence of a solution of \eqref{halfwave} such that
$\limsup_{t\to\infty}\|u(t)\|_{H^1}= +\infty$. 
\item What are the possible growth rates ? From the recent paper \cite{Th2016}, we know that this rate cannot be bigger than $e^{O(t^2)}$, how optimal is it ?
\item Are unbounded solutions in $H^1$ generic ? Is the behavior $\|u(t)\|_{H^1}\td_t,\infty\infty$ generic, or rather is it generic to have infinitely many forward and backward energy cascades, as in the case of the cubic Szeg\H{o} equation on the circle ?
\end{itemize}
To conclude, we hope that Theorem \ref{thmmain} is an important step towards a  better understanding of the role played by interactions of solitons in turbulent transfers of energy.


\subsection{Strategy of the proof}


We outline in this subsection the main steps and difficulties in the proof of Theorem \ref{thmmain}.\\

\noindent{\bf Step 1:} {\it Description of the bifurcation $Q^+\to Q_\beta$}. Our first task is to completely describe the solutions to the singular elliptic traveling wave equation 
$$\frac{|D|-\beta D}{1-\beta}Q_\beta+Q_\beta-Q_\beta|Q_\beta|^2=0$$ 
in the limit $\beta\to 1$. The local existence and uniqueness of the profile $Q_\beta$ for $\beta$ close to 1 in Proposition \ref{prop:QbCV} relies on a classical Lyapunov-Schmidt argument, which itself relies on the non degeneracy of the linearized operator close to $Q^+$ for the Szeg\H{o} problem proved in \cite{Po}. The Lyapunov-Schmidt argument yields the non degeneracy of the linearized operator close to $Q_\beta$ in Proposition \ref{invertibility1}. We then completely describe the profile in space of $Q_\beta$ and, in particular, its long range asymptotics which displays a nontrivial boundary layer at $x\sim \frac{1}{1-\beta}$, see Section \ref{sectionqbeta}. Here we aimed at avoiding logarithmic losses which would be dramatic for the forthcoming analysis, and this requires the consideration of suitable norms and Fourier multipliers.\\

\noindent{\bf Step 2:} {\it Two-soliton ansatz}. We now implement the strategy developed in \cite{KMR} and construct an approximate solution of the form $$u=u_1+u_2$$ after reduction to the slow variables $$u_j(t,x)=\frac{1}{\l_j^{\frac 12}}v_j(s_j,y_j)e^{i\gamma_j},  \ \ \frac{ds_j}{dt}=\frac{1}{\l_j(t)}, \ \ y_j:=\frac{x-x_j(t)}{\l_j(t)(1-\beta_j(t))}, \ \ j=1,2.$$ 
Here we proceed to an expansion of the profiles $v_j$ after separation of variables $$v_j(s_j,y_j)=Q_{\beta_j(s_j)}(y_j)+\sum_{n=1}^NT_{j,n}(y_j,\mathcal P(s_j)),$$ 
where $\mathcal P$ encodes the geometrical parameters of the problem
$$\mathcal P=(\l_1,\l_2,\beta_2,\beta_2,\Gamma,R)$$ and $(\Gamma, R)$ denote the phase shift and relative distance between the waves {\it after renormalization} $$\Gamma=\gamma_2-\gamma_1, \ \ R=\frac{x_2-x_1}{\l_1(1-\beta_1)}$$ which is always large $R\gg 1$. The laws for the parameters are adjusted 
\be
\label{nekoneoneo}
\frac{(\l)_{s_j}}{\l_j}=M_j(\mathcal P), \ \ \frac{(\beta_j)_{s_j}}{1-\beta_j}=B_j(\mathcal P)
\ee 
in order to ensure the solvability of the elliptic system defining $T_{j,n}$;
see Proposition \ref{propwtobublle}. In order to keep control of the various terms produced by this procedure, we need to define a notion of admissible function, see Definition \ref{defadmissible}, which is compatible with the properties of $Q_\beta$ and stable for this nonlinear procedure of construction of the approximate solution. The strategy is conceptually similar to \cite{KMRhartree}, but the functional framework is considerably more challenging due to the slow decay of the solitary wave $Q_\b$ and to the singular nature of the bifurcation $Q^+\to Q_\beta$.\\

\noindent{\bf Step 3:} {\it Leading order dynamics}. We now extract the leading order dynamics for the ODEs predicted by \eqref{nekoneoneo}. This step is more delicate than one would expect, in particular because we need to keep track of a logarithmic instability of the phase shift $\Gamma$ which is essential for the derivation of the turbulent growth. We observe in Proposition \ref{sharpmod} that {\it mimicking the conservation laws} of mass and kinetic momentum for the approximate solution provides nonlinear cancellations and a high order approximation of the dynamical system for $\mathcal P$. Roughly speaking, this reads 
$$\frac{(\beta_1)_t}{1-\beta_1}\sim 0, \ \ \frac{(\beta_2)_t}{1-\beta_2}\sim \frac{2\cos\Gamma}{R(1+(1-\beta_1)R)}, \ \ R\sim t$$ 
which reflects the decay \eqref{decayvbeibvejbi}. Hence, $1-\b_1\sim \eta$ and as long as $\Gamma\sim 0$ and $t\leq \frac{1}{\eta}\sim T^-$, we have the decay $$1-\beta_2(t)\sim \frac{1}{t^2},$$ which saturates for $t\geq T^-$. Keeping the phase under control requires a high order approximation of the modulation equations (Proposition \ref{sharpmod}) and a careful integration of the associated modulation equations; see Subsection \ref{sectionreduced}.\\

\noindent{\bf Step 4:} {\it Backwards integration and energy bounds}. We now solve the problem from $+\infty$ following the backward integration scheme designed in \cite{merlemulti,Martelmulti,MMnls,KMRhartree}. 
In the setting of a suitable bootstrap (Proposition \ref{propboot}), the solution decomposes into two bubbles and radiation $$u(t,x)=\sum_{j=1}^2u_j+\e(t,x), \ \ u_j(t,x)=\frac{1}{\l_j^{\frac 12}}v_j(s_j,y_j)e^{i\gamma_j},$$ where the profiles $v_j$ have been constructed above. 
We pick a sequence $T_n\to+\infty$ and look for uniform backwards estimates for the solution to \eqref{halfwave} with Cauchy data at $T_n$ given by 
\be
\label{esttn}
\e(T_n)=0.
\ee 
The heart of the analysis is to design an energy estimate to control $\e$. Following \cite{MMnls,KMRhartree}, the energy functional is a localization in space of the total conserved energy, with cut-off functions which are adapted to the dramatic change of size of the second bubble. The outcome is an energy bound of the type 
\be
\label{vnionvoneoevo}
\left|\frac{d}{dt}\mathcal G(\eps(t))\right|\lesssim \frac{\mathcal G(\eps)}{t}+\frac{C_N}{t^N}
\ee 
where $N$ is the order of accuracy of the approximate solution and can be made arbitrarily large, and $\matchal G$ is a suitable energy functional with roughly $$\mathcal G(\e)\sim \|\e\|^2_{H^{\frac 12}},$$ see Proposition \ref{coerclinearizedenergy}.  Bootstrapping the bound $\mathcal G(\eps(t))\leq\frac{1}{t^{\frac N2}}$ and integrating in time using the boundary condition \eqref{esttn} yields $$\matchal G(\eps (t))\lesssim \frac{1}{Nt^{\frac N2}},$$ 
which is an improved bound for $N$ universal sufficiently large. The critical point in this argument is the $\frac1t$ loss only in the RHS of \eqref{vnionvoneoevo}. In general, the terms induced by the necessary localization procedure may be difficult to control, and sometimes the only known way out is a symmetry assumption on the behaviour of the bubbles as in \cite{KMRhartree,MRmulti}.
This is not an option here since the turbulent regime is in essence asymmetric. Furthermore, a fundamental difficulty here is that the linearized operator close to $Q_\beta$ depends on $\beta$ and degenerates as $\beta\to 1$, see \eqref{coerc}. We show in Section \ref{sectionenergy} that the above strategy can be implemented with a sharp loss of $\frac 1t$ only, using two new ingredients: a favorable algebra for the localization terms, which seems specific to wave-like problems and is reminiscent of a related algebra in \cite{MMwave}, see the proof of \eqref{poitpsi}, and the splitting of the motion along positive and negative frequencies which move in space differently. Hence the full energy method relies very strongly on the localization {\it both in space and frequency} of the infinite dimensional part of the solution.\\

This paper is organized as follows. In Section \ref{sectiontravellingwaves}, we construct the bifurcation $Q^+\to Q_\beta$ \`a la Lyapunov-Schmidt, and we study in detail the $Q_\beta$ profile in Section \ref{sectionqbeta}. In Section \ref{sectionfour}, we produce the two-bubble approximate solution (Proposition \ref{propwtobublle}) and derive and study the associated dynamical system for the geometrical parameters (Proposition \ref{sharpmod} and Subsection \ref{sectionreduced}). 
In Section \ref{sectionenergy}, we close the control of the infinite dimensional remainder by setting up the bootstrap argument (Proposition \ref{propboot}), and by using in particular the key energetic control given in Proposition \ref{propenergy}. 
The proof of Theorem \ref{thmmain} easily follows from Proposition \ref{propboot} as detailed in Subsection \ref{proofthm}. Appendix A is devoted to simple algebraic formulae involving $Q^+$. 
Appendix B revisits the two-soliton dynamics for the Szeg\H{o} equation on the line studied by  Pocovnicu \cite{Po}. Appendix C establishes some non degeneracy lemma allowing to implement the modulation theory in this context. Appendix D is devoted to basic commutator estimates. Appendix E contains estimates on some cut-off functions which are crucial in our energy method. Finally, Appendix F is devoted to the coercivity of our energy functional.
\vskip 0.5cm


\noindent{\bf Notations}. On $L^2(\R )$, we adopt the real scalar product
\be
\label{realcompes}
(u,v)=\Re\left(\int_\R u\overline{v}dx\right).
\ee
For $x\in \R $, we set $$\la x\ra :=\sqrt{1+x^2}.$$
 If $s>0$ and $f$ is a tempered distribution such that $\hat f$ is locally integrable near $\xi =0$, we define the tempered distribution $\vert D\vert ^sf$ by 
  $$\widehat{|D|^sf}(\xi) = |\xi|^s \widehat{f}(\xi)\ .$$ 
We define the differential operators $$\Lambda_xf:=x\pa_x f, \ \ \Lambda f:=\frac 12f+\Lambda_x f, \ \ \tL_\beta:=(1-\beta)\pa_\beta$$ and the function $$  \Phi_\beta:=y\pa_yQ_\beta+(1-\beta)\frac{\pa Q_\beta}{\pa \beta}.$$ We use the Sobolev norm $$\|f\|_{W^{k,\infty}}=\Sigma_{j=0}^k\|\pa_x^kf\|_{L^{\infty}}, \ \ k\in \Bbb N.$$

\vskip 0.5 cm
\noindent{\bf Acknowledgements}. P.G. is supported by Grant ANAE of French ANR, and partially supported by the ERC-2014-CoG 646650 SingWave. E.L. is supported by the Swiss National Science Foundation (SNF) through Grant No. 200021-149233. O.P. was supported by the NSF grant under agreement No. DMS-1128155 during the year 2013-2014 that she spent at the Institute for Advanced Study. Any opinions, findings, and conclusions or recommendations expressed in this material are those of the authors and do not necessarily reflect the views of the NSF. P.R is supported by the ERC-2014-CoG 646650 SingWave and is a junior member of the Institut Universitaire de France.   Part of this work was done while P.R was visiting the Mathematics Department at MIT, Boston, which he would like to thank for its kind hospitality. Another part was done while P.G., O.P., and P.R. were in residence at MSRI in Berkeley, California, during the Fall 2015 semester, and were supported by the NSF under Grant No. DMS-1440140.\\


\section{Existence and uniqueness of traveling waves}
\label{sectiontravellingwaves}



\subsection{The limiting Szeg\H{o} profile}


We  consider 
$$H^{\frac 12}_+(\R ):=\{ u\in H^{\frac 12}(\R ): supp (\hat u)\subset \R _+\} \ ,$$
and, for every $u\in H^{\frac 12}_+(\R )\setminus \{ 0\} $, 
$$J^+(u):=\frac { (Du,u)\Vert u\Vert _{L^2}^2}{\Vert u\Vert _{L^4}^4}\ ,\ I^+:=\inf _{u\in H^{\frac 12}_+(\R )\setminus \{ 0\} } J^+(u)\ .$$
It is known (\cite{Po}) that $I^+$ is a minimum and that its minimizers are exactly 
$$Q(x)=\frac C{x+p}\ ,\ \Im p>0 \ .$$
Moreover, those minimizers which satisfy the following Euler--Lagrange equation 
$$DQ+Q-\Pi _+(\vert Q\vert ^2Q)=0\ ,$$
are given by
\be
\label{defqplus}
Q(x)={\rm e}^{i\gamma} Q^+(x+x_0)\ ,\ Q^+(x):=\frac{2}{2x+i}\ ,\ (\gamma , x_0) \in \T \times \R .
\ee


\subsection{Existence of traveling waves}


To show the existence of nontrivial traveling waves $Q_\beta$ satisfying \eqref{travellingwave}, we consider the minimization problem
$$
J_\beta (u):=\frac{((\vert D\vert -\beta D)u,u)\Vert u\Vert _{L^2}^2}{\Vert u\Vert _{L^4}^4}\ ,\ \ I_\beta :=\inf _{u\in H^{\frac 12}(\R )\setminus\{ 0\} } J_\beta (u)\ .
$$
From \cite{KrLeRa} and a simple scaling argument, we have the following result:

\begin{proposition}[Small traveling waves]
For all $0\leq \beta<1$, the infimum $I_\beta$ is attained. Moreover, any minimizer $Q_\beta$ for $J_\beta(u)$ such that
\begin{equation}\label{idQbeta}
\Vert Q_\beta \Vert _{L^2}^2=\frac 12 \Vert Q_\beta \Vert _{L^4}^4=\frac {((\vert D\vert -\beta D)Q_\beta ,Q_\beta )}{1-\beta}=\frac{2I_\beta}{1-\beta}\ 
\end{equation}
satisfies the following equation:
\[\frac{|D|-\b D}{1-\b}Q_\b+Q_\b=|Q_\b|^2Q_\b.\]
\end{proposition}

In what follows, let $\mathcal{Q}_\beta$ denote the set of minimizers $Q_\beta$ of $J_\beta(u)$ such that \eqref{idQbeta} holds.


\begin{proposition}[Profile of $Q_\beta$] \label{prop:QbCV}
If $Q_\beta \in \mathcal Q_\beta$ and $\beta \rightarrow 1, \beta <1$, there exist $x(\beta )\in \R $ and $\gamma \in \T $ such that,
up to a subsequence,
$$Q_\beta (x-x(\beta ))\rightarrow {\rm e}^{i\gamma }Q^+(x)\ ,$$
strongly in $H^{\frac 12}(\R)$. 
More precisely, for $\beta$ sufficiently close to $1$, we have
\begin{align}\label{Q_b-Q^+}
\big\|Q_\beta (x-x(\beta ))- {\rm e}^{i\gamma }Q^+(x)\big\|_{H^{\frac 12}}\leq C (1-\b)^{1/2}\vert \log (1-\beta)\vert ^{\frac 12}.
\end{align}
\end{proposition}
\begin{proof}
First observe that, since $\vert D\vert -\beta D\ge (1-\beta )\vert D\vert$,
$$I_\beta \ge (1-\beta )I_0\ ,$$
and, by plugging $u=Q^+$ in $J_\beta $,
$$I_\beta \leq (1-\beta )I^+\ .$$
We claim that indeed,
$$\frac{I_\beta }{1-\beta} \rightarrow I^+\ .$$
Decompose
$$Q_\beta =Q_\beta ^++Q_\beta ^-\ ,\ Q_\beta ^{\pm }:=\Pi _\pm (Q_\beta )\ .$$
Then  identities (\ref{idQbeta}) read
$$\Vert Q_\beta ^+\Vert _{L^2}^2+\Vert Q_\beta ^-\Vert _{L^2}^2=\frac 12 \Vert Q_\beta ^++Q_\beta ^-\Vert _{L^4}^4=(DQ_\beta ^+,Q_\beta ^+)+\frac{1+\beta }{1-\beta} (\vert D\vert Q_\beta ^-,Q_\beta ^-)=
\frac{2I_\beta}{1-\beta}\ .$$
This implies in particular
$$\Vert Q_\beta ^-\Vert _{L^2}^2\leq 2I^+\ ,\  (\vert D\vert Q_\beta ^-,Q_\beta ^-)\le 2I_+(1-\beta )\ ,\ \Vert Q_\beta ^-\Vert _{L^4}^4\leq \frac{4I_+^2}{I_0}(1-\beta )\rightarrow 0\ .$$
We are going to improve these estimates on $Q_\b^-$, using the following identity on Fourier transforms, which is an immediate consequence 
of the equation for $Q_\b $ in Proposition \ref{prop:QbCV},
$$\widehat {Q_\b}(\xi )=\frac{1}{1+\frac{\vert \xi \vert -\b \xi}{1-\b }}\widehat {\vert Q_\b \vert ^2Q_\b }(\xi )\ .$$
In particular,
\begin{equation}\label{Qbmoins}
\widehat {Q_\b}^-(\xi )=\frac{{\bf 1}_{\{ \xi <0\}}}{1+\frac{1+\b }{1-\b }\vert \xi \vert }\widehat {\vert Q_\b \vert ^2Q_\b }(\xi )\ .
\end{equation}
From (\ref{Qbmoins}) and the Plancherel formula, we immediately get
\begin{equation}\label{QbmoinsL2}
\Vert Q_\b ^-\Vert _{L^2}^2=\frac{1}{2\pi}\int _{-\infty }^0 \frac{1}{\left (1+\frac{1+\b }{1-\b }\vert \xi \vert \right )^2}\vert \widehat {\vert Q_\b \vert ^2Q_\b }(\xi )\vert ^2\, d\xi \ \le C(1-\beta)\ ,
\end{equation}
where we used a bound on $Q_\b $ in $L^3$, which is a consequence of identities (\ref{idQbeta}) and of the 
estimate $I_\beta \leq (1-\beta )I^+\ .$
Similarly, we have
\begin{equation}\label{QbmoinsH}
(DQ_\b ^-,Q_\b ^-)=\frac{1}{2\pi}\int _{-\infty }^0 \frac{\vert \xi \vert}{\left (1+\frac{1+\b }{1-\b }\vert \xi \vert \right )^2}\vert \widehat {\vert Q_\b \vert ^2Q_\b }(\xi )\vert ^2\, d\xi \ \le C(1-\beta)^2\vert \log (1-\b)\vert \ ,
\end{equation}
because of the logarithmic divergence of the integral at $\xi =0$. This already implies
$$\Vert Q_\beta ^-\Vert _{L^4}^4\leq C(1-\b )^3\vert \log (1-\b )\vert \ .$$
Finally, using the bound on $Q_\b $ in all the $L^p$-norms with $p$ finite, we have
\begin{align*}
\Vert Q_\b ^+\Vert _{L^4}^4& = \|Q_\b-Q_\b^-\|_{L^4}^4= \|Q_\b\|_{L^4}^4-4{\rm Re}\left ( \int _\R\vert Q_\b \vert ^2\overline Q_\b Q_\b ^-\, dx\right )+O(\Vert Q_\b ^-\Vert _{L^4}^2)\\
& = \|Q_\b\|_{L^4}^4-4{\rm Re}\left ( \frac{1}{2\pi}\int _{-\infty}^0 \frac{1}{1+\frac{1+\b }{1-\b }\vert \xi \vert }\vert \widehat {\vert Q_\b \vert ^2Q_\b }(\xi )\vert ^2\, d\xi \right ) +O((1-\b )^{3/2}\vert \log (1-\b )\vert ^{1/2})\\
&= \|Q_\b\|_{L^4}^4-O((1-\b )\vert \log (1-\b )\vert ).
\end{align*}
\noindent
Therefore
\begin{align*}
I^+&\le J^+ (Q_\beta ^+)=\frac{(DQ_\b^+,Q_\b^+)\|Q_\b^+\|_{L^2}^2}{\|Q_\b^+\|_{L^4}^4}
= \frac{\big(\frac{2I_\b}{1-\b}\big)^2}{\|Q_\b\|_{L^4}^4-O((1-\b)\vert \log (1-\b )\vert)}\\
&=  \frac{I_\beta }{1-\beta}-O((1-\b)\log (1-\b))\leq I^++O((1-\b)\vert \log (1-\b )\vert)).
\end{align*}
\noindent
Summing up, we have proved
\begin{align}
0\le I_+-\frac{I_\b}{1-\b}&\lesssim (1-\b)\vert \log (1-\b )\vert\ ,\notag\\
\Big|\Vert Q_\beta ^+\Vert _{L^2}^2- 2I^+\Big|
+\Big| \Vert Q_\beta ^+\Vert _{L^4}^4-4I^+\Big|
&+\Big| (DQ_\beta ^+,Q_\beta ^+)- 2I^+\Big|\notag \\
&\lesssim (1-\b)\vert \log (1-\b)\vert\label{boundsQ_b}\\
\|Q_{\b}^-\|_{\dot{H}^{1/2}}&\lesssim (1-\b)\vert \log (1-\b)\vert ^{\frac 12}\label{dotH1/2Q-}\\
\Vert Q_\beta ^-\Vert _{H^{\frac 12}}&\lesssim (1-\b)^{\frac 12}.\notag
\end{align}
By a concentration-compactness argument on the space $H^{\frac 12}_+$ (see e.g. \cite{Po}, Prop. 5.1), this yields \eqref{Q_b-Q^+}.
\end{proof}

By a straightforward argument, we upgrade the convergence of $Q_\beta$ to any $H^s$.

\begin{proposition} \label{prop:reg}
Let $\beta_n \to 1$, $\beta_n < 1$, and  suppose that $Q_{\beta_n} \in \mathcal{Q}_{\beta_n}$ satisfies $Q_{\beta_n} \to Q^+$ in $H^{\frac 1 2}(\R)$. Then, for any $s \geq 0$, we have 
$$
\| Q_{\beta_n} \|_{H^s} \leq C_s .
$$
In particular, $\|Q_{\b_n}\|_{L^{\infty}}\leq C$ and it holds that
$$
Q_{\beta_n} \to Q^+ \ \ \mbox{in} \ \ H^s(\R) \ \ \mbox{for all $s \geq 0$}.
$$
\end{proposition}

\begin{proof} It suffices to prove the claim for integer $s \in \N$. By applying $\nabla^s$ to the equation satisfied by $Q_n := Q_{\beta_n}$, we obtain that
\begin{equation} \label{eq:boot}
 \nabla^s Q_{n} = \frac{\nabla}{\frac{|D| - \beta_n D}{1-\beta_n} +1} \nabla^{s-1} ( |Q_{n}|^2 Q_{n} ) =: A_{\beta_n} \nabla^{s-1} ( |Q_n|^2 Q_n ) .
\end{equation} 
Using the simple fact that $|\xi| - \beta \xi \geq (1-\beta) |\xi|$, we see that $\| A_{\beta_n} \|_{L^2 \to L^2} \leq C$ holds. Thus, by choosing $s=1$, we obtain the uniform bound
$$
\| \nabla Q_n \|_{L^2} \leq C \| Q_n \|_{L^6}^3 \leq C,
$$
since $\| Q_n \|_{L^6} \leq C$ because of $Q_n \to Q^+$ in $H^{\frac 1 2}$. Hence we obtain the uniform bounds $\| Q_n \|_{H^1} \leq C$ and $\|Q_n \|_{L^\infty} \leq C$ (by Sobolev embedding). Now, by induction over $s \in \N$, Leibniz' rule, and the uniform bounds $\| Q_n \|_{L^\infty} \leq C$, we find 
$$
\| Q_n \|_{H^k} \leq C_k
$$
for any $k \in \N$. By interpolation, this bound implies that $Q_n \to Q^+$ in $H^s$ for any $s \geq 0$, since $Q_n \to Q^+$ in $H^{\frac 1 2}$ by assumption.
\end{proof}


\subsection{Invertibility of the linearized operator}


In this section, we fix a solitary wave $Q_\b\in\mathcal{Q}_\b$. 
Let the linearized operator close to this solitary wave be
\be
\label{deflbeta}
\LQb \eps = \frac{|D|- \beta D}{1-\beta} \eps + \eps - 2 |Q_\beta|^2 \eps - Q_\beta^2 \overline{\eps}.
\ee

We may now invert $\matchal L_{\beta}$ and prove the continuity of the inverse in suitable weighted norms.

\begin{proposition}[Invertibility of $\LQb$]
\label{invertibility1}
There exist $\beta _*\in (0,1)$ such that for all $ \beta \in (\beta _*,1)$ and for all $Q_\b\in\mathcal{Q}_\b$, the following holds. 
There exists $C>0$ such that for all $f\in H^{\frac 12}$ we have
\be \label{estLbeta}
\Vert f\Vert _{H^{\frac 12}}\le C\left (\Vert \LQb f\Vert _{H^{-\frac 12}}+\vert (f,iQ_\beta )\vert +\vert (f,\partial _xQ_\beta )\vert \right )\ .
\ee

\noindent
Let $g\in H^{-\frac 12}$ with 
\be
\label{orhoghog}
(g,iQ_\beta )=(g,\partial _xQ_\beta )=0.
\ee 
Then, there exists a unique solution to 
\be
\label{cneocneoneonoe}
\mathcal{L}_\b f=g,\ \ (f,iQ_\beta)=(f,\pa_xQ_\beta)=0, \ \ f\in H^{\frac 12}
\ee 
and 
\be
\label{esthoin}
\|f\|_{H^{\frac 12}}\lesssim \|g\|_{H^{-\frac12}}.
\ee 

\end{proposition}

\begin{proof}[Proof of Proposition \ref{invertibility1}]
The invertibility claim 
follows easily 
once one proves \eqref{estLbeta}. 
Indeed, denote by $P_\beta $ the orthogonal projection onto $V_\beta :={\rm span} _\R (iQ_\beta ,\partial _xQ_\beta )$. Since $V_\beta \subset \ker \LQb $ from the invariance of
the equation on $Q_\beta $ by translation and phase shift, we have
$$f\in \ker \LQb \Rightarrow f-P_\beta f\in \ker \LQb \ .$$
Applying estimate (\ref{estLbeta}) to $f-P_\beta f$, we conclude that $f-P_\beta f=0$, namely $f\in V_\beta $. Therefore,
$\ker \mathcal{L}_\b=V_\beta$. The rest of the statement is just Fredholm alternative applied to the self-adjoint Fredholm operator
$\LQb$.

In the remaining we will prove \eqref{estLbeta}.

\medskip

\noindent
{\bf Step 1:} We first claim that
\be \label{estL}
\forall f\in H^{\frac 12}_+\ ,\  \Vert f\Vert _{H^{\frac 12}}\le C\left (\Vert \mathcal Lf\Vert _{H^{-\frac 12}}+\vert (f,iQ^+)\vert +\vert (f,\partial _xQ^+)\vert \right )\ ,
\ee
where $\mathcal L$ denotes the linearized operator for the equation on $Q^+$,
\begin{equation}\label{defL}
\L\eps :=D\eps +\eps -\Pi _+(2\vert Q^+\vert ^2\eps +(Q^+)^2\overline \eps)\ , \eps \in H^{\frac 12}_+\ .
\end{equation}
To prove this estimate, we closely follow Section 5 of \cite{Po2}. More precisely, we decompose $f\in H^{\frac 12}_+$ according to the orthogonal decomposition
$$L^2_+=(V\oplus iV)^\perp \oplus iV \oplus V ,\ V:={\rm span}_\R (iQ^+, \partial _xQ^+)\ ,$$
which reads
$$f=f'+f''_1+f''_2\ .$$
By translation invariance and phase shift invariance, $\L =0$ on $V$. Moreover, an exact computation yields
$$\L (Q^+)=-2(DQ^++Q^+)\ ,\ \L (DQ_+)=-2DQ^+-4Q^+\ .$$
Consequently, $\L :iV\rightarrow iV $ is one to one. Finally, $\L :(V\oplus iV)^\perp\rightarrow (V\oplus iV)^\perp$ and is coercive (as shown in \cite{Po2}),
\be
\label{ceroclimiting}
(\L f',f')\ge c\Vert f'\Vert _{H^{\frac 12}}^2\ ,
\ee
and consequently,
$$\forall f'\in H^{\frac 12}_+\cap (V\oplus iV)^\perp\ ,\ \Vert f'\Vert _{H^{\frac 12}}\le C\Vert \L f'\Vert _{H^{-\frac 12}}\ .$$
We now proceed by contradiction. Assume (\ref{estL}) fails. Then there exists a sequence $(f_n)$ of $H^{\frac 12}_+$ such that
$$\Vert f_n\Vert _{H^{\frac 12}}=1\ ,\ \Vert \L f_n\Vert _{H^{-\frac 12}}\rightarrow 0\ ,\ \vert (f_n,iQ^+)\vert +\vert (f_n,\partial _xQ^+)\vert \rightarrow 0\ .$$
Decomposing $f_n=f'_n+f''_{n1}+f''_{n2}$, we notice that the last condition exactly means $f''_{n2}\rightarrow 0$ in the plane $V$. Moreover, since 
$\Vert f''_{n1}\Vert _{L^2}\le \Vert f_n\Vert _{L^2}$, we may assume that $f''_{n1}\rightarrow f''_1$ in the plane $iV$. Since
$$\L f_n=\L f'_n+\L f''_{n1}\ ,$$
we have, for every $g\in iV$,
$$(\L f''_{n1}, g)=(\L f_n,g)\rightarrow 0\ ,$$
whence  $(\L f''_1,g)=0$, or $\L f''_1=0$, which implies $f''_1=0$ since $\L :iV\rightarrow iV$ is one to one. Finally, we conclude that
$\L f'_n\rightarrow 0$ in $H^{-\frac 12}$, which implies $f'_n\rightarrow 0$ in $H^{\frac 12}$, and finally $f_n\rightarrow 0$ in $H^{\frac 12}$, a contradiction.\\

\medskip

\noindent
\noindent{\bf Step 2:} Proof of \eqref{estLbeta}. This now follows from a standard perturbation argument. Indeed, since (\ref{esthoin}) is translation and phase-shift invariant, it is enough to 
prove it for $Q_\beta =Q_{\beta _n}\rightarrow Q^+, \, \beta _n\rightarrow 1, \, n\ge N$ sufficiently large. In the following, we write 
$$Q_n = Q_{\beta_n}.$$ 
For $f\in H^{\frac 12}$, we observe that
$$\Vert \LQbn f\Vert _{H^{-\frac 12}} ^2= \Vert \Pi _+\LQbn f\Vert _{H^{-\frac 12}}^2+\Vert \Pi _-\LQbn f\Vert _{H^{-\frac 12}}^2\ .$$
Write $f^\pm :=\Pi _\pm (f)$. We have
$$\Pi _-(\LQbn f)=\frac{1+\beta _n}{1-\beta _n}\vert D\vert f^- +f^--\Pi _-(2\vert Q_n\vert ^2f+Q_n^2\overline f)\,$$
hence, using the $L^4$ bound for $Q_n$,
$$(\Pi _-(\LQbn f), f^-)\ge \frac{1+\beta _n}{1-\beta _n}(\vert D\vert f^-, f^-)+\Vert f^-\Vert _{L^2}^2-\Or (1)\Vert f\Vert _{L^4}\Vert f^-\Vert _{L^4}\ .$$
Using the Gagliardo-Nirenberg inequality for $f^-$ and $\beta _n$ close to $1$, we can absorb $\Vert f^-\Vert _{L^4}^2$ with a large factor and get
$$(\Pi _-(\LQbn f), f^-)\ge \frac{1}{1-\beta _n}(\vert D\vert f^-, f^-)+\Vert f^-\Vert _{L^2}^2-o(1)\Vert f^+\Vert _{L^4}^2\ ,$$
and finally
$$\Vert \Pi _-(\LQbn f)\Vert _{H^{-\frac 12}}^2\ge  c\left (\frac{1}{1-\beta _n}(\vert D\vert f^-, f^-)+\Vert f^-\Vert _{L^2}^2\right )- o(1)\Vert f^+\Vert _{L^4}^2\ .$$
On the other hand,
$$\Pi _+(\LQbn f)=\Pi _+(\LQbn f^+)+\Pi _+(\LQbn f^-)=\L f^++r^++r^-\ ,$$
with
\bee
r^- &=&-\Pi _-(2\vert Q_n\vert ^2f^-+Q_n^2\overline {f^-})\ ,\ \Vert r^-\Vert _{H^{-\frac 12}}\le \Vert r^-\Vert _{L^2}\le \Or (1)\Vert f^-\Vert _{L^4}\ ,\\
r^+ &=&-\Pi _+(2(\vert Q_n\vert ^2-\vert Q^+\vert ^2)f^++(Q_n^2-(Q^+)^2)\overline {f^+})\ ,\ \Vert r^+\Vert _{H^{-\frac 12}}\le \Vert r^+\Vert _{L^2}\le \ o(1)\Vert f^+\Vert _{L^4}\ ,
\eee
where we have used uniform estimates on $Q_n$ and the fact that $Q_n\rightarrow Q^+$ in $L^p$ for every $p$. Finally, 
\begin{equation}\label{eq: estim P_+}
\Vert \Pi _+(\LQbn f)\Vert _{H^{-\frac 12}}^2\ge \Vert \L f^+\Vert _{H^{\frac 12}}^2-o(1)\Vert f^+\Vert _{L^4}^2-\Or (1)\Vert f^-\Vert _{L^4}^2\ .
\end{equation}

Summing up, we get, using again the absorption of $\Vert f^-\Vert _{L^4}$, 
$$\Vert \LQbn f\Vert _{H^{-\frac 12}}^2\ge c\left (\frac{1}{1-\beta _n}(\vert D\vert f^-, f^-)+\Vert f^-\Vert _{L^2}^2\right )+\Vert \L f^+\Vert _{H^{\frac 12}}^2-o(1)\Vert f^+\Vert _{L^4}^2\ .$$
On the other hand,
$$\vert (f,\partial _xQ_n)\vert ^2+\vert (f,iQ_n)\vert ^2\geq \vert (f^+,\partial _xQ^+)\vert ^2+\vert (f^+,iQ^+)\vert ^2-o(1)\Vert f\Vert _{L^2}^2\ .$$
Summing the last two inequalities and using estimate (\ref{estL}) for $f^+$, we absorb the term $o(1)(\Vert f^+\Vert _{L^4}^2+\Vert f\Vert _{L^2}^2)$ and obtain the desired estimate.
\end{proof}

\begin{remark}\label{estLbetaplus}
We also have the estimate 
\begin{equation}\label{estLbetaplusplus}
\Vert f\Vert _{H^{\frac 12}}\le C\left (\Vert \LQb f\Vert _{H^{-\frac 12}}+\vert (f,iQ^+)\vert +\vert (f,\partial _xQ^+)\vert \right )\ ,
\end{equation}
if $\beta $ is close enough to $1$ and $Q_\beta $ is close enough to $Q^+$. This will be useful in the next subsection  for defining a smooth branch of $Q_\beta $.
\end{remark}


\subsection{Uniqueness of traveling waves for $\beta\in (\b_\ast, 1)$ close to $1$}


\begin{proposition}\label{prop: charac Q_beta}
There exists $\beta _*\in (0,1)$ such that the following holds.
\begin{itemize}
\item For every $\beta \in (\beta _*,1)$, for every $Q_\beta ,\tilde Q_\beta $ in $\mathcal Q_\beta $, there exists $(\gamma ,y)\in \T \times \R $ such that
$$\tilde Q_\beta (x)={\rm e}^{i\gamma} Q_\beta (x-y)\ .$$
\item There exists a neighborhood $U$ of $Q^+$ in $H^{\frac 12}$ such that, for every $\beta \in (\beta _*,1)$, $\mathcal Q_\beta \cap U$ contains a unique point $Q_\beta $ satisfying
$$(Q_\beta ,iQ^+)=(Q_\beta ,\pa _xQ^+)=0\ .$$
Moreover, we have
\be \label{QbQ+}
\Vert Q_\b -Q^+\Vert _{H^{1}}=O\left (\vert 1-\b\vert ^{\frac12}\vert \log(1-\b)\vert ^{\frac12}\right )\ .
\ee
The map $\beta \in (\beta _*,1)\mapsto Q_\beta \in H^{\frac 12}$ is smooth, tends to $Q^+$ as $\beta $ tends to $1$, 
and its derivative is uniquely determined by
\bea \label{eqQbdot}
\begin{cases}
\LQb (\pa _\beta Q_\beta )=\frac{2}{1-\beta ^2}(Q_\b^- -\Pi_-(\vert Q_\b \vert^2Q_\b))\cr
\ (\pa _\beta Q_\beta ,iQ^+)=(\pa _\beta Q_\beta ,\pa_xQ^+)=0
\end{cases}
\eea
\end{itemize}
\end{proposition}

\begin{proof}
Let us prove the first item. We may assume that $Q_\beta $ and $\tilde Q_\beta $ tend to  $Q^+$ as $\beta $ tends to $1$. For $(\gamma ,y)\in \T \times \R $, we then define
$$\eps (x,\gamma ,y,\beta ):= \tilde Q_\beta (x)-{\rm e}^{i\gamma}Q_\beta (x-y)\ ,$$
and 
$$f(\gamma ,y,\beta ):= (\eps (.,\gamma ,y,\beta ), i\tilde Q_\beta )\ ,\ g(\gamma ,y,\beta ):= (\eps (.,\gamma ,y,\beta ), \pa _x\tilde Q_\beta )\ .$$
These two functions are smooth in $(\gamma ,y)$ and their Jacobian matrix at $(\gamma,y)=(0,0)$ is close to 
\bee \begin{pmatrix}
(-iQ^+,iQ^+)& (\pa _xQ^+, iQ^+)\cr
(-iQ^+,\pa _xQ^+)& (\pa _xQ_+, \pa _xQ^+)
\end{pmatrix}
=
\begin{pmatrix}
-2\pi & 2\pi \cr
-2\pi & 4\pi 
\end{pmatrix}
\eee
therefore it is uniformly invertible. Moreover, as $\beta $ goes to $1$, $f(0,0,\beta )$ and $g(0,0,\beta )$ tend to $0$.
By the implicit function theorem, we conclude that there exist functions $\gamma (\beta ), y(\beta )$ with values near $(0,0)$ such that
$$f(\gamma (\beta ),y(\beta ),\beta )=g(\gamma (\beta ),y(\beta ), \beta )=0\ .$$
Then, coming back to the equations satisfied by $Q_\beta $ and $\tilde Q_\beta $,
we infer  that  $\eps (x,\beta ):=\eps (x,\gamma (\beta ) ,y(\beta ),\beta )$ satisfies
$$\Vert \mathcal L_{\tilde Q_\beta} \eps (\, .\, ,\beta )\Vert _{H^{-\frac 12}}\leq Co(1) \Vert \eps (\, .\, ,\beta )\Vert _{H^{\frac 12}}\ ,$$
and, using estimate (\ref{esthoin}), we conclude that $\eps (x,\beta )=0$.
\vskip 0.25cm
Let us come to the second item. Select a family $(Q_\beta ^0)$, with $Q_\beta ^0\in \mathcal Q_\beta $,  which tends to $Q^+$ as $\beta $ tends to $1$. Applying the implicit function theorem as before
to the functions 
$$\tilde f (\gamma, y, \beta ):=({\rm e}^{i\gamma}Q_\beta ^0(.-y), iQ^+)\ ,\ \tilde g(\gamma ,y,\beta ):=({\rm e}^{i\gamma}Q_\beta ^0(.-y), \pa _xQ^+)\ ,$$
we find functions $\tilde \gamma (\beta ), \tilde y(\beta )$ valued near $(0,0)$ which cancel $\tilde f, \tilde g$. This provides the existence of $Q_\beta $.
The uniqueness comes from Remark \ref{estLbetaplus}. Furthermore, as a consequence of \eqref{Q_b-Q^+}, we get
$$\Vert Q_\b -Q^+\Vert _{H^{\frac12}}=O\left (\vert 1-\b\vert ^{\frac12}\vert \log(1-\b)\vert ^{\frac12}\right )\ .$$
Coming back to the equation satisfied by $Q_\b$,
$$Q_\b =\left (\frac{\vert D\vert -\b D}{1-\b}+1\right )^{-1}(\vert Q_\b \vert ^2Q_\b)\ ,$$
and expanding in the $L^2$-norm
$$\vert Q_\b \vert ^2Q_\b=\vert Q^+\vert ^2Q^++O((1-\b)^{\frac 12}\vert \log(1-\b)\vert ^{\frac12})\ ,$$
we infer, in the $L^2$ norm,
$$DQ_\b=D(D+1)^{-1}\Pi ^+(\vert Q^+\vert ^2Q^+)+D\left (\frac{(1+\b )\vert D\vert}{1-\b}+1\right )^{-1}\Pi ^-(\vert Q^+\vert ^2Q^+)+O((1-\b)^{\frac 12}\vert \log(1-\b)\vert ^{\frac12})\ ,$$
and finally
$$DQ_\b=DQ^++O((1-\b)^{\frac 12}\vert \log(1-\b)\vert ^{\frac12})\ ,$$
in the $L^2$ norm, which completes the proof of \eqref{QbQ+}.
\vskip 0.25cm
Using again the equation satisfied by $Q_\beta $ and the estimate from Remark \ref{estLbetaplus},
it is then straightforward to prove that the map $\beta \mapsto Q_\beta $ is smooth on $(\beta _*,1)$ and that its derivative satisfies
$$\LQb (\pa _\beta Q_\beta )+\frac {(\vert D\vert -D)Q_\beta }{(1-\beta )^2}=0\ ,\ (\pa _\beta Q_\beta ,iQ^+)=(\pa_\beta Q_\beta ,\pa_xQ^+)=0.
$$
Notice that $(\vert D\vert -D)Q_\beta=-2DQ_\b^-$. Projecting the equation for $Q_\b$
onto negative Fourier modes, we get
$$\frac{2DQ_\b^-}{1-\b}=\frac{2}{1+\b}(Q_\b^- -\Pi _-(\vert Q_\b\vert^2Q_\b))\ ,$$
which, plugged into the equation on $\pa _\b Q_\b $, leads to (\ref{eqQbdot}). 
\end{proof}


\section{Properties of $Q_\beta$}
\label{sectionqbeta}


We collect in this section information on $Q_\beta$ which will be essential for the construction of the two-bubble approximate solutions.


\subsection{Weighted norms and Fourier multipliers}


For every function $f$ on $\R $ and $\b \in (\b _*,1)$, we define the following weighted norm,

$$\norm{f}_\b:=\sup _{x\in \R }\la x\ra (1+(1-\beta )\vert x\vert )\vert f(x)\vert \ .$$

The next lemma will be crucial in all our estimates.

\begin{lemma}\label{lemma:iteration}
Let $\{ m_\b\} _{\b _*<\b <1}$ be a family of functions on $\R $ such that
\begin{align}
\sup _{\b }\norm {m_\b }_ {L^2}&\le M_0\ ,\label{mb1}\\
\vert xm_\b (x)\vert &\le \frac{M_0}{1+(1-\b )\vert x\vert }\ ,\label{mb2}
\end{align}
for some $M_0>0$.  Assume $\{ a_\b , b_\b \}_{\b _*<\b <1}$ is  bounded in $L^\infty $ and is tight in  $L^2$, namely
$$\sup_{\b _*<\b <1}\int _{\vert x\vert >R}[\vert a_\b(x)\vert ^2+\vert b_\b (x)\vert ^2]\, dx \td_R,\infty 0\ .$$
 Then there exists a constant $A>0$
independent of $\beta $ such that, if $f,h\in L^2$ satisfy
$$f=m_\b *(a_\b f+b_\b \overline f)+h\ ,$$
the following estimate holds,
$$\norm{f}_\b\le  A[(\norm{a_\b }_{L^\infty}+\norm{a_\b }_{L^2}+\norm{b_\b }_{L^\infty}+\norm{b_\b }_{L^2})\norm{f}_{L^2}+\norm{h}_\b]\ .$$
\end{lemma}

\begin{proof}
First of all, we have trivially
$$\norm{f}_{L^\infty}\le \norm{m_\b}_{L^2}(\norm{a_\b }_{L^\infty}+\norm{b_\b }_{L^\infty})\norm{f}_{L^2}+\norm{h}_{L^\infty}\ ,$$
hence it is enough to estimate $\vert f(x)\vert$ for $x$ large enough. Let $R_0>0$ such that
$$\sup _\b \norm{m_\b}_{L^2}\left [\left (\int _{\vert y\vert \ge R_0/2}\vert a_\b (y)\vert ^2\, dy\right )^{1/2}+\left (\int _{\vert y\vert \ge R_0/2}\vert b_\b (y)\vert ^2)\, dy\right )^{1/2}\right ]\le \frac 18\ .$$
For every $R>0$, we set 
$$M(R):=\sup _{\vert x\vert \ge R}\vert f(x)\vert \ .$$
For $\vert x\vert \ge R$, and $R\ge R_0$, we write
\bee
\vert m_\b *(a_\b f+b_\b \overline f)(x)\vert &\le &\left \vert \int _{\vert y\vert \le \frac R2}m_\b (x-y)(a_\b (y)f(y)+b_\b (y)\overline {f(y)})\, dy\right \vert \\
& +&\left \vert \int _{\vert y\vert \ge \frac R2}m_\b (x-y)(a_\b (y)f(y)+b_\b (y)\overline {f(y)}) dy\right \vert  \\
&\le & \frac C{R(1+(1-\b)R)} (\norm{a_\b}_{L^2}+\norm{b_\b}_{L^2})\norm{f}_{L^2}+\frac 18 M\left (\frac R2\right )\ .
\eee
This implies, for every $R\ge R_0$, 
$$M(R)\le \frac{C(\norm{a_\b}_{L^2}+\norm{b_\b}_{L^2})\norm{f}_{L^2}+\norm{h}_\b}{R(1+(1-\b)R)}+\frac 18 M\left (\frac R2\right )\ .$$
Applying this to $R=2^n$ for $n\ge n_0$, we obtain
$$M(2^n)\le K2^{-n}(1+(1-\b)2^n)^{-1}+\frac 18 M(2^{n-1})\ ,\ K:=C(\norm{a_\b}_{L^2}+\norm{b_\b}_{L^2})\norm{f}_{L^2}+\norm{h}_\b\ .$$
Iterating, we get
\bee M(2^n)&\le & K\sum _{p=0}^{n-n_0}2^{-(n-p)}(1+(1-\b)2^{n-p})^{-1}\left (\frac 18\right )^p+\left (\frac 18\right )^{n-n_0+1}M(2^{n_0-1})\\
&\le & K2^{-n}(1+(1-\b)2^n)^{-1}\sum _{p=0}^{n-n_0}2^{-p}+\left (\frac 18\right )^{n-n_0+1}M(2^{n_0-1})\\
&\le & (2K+4^{n_0}M(2^{n_0-1}))\, 2^{-n}(1+(1-\b)2^n)^{-1}\ .
\eee
Since $\vert x\vert \sim 2^n$ for $2^n\le \vert x\vert \le 2^{n+1}$, this completes the proof of the lemma.
\end{proof}

We now introduce an important class of  families $\{ m_\b\} _{\b _*<\b<1}$ satisfying estimates (\ref{mb1}), (\ref{mb2}). Denote by $\mathcal M$  the class of families $\{ \mu _\b\} _{\b _*<\b<1}$ such that the Fourier transform is given by
\be \label{mub}
\hat{\mu}_\b (\xi )=A_\b \left (f_+(\xi )\, {\bf 1}_{\xi >0}+f_-\left (-\frac{1+\b }{1-\b}\xi \right )\, {\bf 1}_{\xi <0}   \right )\ ,
\ee
where $f_\pm \in C^\infty ([0,+\infty ))$ satisfy the following requirements,
$$\forall j\ge 0, \forall \zeta \in (0,+\infty), \vert f^{(j)}_\pm(\zeta )\vert \le C_j(1+ \zeta )^{-j-1}\ ,\ f_+(0)=f_-(0)\ ,$$
and where $\b \mapsto A_\b $ is smooth on $(\b _*,1)$ and  is  bounded with bounded derivatives of any order. Indeed, the $L^2$-estimate (\ref{mb1}) on $\mu _\b $ is provided by 
$$\vert f_\pm(\zeta )\vert \le C_0(1+\zeta  )^{-1}\ ,$$
while (\ref{mb2}) comes from 
$$x\mu_\b (x)=A_\b \left (F_+(x)-F_-\left (-\frac{1-\b}{1+\b}x\right )\right )\ ,\ F_\pm(y):=\int _0^{+\infty}if'_\pm (\zeta )\, e^{iy\zeta}\, \frac{d\zeta}{2\pi}=O\left ( \frac{1}{1+\vert y\vert }\right )\ . $$
The advantage of the class $\mathcal M$ is that it is stable through various important operations. The first one is of course the product of convolution, which corresponds to the product of functions $\b \mapsto A_\b$ and $\zeta \mapsto f_\pm (\zeta)$. The second one is the operator $x\pa _x+1$, which corresponds to replacing $f_\pm $ by $-\zeta f'_\pm $. Finally, if $\{ \mu _\b\} _{\b _*<\b<1}$ belongs to class $\mathcal M$, then
\bea
(1-\b)\pa _\b\hat{\mu}_\b (\xi )&=&(1-\b)A'_\b\left (f_+(\xi )\, {\bf 1}_{\xi >0}+f_-\left (-\frac{1+\b }{1-\b}\xi \right )\, {\bf 1}_{\xi <0}   \right )\\
&&+\frac{2A_\b}{1+\b}\, g_-\left (-\frac{1+\b}{1-\b}\xi  \right )\, {\bf 1}_{\xi <0} \ ,
\eea 
where $g_-(\zeta ):=\zeta f_-'(\zeta)$. Hence the family
$$\{ (1-\b)\pa _\b \mu _\b \} _{\b _*<\b<1}$$
is a sum of elements of class $\mathcal M$.

A typical example of a family in class $\mathcal M$ is 
$$m_\b=\mathcal F^{-1}\left (\frac{1}{1+\frac{\vert \xi\vert -\b \xi}{1-\b}}\right )\ ,$$
which corresponds to
$$A_\b =1\ ,\ f_+(\zeta )=f_-(\zeta )=(1+\zeta)^{-1}\ .$$
The above considerations lead to the following result, which will be of constant use in the sequel.
\begin{lemma}\label{multipliers}
All the multipliers $$m_{\b ,p,q}:=(x\pa _x)^p((1-\b)\pa _\b)^qm_\b ,\ p,q\ge 0\ ,$$ and any convolution products between them satisfy properties (\ref{mb1}) and (\ref{mb2}). 
\end{lemma}

We complete this subsection with three auxiliary results. The first one is the crucial estimate for $\LQb$ regarding the weighted norm $\Vert \ \Vert _\b$.
\begin{proposition}[Continuity of $\LQb^{-1}$ in weighted norms]
\label{invertibility}
 Let $\beta \in (\beta _*,1)$ and $g\in H^{-\frac 12}$ with 
$$
(g,iQ_\beta )=(g,\partial _xQ_\beta )=0.
$$ 
Then any  solution $f$ to 
$$
\LQb f=g,\  f\in H^{\frac 12}
$$  satisfies:
\be
\label{linfinitycontrolbisbistierce}
\|f\|_\b\le C(\Vert g\Vert _{H^{-\frac 12}}+\vert (f,iQ_\b)\vert +\vert (f,\pa _xQ_\b)\vert +\Vert m_\b *g\Vert _\b)
\ee
where $$m_\b=\mathcal F^{-1}\left (\frac{1}{1+\frac{\vert \xi\vert -\b \xi}{1-\b}}\right )\ .$$
\end{proposition}
\begin{proof}
The equation reads
$$f=m_\b*g +m_\b *(2\vert Q_\b\vert ^2f+Q_\b ^2\overline f)\ ,$$
so we are in position to apply Lemma \ref{lemma:iteration} with $a_\b =2\vert Q_\b \vert ^2$, $b_\b =Q_\b ^2$, $h=m_\b *g$. In view of the $L^\infty $-estimates and the tightness property for the family $Q_\b $ obtained from Proposition \ref{prop: charac Q_beta}, we infer
$$\Vert f\Vert _\b\le B(\Vert f\Vert _{L^2}+\Vert m_\b*g\Vert _\b)\ .$$
On the other hand, by Proposition \ref{invertibility1}, 
$$\Vert f\Vert _{L^2}\le \Vert f\Vert _{H^{\frac 12}}\lesssim \Vert g\Vert _{H^{-\frac12}}+\vert (f,iQ_\b)\vert +\vert (f,\pa _xQ_\b)\vert  .$$
This completes the proof.
\end{proof}
\begin{remark}\label{invertibility+}
 In view of Remark \ref{estLbetaplus}, one can replace $$\vert (f,iQ_\b)\vert +\vert (f,\pa _xQ_\b)\vert $$ by $$\vert (f,iQ_+)\vert +\vert (f,\pa _xQ_+)\vert $$
in the right hand side of the estimate \eqref{linfinitycontrolbisbistierce}.
\end{remark}

The second result is the following lemma.

\begin{lemma}\label{mub123}
Assume $\mu _\b$ satisfies (\ref{mb1}) and (\ref{mb2}). Then
$$\Vert \mu _\b *(h_1h_2)\Vert _\b\lesssim \Vert h_1\Vert _\b\, \Vert h_2\Vert _\b \ .$$
\end{lemma}
\begin{proof}
First of all, the $L^\infty$-bound is an easy consequence of $L^2*L^2\subset L^\infty $, so we may assume $\vert x\vert \ge 1$. Then we split
\bee
 \mu _\b *(h_1h_2)(x)&=&\int _{\vert y\vert <\frac{\vert x\vert}{2}} \mu _\b (x-y)h_1(y)h_2(y)\, dy +\\
 &&\int _{\vert y\vert \ge\frac{\vert x\vert}{2}}  \mu _\b (x-y)h_1(y)h_2(y)\, dy\ \\
&=&O(\vert x\vert ^{-1} (1+(1-\b)\vert x\vert )^{-1}  )\Vert h_1h_2\Vert _{L^1}\\
&&+\Vert \mu _\b\Vert _{L^2}\, \Vert h_1 h_2 \Vert _{L^2(\vert y\vert >\vert x\vert /2)}\\
&\le & O(\vert x\vert ^{-1} (1+(1-\b)\vert x\vert )^{-1}  )\Vert h_1\Vert _{L^2} \Vert h_2\Vert _{L^2}\\
&&+O(\vert x\vert ^{-3/2} (1+(1-\b)\vert x\vert )^{-2}  )\Vert h_1\Vert _\b \Vert h_2\Vert _\b \ , 
\eee
and the lemma follows.
\end{proof}

The third result concerns the $L^p$ norm of elements of class $\mathcal M$.

\begin{lemma}\label{mubLp}
If $\{ \mu _\b\} _{\b _*<\b<1}$ belongs to class $\mathcal M$, then there exists $C>0$ such that, for every $p\in (1,\infty )$,
for every $\b \in (\b _*,1)$, 
$$\Vert \mu _\b \Vert _{L^p}\le C\, \max \left (\frac{1}{p-1},p\right )\ .$$
\end{lemma}
\begin{proof}
From \eqref{mub}, the following holds,
$$\mu_\b (x)=A_\b \left (\mu _+(x)+\frac{1-\b}{1+\b}\mu _-\left ( -\frac{1-\b}{1+\b}x \right )   \right )\ ,\ \mu _\pm :=\mathcal F^{-1}(f_\pm)\ .$$
It is therefore sufficient to prove that, for every $f\in C^1(\R _+)$ such that 
$$\vert f(\xi )\vert \le \frac{C}{1+\xi}\ ,\ \vert f'(\xi )\vert \le \frac{C}{(1+\xi)^2}\ ,$$
the inverse Fourier transform $\mu =\mathcal F^{-1}(f)$ satisfies
$$\forall p\in (1,\infty)\ ,\ \Vert \mu\Vert _{L^p(\R)}\le \tilde C\, \max \left (\frac{1}{p-1},p\right )\ .$$
First, an integration by part leads to
$$x\mu (x)=\frac{if(0)}{2\pi}+i\int _0^\infty e^{ix\xi}f'(\xi)\, \frac{d\xi }{2\pi}\ ,$$
which provides the bound
$$\vert \mu (x)\vert \lesssim \frac{1}{|x|}\ .$$
Secondly, if $x$ is close to $0$, introducing a cut-off function $\varphi $ such that $\varphi =1$ near $0$,
and writing
$$\mu (x)=\int _0^\infty e^{ix\xi}\varphi (x\xi)f(\xi)\, \frac{d\xi }{2\pi}+\int _0^\infty e^{ix\xi}(1-\varphi (x\xi))f(\xi)\, \frac{d\xi }{2\pi}\ :=\mu _<(x)+\mu_>(x)\ ,$$
we observe that $$\vert \mu _<(x)\vert \lesssim \log \left (\frac 1{\vert x\vert}\right )\ ,$$
while 
$$\vert x\mu _>(x)\vert \lesssim \int _{\R}\left \vert \frac{d}{d\xi}[\varphi (x\xi)f(\xi)]\right \vert \, d\xi \lesssim \vert x\vert \ .$$
We infer that, near $x=0$, 
$$\vert \mu (x)\vert \lesssim  \log \left (\frac 1{\vert x\vert}\right )\ .$$
Consequently,
\bee \Vert \mu \Vert _{L^p}^p &\lesssim &\int _{\vert x\vert \le 1}\left (\log \left (\frac 1{\vert x\vert}\right )\right )^p\, dx+\int _{\vert x\vert \ge 1}\frac{dx}{\vert x\vert ^p}\\
&\lesssim&p^p +\frac{1}{p-1}\ .
\eee
This completes the proof.
\end{proof}

\subsection{Weighted estimates on $Q_\b$}

\begin{proposition}\label{boundQb}
For every $p,q\in \N $, there exists $C_{p,q}$ such that 
$$\forall \b \in (\b _*,1)\ , \ \norm{(x\pa _x)^p ((1-\b)\pa _\b)^qQ_\b }_\b\leq C_{p,q}\ .$$
\end{proposition}
\begin{proof} First assume $p=q=0$. We use  the identity
$$Q_\b =m_\b *(Q_\b \vert Q_\b \vert ^2)\ ,$$
and Lemma \ref{lemma:iteration} with 
$$m_\b=\mathcal F^{-1}\left (\frac{1}{1+\frac{\vert \xi\vert -\b \xi}{1-\b}}\right )\ ,\ a_\b =\vert Q_\b \vert ^2\ ,\   b_\b =0\ ,\ h=0\ , $$
and we easily obtain
$$\norm{Q_\b}_\b\le C_{0,0}\ .$$
Now let us prove the estimate for $p=0$ and every $q$. Set $\tL_\b:=(1-\b)\pa _\b$.  From equation (\ref{eqQbdot}), we have
\bee
\begin{cases}
\LQb(\tL_\b Q_\b)=\frac{2}{1+\beta }(Q_\b^- -\Pi_-(\vert Q_\b \vert^2Q_\b))\cr
(\tL_\b Q_\b ,iQ^+)=(\tL_\b Q_\b ,\pa_xQ^+)=0
\end{cases}
\eee
From a priori $H^s$ estimates on $Q_\b$ and inequality (\ref{linfinitycontrolbisbistierce}) --- in fact Remark \ref{invertibility+}--- we infer
$$\norm{\tL_\b Q_\b}_\b\le C \left (1+\norm{m_\b*(Q_\b^- -\Pi_-(\vert Q_\b \vert^2Q_\b))}_\b\right )\ .$$
From the equation \eqref{travellingwave} of $Q_\b$, we have
$$Q_\b=m_\b*(\vert Q_\b\vert ^2Q_\b)\ ,$$
so that, with $m_\b^-:=\Pi_-m_\b$,
$$m_\b*(Q_\b^- -\Pi_-(\vert Q_\b \vert^2Q_\b))=(m_\b^-*m_\b^--m_\b^-)*(\vert Q_\b\vert ^2Q_\b)\ .$$
Notice that
$$\mathcal F(m_\b^-*m_\b^--m_\b^-)(\xi)={\bf 1}_{\xi <0}\, \frac{\frac{1+\b}{1-\b}\xi}{(1-\frac{1+\b}{1-\b}\xi)^2}\ ,$$
so that $\{ m_\b^-*m_\b^--m_\b^-\} _{\b_*<\b<1}$ belongs to class $\mathcal M$, and therefore Lemma \ref{mub123} yields 
$$\norm{\tL_\b Q_\b}_\b\leq C_{0,1}\ .$$
For further reference, we are going to estimate $\norm{D\pa _\b Q_\b^-}_{L^2}$. Projecting the equation of $\tL_\b Q_\b$ onto negative Fourier modes, 
we get
$$(1+\b)D\pa _\b Q_\b^-=\tL_\b Q_\b^- - \Pi _-(2\vert Q_\b\vert ^2\tL _\b Q_\b+Q_\b^2\overline{\tL _\b Q_\b})-\frac{2}{1+\b}(Q_\b^--\Pi_-(|Q_\b|^2Q_\b))\ .$$
 From the estimate on $\tL_\b Q_\b$ we just established, we infer
 $$\norm{D\pa _\b Q_\b^-}_{L^2}\le C'_1\ .$$
 Let us prove by induction on $q\ge 1$ that 
 \be \label{inductionq}
 \norm{\tL_\b^qQ_\b}_\b \le C_{0,q}\ ,\ \norm{D\pa _\b\tL_\b^{q-1}Q_\b}_{L^2}\le C'_q\ ,
 \ee
 where $C_{0,q}$ and $C'_q$ are independent of $\b $. Notice that we just proved the case $q=1$. 
In order to deal with higher orders, we observe that, for every function $f_\b$ depending smoothly on $\b$, 
$$\LQb(\tL_\b f_\b)=\tL_\b(\LQb f_\b)+\frac{2Df_\b^-}{1-\b}+4{\rm Re}(\overline Q_\b\tL_\b Q_\b)f_\b+2Q_\b\tL_\b Q_\b\overline f_\b\ .$$
From this identity and the formula for $\tL_\b Q_\b$, we infer that
$\LQb((\tL_\b)^{q+1}Q_\b)$ is a linear combination of terms of the following form.
\begin{itemize}
\item $D\pa _\b(\tL_\b)^rQ_\b$, with $r\le q-1$.
\item $A_\b (\tL_\b )^rQ_\b^-$ for $r\le q$ and $A_\b $ depends smoothly on $\b$, is bounded as well as its derivatives.
\item $B_\b \Pi _-\left ((\tL_\b)^aQ_\b(\tL_\b)^bQ_\b\overline {(\tL_\b)^cQ_\b}\right )$, where $a+b+c\le q$, and $B_\b $ depends smoothly on $\b$, is bounded as well as its derivatives.
\item $C_\b (\tL_\b)^aQ_\b(\tL_\b)^bQ_\b\overline {(\tL_\b)^cQ_\b}$, where $a+b+c\le q+1$, $a,b,c\leq q$, and $C_\b $ depends smoothly on $\b$, is bounded as well as its derivatives.
\end{itemize}
Since all these terms are bounded in $L^2$ by the induction assumption, and since $((\tL_\b)^{q+1} Q_\b ,iQ^+)=((\tL_\b)^{q+1}Q_\b ,\pa_xQ^+)=0$,
we infer from inequality (\ref{linfinitycontrolbisbistierce}) --- in fact Remark \ref{invertibility+}--- that 
$\norm{(\tL_\b)^{q+1}Q_\b}_{L^2}$ is bounded independently of $\b $.  \\
Now let us prove (\ref{inductionq}) at step $q+1$. Applying $(\tL_\b)^{q+1}$ to
$$Q_\b=m_\b*(\vert Q_\b\vert ^2Q_\b)\ ,$$
we obtain
$$(\tL_\b)^{q+1}Q_\b=m_\b*\left ( 2\vert Q_\b\vert^2 (\tL_\b)^{q+1}Q_\b+Q_\b^2\overline{(\tL_\b)^{q+1}Q_\b}  \right )+R_{\b,q}\ ,$$
where $R_{\b,q}$ is a finite sum of terms of the form $$(\tL_\b)^am_\b*\left [(\tL_\b)^bQ_\b(\tL_\b)^cQ_\b\overline{(\tL_\b)^dQ_\b}\right ]\ ,\ a+b+c+d=q+1\ ,\  \max(b,c,d)\le q\ .$$
  Using Lemma \ref{lemma:iteration}, the $L^2$ estimate on $(\tL_\b)^{q+1}Q_\b$, and Lemmas \ref{multipliers} and \ref{mub123}, as well as the induction assumption, we infer
$$\norm{(\tL_\b)^{q+1}Q_\b}_\b \le C_{0,q+1}\ .$$
Furthermore,
$$D\pa _\b (\tL_\b)^qQ_\b^-=\frac{Dm_\b}{1-\b}*\left ( 2\vert Q_\b\vert^2 (\tL_\b)^{q+1}Q_\b+Q_\b^2\overline{(\tL_\b)^{q+1}Q_\b}  \right )+(1-\b)^{-1}DR_{\b,q}\ ,$$
where $(1-\b)^{-1}DR_{\b,q}$ is a finite sum of terms of the form 
$$(1-\b)^{-1}D(\tL_\b)^am_\b*\left [(\tL_\b)^bQ_\b(\tL_\b)^cQ_\b\overline{(\tL_\b)^dQ_\b}\right ]\ ,\  a+b+c+d=q+1\ ,\  \max(b,c,d)\le q\ .$$
 It remains to observe that, if $\{ \mu _\b\} $ is an element of class $\mathcal M$, then
$$(1-\b)^{-1}\widehat{D\mu_\b ^-}(\xi )={\bf 1}_{\xi <0}\frac{i\xi}{1-\b}f_-\left (-\frac{1+\b}{1-\b}\xi \right )$$
is uniformly bounded in $L^\infty$, therefore the convolution with $(1-\b)^{-1}D\mu _\b^-$ is uniformly bounded on $L^2$. This proves the $L^2$-estimate on $D\pa _\b (\tL_\b)^qQ_\b^-$, and completes the proof of (\ref{inductionq}) at step $q+1$. 

Finally, we prove the estimate for every $p,q$, by induction on $p+q$. Assume that
$$\norm{\La _x^r(\tL _\b)^sQ_\b}_\b \le C_{r,s}\ ,\ r+s\le n\ ,$$
and let us prove the inequality for $r+s=n+1$. Since the case $r=0$ is already known, we may assume $r=p+1, s=q$ with $p+q=n$.
Recall that $\La _x:=x\pa _x$. We use the identity
\be \label{Lambdastar}
\Lambda _x (f*g)=\Lambda _x(f)*g+f*\Lambda _x(g)+f*g=(\Lambda _x+I)f*g+f*\Lambda _x(g)\ 
\ee
to obtain
$$\La _x^p(\tL_\b)^q Q_\b=m_\b*\left ( 2\vert Q_\b\vert^2 \La _x^p(\tL_\b)^qQ_\b+Q_\b^2\overline{\La _x^p(\tL_\b)^qQ_\b}  \right )+R_{\b,p,q}\ ,$$
where $R_{\b,p,q}$ is a finite sum of terms of the  form
\bee
&&(\La _x+I)^{a'}(\tL_\b)^am_\b*\left [\La _x^{b'}(\tL_\b)^bQ_\b \La _x^{c'}(\tL_\b)^cQ_\b \overline{\La _x^{d'}(\tL_\b)^dQ_\b}\right ]\ ,\\
&& a+b+c+d=q\ ,\ a'+b'+c'+d'=p\ ,\ \  \max(b,c,d)\le q-1\ {\rm or}\  \max(b',c',d')\le p-1\ .
\eee
Let us first prove that $\La _x^{p+1}(\tL_\b)^q Q_\b$ is uniformly bounded in $L^2$. We apply $\La _x $ to the above formula giving $\La _x^p(\tL_\b)^q Q_\b$. We expand $\La _xR_{\b,p,q}$ using again identity (\ref{Lambdastar}), and we get, by the induction assumption, that $\La _xR_{\b,p,q}$
is uniformly bounded in $L^2$. As for the term 
$$m_\b*\left ( 2\vert Q_\b\vert^2 \La _x^p(\tL_\b)^qQ_\b+Q_\b^2\overline{\La _x^p(\tL_\b)^qQ_\b}  \right )\ ,$$
we write
$$x\pa _x[m_\b*f]=x\pa _xm_\b *f+\pa _xm_\b*(xf)\ .$$
From the induction assumption, we easily get that 
$$\norm{x\pa _xm_\b *\left ( 2\vert Q_\b\vert^2 \La _x^p(\tL_\b)^qQ_\b+Q_\b^2\overline{\La _x^p(\tL_\b)^qQ_\b}  \right )}_{L^2}\le A_{p,q}\ .$$
On the other hand, since 
$$\widehat{\pa _xm_\b}(\xi)=\frac{i\xi}{1+\frac{\vert \xi\vert -\b\xi}{1-\b}}$$
is uniformly bounded, the uniform bounds on
$$\norm{xQ_\b}_{L^\infty}, \norm{Q_\b}_{L^\infty}\ ,\ \norm{\La _x^p(\tL_\b)^qQ_\b}_{L^2}$$
imply
$$\norm{\pa _xm_\b *\left ( 2x\vert Q_\b\vert^2 \La _x^p(\tL_\b)^qQ_\b+xQ_\b^2\overline{\La _x^p(\tL_\b)^qQ_\b}  \right )}_{L^2}\le B_{p,q}\ .$$
Summing up, we have proved that $\La _x^{p+1}(\tL_\b)^q Q_\b$ is uniformly bounded in $L^2$. It remains to prove a uniform bound of the weighted norm. But this is now a consequence of the formula
$$\La _x^{p+1}(\tL_\b)^q Q_\b=m_\b*\left ( 2\vert Q_\b\vert^2 \La _x^{p+1}(\tL_\b)^qQ_\b+Q_\b^2\overline{\La _x^{p+1}(\tL_\b)^qQ_\b}  \right )+R_{\b,p+1,q}\ ,$$
of Lemmas \ref{lemma:iteration}, \ref{multipliers}, \ref{mub123} and of the induction assumption. The proof is complete.
\end{proof}


\subsection{Inverting $\mathcal L_\b$ with a special right hand side.}

In this section, we consider the equation

\be
\label{nceionceoneo}
\matchal L_{\beta}(i\rho_\b)=i\pa_{y}Q_{\beta}\ ,\ 
(i\rho_\b,iQ_{\beta})=(i\rho_\b,\pa_{y}Q_{\beta})=0\ .
\ee

\noindent
Since $i\pa_{y}Q_{\beta}$ is orthogonal to $iQ_\b $ and $\pa_y Q_\b $, this equation has a unique solution
given by Proposition \ref{invertibility}. The next lemma describes this solution as $\b $ tends to $1$.

\begin{lemma}
\label{lemmacalculrho}
Let $i\rho_\b$ be defined by \eqref{nceionceoneo}. 
Then,
\be
\label{irho}
i\rho_\b= Q_\b+\frac{i}{2}\pa_y Q_\b + O((1-\b)^{\frac 12}\vert \log (1-\b)\vert ^{\frac 12}) \text{ in } H^{\frac 12}(\R).
\ee

\end{lemma}

\begin{proof}
A computation based on the equation satisfied by $Q_\b$ shows that
\begin{equation*}
\mathcal{L}_{\b}\Big(Q_\b+\frac{i}{2}\pa_y Q_\b\Big)=-2|Q_\b|^2Q_\b+iQ_\b^2\overline{\pa_y Q_\b}.
\end{equation*}

\noindent
On the other hand, we have
\begin{equation*}
\mathcal{L}_{\b}(\overline{Q_\b})=\overline{\frac{2\b DQ_\b}{1-\b}}-|Q_\b|^2\overline{Q_\b}-Q_\b^3.
\end{equation*}

\noindent
From the last two equations, we conclude that
\begin{align*}
\mathcal{L}_{\b}\Big(Q_\b+\frac{i}{2}\pa_y Q_\b+\frac{1}{2}(1-\b)\overline{Q_\b}\Big)
&=-2|Q_\b|^2Q_\b+iQ_\b^2\overline{\pa_y Q_\b}
+i\b \pa_y \overline{Q_\b}\\
&-\frac{1}{2}(1-\b)|Q_\b|^2\overline{Q_\b}-\frac{1}{2}(1-\b)Q_\b^3=:RHS
\end{align*}

\noindent
Using \eqref{QbQ+} and Proposition \ref{prop:reg}, we then notice that 
\begin{align*}
RHS&=-2|Q^+|^2Q^+-i|Q^+|^4-i\overline{Q^+}^2+O((1-\b)^{1/2}\vert \log  (1-\b)\vert ^{1/2}) \text{ in } H^{-1/2}\\
&=i\pa_y Q^++ O((1-\b)^{1/2}\vert \log (1-\b)\vert ^{1/2}) \text{ in } H^{-1/2}\\
&=i\pa_y Q_\b+O((1-\b)^{1/2}\vert \log (1-\b)\vert ^{1/2}) \text{ in } H^{-1/2}
\end{align*}

\noindent
Thus, denoting
\begin{equation*}
g_\b:=Q_\b+\frac{i}{2}\pa_y Q_\b+\frac{1}{2}(1-\b)\overline{Q_\b},
\end{equation*}

\noindent
we have that
\begin{equation*}
\mathcal{L}_{\b}(g_\b)=i\pa_y Q_\b+O((1-\b)^{1/2}\vert \log (1-\b)\vert ^{1/2}) \text{ in } H^{-1/2}.
\end{equation*}

\noindent
Notice that 
\[\Big(Q_\b+\frac{i}{2}\pa_y Q_\b, iQ_\b\Big)=\Big(Q_\b+\frac{i}{2}\pa_y Q_\b, \pa_y Q_\b\Big)=0.\]

\noindent
Then, considering 
\begin{equation}\label{tildeg}
\tilde{g}_\b:=g_\b-\frac{1}{2}(1-\b)\textup{Proj}_{(iQ_{\b},\pa_y Q_\b)}\overline{Q_\b}\ ,
\end{equation}

\noindent
we have that $(\tilde{g}_\b, iQ_\b)=(\tilde{g}_\b,\pa_y Q_\b)=0$ and
\begin{equation*}
\mathcal{L}_{\b}(\tilde{g}_\b)=\mathcal{L}_{\b}(g_\b)=i\pa_y Q_\b+O((1-\b)^{1/2}\vert \log (1-\b)\vert ^{1/2}) \text{ in } H^{-1/2}\ .
\end{equation*}

\noindent
Since $\mathcal{L}_{\b} (i\rho_{\b})=i\pa_{y} Q_{\b}$, 
it follows that $(i\rho_{\b}-\tilde{g}_{\b}, iQ_{\b})=(i\rho_{\b}-\tilde{g}_{\b},\pa_{y} Q_{\b})=0$ and
\[\mathcal{L}_{\b} (i\rho_{\b}-\tilde{g}_{\b})=O((1-\b)^{1/2}\vert \log (1-\b)\vert ^{1/2}) \text{ in } H^{-1/2}.\]

\noindent
Then, by Proposition \ref{invertibility1},
we have that
\[i\rho_{\b}-\tilde{g}_{\b}=O((1-\b)^{1/2}\vert \log (1-\b)\vert ^{1/2}) \text{ in } H^{1/2}\] 

\noindent
In view of  \eqref{tildeg}, we have $\tilde g_\b =Q_\b+\frac{i}{2}\pa_y Q_\b + O(1-\b) \text{ in } H^{\frac 12}(\R)$, thus  \eqref{irho} is proved.
\end{proof}


\subsection{The profiles of $Q_\b (x)$ and of $\pa _xQ_\b (x)$ at infinity}


\begin{proposition}\label{prop:equivalent Q_b}
Consider the following function,
\be
F(x)=\int _0^\infty \frac{\a \, e^{-\a }}{\a -ix}\, d\a \ ,\ x\in \R\ ,\label{functionF}
\ee
and the quantity
\begin{equation}
c_\b :=\frac{i}{2\pi}\int _\R \vert Q_\b (x)\vert ^2Q_\b (x)\, dx\ .
\end{equation}
Then, as $\b \rightarrow 1$ and $\vert x\vert \rightarrow \infty$, we have
\begin{align}
Q_\b (x)&= \frac{c_\b}{x}F\left (-\frac{1-\b}{1+\b}x\right )+O\left (\frac 1{x^2}\right ) \label{Qbx}\\
\pa _xQ_\b (x)&= \frac{ic_\b}{x}\frac{1-\b}{1+\b} F\left (-\frac{1-\b}{1+\b}x\right )-\frac{c_\b}{x^2}
+O\left ( \frac{1-\b}{x^2}+\frac{ \log \vert x\vert}{\vert x\vert ^{3}} \right ).
\label{DQbx}
\end{align}
\end{proposition}
\begin{remark}
\begin{enumerate}
\item From the previous section and by Lemma \ref{lemmaalgebraSzeg\H{o}}, we know that $c_\b$ tends to $1$ as $\b $ tends to 1. In the next subsection --- see (\ref{varcb})--- we will prove  
that
\be
\label{estcbeta}
c_\b =1+O((1-\b)\vert \log(1-\b)\vert )\ .
\ee
\item Notice that $F(x)=1+O(|x||\log |x||)$
as $x\to 0 $ and $|F(x)|\lesssim \frac{1}{|x|}$ for all $|x|>0$. Therefore, as $\b\to 1$ and $|x|\to \infty$, we infer from \eqref{Qbx}  that
\begin{equation}\label{boundQ_b}
|Q_\b(x)|\lesssim \frac{1}{|1-\b|x^2}, \quad \forall |x|>0.
\end{equation}
Furthermore, if $0<1-\b\ll1$, $|x|\gg 1$, and $(1-\b)|x|\ll 1$, the following asymptotics follows from \eqref{Qbx} and \eqref{estcbeta}:
\begin{align}\label{asymQ_b}
Q_\b(x)&=\frac{1+O((1-\b)|\log (1-\b)|)}{x}\left[1+O\left((1-\b)|x|\left|\log ((1-\b)|x|)\right|\right)\right]\notag\\
&\hphantom{XX}+O\left(\frac{1}{x^2}\right).
\end{align}

\item In view of the identity
\be \label{FF'}
F'(x)=\left (\frac 1x-i\right )F(x)-\frac 1x\ ,
\ee
the main term in the asymptotics \eqref{DQbx} for $\pa _xQ_\b(x)$ is indeed obtained by deriving the main term in the asymptotics \eqref{Qbx} of 
$Q_\b (x)$.
\end{enumerate}
\end{remark}
\begin{proof}
The starting point is again the formula
$$Q_\b =m_\b *(\vert Q_\b\vert ^2Q_\b )\ ,$$
where
$$m_\b =\mathcal F^{-1}\left (\frac{1}{1+\frac{\vert \xi \vert -\b \xi}{1-\b}}\right )\ .$$
We notice that, for $x\ne 0$,  
$$m_\b (x)=\frac{1}{2\pi}\left (G(x)+\frac{1-\b}{1+\b} G\left (-\frac{1-\b}{1+\b}x\right ) \right )\ ,$$
where $$G(x)=\int _0^\infty \frac{e^{ix\xi}}{1+\xi }\, d\xi =\int _0^{\infty}\frac{e^{-\a}}{\a -ix}\, d\a ,$$
the second integral being obtained from the former by writing
$$\frac{1}{1+\xi}=\int _0^\infty e^{-\a(1+\xi)}\, d\a \ .$$
It is easy to check that $G$ is smooth outside $x=0$, $G(x)= ix^{-1}+x^{-2}+O(x^{-3})$ as $x\rightarrow \infty$, and $G(x)\sim \log \vert x\vert $
as $x\rightarrow 0$. In particular, $G\in L^p(\R )$ for every $p\in (1,\infty)$, with 
\begin{equation}\label{GLp}
\Vert G\Vert _{L^p}\le C \max \left (\frac{1}{p-1}, p\right )\ .
\end{equation}
Next we split
$$
xQ_\b (x)=\int _\R (x-y)m_\b (x-y)\vert Q_\b (y)\vert ^2Q_\b(y)\, dy+\int _\R m_\b (x-y)y\vert Q_\b (y)\vert ^2Q_\b(y)\, dy\ .
$$
Let us estimate the second integral in the right hand side, writing
\bee
\int _\R m_\b (x-y)y\vert Q_\b (y)\vert ^2Q_\b(y)\, dy&=&\int _{\vert y\vert \le \vert x\vert /2} m_\b (x-y)y\vert Q_\b (y)\vert ^2Q_\b(y)\, dy+\\
&&\int _{\vert y\vert > \vert x\vert /2}  m_\b (x-y)y\vert Q_\b (y)\vert ^2Q_\b(y)\, dy\ .
\eee
From H\"older's inequality and the uniform bound $\vert Q_\b (x)\vert \la x\ra $ from Proposition \ref{boundQb}, we have, for every $p>1$, close to $1$,
$$\left \vert \int _{\vert y\vert > \vert x\vert /2}  m_\b (x-y)y\vert Q_\b (y)\vert ^2Q_\b(y)\, dy\right \vert \lesssim \frac{1}{p-1}\vert x\vert ^{-1-\frac 1p  }\ ,$$
which, by choosing 
$$p=1+\frac{1}{\log\vert x\vert}\ ,$$
yields
$$\left \vert \int _{\vert y\vert > \vert x\vert /2}  m_\b (x-y)y\vert Q_\b (y)\vert ^2Q_\b(y)\, dy\right \vert \lesssim \frac{\log \vert x\vert}{x^{2-\frac{1}{\log|x|}}}\lesssim \frac{\log |x|}{|x|^{2}}\ .$$
On the other hand, because of the bounds on $G$, we have
$$\vert m_\b (x)\vert \lesssim \vert x\vert ^{-1}\ ,\ \vert x\vert \rightarrow \infty \ .$$
Indeed, the only non trivial case is $(1-\b)\vert x\vert \le 1$, so that
$$\left \vert \frac{1-\b}{1+\b} G\left (-\frac{1-\b}{1+\b}x\right )\right \vert \lesssim (1-\b)\vert \log[(1-\b)\vert x\vert]\lesssim \frac{1}{\vert x\vert }\ .$$
We conclude that
$$\left \vert \int _{\vert y\vert \le \vert x\vert /2}  m_\b (x-y)y\vert Q_\b (y)\vert ^2Q_\b(y)\, dy\right \vert \lesssim \frac{1}{\vert x\vert}\ ,$$
so that
$$Q_\b (x)=\frac{1}{x}\int _\R (x-y)m_\b (x-y)\vert Q_\b (y)\vert ^2Q_\b(y)\, dy+O\left (\frac{1}{x^2}\right )\ .$$
We come to the first integral. We observe that
$$\widehat{xm_\b}(\xi )=i\pa _\xi \left (\frac{1}{1+\frac{\vert \xi \vert -\b \xi}{1-\b}} \right )=i\left (\frac{1+\b}{1-\b}\right ){{\bf 1}_{\xi <0}}{\left(1+\frac{1+\b}{1-\b}\vert \xi\vert \right)^{-2}}-i\frac{{\bf 1}_{\xi >0}}{(1+\xi)^2}\ ,$$
so that
\bea
xm_\b (x)&=&\frac{i}{2\pi}\left (F\left (-\frac{1-\b}{1+\b}x\right ) -F(x)\right )\ ,\\ F(x)&:=&\int _0^\infty \frac{e^{ix\xi}}{(1+\xi)^2}\, d\xi =\int_0^\infty \frac{\a\, e^{-\a}}{\a -ix}\, d\a =1+ixG(x)\ .
\eea
This leads to
\bee
\int _\R (x-y)m_\b (x-y)\vert Q_\b (y)\vert ^2Q_\b(y)\, dy&=&-\int _\R F(x-y)g_\b (y)\, dy\\
&+&\int _\R F\left (-\frac{1-\b}{1+\b}(x-y)\right )g_\b (y)\, dy\ ,\\
\ g_\b &:=&\frac{i}{2\pi}\vert Q_\b \vert ^2Q_\b \ .
\eee
Again we are going to estimate the above two integrals by using the properties of $F$, namely that $F$ is smooth outside the origin, it is bounded near $0$, $F(x)=O(x^{-1})$ at infinity, while $\vert F'(x)\vert =O(\vert  \log\vert x\vert \vert )$ near $0$ and $F'(x)=O(x^{-2})$ at infinity. Furthermore, let us recall from Proposition \ref{boundQb} that
$$g_\b (y)=O(\la y\ra ^{-3})\ .$$
We infer the following estimates,
\bee 
\int _\R F(x-y)g_\b (y)\, dy&=&\int _{\vert y\vert \le \frac{\vert x\vert}{2}} F(x-y)g_\b (y)\, dy+\int _{\vert y\vert > \frac{\vert x\vert}{2}} F(x-y)g_\b (y)\, dy\\
&=&O(\vert x\vert ^{-1})+O(x^{-2})\ ,\\
\int _\R F\left (-\frac{1-\b}{1+\b}(x-y)\right )g_\b (y)\, dy&=& F\left (-\frac{1-\b}{1+\b}x\right )\int _{\vert y\vert \le \frac{\vert x\vert}{2}}g_\b (y)\, dy\\
&+&\int _{\vert y\vert > \frac{\vert x\vert}{2}} F\left (-\frac{1-\b}{1+\b}(x-y)\right )g_\b (y)\, dy\\
&+&\int _{\vert y\vert \le \frac{\vert x\vert}{2}} \left (F\left (-\frac{1-\b}{1+\b}(x-y)\right )
- F\left (-\frac{1-\b}{1+\b}x\right ) \right )g_\b (y)\, dy\\
&=&c_\b F\left (-\frac{1-\b}{1+\b}x\right )+O(x^{-2})+O(\vert x\vert ^{-1}\omega ((1-\b)\vert x\vert ))\ ,\\
\omega (s)&:=&\begin{cases} s\vert \log s\vert &\textrm{if}\ 0<s\le \frac 12\\
\frac{1}{s}&\textrm{if}\ \frac 12\le s\ \end{cases}\ .
\eee
This completes the proof of (\ref{Qbx}). Let us come to the proof of (\ref{DQbx}). Notice that
\bee
\widehat{\pa _xm_\b}(\xi )&=&\frac{i\xi}{1+\frac{\vert \xi\vert -\b \xi}{1-\b}}\\
&=&i\left( {\bf 1}_{\xi >0}-\frac{1-\b}{1+\b}{\bf 1}_{\xi <0}-\frac{ {\bf 1}_{\xi >0}}{1+\xi}+ \frac{1-\b}{1+\b}\frac{{\bf 1}_{\xi <0}}{\left (1+\frac{\vert \xi\vert -\b \xi}{1-\b}\right )}\right )\ ,
\eee
so that, using the formulae $\mathcal F^{-1}({\bf 1}_{\pm \xi >0})=\mp \frac{1}{2\pi i}pv \left (\frac 1x\right )+\frac 12  \delta _0$,
$$
\pa _xm_\b (x)=\frac{-1}{\pi (1+\b)}pv\left (\frac {1}{x}\right )+\frac{i\b}{2(1+\b)}\delta _0-\frac{i}{2\pi}G(x)+\frac{i}{2\pi}\left (\frac{1-\b}{1+\b}\right )^2
G\left (-\frac{1-\b}{1+\b}x \right )\ ,
$$
and
\bee \pa _xQ_\b (x)&=&\frac{2i}{1+\b}pv\left (\frac 1x\right )*g_\b+\frac{i\b}{2(1+\b)}\vert Q_\b(x)\vert ^2Q_\b(x)\\
&+&\int _\R\left (\left (\frac{1-\b}{1+\b}\right )^2
 G\left (-\frac{1-\b}{1+\b}(x-y) \right )-G(x-y)\right )g_\b(y)\, dy\ .
 \eee
Using, similarly as above, the estimates on $G$, and Proposition \ref{boundQb} for $g_\b$ , we have
\bee
\int _{\vert y\vert > \vert x\vert /2} G(x-y)g_\b (y)\, dy&=&O\left (\frac{\log \vert x\vert}{\vert x\vert ^{3}} \right )\ ,\\
\int _{\vert y\vert \le  \vert x\vert /2} G(x-y)g_\b (y)\, dy&=&\frac{i}{x}\int _\R g_\b +\frac{1}{x^2}\left (i\int _\R yg_\b +\int _\R g_\b\right )+O\left (\frac{\log \vert x\vert}{\vert x\vert ^3} \right )\ .
\eee
On the other hand,
$$pv\left (\frac 1x\right )*g_\b =\frac 1x\int _\R g_\b +\frac{1}{x^2}\int _\R yg_\b +\frac{1}{x^2}\, pv\left (\frac 1x\right )*(y^2g_\b)\ .$$
From Proposition \ref{boundQb},  $h_\b (y):=y^2g_\b (y)$ satisfies $h_\b(y)=O(\la y\ra ^{-1}) $ and $h_\b'(y)=O(\la y\ra ^{-2})$. We infer
\bee
pv\left (\frac 1x\right )*h_\b &=&\int _0^\infty \frac{h_\b (x-z)-h_\b (x+z)}{z}\, dz\\
&=&\int _{\vert \vert x\vert -z\vert >\vert x\vert /2}\frac{h_\b (x-z)-h_\b (x+z)}{z}\, dz\\
&&+\int _{\vert \vert x\vert -z\vert \le \vert x\vert /2}\frac{h_\b (x-z)-h_\b (x+z)}{z}\, dz\\
&\lesssim &\int _{\vert \vert x\vert -z\vert >\vert x\vert /2}\frac{dz}{\la \vert x\vert -z\ra ^2}+\int _{\vert \vert  x\vert -z\vert \le \vert x\vert /2}\frac{dz}{\vert x\vert \la \vert x\vert -z\ra}\\
&=&O(\vert x\vert ^{-1})+O(\vert x\vert ^{-1}\log \vert x\vert )\ .
\eee
Summing up, we have proved that, as $x\rightarrow \infty$, 
$$\frac{2i}{1+\b}pv\left (\frac 1x\right )*g_\b-G*g_\b=\frac{i(1-\b)}{(1+\b)x}\int _\R g_\b -\frac{1}{x^2}\int _\R g_\b +O\left (\frac{1-\b}{x^2}+\frac{\log \vert x\vert}{\vert x\vert ^3}  \right ) \ .   $$
It remains to study the last integral, namely
$$\int _\R \left (\frac{1-\b}{1+\b}\right )^2 G\left (-\frac{1-\b}{1+\b}(x-y) \right )\, g_\b(y)\, dy=\int _{\vert y\vert \le \vert x\vert /2}...+\int _{\vert y\vert > \vert x\vert /2}...\ .$$
Using again H\"older's inequality and optimizing on the power, we get
$$\left \vert \int _{\vert y\vert > \vert x\vert /2}\left (\frac{1-\b}{1+\b}\right )^2 G\left (-\frac{1-\b}{1+\b}(x-y) \right )\, g_\b(y)\, dy\right \vert \lesssim \frac{(1-\b)\log \vert x\vert}{\vert x\vert ^{3}}\ .$$
On the other hand, because of the estimates on $G'$, we have
\bee 
&&\int _{\vert y\vertÊ\le \vert x\vert /2}\left (\frac{1-\b}{1+\b}\right )^2 G\left (-\frac{1-\b}{1+\b}(x-y) \right )\, g_\b(y)\, dy=\\
&&\left (\frac{1-\b}{1+\b}\right )^2 G\left (-\frac{1-\b}{1+\b}x \right )\int _\R g_\b (y)\, dy+O\left (\frac{1-\b}{x^2}\right )\ .
\eee
In view of the identity
$$G(x)=\frac{F(x)-1}{ix}\ ,$$
this completes the proof of (\ref{DQbx}). 
\end{proof}


\subsection{Further estimates on $\pa _\b Q_\b$}


In this subsection, we improve some  the estimates on $\dot Q_\b:=\pa _\b Q_\b$ deduced in Proposition \ref{boundQb}. 

\begin{proposition}\label{furtherQbdot}
The following estimates hold as $\b $ tends to $1$.
\begin{align}
&\Vert \dot Q_\b^+\Vert _{H^{\frac 12}} \lesssim \vert \log(1-\b)\vert \ , \label{Qbdotplus}\\
&\vert \widehat{\dot Q_\b ^-}(\xi)\vert \le \frac{C}{1-\b+(1+\b)\vert \xi\vert} \ .\label{Qbdotmoins}
\end{align}
Furthermore, if $H_\b =(1-\b)\pa _\b^2Q_\b $ or $H_\b =\pa _\b y\pa _yQ_\b$, we have similarly
\begin{align}
&\Vert H_\b^+\Vert _{H^{\frac 12}} \lesssim \vert \log(1-\b)\vert \ , \label{Hbplus}\\
&\vert \widehat{H_\b ^-}(\xi)\vert \le \frac{C}{1-\b+(1+\b)\vert \xi\vert} \ .\label{Hbmoins}
\end{align}
In particular, 
\begin{align}
& \frac{d}{d\b}\Vert Q_\b\Vert _{L^2}^2  =O(\vert \log(1-\b)\vert )\ ,\label{varL2b}\\ 
& \frac{d}{d\b}(DQ_\b, Q_\b) =O(\vert \log(1-\b)\vert )\ ,\label{varHdemib}\\ 
& \frac{d}{d\b}\int _\R \vert Q_\b\vert ^2Q_\b  =O(\vert \log(1-\b)\vert )\ ,\label{varcb}
\end{align}
and, if $H_\b$ is as above, and $\rho_\b$ is defined by \eqref{nceionceoneo}, we have
\be \vert (H_\b,Q_\b)\vert +\vert (H_\b,DQ_\b)\vert +\vert (H_\b,i\rho_\b)\vert =O(\vert \log(1-\b)\vert ).\label{Hbinners}
\ee
\end{proposition}
\begin{proof}
We  project the equation (\ref{eqQbdot}) for $\dot Q_\b$ onto the negative and positive modes. This gives
\begin{align*}
&(1+\b)\vert D\vert \dot Q_\b^-+(1-\b)\dot Q_\b^-=\\  
&\frac{2}{1+\b}[Q_\b^--\Pi _-(\vert Q_\b\vert ^2Q_\b )]+\Pi_-[ 2\vert Q_\b\vert ^2(1-\b)\dot Q_\b+Q_\b^2(1-\b)\overline {\dot Q_\b}  ]\ ,\\
&D\dot Q_\b^+ +\dot Q_\b^+-\Pi _+[2\vert Q_\b\vert ^2\dot Q_\b ^++Q_\b^2\overline{\dot Q_\b^+}]=\Pi _+[2\vert Q_\b\vert ^2\dot Q_\b ^-+Q_\b^2\overline{\dot Q_\b^-}]\ ,\\
&(\dot Q_\b^+,iQ^+)=(\dot Q_\b^+,\pa _xQ^+)=0\ .
\end{align*}
Using the last equation, the invertibility \eqref{estL} of $\mathcal L$ defined in \eqref{defL}, and a perturbation argument as in Proposition \ref{invertibility1}, we can estimate $\dot Q_\b ^+$ by means of  $\dot Q_\b ^-$ as follows,
\be \label{Qbdotplusmoins}
\Vert \dot Q_\b ^+\Vert _{H^{\frac12}}\lesssim \Vert \vert Q_\b\vert ^2\dot Q_\b ^-+Q_\b^2\overline{\dot Q_\b^-}\Vert _{H^{-\frac 12}}\lesssim \Vert Q_\b^2\dot Q_\b^-\Vert _{L^2}\ .
\ee
On the other hand, the first equation leads to
\bee 
\widehat {\dot Q_\b ^-}(\xi )&=& \frac{  \hat \ell _\b (\xi)}{1-\b+(1+\b)\vert \xi\vert}\ ,\\ 
 \ell_\b &:=&\frac{2}{1+\b}\Pi_-[m_\b*(\vert Q_\b\vert ^2Q_\b )- \vert Q_\b\vert ^2Q_\b ]+\Pi_-[ 2\vert Q_\b\vert ^2(1-\b)\dot Q_\b+Q_\b^2(1-\b)\overline {\dot Q_\b}  ]\ .
\eee
Using the $L^2$ bound on $(1-\b)\dot Q_\b$ from Proposition \ref{boundQb}, the above expression of $\ell _\b$ implies 
$$\Vert \hat \ell _\b\Vert _{L^\infty}\le C\ ,$$
which proves (\ref{Qbdotmoins}). 
Coming back to (\ref{Qbdotplusmoins}), we infer, using the $L^1$ and the $L^2$ bound on $\widehat{Q_\b^2}$, and from Young's $L^1*L^2\subset L^2$ inequality,
\bee
\Vert \dot Q_\b ^+\Vert _{H^{\frac12}}&\lesssim &\left [ \int _\R\left \vert \int _\R \frac{\vert \widehat{Q_\b^2}(\xi -\eta)\vert }
{1-\b +(1+\b)\vert \eta \vert}  \, d\eta \right \vert ^2  \, d\xi \right ]^{\frac 12}\\
&\lesssim & \left (\int _{\vert \eta\vert >1} \frac{d\eta }{(1-\b +(1+\b)\vert \eta \vert)^2} \right )^{\frac 12}+\int_{\vert \eta\vert \le 1}\frac{d\eta}{1-\b +(1+\b)\vert \eta \vert}\\
&\lesssim& \vert \log(1-\beta)\vert .
\eee

This proves (\ref{Qbdotplus}). 
\par\noindent
Next we prove \eqref{Hbplus} and \eqref{Hbmoins}. We  apply $(1-\b)\pa _\b$ to the above equations on $Q_\b^+$ and $Q_\b^-$. With $H_\b:=(1-\b)\pa _\b ^2Q_\b$, we infer
\begin{align*}
&(1+\b)\vert D\vert H_\b^-+(1-\b)H_\b^-+(\vert D\vert -1)(1-\b)\dot Q_\b^-=\\  
&(1-\b)\pa _\b \left (\frac{2}{1+\b}[Q_\b^--\Pi _-(\vert Q_\b\vert ^2Q_\b )]+\Pi_-[ 2\vert Q_\b\vert ^2(1-\b)\dot Q_\b+Q_\b^2(1-\b)\overline {\dot Q_\b}  ]\right )\ ,\\
&DH_\b^+ +H_\b^+-\Pi _+[2\vert Q_\b\vert ^2H_\b ^++Q_\b^2\overline{H_\b^+}]=\\
&\Pi _+[2\vert Q_\b\vert ^2H_\b ^-+Q_\b^2\overline{H_\b^-}
+2(1-\b)\pa _\b(\vert Q_\b\vert^2)\dot Q_\b+(1-\b)\pa _\b(Q_\b^2)\overline {\dot Q_\b}]\ ,\\
&(H_\b^+,iQ^+)=(H_\b^+,\pa _xQ^+)=0\ .
\end{align*}
In view of \eqref{Qbdotmoins}, the Fourier transform of $(\vert D\vert -1)(1-\b)\dot Q_\b^-$ is uniformly bounded. Furthermore, using again \eqref{Qbdotmoins} and the $L^2$ bound on $[(1-\b)\pa _\b]^k Q_\b$ from Proposition \ref{boundQb}, the Fourier transform of the right hand side of the equation on $H_\b ^-$ is uniformly bounded. This provides estimate \eqref{Hbmoins}. In order to obtain \eqref{Hbplus}, we use the equation
on $H_\b^+$. Notice that, again by \eqref{Qbdotplus} and \eqref{Qbdotmoins} combined with the Hausdorff--Young inequality,
\begin{align*}
&\Vert  2(1-\b)\pa _\b(\vert Q_\b\vert^2)\dot Q_\b+(1-\b)\pa _\b(Q_\b^2)\overline {\dot Q_\b}  \Vert _{L^2}\lesssim (1-\b)\Vert \dot Q_\b\Vert _{L^4}^2\\
& \lesssim (1-\b)\left (\int _\R \frac{d\xi}{((1-\b)+(1+\b)\vert \xi\vert)^{4/3}}\right )^{3/2}+(1-\b)|\log (1-\b)|^2\\
&\lesssim (1-\b)^{1/2}\ .
\end{align*}
By the perturbation argument of Proposition \ref{invertibility1}, we infer
$$\Vert H_\b ^+\Vert _{H^{\frac12}}\lesssim  \Vert Q_\b^2H_\b^-\Vert _{L^2} +O((1-\b)^{1/2})\ ,$$
and we obtain \eqref{Hbplus} exactly as we obtained \eqref{Qbdotplus} above. 
\vskip0.25cm
Next we deal with the case of $H_\b:=y\pa _y\dot Q_\b$. Applying $y\pa _y$ to the equation on $Q_\b$, we get
$$\mathcal L_\b (y\pa _yQ_\b)=\vert Q_\b\vert ^2Q_\b -Q_\b\ ,$$
and, taking the derivative with respect to $\b$ and projecting on the negative and positive modes, we obtain
\begin{align*}
&(1+\b)\vert D\vert H_\b^-+(1-\b)H_\b^-=\\  
&\Pi_-[ 2\vert Q_\b\vert ^2(1-\b)H_\b+Q_\b^2(1-\b)\overline {H_\b}  ]
+\frac{2}{1+\b}\Pi_-(y\pa _yQ_\b)\\
&+\frac{2}{1+\b}[Q_\b^--\Pi _-(\vert Q_\b\vert ^2Q_\b )]+\Pi_-[ 2\vert Q_\b\vert ^2(1-\b)\dot Q_\b+Q_\b^2(1-\b)\overline {\dot Q_\b}  ]-(1-\b)\dot Q_\b^-\\
&-\frac{2}{1+\b}\Pi_-[2|Q_\b|^2y\pa_y Q_\b + Q_\b^2\overline{y\pa_y Q_\b}]\\
&+2\Pi_-[(1-\b)\overline{\dot Q_\b}Q_\b y\pa_y Q_\b
+(1-\b)\dot Q_\b\overline{Q_\b}y\pa_y Q_\b
+(1-\b)\dot Q_\b Q_\b \overline{y\pa_y Q_\b}],\\
&DH_\b^+ +H_\b^+-\Pi _+[2\vert Q_\b\vert ^2H_\b ^++Q_\b^2\overline{H_\b^+}]=\\
&\Pi _+[2\vert Q_\b\vert ^2H_\b ^-+Q_\b^2\overline{H_\b^-}
+2\vert Q_\b\vert^2\dot Q_\b+Q_\b^2\overline {\dot Q_\b}]-\dot Q_\b^+\\
&+2\Pi_+[\overline{\dot Q_\b} Q_\b y\pa_y Q_\b
+\dot Q_\b\overline{Q_\b} y\pa_y  Q_\b
+\dot Q_\b Q_\b\overline{y\pa_y Q_\b}],\\
&(H_\b^+,iQ^+)=(H_\b^+,\pa _xQ^+)=0\ .
\end{align*}
Again, from Proposition \ref{boundQb}, we notice that the Fourier transform of the right hand side of the equation on $H_\b^-$ is bounded.
This provides \eqref{Hbmoins}. Using again the perturbation argument of Proposition  \ref{invertibility1}, we infer
$$\Vert H_\b ^+\Vert _{H^{\frac12}}
\lesssim  \Vert Q_\b^2H_\b^-\Vert _{L^2} +\Vert Q_\b^2\dot Q_\b\Vert _{L^2}+\Vert \dot Q_\b^+\Vert _{L^2} +\|\dot Q_\b Q_\b y\pa_y Q_\b\|_{L^2}
$$
and \eqref{Hbplus} again follows from \eqref{Hbmoins}, \eqref{Qbdotmoins}, \eqref{Qbdotplus},
and the $L^1$- and $L^2$- bounds on $\widehat{Q_\b y\pa_y Q_\b}$.
\vskip0.25cm
Let us come to the proof of (\ref{varL2b}). We have
$$\frac{d}{d\b}\Vert Q_\b\Vert _{L^2}^2=2(Q_\b,\dot Q_\b)=2(Q_\b^+,\dot Q_\b^+)+2(Q_\b^-,\dot Q_\b^-)\ .$$
From (\ref{Qbdotplus}) and the $L^2$ bound on $Q_\b$, we infer
$$\vert (Q_\b^+,\dot Q_\b^+)\vert \lesssim \vert \log(1-\b)\vert \ .$$
From (\ref{Qbdotmoins}) and the representation of $Q_\b^-$, we infer
$$\vert (Q_\b^-,\dot Q_\b^-)\vert \lesssim \int _\R \frac{(1-\b)\vert \widehat{\vert Q_\b\vert ^2Q_\b}(\xi)\vert}{(1-\b+(1+\b)\vert \xi\vert)^2}\, d\xi =O(1)\ .$$
This completes the proof of (\ref{varL2b}). The proof of (\ref{varHdemib}) is similar. As for (\ref{varcb}), we write
$$\frac{d}{d\b}\int _\R \vert Q_\b\vert ^2Q_\b =2\int _\R \vert Q_\b\vert ^2\dot Q_\b+\int_\R Q_\b^2\overline{\dot Q_\b}\ .$$
Write $\dot Q_\b=\dot Q_\b^++\dot Q_\b^-$ in the two integrals of the above right hand side. The contribution of $\dot Q_\b^+ $ is $O(\vert \log(1-\b)\vert)$ because of (\ref{Qbdotplus}). As for the contribution of $\dot Q_\b^-$, we evaluate it by means of the Plancherel theorem. In view of (\ref{Qbdotmoins}),
it is $O(\vert \log(1-\b)\vert)$. This completes the proof of (\ref{varcb}).\\
The proof of the first two estimates of \eqref{Hbinners} follows exactly the same lines as  (\ref{varL2b}). As for the last estimate, we recall from \eqref{irho}
that
$$\Vert i\rho_\b -Q_\b-\frac 12DQ_\b\Vert _{L^2}\lesssim (1-\b)^{1/2}\vert \log(1-\b)\vert^{1/2}\ ,$$
so that
$$\vert (H_\b,i\rho_\b)\vert \lesssim \vert \log(1-\b)\vert +(\Vert H_\b^+\Vert _{L^2}+\Vert H_\b^-\Vert _{L^2})(1-\b)^{1/2}\vert \log(1-\b)\vert^{1/2}\ ,$$
and the proof is completed by using \eqref{Hbplus} and \eqref{Hbmoins}.
\end{proof}


\section{The two-bubble approximate solution}
\label{sectionfour}


This section is devoted to the construction of the two-bubble approximate solution. The general strategy follows the lines of \cite{KMRhartree} for the Hartree problem with the additional difficulties of keeping very carefully track of the leading order terms generated by the critically slow decay of the solitary wave and getting estimates which are uniform in the singular limit $\beta\to 1$.


\subsection{Renormalization and slow variables}


Let $$u_j(t,x)=\frac{1}{\l_j^{\frac 12}}v_j(s_j,y_j)e^{i\gamma_j},  \ \ \frac{ds_j}{dt}=\frac{1}{\l_j(t)}, \ \ y_j:=\frac{x-x_j(t)}{\l_j(t)(1-\beta_j(t))},$$ 
for $j=1,2$. We have 
\bee
&&i\pa_tu_j-|D|u_j+u_j|u_j|^2\\
\nonumber & = & \frac{1}{\l_j^{\frac 32}}\left[i\pa_{s_j}v_j-\frac{(|D|-\beta_j D)v_j}{1-\b_j}-i\frac{(\l_j)_{s_j}}{\l_j}\Lambda v_j-\frac{i}{1-\b_j}
\left(\frac{(x_j)_{s_j}}{\l_j}-\beta_j\right)\pa_{y_j}v_j\right.\\
& + & \left. \frac{i(\b_j)_{s_j}}{1-\b_j}y_j\pa_{y_j}v_j-(\gamma_j)_{s_j}v_j+v_j|v_j|^2\right]e^{i\gamma_j}(s_j,y_j).
\eee
Let us define the relative numbers 
$$ \ \ X=x_2-x_1, \ \ \ \ \mu=\frac{\l_2}{\l_1}, \ \ \Gamma=\gamma_2-\gamma_1, 
$$ and 
\be
\label{defbr}
b=\frac{1-\beta_2}{1-\beta_1},\ \ R=\frac{X}{\l_1(1-\beta_1)}.
\ee
We observe the relation
\be
\label{estyunytwo}
y_1=R+\mu b y_2.
\ee
We then decompose $u(t,x)=u_1(t,x)+u_2(t,x)$, expand the nonlinearity
\bee
u|u|^2& = & u_1(|u_1|^2+2|u_2|^2+u_1\bar{u_2})+u_2(|u_2|^2+2|u_1|^2+u_2\bar{u_1})
\eee
and split the contributions of crossed terms using a cut off function 
\be
\label{defchir}
\chi_R(x)=\chi\left(\frac{y_1}{R}\right)=\chi\left(1+\frac{\mu b}{R}y_2\right)
\ee to obtain:

$$i\pa_tu-|D|u+u|u|^2=\frac{1}{\l_1^{\frac 32}}\mathcal E_1(s_1,y_1)e^{i\gamma_1}+\frac{1}{\l_2^{\frac 32}}\mathcal E_2(s_2,y_2)e^{i\gamma_2}$$ with \bee
\mathcal E_1& = & i\pa_{s_1}v_1-\frac{(|D|-\beta_1 D)v_1}{1-\b_1}-v_1+ v_1|v_1|^2\\
& - & i\lslone\Lambda v_1-\frac{i}{1-\b_1}\left(\xslone-\beta_1\right)\pa_{y_1}v_1+\frac{i(\b_1)_{s_1}}{1-\b_1}y_1\pa_{y_1}v_1-[(\gamma_1)_{s_1}-1]v_1\\
& + &\chi_R\left[\frac{2}{\mu}v_1|v_2|^2+\frac{e^{-i\Gamma}}{\sqrt{\mu}}v^2_1 \overline{v_2}+2\frac{e^{i\Gamma}}{\sqrt \mu}\vert v_1\vert ^2v_2+\frac{e^{2i\Gamma}}{\mu}\overline{v_1}v_2^2\right ]\ ,\\
\eee
\bee
\mathcal E_2& = & i\pa_{s_2}v_2-\frac{(|D|-\beta_2 D)v_2}{1-\b_2}-v_2+ v_2|v_2|^2\\
& - & i\lsltwo\Lambda v_2-\frac{i}{1-\b_2}\left(\xsltwo-\beta_2\right)\pa_{y_2}v_2+\frac{i(\b_2)_{s_2}}{1-\b_2}y_2\pa_{y_2}v_2-[(\gamma_2)_{s_2}-1]v_2\\
& + & (1-\chi_R)\left[2\mu v\vert v_1\vert ^2v_2+2\sqrt{\mu}e^{-i\Gamma}v_1|v_2|^2+\sqrt{\mu}e^{i\Gamma}\overline{v_1}v_2^2+\mu e^{-2i\Gamma}v_1^2\overline{v_2}\right].
\eee
The full vector of parameters is denoted by
 \be
 \label{defp}
 \mathcal P=(\l_1,\l_2,\beta_1,\beta_2,\Gamma,R).
 \ee
Following \cite{KMRhartree}, we now look for a solution to 
$$\mathcal E_1=\mathcal E_2=0$$
 in the form of a slowly modulated two-bubble, i.e. 
$$v_j(s_j,y_j)=V_j(y_j,\mathcal P(s_j))$$ where the time dependence of the parameters is frozen for translation and phase invariances:
\be
\label{lawsfrozen}
 \frac{(x_j)_{s_j}}{\l_j}=\beta_j, \ \ (\gamma_j)_{s_{j}}=1,
 \ee the dependence of scaling and speed is computed iteratively according to a dynamical system 
 \be
 \label{lawscomputed}
 \frac{(\l_j)_{s_j}}{\l_j}=M_j(\mathcal P), \ \ \frac{(\beta_j)_{s_j}}{1-\beta_j}=B_j(\mathcal P),
 \ee
 \be
 \label{actiongamma}
 \Gamma_{s_1}=\frac{1}{\mu}-1, \ \ \Gamma_{s_2}=1-\mu,\ \  X_t=\beta_2-\beta_1
 \ee
 and the remaining time derivatives for $(b,R)$ are modeled after \fref{defbr}, \fref{lawsfrozen}, \fref{lawscomputed}:
 \be
 \label{actionR}
 R_{s_1}=1-b+(B_1-M_1)R, \ \ R_{s_2}=\mu(1-b+(B_1-M_1)R).
 \ee
 Hence
 \bea
  \label{expressioneone}
\nonumber \mathcal E_1& = & -\frac{(|D|-\beta_1 D)V_1}{1-\b_1}-V_1+ V_1|V_1|^2-iM_1\Lambda V_1+iB_1\left[y_1\pa_{y_1}V_1+(1-\beta_1)\frac{\pa V_1}{\pa \beta_1}\right]\\
& + & i\l_1M_1\frac{\pa V_1}{\pa \l_1}+i\l_1M_2\frac{\pa V_1}{\pa \l_2}+i\frac{(1-\beta_2)B_2}{\mu}\frac{\pa V_1}{\pa \beta_2}\\
\nonumber & + & i\frac{1-\mu}{\mu}\frac{\pa V_1}{\pa \Gamma}+i(1-b+(B_1-M_1)R)\frac{\pa V_1}{\pa R}\\
\nonumber & + &\chi_R\left[\frac{2}{\mu}V_1|V_2|^2+\frac{e^{-i\Gamma}}{\sqrt{\mu}}V^2_1 \overline{V_2}+2\frac{e^{i\Gamma}}{\sqrt{\mu}}|V_1|^2V_2+\frac{e^{2i\Gamma}}{\mu}\overline{V_1}V_2^2\right],
\eea
 \bea
 \label{expressionetwo}
 \nonumber \mathcal E_2& = & -\frac{(|D|-\beta_2 D)V_2}{1-\b_2}-V_2+ V_2|V_2|^2-iM_2\Lambda V_2+iB_2\left[y_2\pa_{y_2}V_2+(1-\beta_2)\frac{\pa V_2}{\pa \beta_2}\right]\\
 &+& i\l_2M_2\frac{\pa V_2}{\pa \l_2}+i\l_2M_1\frac{\pa V_2}{\pa \l_1}+i\mu(1-\beta_1)B_1\frac{\pa V_2}{\pa \beta_1}\\
 \nonumber & + & i(1-\mu)\frac{\pa V_2}{\pa \Gamma}+i\mu(1-b+(B_1-M_1)R)\frac{\pa V_2}{\pa R}\\
\nonumber & + & (1-\chi_R)\left[2\mu \vert V_1\vert ^2V_2+2\sqrt{\mu}e^{-i\Gamma}V_1|V_2|^2+\sqrt{\mu}e^{i\Gamma}\overline {V_1}V_2^2+\mu e^{-2i\Gamma}V_1^2\overline{V_2}\right].
\eea
 and we need to solve the system of {\it nonlinear elliptic equations} in $V_1,V_2$, 
 \be
 \label{solutioneone}
 \left\{\begin{array}{ll} \mathcal E_1(y_1)=0 \ \ \mbox{with}\ \ y_2=\frac{y_1-R}{b\mu},\\ 
 \mathcal E_2(y_2)=0\ \ \mbox{with}\ \ y_1=R+b\mu y_2.
 \end{array}\right.
 \ee
 in a suitable range of parameters $\mathcal P$.


\subsection{Definition of admissible functions}


We define the open set of parameters: 
\be
\label{bootass}
\mathcal P\in \mathcal O\equiv \left\{\begin{array}{llll}
\beta_j\in (\beta^*,1) , \ \ j=1,2\\
R>R_*\\
|1-\l_1|+|1-\l_2|<\eta^\delta\\
\frac{\eta}{2}<1-\beta_1<2\eta, \ \ 1-\b _2\ge e^{-R}, \ \ 0<b<\delta
 \end{array}\right.
\ee
for some universal constants $R_*\gg1$, $0<\eta,\delta\ll 1$ to be chosen later.\\

We now define a suitable topology:

\begin{definition}[Admissible function] 
\label{defadmissible}
 We consider functions $g=g(y,\matchal P):\Bbb R\times \mathcal O\rightarrow \Bbb C$.\\
{\em (i) (}$L^\infty $-admissibility{\rm )}. We say that $g$ is $L^\infty$-admissible  if  $\forall \alpha\in \Bbb N^7$, $\exists A_\alpha $ , $\forall \mathcal P\in \mathcal O$,
\be
\label{normeginfty}
\left\|\Lambda_y^{\alpha_1}\Lambda_R^{\alpha_2}\pa^{\alpha_3}_{\l_1}\pa^{\alpha_4}_{\l_2}\pa^{\alpha_5}_{\Gamma} \tL_{\beta_1}^{\alpha_6}\tL_{\beta_2}^{\alpha_7}g(\cdot,\mathcal P)\right\|_\infty \leq A_\alpha. 
\ee
{\em (ii) (}Admissibility with respect to a bubble{\rm ).} Let $j\in \{ 1,2\} $. We say that $g$ is admissible with respect to the bubble $j$ --- or $j$--admissible --- if  $\forall \alpha\in \Bbb N^7$, $\exists A_\alpha $, $\forall \mathcal P\in \mathcal O$,
\begin{align}
\label{normeg}
\left\|\Lambda_y^{\alpha_1}\Lambda_R^{\alpha_2}\pa^{\alpha_3}_{\l_1}\pa^{\alpha_4}_{\l_2}\pa^{\alpha_5}_{\Gamma}\tL_{\beta_1}^{\alpha_6}\tL_{\beta_2}^{\alpha_7} g(\cdot,\mathcal P)\right\|_{\b _j}
\leq A_\alpha .
\end{align}
{\em (iii) (}Strong admissibility with respect to a bubble{\rm ).} Let $j\in \{ 1,2\} $. We say that $g$ is strongly admissible with respect to the bubble $j$ --- or $j$--strongly admissible --- if it is $j$--admissible and if, for every family $\{ \mu _\b\} _{\b \in (\b *,1)}$ of multipliers  in the class $\mathcal M$, the convolution product
$$\mu _{\b _j}*g(.,\mathcal P)$$
is $j$--admissible.
\end{definition}
Notice that admissibility with respect to the bubble $j$ implies $L^\infty $-admissibility. Furthermore, we have the following
fundamental property.
\begin{lemma}[Admissibility of $Q_\beta$] \label{lemma: admis Q}
For $j=1,2$, $Q_{\b _j}$ is strongly $j$--admissible. 
\end{lemma}
\begin{proof}
Admissibility of $Q_{\b _j}$ with respect to the bubble $j$ is a straightforward consequence of Proposition \ref{boundQb}. 
Given $\{ \mu _\b\} _{\b \in (\b *,1)}$ a family of multipliers  in the class $\mathcal M$, let us come to the $j$--admissibility of $\mu _{\b _j}*Q_{\b _j}$. From the identity
$$Q_\b =m_\b *(\vert Q_\b\vert ^2Q_\b )\ ,$$
and the invariance of $\mathcal M$ by convolution, we infer that
$$\mu _\b *Q_\b =\tilde \mu _\b *(\vert Q_\b\vert ^2Q_\b )\ ,$$
where $\{ \tilde \mu _\b\} _{\b \in (\b *,1)}$ belongs to $\mathcal M$. Then, applying $\tilde \Lambda _\b ^p\Lambda _y^q$
to this identity, and using the stability properties of class $\mathcal M$ through these operations, the $j$--admissibility of  $\mu _{\b _j}*Q_{\b _j}$ follows from the $j$--admissibility of $Q_{\b _j}$ and from Lemma \ref{mub123}. 
\end{proof}


\subsection{Stability properties of admissible functions}


 We now prove some elementary stability properties of admissible functions. 
  
 \begin{lemma}[Stability properties of admissible functions]\label{stabprop}
The following stability properties hold.

\noindent
{\em (i) (}Stability by derivation{\rm ).} Assume $g$ is $j$--admissible (resp. strongly $j$--admissible).  Then 
 \begin{equation}\label{deriv}
 \Lambda_yg, \Lambda_Rg, \pa_{\l_j}g, \pa_{\Gamma} g, \tL_{\beta_j}g
\end{equation}
are $j$--admissible (resp. strongly $j$--admissible).

\noindent
 {\em (ii) (}Stability by multiplication{\rm ).} If  $g$ is $j$--admissible, $h$ is $L^\infty $-admissible, then $gh$ is $j$--admissible.
Furthermore, if $g$ and $h$ are $j$--admissible, then $gh$ is strongly $j$--admissible. 

\noindent
 {\em (iii) (}Exchange of variables{\rm ).} Given a function $g=g(y)$, we define
 \be
 \label{firstestiem}
g^\sharp (y_1): = g\left (\frac{y_1-R}{b\mu}\right )
 \ee
 and  
  \be
  \label{firstestiembis}
g^\flat (y_2):=g(R+b\mu y_2)\ .
  \ee
If $g_2$ is $2$--admissible, then  $R(1+(1-\b_1)R)b^{-1}\chi _Rg_2^\sharp $ is $L^\infty$-admissible, and $b^{-1}\chi _Rg_2^\sharp $ is $1$--admissible.  If  $g_1 $ is $1$--admissible, then $R(1+(1-\b_1)R)((1-\chi _R)g_1)^\flat $ is $L^\infty$-admissible. 

\noindent
 {\em (iv) (}Stability by scalar product{\rm ).} If $g$ is $j$--admissible, then 
 $(g,iQ_{\b _j})$ and $(g,\pa _yQ_{\b _j})$ are $L^\infty$-admissible.
 
 \noindent{\em (v) (}Stability by convolution{\rm ).} If $g$ is strongly $j$--admissible and if $\{ \mu _\b\} _{\b \in (\b *,1)}$ belongs to class $\mathcal M$, then $\mu _{\b _j}*g$ is strongly $j$--admissible.
 
 \noindent{\em (vi) (}Mixed cubic nonlinearity and convolution{\rm ).} Assume $g_1,h_1$ are $1$--admissible, and $g_2, h_2$ are $2$--admissible.  Then 
 $$R(1+(1-\b_1)R)b^{-1}\chi _Rg_1g_2^\sharp h_2^\sharp \ ,\  R(1+(1-\b_1)R)b^{-1}\chi _Rg_1h_1g_2^\sharp $$
 are strongly $1$--admissible, and 
 $$R(1+(1-\b_1)R)((1-\chi _R)g_1)^\flat g_2h_2\ ,\ R(1+(1-\b_1)R)((1-\chi _R)g_1h_1)^\flat g_2$$
 are strongly $2$--admissible.

  \end{lemma}

 \begin{proof}[Proof of Lemma \ref{stabprop}] The first two properties are almost immediate --- notice that the strong admissibility of $gh$ is a consequence of Lemma \ref{mub123}. \\ 
Property (iii) is established 
 by first observing that $|y_1|\le \frac R2$ on the support of $\chi _R\, g_2^\sharp $, so that
 $$(1-\b_2)\frac{\vert y_1-R\vert }{b\mu}\ge (1-\b_1)\frac{R}{2\mu}\ .$$
 Similarly, $R+b\mu y_2\ge R/4$ on the support of  $((1-\chi _R)g_1)^\flat$, so that 
 $$(1-\b_1)\vert R+b\mu y_2\vert \ge (1-\b_1)\frac{R}{4}\ .$$
  In the first case, we also have, on the support of  $\chi _R\, g_2^\sharp $,
 $$\vert y_1-R\vert \ge \frac 14(\vert y_1\vert +R)\ ,\ $$
 so that 
 $$\Vert \chi _Rg_2^\sharp \Vert _{\b _1}\lesssim b\Vert g_2\Vert _{\b _2}
 \text{ and } R(1+(1-\b_1)R)\| \chi _Rg_2^\sharp\|_{L^\infty}\lesssim b\|g_2\|_\b.$$
 We argue similarly for $((1-\chi _R)g_1)^\flat $.
 Furthermore, 
 $$\Lambda _{y_1}(g_2^\sharp)=(\Lambda _{y_2}g_2)^\sharp +\frac{R}{b\mu}(\pa _{y_2}g_2)^\sharp\ ,\  \Lambda _{y_2}g_1^\flat =(\Lambda _{y_1}g_1)^\flat -R(\pa _{y_1}g_1)^\flat \ ,$$
with similar formulae for derivatives $\tL_{\b _j}, \Lambda _R, \pa _{\l_j}$. Since 
$$\pa _{y_j}^kg_j(y_j)=O(\la y_j\ra ^{- k-1})\ ,$$
this  provides the correct decay of derivatives of $\chi _R\, g_2^\sharp $ and of $((1-\chi _R)g_1)^\flat$.\\
Let us prove  property (iv). The $L^\infty$-admissibility of $(g,iQ_{\b _j})$ is a consequence of the Cauchy--Schwarz inequality and of the $j$--admissibility of  $g$ and $Q_{\b_j}$. As for the $L^\infty$-admissibility of $(g,\pa _yQ_{\b _j})$, it is a consequence of  the $j$--admissibility of  $g$ and of  the boundedness in $L^2$ of
$\tL_{\b}^aQ_\b \tL_{\b}^bQ_\b \tL_{\b}^cQ_\b $. The latter fact follows from  the identity
$$\pa _yQ_\b =\pa _ym_\b  *(\vert Q_\b \vert ^2Q_\b)\ ,$$
and of the boundedness of the Fourier transforms of $\pa _y\tL_{\b}^qm_\b $.\\
Property (v) is an immediate consequence of the invariance of class $\mathcal M$ by convolution.\\
Finally, let us prove property (vi). By properties (iii) and (ii), we immediately get that 
 $$R(1+(1-\b_1)R)b^{-1}\chi _Rg_1g_2^\sharp h_2^\sharp \ ,\  R(1+(1-\b_1)R)b^{-1}\chi _Rg_1h_1g_2^\sharp $$
 are strongly $1$--admissible, and 
 $$R(1+(1-\b_1)R)((1-\chi _R)g_1)^\flat g_2h_2$$ is strongly $2$--admissible. Furthermore, $R^2((1-\chi _R)g_1h_1)^\flat g_2$ is $2$--admissible for the same reasons.\\
  The strong admissibility of $R(1+(1-\b_1)R)((1-\chi _R)g_1h_1)^\flat g_2$ requires a 
 specific proof, as follows. We proceed as in the proof of Lemma \ref{mub123}. First of all, the $L^\infty$-bound 
 of $\mu _{\b _2}* R^2((1-\chi _R)g_1h_1)^\flat g_2)$ is a consequence of $L^2*L^2\subset L^\infty $. Then we consider the case $\vert y_1\vert \ge 1$. We split
\bee 
&&\mu _{\b _2}*((1-\chi _R)^\flat g_1^\flat h_1^\flat g_2)(y_2)=\\ &&\int _{\vert y'_2\vert <\frac{\vert y_2\vert}{2}}\mu _{\b _2}(y_2-y'_2) (1-\chi _R)(R+\mu by'_2)g_1(R+\mu by'_2)h_1(R+\mu by'_2)g_2(y'_2)\, dy'_2\\
&+& \int _{\vert y'_2\vert \ge \frac{\vert y_2\vert}{2}}  \mu _{\b _2}(y_2-y'_2) (1-\chi _R)(R+\mu by'_2)g_1(R+\mu by'_2)h_1(R+\mu by'_2)g_2(y'_2)\, dy'_2\ . 
\eee 
In view of decaying properties of $\mu _\b$ and of the $L^\infty $-bound on $(1-\chi _R)g_1h_1$, the first term in the right hand side is bounded by
\bee 
&&\frac{\Vert g_2\Vert _{\b _2}}{\vert y_2\vert (1+(1-\b_2)\vert y_2\vert) R^2(1+(1-\b_1)R)^2}\int _\R\frac{ dy'_2}{(1+\vert y'_2\vert +(1-\b _2)\vert y'_2\vert ^2)}\, dy'_2\\ 
&\lesssim &
 \frac{\vert \log (1-\b _2)\vert }{\vert y_2\vert (1+(1-\b_2)\vert y_2\vert )R^2(1+(1-\b_1)R)^2}\ .
\eee
For the second term, we need the following $L^p$ bound on $\mu _\b$, proved in Lemma \ref{mubLp},
$$\Vert \mu _\b\Vert _{L^p(\R)}\le \frac{C}{p-1}\ ,\ 1<p\le 2.$$
Using this bound and H\"older's inequality, we infer that, for $2\le q<\infty $, the second term is bounded by 
\bee
&&\frac{Cq}{R^2(1+(1-\b_1)R)(1+\vert y_2\vert )(1+(1-\b _2)\vert y_2\vert )}\left (\int _\R \frac{dy'_2}{(1+(1-\b _2)\vert y'_2\vert)^q}\right )^{\frac1q} \\
&&\lesssim \frac{Cq(1-\b _2)^{-1/q}}{R^2(1+(1-\b_1)R)(1+\vert y_2\vert )(1+(1-\b _2)\vert y_2\vert )}\ .
\eee
Optimizing on $q$, we get the bound 
$$ \frac{\vert \log (1-\b _2)\vert }{(1+\vert y_2\vert )(1+(1-\b_2)\vert y_2\vert )R^2(1+(1-\b_1)R)}.$$
We conclude that
$$\Vert \mu _{\b _2}*((1-\chi _R)^\flat g_1^\flat h_1^\flat g_2)\Vert _{\b _2}\lesssim \frac{\vert \log (1-\b _2)\vert }{R^2(1+(1-\b_1)R)}\le \frac{1}{R(1+(1-\b_1)R)}$$
because of the assumption 
\be \label{bootassspe}
1-\b _2\ge e^{-R}
\ee
from \eqref{bootass}. Similar estimates hold for the derivatives. This completes the proof.
\end{proof}
\begin{remark}\label{dissymetry}
Because $b$ is bounded but can be small in the set of parameters $\mathcal O$, there is some asymmetry between bubble 1 and bubble 2, which is reflected by the specificity of the last case in property (vi), for which we had to 
introduce assumption \eqref{bootassspe}.
\end{remark}


\subsection{Continuity of $\mathcal L_\beta^{-1}$ on admissible functions}


We claim a uniform continuity property of $\mathcal L_\beta^{-1}$ with respect to Schwartz-like norms which will be essential to control the error in the construction of the approximate 2-bubble. Recall that
$$\Phi _\b :=y\pa _yQ_\b+(1-\b)\pa _\b Q_\b\ .$$

\begin{lemma}[Generalized invertibility]
\label{generalinvert}
Let $j=1$ or $j=2$,  let $d$ be a nonnegative integer, and $\alpha \in \R $ such that $\vert \alpha \vert <\alpha _*(d)$. If $\eta <\eta _*(d)$ and if $g$  is of the form
$$g(y,\mathcal P)=\sum _{r=-d}^d g_r(y,\mathcal P^*)\, e^{ir\Gamma}\ ,$$
where $\mathcal P^*:=(\l_1,\l_2,\b_1,\b_2,R)$, and each $g_r\ ,\ r=-d,\dots,d,$ is strongly $j$--admissible, then  the problem
$$\mathcal L_{\b _j}f  -i\alpha \pa _\Gamma f=g-iM(\mathcal P)\Lambda Q_{\b _j} +iB(\mathcal P)\Phi _{\b _j}\ ,\ (f,iQ_{\b _j})=(f,\pa_yQ_{\b _j})=0\ ,$$
admits a unique solution $(f,M,B)$, where $M(\mathcal P), B(\mathcal P)$ are real valued, and 
$$f(y,\mathcal P)= \sum _{r=-d}^d f_r(y,\mathcal P^*)\, e^{ir\Gamma}\ ,$$
where each $f_r\ ,\ r=-d,\dots,d,$ is in $H^{\frac 12}$ in the variable $y$. Furthermore, $M,B$ are $L^\infty$-admissible, and $f$ is strongly $j$--admissible. 
\end{lemma}

\begin{proof}
Since $\mathcal L_\b$ is not $\C $--linear, it is preferable to use the Fourier expansion in cosines and sines,
so we write 
\bee
g(y,\mathcal P^*)=g_0(y,\mathcal P^*)+\sum _{r=1}^d[g_r^+(y,\mathcal P^*)\cos (r\Gamma )+g_r^-(y,\mathcal P^*)\sin (r\Gamma)]\ ,\\
f(y,\mathcal P^*)=f_0(y,\mathcal P^*)+\sum _{r=1}^d[f_r^+(y,\mathcal P^*)\cos (r\Gamma )+f_r^-(y,\mathcal P^*)\sin (r\Gamma)]\ ,\\
M(\mathcal P)=M_0(\mathcal P^*)+\sum _{r=1}^d[M_r^+(\mathcal P^*)\cos(r\Gamma)+M_r^-(\mathcal P^*)\sin(r\Gamma)]\ ,\\
B(\mathcal P)=B_0(\mathcal P^*)+\sum _{r=1}^d[B_r^+(\mathcal P^*)\cos(r\Gamma)+B_r^-(\mathcal P^*)\sin(r\Gamma)]\ .
\eee
The problem on $f, M,B$ is therefore equivalent to the following family of problems
\begin{align}
\label{f0}&\mathcal L_{\b _j}f_0=iM_0\Lambda Q_{\b _j}-iB_0\Phi _{\b _j}+g_0\ ,\ (f_0,iQ_{\b _j})=(f_0,\pa _{y}Q_{\b _j})=0\ ,\\
\label{fr}&\left \{ \begin{array}{ll}\mathcal L_{\b _j}f_r^+-i\alpha rf_r^-=iM_r^+\Lambda Q_{\b _j}-iB_r^+\Phi _{\b _j}+g_r^+\ ,\ (f_r^+,iQ_{\b _j})=(f_r^+,\pa _{y}Q_{\b _j})=0\ ,\\ \mathcal L_{\b _j}f_r^-+i\alpha rf_r^+=iM_r^-\Lambda Q_{\b _j}-iB_r^-\Phi _{\b _j}+g_r^-\  ,\ (f_r^-,iQ_{\b _j})=(f_r^-,\pa _{y}Q_{\b _j})=0\ ,
\end{array}\right .
\end{align}
Let us first deal with (\ref{f0}).  Recall from Proposition \ref{invertibility1} that 
$$\ker \mathcal L_{\b _j}=span _{\R} \{ iQ_{\b _j},\pa_{y_j} Q_{\b _j}\} ,$$
and that the range of $\mathcal L_{\b _j}$ coincides with the orthogonal of $span _{\R} \{ iQ_{\b _j},\pa_{y_j} Q_{\b _j}\} .$

Consequently,  the real numbers $M_0, B_0$ must satisfy the
orthogonality conditions
$$(g_0-iM_0\Lambda Q_{\b _j} +iB_0\Phi _{\b _j},iQ_{\b _j})=(g_0-iM_0\Lambda Q_{\b _j} +iB_0\Phi _{\b _j},\pa _{y_j}Q_{\b _j})= 0\ .$$
Notice that, in view of \eqref{varL2b}, \eqref{varHdemib}, 
\bee
(i\Lambda Q_\b, iQ_\b)&=&(\Lambda Q_\b,Q_\b)=0\ ,\\
(i\Phi _\b,iQ_\b)&=&(\Phi_\b,Q_\b)=\frac{1-\b}2\frac{d}{d\b}\Vert Q_\b\Vert_{L^2}^2-\frac 12\Vert Q_\b\Vert_{L^2}^2=-\pi +O((1-\b)\vert \log(1-\b)\vert)\ ,\\
(i\Lambda Q_\b,\pa _yQ_\b)&=&\frac 12(Q_\b,DQ_\b)=\pi +O((1-\b)\vert \log(1-\b)\vert)\ ,\\
(i\Phi _\b,\pa_yQ_\b)&=&\frac{1-\b}{2}\frac{d}{d\b}(Q_\b,DQ_\b)=O((1-\b)\vert \log(1-\b)\vert)\ .
\eee
In view of these identities, we infer that $M_0,B_0$ are characterized for $\b _j$ close enough to $1$ --- hence for $\eta $ small enough ---, given by the following formulae
\begin{align}
B_0&=  \frac{2(g_0,iQ_{\b_j})}{ \Vert Q_{\b _j}\Vert_{L^2}^2-\tL_{\b _j} \Vert Q_{\b _j}\Vert_{L^2}^2}\ ,\  \label{B0}\\
M_0&=  \frac{2(g_0,\pa_yQ_{\b _j})}{(Q_{\b _j},DQ_{\b _j})}+\frac{ 2(g_0,iQ_{\b_j}) \tL_{\b _j} \Vert Q_{\b _j}\Vert_{L^2}^2}{ (Q_{\b _j},DQ_{\b _j})( \Vert Q_{\b _j}\Vert_{L^2}^2-\tL_{\b _j} \Vert Q_{\b _j}\Vert_{L^2}^2)}\ .      \label{M0}
\end{align}
In view of these formulae and of property (v) in Lemma \ref{stabprop}, we conclude that $M_0$ and $B_0$ are $L^\infty $-admissible.\\
Then Proposition \ref{invertibility} provides existence and uniqueness of function $f_0$, as well as the estimate
$$\Vert f_0\Vert _{\b _j}\lesssim \Vert g_0\Vert _{L^2}+\Vert m_{\b _j}*g_0\Vert _{\b _j}\ .$$
Applying inductively $\Lambda _y^p\tL_{\b _j}^q $ to the identity
$$f_0=m_{\b _j}*(iM_0\Lambda Q_{\b _j}-iB_0\Phi _{\b _j}+g_0)+m_{\b _j}*(2\vert Q_{\b _j}\vert ^2f_0+Q_{\b _j}^2\overline f_0)\ ,$$
and using that $\Lambda Q_{\b _j}, \Phi _{\b _j}$ and $g_0$ are strongly $j$--admissible, we conclude from Lemma \ref{lemma:iteration} that $f_0$ is  strongly $j$--admissible.
\medskip

\noindent Let us come to the systems (\ref{fr}). Given $g\in H^{-\frac12}$, define 
\begin{align*}
B[g]&:=  \frac{2(g,iQ_{\b_j})}{ \Vert Q_{\b _j}\Vert_{L^2}^2-\tL_{\b _j} \Vert Q_{\b _j}\Vert_{L^2}^2}\ ,\  \\
M[g]&:=  \frac{2(g_,\pa_yQ_{\b _j})}{(Q_{\b _j},DQ_{\b _j})}+\frac{ 2(g,iQ_{\b_j}) \tL_{\b _j} \Vert Q_{\b _j}\Vert_{L^2}^2}{ (Q_{\b _j},DQ_{\b _j})( \Vert Q_{\b _j}\Vert_{L^2}^2-\tL_{\b _j} \Vert Q_{\b _j}\Vert_{L^2}^2)}\ .      
\end{align*}
and let $$\mathcal L_{\b }^{-1}:H^{-\frac 12}\cap (\ker \mathcal L_\b)^{\perp}\rightarrow H^{\frac 12}\cap (\ker \mathcal L_\b)^{\perp}$$
be the $\R $--linear isomorphism provided by Proposition \ref{invertibility1} . Then the system \eqref{fr} is equivalent to
$$
\left \{ \begin{array}{ll}f_r^+=\mathcal L_{\b_j}^{-1}(g_r^++i\alpha rf_r^-+iM[ g_r^++i\alpha rf_r^- ]\Lambda Q_{\b _j}-iB[g_r^++i\alpha rf_r^-]\Phi _{\b _j}), \\ 
f_r^-=\mathcal L_{\b_j}^{-1}(g_r^--i\alpha rf_r^++iM[g_r^--i\alpha rf_r^+]\Lambda Q_{\b _j}-iB[g_r^--i\alpha rf_r^+]\Phi _{\b _j})\ .\end{array}\right .
$$
The right hand side in the above side defines a mapping of $(f_r^+,f_r^-)\in H^{1/2}\times H^{1/2}$ which is contracting if $\alpha r$ is small enough. This provides existence and uniqueness of $(f_r^+,f_r^-)$ as well as uniform bounds in $H^{1/2}$, and the formulae
$$M_r^+=M[ g_r^++i\alpha rf_r^- ], B_r^+=B[ g_r^++i\alpha rf_r^-], M_r^-=M[ g_r^--i\alpha rf_r^+ ]\ ,\ B_r^-=B[g_r^--i\alpha rf_r^+  ]\ .$$

 The strong $j$--admissibility of $f_r^+$ and $f_r^-$ and the $L^\infty $-admissibility of $M_r^\pm, B_r^\pm $ are then obtained from the system
$$
\left \{ \begin{array}{ll}f_r^+=m_{\b _j}*(iM_r^+\Lambda Q_{\b _j}-iB_r^+\Phi _{\b _j}+g_r^++i\alpha rf_r^-)+m_{\b _j}*(2\vert Q_{\b _j}\vert ^2f_r^++Q_{\b _j}^2\overline {f_r^+})\ ,\\ 
f_r^-=m_{\b _j}*(iM_r^-\Lambda Q_{\b _j}-iB_r^-\Phi _{\b _j}+g_r^--i\alpha rf_r^+)+m_{\b _j}*(2\vert Q_{\b _j}\vert ^2f_r^-+Q_{\b _j}^2\overline {f_r^-})\ , \end{array}\right .
$$
applying again Lemma \ref{lemma:iteration}.
\end{proof}


\subsection{Construction of the approximate solution}


We are now in position to construct the approximate two-bubble solution.

\begin{proposition}[Construction of the two-bubble]
\label{propwtobublle}
Let $N$ be a positive integer, $0<\eta \ll \eta_*(N)$.
We can find an expansion of the slowly modulated two-bubble for $j=1,2$: 
\bee
V_j^{(N)}(y_j,\matchal P)&=&\sum_{n=0}^NT_{j,n}(y_j,\mathcal P), \\
M_j^{(N)}(\matchal P)&=&\sum_{n=0}^NM_{j,n}(\mathcal P), \\
 B_j^{(N)}(\matchal P)&=&\sum_{n=0}^NB_{j,n}(\mathcal P)
 \eee
  such that the following holds:
  \begin{enumerate}
\item {\rm (}Initialization{\rm ).} For $j=1,2$, $T_{j,0}=Q_{\beta_j}(y_j)$, $M_{j,0}=B_{j,0}=0.$\\
\item {\rm (}Control of the error{\rm ).}  Let $0\leq n\leq N$ and $(\mathcal E_{j,n})_{j=1,2}$ be given by \fref{expressioneone}, \fref{expressionetwo} with $V_j= V_j^{(n)}$. Then\\
 $b^{-1}(1+(1-\b _1)R)R^{n+1}\mathcal E_{1,n}$ is strongly $1$--admissible, and $(1+(1-\b _1)R)R^{n+1}\mathcal E_{2,n}$ is strongly $2$--admissible.
\item  {\rm (}Control of the profile{\rm ).} For all $0\leq n\leq N$, $j=1,2$, \\
$b^{-1}(1+(1-\b _1)R)R^nT_{1,n}$ is strongly $1$--admissible, and $(1+(1-\b _1)R)R^nT_{2,n}$ is strongly $2$--admissible.
\item {\rm (}Orthogonality{\rm ).} For $j=1,2$, $n\geq 1$, $(T_{j,n},iQ_{\beta_j})=(T_{j,n},\pa_{y_j}Q_{\beta_j})=0.$
\item {\rm (}Control of the modulation equations{\rm ).} For all $0\le n\le N$,\\ $b^{-1}(1+(1-\b _1)R)R^nB_{1,n}$, $b^{-1}(1+(1-\b _1)R)R^nM_{1,n}$, $(1+(1-\b _1)R)R^nB_{2,n}$,  and $(1+(1-\b _1)R)R^nM_{2,n}$ are $L^\infty$-admissible.
\end{enumerate}
\end{proposition}

\begin{proof}[Proof of Proposition \ref{propwtobublle}] 
We argue by induction on $N$. In order to deal with the dependence on the phase $\Gamma$, we need a more refined description of the error and claim inductively:
\be
\label{controletjn}
T_{j,n}= \sum_{r=-d_n}^{d_n} T_{j,n,r}\, e^{ir\Gamma}
\ee
where $d_n$ is an integer, $b^{-1}(1+(1-\b _1)R)R^nT_{1,n,r}$ is strongly $1$--admissible,\\ $(1+(1-\b _1)R)R^nT_{2,n,r}$ is strongly $2$--admissible,
and they do not depend on $\Gamma$. 
Moreover,
\be
\label{inductionnbis}
\mathcal E_{j,n}=\sum_{r=-d_{n+1}}^{d_{n+1}} \E_{j,n,r}\, e^{ir\Gamma}
\ee
where  $b^{-1}(1+(1-\b _1)R)R^{n+1}\E_{1,n,r}$ is strongly $1$--admissible, $(1+(1-\b _1)R)R^{n+1}\E_{2,n,r}$  is strongly $2$--admissible, and they do not depend on $\Gamma$. Finally,
\begin{align*}
M_{j,n}&= \sum_{r=-d_n}^{d_n} M_{j,n,r}\, e^{ir\Gamma} \\
B_{j,n}&= \sum_{r=-d_n}^{d_n} B_{j,n,r}\, e^{ir\Gamma}
\end{align*}
where $b^{-1}(1+(1-\b _1)R)R^{n}M_{1,n,r}, b^{-1}(1+(1-\b _1)R)R^{n}B_{1,n,r}, (1+(1-\b _1)R)R^{n}M_{2,n,r},$ $(1+(1-\b _1)R)R^{n}B_{2,n,r}$ are $L^\infty$--admissible and do not depend on $\Gamma$ nor $y$.

\noindent{\bf Step 1:} Initialization $N=0$. We inject the decomposition 
$$V_j=V_j^{(0)}=Q_{\beta_j}(y_j), \ \ M_{j,0}=B_{j,0}=0$$ 
$j=1,2$, into the definitions \eqref{expressioneone} and \eqref{expressionetwo}
 of the errors and compute from the equation of $Q_{\beta_j}$:
\bee
\mathcal E_{1,0}&=&\chi_R\left[\frac{2}{\mu}Q_{\beta_1}|Q_{\beta_2}|^2+\frac{e^{-i\Gamma}}{\sqrt{\mu}}Q_{\beta_1}^2\overline{Q_{\beta_2}}+2\frac{e^{i\Gamma}}{\sqrt{\mu}}|Q_{\beta_1}|^2Q_{\beta_2}+\frac{e^{2i\Gamma}}{\mu}\overline{Q_{\beta_1}}Q_{\beta_2}^2\right]\ ,\\
 \mathcal E_{2,0}&=&(1-\chi_R)\left[2\mu \vert Q_{\b _1}\vert ^2Q_{\b _2}+2\sqrt{\mu}e^{-i\Gamma}Q_{\b_1}|Q_{\b_2}|^2+\sqrt{\mu}e^{i\Gamma}\overline{Q_{\b_1}}Q_{\b_2}^2+ \mu e^{-2i\Gamma}Q_{\b _1}^2\overline{Q_{\b _2}}\right].
 \eee
 We now recall from that $Q_{\beta_j}$ is strongly $j$--admissible. Therefore, a direct application of Lemma \ref{stabprop}, property (vi),  ensures that  $b^{-1}(1+(1-\b _1)R)R\matchal E_{1,0}$ is strongly $1$--admissible, and $(1+(1-\b _1)R)R\matchal E_{2,0}$ is strongly $2$--admissible. Notice that we have \eqref{inductionnbis}
 with $n=0$, $d_1=2$, and that the admissibility properties transfer to the Fourier coefficients by integration in the $\Gamma $ variable.

 
 \noindent{\bf Step 2:} Induction. We assume the claim for $N=n$ and prove it for $N=n+1$. We expand
\be
\label{vnvnonplusone}
 V^{(n+1)}_{j}=V^{(n)}_{j}+T_{j,n+1}, \ \ j=1,2
 \ee 
 and show how to choose $(T_{j,n+1},M_{j,n+1},B_{j,n+1})$ so that the corresponding errors $\mathcal E_{j,n+1}$ are such that $b^{-1}(1+(1-\b_1)R)R^{n+2}\mathcal E_{1,n+1}$ is strongly $1$--admissible, and $(1+(1-\b_1)R)R^{n+2}\mathcal E_{2,n+1}$ is strongly $1$--admissible. We focus onto the first bubble, the computations for the second bubble are completely analogous, except that  there is no gain of a $b$ factor.\\
In general, we split the  error term $\mathcal E_1$ into four contributions: the nonlinear term,
\be
\label{nonlinearterm}
{\rm NL}_1=-\frac{(|D|-\beta_1 D)V_{1}}{1-\b_1}-V_{1}+ V_{1}|V_{1}|^2,
\ee
the interaction term,
\be
\label{interactionterm}
{\rm Int}_{1}=  \chi_R\left[\frac{2}{\mu}V_{1}|V_2|^2+\frac{e^{-i\Gamma}}{\sqrt{\mu}}(V_{1})^2 \overline{V_2}
+ 2\frac{e^{i\Gamma}}{\sqrt{\mu}}|V_{1}|^2V_2+\frac{e^{2i\Gamma}}{\mu}\overline{V_{1}}(V_2)^2\right],
\ee
the leading order term for modulation equations,
\be
\label{modulationterm}
 {\rm Mod}_{1}=-iM_1\Lambda V_{1}+iB_1[\Lambda _{y_1}V_1+\tL_{\b _1}V_1] + i\frac{1-\mu}{\mu}\frac{\pa V_{1}}{\pa \Gamma}
\ee
and the lower order term for modulation equations,
\bea
\label{lowerordermodulation}
 {\rm Modlow}_{1}&=& i\l_1(M_1\pa _{\l_1}V_{1}+M_2\pa _{\l_2}V_{1})+i\frac{B_2}{\mu}\tL_{\b_2} V_{1}+ \\
\nonumber  &+&i\left(\frac {1-b}R+B_1-M_1\right )\Lambda _R V_{1}\ .
\eea
Notice that we dropped the notation $V^\sharp $ and $V^\flat $ in these formulae, since the indices $1,2$ unambiguously suggest the arguments $y_1,y_2$.

\noindent{\bf Step 3:} Choice of $T_{1,n+1},M_{1,n+1},B_{1,n+1}$. We inject the decomposition \fref{vnvnonplusone} into \eqref{nonlinearterm} - \eqref{lowerordermodulation} and define $\mathcal E^{(k)}_{1,n+1}$, $k=1,\dots ,4$ by
\bee
&&{\rm NL}_{1,n+1}={\rm NL}_{1,n}-  \L_{\beta_1}T_{1,n+1} + \mathcal E^{(1)}_{1,n+1}\\
&&{\rm Int}_{1,n+1}={\rm Int}_{1,n}+\mathcal E^{(2)}_{1,n+1}\\
&&{\rm Mod}_{1,n+1}={\rm Mod}_{1,n}+\left\{-iM_{1,n+1}\Lambda Q_{\beta_1}+iB_{1,n+1}\Phi_{\beta_1}\right\}+\frac{i(1-\mu)}{\mu}\frac{\pa T_{1,n+1}}{\pa \Gamma}+\mathcal E^{(3)}_{1,n+1}\\
&&{\rm Modlow}_{1,n+1}={\rm Modlow}_{1,n}+\mathcal E^{(4)}_{1,n+1}.
\eee
Therefore 
\bee
\mathcal E_{1,n+1}&=&\mathcal E_{1,n}- \L_{\beta_1}T_{1,n+1}+\frac{i(1-\mu)}{\mu}\frac{\pa T_{1,n+1}}{\pa \Gamma} -iM_{1,n+1}\Lambda Q_{\beta_1}+iB_{1,n+1}\Phi_{\beta_1}\\
&+& \Sigma_{k=1}^4\matchal E_{1,n+1}^{(k)}.
\eee
The smallness assumption on $\eta $ and the definition of $\mathcal O$ imply that $1-\mu $ is small enough with respect to $n$, and we may therefore use Lemma \ref{generalinvert} to solve the equation
$$\mathcal L_{\b _1}T_{1,n+1}+iM_{1,n+1}\Lambda Q_{\b _1}-iB_{1,n+1}\Phi _{\b _1}-i\frac{1-\mu}{\mu }\pa _\Gamma T_{1,n+1}=\mathcal E_{1,n}\ .$$
From the inductive assumption on $\mathcal E_{1,n}$ and Lemma \ref{generalinvert}, we infer that $b^{-1}(1+(1-\b_1)R)R^{n+1}T_{1,n+1}$ is 
strongly $1$--admissible, and that $b^{-1}(1+(1-\b_1)R)R^{n+1}M_{1,n+1}$, $b^{-1}(1+(1-\b_1)R)R^{n+1}B_{1,n+1}$ are $L^\infty$-admissible. Furthermore,
$T_{1,n+1}, M_{1,n+1}, B_{j,n+1}$ are trigonometric polynomials of degree $d_{n+1}$.


\noindent{\bf Step 4:} Estimating $\mathcal{E}^{(1)}_{1,n+1}$. Explicitly:
\bea
\label{defeonnplusone}
\mathcal{E}^{(1)}_{1,n+1}& = & 2\left[|V_1^{(n)}|^2-|Q_{\beta_1}|^2\right]T_{1,n+1}+\left[(V_1^{(n)})^2-Q_{\beta_1}^2\right]\overline{T_{1,n+1}}\\
\nonumber & + & 2V_1^{(n)}|T_{1,n+1}|^2+\overline{V_1^{(n)}}T_{1,n+1}^2+T_{1,n+1}|T_{1,n+1}|^2.
\eea
First of all, we observe that $\mathcal{E}^{(1)}_{1,n+1}$ is a trigonometric polynomial in $\Gamma$, with a degree $d_{n+2}^{(1)}$ depending only on $n$. Secondly, using Lemma \ref{stabprop}, the $1$--admissibility of $b^{-1}(1+(1\b_1)R)R^kT_{1,k} $, and the $2$--admissibility of $(1+(1\b_1)R)R^kT_{2,k} $
for $k\le n+1$, we conclude that $b^{-1}(1+(1-\b_1)R)R^{n+2}\mathcal{E}^{(1)}_{1,n+1}$ is strongly $1$--admissible.


\noindent{\bf Step 5:} Estimating $\mathcal E^{(2)}_{1,n+1}$. First of all, we observe that $\mathcal{E}^{(2)}_{1,n+1}$ is a trigonometric polynomial in $\Gamma$, with a degree $d_{n+2}^{(2)}$ depending only on $n$. We then expand the interaction term ${\rm Int}_{1,n+1}$ \fref{interactionterm}. Notice that each term contains an exchange of variables.  Let us consider the term
 $$\frac{e^{2i\Gamma}}{\mu}\chi_R\overline{V_1^{(n)}} T_{2,n+1}V_2^{(n)}.$$ Recall that $V_j^{(n)}$ is $j$--admissible by the induction assumption, and that \\ $(1+(1-\b_1)R)R^{n+1}T_{2,n+1}$ is $1$ admissible by step 3.  By Lemma \ref{stabprop}, (vi), we infer that 
$$b^{-1}(1+(1-\b_1)R)R^{n+2}\frac{e^{2i\Gamma}}{\mu}\chi_R\overline{V_1^{(n)}} T_{2,n+1}V_2^{(n)}$$ is strongly $1$-admissible. 
 The other terms can be treated similarly. We therefore conclude that $b^{-1}(1+(1-\b_1)R)R^{n+2}\mathcal{E}^{(2)}_{1,n+1}$ is strongly $1$--admissible.

\noindent{\bf Step 6:} Estimating $\mathcal E^{(3)}_{1,n+1}$. Again, $\mathcal{E}^{(3)}_{1,n+1}$ is a trigonometric polynomial in $\Gamma$, with a degree $d_{n+2}^{(3)}$ depending only on $n$.\\
 Let us first observe that the term 
$i\frac{1-\mu}{\mu}\frac{\pa T_{1,n+1}}{\pa \Gamma}$ is absent in 
$\mathcal E^{(3)}_{1,n+1}$ since it is now a part of the equation of $T_{1,n+1}$.
For example, let us deal with the contribution of the term $-iM_1\Lambda V_1$ to $\mathcal{E}^{(3)}_{1,n+1}$.
The other contributions can be handled similarly. We have
$$M_1^{(n+1)}\Lambda V_1^{(n+1)}-M_1^{(n)}\Lambda V_1^{(n)}-M_{1,n+1}\Lambda Q_{\b _1}=M_{1,n+1}\Lambda (V_1^{(n)}-Q_{\b _1})+M_1^{(n+1)}\Lambda T_{1,n+1}\ .$$
Let us consider the first term $M_{1,n+1}\Lambda (V_1^{(n)}-Q_{\b _1})$ in the right hand side. 
By step 3, we know that $b^{-1}(1+(1-\b_1)R)R^{n+1}M_{1,n+1}$  is $L^\infty $-admissible, and independent on $y_1$. On the other hand, $R\Lambda (V_1^{(n)}-Q_{\b _1})$ is strongly $1$--admissible. Hence $b^{-1}(1+(1-\b_1)R)R^{n+2}M_{1,n+1}\Lambda (V_1^{(n)}-Q_{\b _1})$ is strongly $1$--admissible. \\
Let us come to the second term $M_1^{(n+1)}\Lambda T_{1,n+1}$ in the right hand side. From step 3, $b^{-1}(1+(1-\b_1)R)R^{n+1}T_{1,n+1}$ is strongly $1$--admissible, while, from step 3 and the induction hypothesis
$$b^{-1}(1+(1-\b_1)R)RM_1^{(n+1)}=b^{-1}(1+(1-\b_1)R)R\sum _{k=1}^{n+1}M_{1,k}$$
is $L^\infty$-admissible and independent on $y_1$. We infer that $b^{-1}(1+(1-\b_1)R)R^{n+2}M_1^{{n+1}}\Lambda T_{1,n+1}$ is 
strongly $1$--admissible. \\
Summing up, $b^{-1}(1+(1-\b_1)R)R^{n+2}\mathcal{E}^{(3)}_{1,n+1}$ is strongly $1$--admissible. 

\noindent{\bf Step 7:} Estimating $\mathcal E^{(4)}_{1,n+1}$. Finally, we deal with  $b^{-1}(1+(1-\b_1)R)R^{n+2}\mathcal{E}^{(4)}_{1,n+1}$ via the lower order term for modulation equations \fref{lowerordermodulation}. In fact, the worst behavior occurs in this part, and comes  from the term
$$i\frac{1-b}{R}\Lambda _RT_{1,n+1}\ .$$
Indeed, this one only provides a gain of $R$, so we get exactly that 
$$b^{-1}(1+(1-\b_1)R)R^{n+2}i\frac{1-b}{R}\Lambda _RT_{1,n+1}$$
is strongly $1$--admissible. The other terms are easier and left to the reader.

Defining $d_{n+2}:=\max\{ d^{(k)}_{n+2}, k=1,\dots ,4\} $, this completes the proof.

\end{proof}

As a consequence of Proposition \ref{propwtobublle}, we establish some additional estimates which will be useful in Section 5.

\begin{corollary}\label{DPiminusDprimeV}
If $V_j=V_j^{(N)}$ as in Proposition \ref{propwtobublle}, and if $$\pa '\in \{\pa _\Gamma, \Lambda _R, \pa _{\lambda _{j+1}},(1-\b_{j+1})\pa _{\b _{j+1}} \} $$ with
$\{j,j+1\} =\{1,2\}$, we have 
$$\Vert D\Pi ^-\pa 'V_j\Vert _{L^2}\lesssim \frac{1-\b_j}{R}\ .$$
\end{corollary}
\begin{proof}
From Proposition \ref{propwtobublle}, we know  that $V_j$ is $j$--admissible, and that $R(V_j-Q_{\b_j})$ is $j$--admissible. Moreover, $R^{N+1}\mathcal E_j$ is $j$--admissible, and $RM_j, RB_j$ are $L^\infty$- admissible. Consequently, in view of the expressions \eqref{expressioneone}, \eqref{expressionetwo} of $\mathcal E_j$ and of Lemma \ref{stabprop}, we conclude that 
$$\left (\frac{\vert D\vert -\b_jD}{1-\b_j}+1\right )V_j-\vert V_j\vert ^2V_j=F_j\ ,$$
where $RF_j $ is $j$--admissible. Furthermore, since $\pa 'Q_{\b_j}=0$, $R\pa'V_j$ is $j$--admissible, and so is $R\pa '(\vert V_j\vert ^2V_j)$. This implies in particular
$$\left \Vert \left (\frac{\vert D\vert -\b_jD}{1-\b_j}+1\right )V_j\right \Vert _{L^2}\lesssim \frac{1}{R}\ .$$
The proof is completed by observing that the operator
$$D\Pi ^-\left (\frac{\vert D\vert -\b_jD}{1-\b_j}+1\right )^{-1}$$
has a norm $O(1-\b_j)$ on $L^2$. 
\end{proof}
\begin{corollary}\label{5.77}
If $M_2=M_2^{(N)}$ as in Proposition \ref{propwtobublle}, we have
$$\vert \pa _\Gamma M_2\vert +\vert R\pa_RM_2\vert +\sum_{k=1}^2(1-\b_k)\vert \pa_{\b_k}M_2\vert \lesssim \frac{\vert 1-\mu\vert +(1-\b_2)\vert \log(1-\b_2)\vert +R^{-1}}{R(1+(1-\b_1)R)}\ .$$
\end{corollary}
\begin{proof}
Since $R^2(1+(1-\b_1)R)(M_2-M_{2,1})$ is $L^\infty $-admissible from Proposition \ref{propwtobublle}, we just have to prove the estimate for $M_{2,1}$.
From the construction of Proposition \ref{propwtobublle} --- see also the proof of Lemma \ref{generalinvert}, we have 
$$M_{2,1}=\frac{2(\mathcal E_{2,0}+i(1-\mu)\pa_\Gamma T_{2,1},\pa_{y_2}Q_{\b_2})}{(Q_{\b_2},DQ_{\b_2})}+\frac{2(\mathcal E_{2,0}+i(1-\mu)\pa_\Gamma T_{2,1},iQ_{\b_2})\tilde \Lambda_{\b_2}\Vert Q_{\b_2}\Vert _{L^2}^2}{(Q_{\b_2},DQ_{\b_2})(\Vert Q_{\b_2}\Vert _{L^2}^2-\tilde \Lambda_{\b_2}\Vert Q_{\b_2}\Vert _{L^2}^2)}\ .$$
Since $Q_{\b_2}$,  and $R(1+(1-\b_1)R)T_{2,1}$  are $2$--admissible, and since $(Q_{\b_2},DQ_{\b_2})^{-1}$,  $(\Vert Q_{\b_2}\Vert _{L^2}^2-\tilde \Lambda_{\b_2}\Vert Q_{\b_2}\Vert _{L^2}^2)^{-1}$ are $L^\infty $-admissible, the only terms to be estimated are 
$$(\mathcal E_{2,0}, \pa_{y_2}Q_{\b_2}), (\mathcal E_{2,0}, iQ_{\b_2})\tilde \Lambda_{\b_2}\Vert Q_{\b_2}\Vert _{L^2}^2\ ,$$
with
$$\mathcal E_{2,0}= (1-\chi_R)\left (2\mu \vert Q_{\b_1}\vert ^2Q_{\b_2}+2\sqrt{\mu}e^{-i\Gamma}Q_{\b_1}\vert Q_{\b_2}\vert ^2+\sqrt{\mu}e^{i\Gamma}
\overline{Q_{\b_1}}Q_{\b_2}^2+\mu e^{-2i\Gamma}Q_{\b_1}^2\overline{Q_{\b_2}}\right )\ .$$
We already know that $R(1+(1-\b_1)R)\mathcal E_{2,0}$ is $2$--admissible. Furthermore, from Proposition \ref{furtherQbdot}, we have
$$\vert \tilde \Lambda_{\b_2}\Vert Q_{\b_2}\Vert _{L^2}^2\vert +\vert \tilde \Lambda_{\b_2}^2\Vert Q_{\b_2}\Vert _{L^2}^2\vert \lesssim (1-\b_2)\vert \log(1-\b_2)\vert \ .$$
This implies the claimed estimate for $(\mathcal E_{2,0}, iQ_{\b_2})\tilde \Lambda_{\b_2}\Vert Q_{\b_2}\Vert _{L^2}^2$. As for $(\mathcal E_{2,0}, \pa_{y_2}Q_{\b_2})$, since $R(1-\chi_R)Q_{\b_1}$ is $L^\infty$-admissible, we just have to study the contribution of the terms with only one factor $Q_{\b_1}$, namely
$$2\sqrt{\mu}((1-\chi_R)e^{-i\Gamma}Q_{\b_1}\vert Q_{\b_2}\vert^2,\pa_{y_2}Q_{\b_2})+\sqrt{\mu}((1-\chi_R)e^{i\Gamma}\overline{Q_{\b_1}}Q_{\b_2}^2,\pa_{y_2}Q_{\b_2})\ .$$
After integrating by parts, this quantity is equal to
$$-\sqrt{\mu}\, {\rm Re}\left (e^{-i\Gamma}\int_{\R} \pa_{y_2}((1-\chi_R)Q_{\b_1})\vert Q_{\b_2}\vert ^2\overline{Q_{\b_2}}\, dy_2   \right )\ .$$
Since $R^2(1+(1-\b_1)R)\pa_{y_2}((1-\chi_R)Q_{\b_1})$ is $L^\infty$-admissible, this completes the proof.
\end{proof}

\subsection{Improved decay for $T_{2,1}$}

In this subsection, we improve some estimates of the first correction $T_{2,1}$ to $Q_{\b _2}$ in the approximate solution we have constructed in the previous paragraph.

\begin{lemma}\label{lemma:pa_Gamma_T21}
We have
$$\Big(i\frac{\pa T_{2,1}}{\pa \Gamma}, \pa_{y_2} Q_{\b_2}\Big)=-2\pi\Re (e^{i\Gamma}\overline{Q_{\b_1}(R)})
+O\Big(\frac{|1-\mu|+(1-\b_2)^{1/2}|\log(1-\b _2)\vert ^{1/2}+R^{-1}}{R(1+(1-\b_1)R)}\Big).
$$

\end{lemma}

\begin{proof}
Writing $i\pa_{y_2}Q_{\b_2}=\mathcal L_{\b_2}(i\rho_{\b_2})$, we have
\begin{align}\label{initial}
&(i\pa_\Gamma T_{2,1}, \pa_{y_2}Q_{\b_2})=-(\pa_\Gamma T_{2,1}, i \pa_{y_2}Q_{\b_2})=-(\pa_\Gamma T_{2,1}, \mathcal{L}_{\b_2}i\rho_{\b_2})=-(\pa_\Gamma \mathcal L_{\b_2} (T_{2,1}), i\rho_{\b_2})\notag\\
&=-\Big(\pa_{\Gamma}\mathcal{E}_{2,0}-i\pa_{\Gamma} M_{2,1}\Lambda Q_{\b_2}+i\pa_{\Gamma} B_{2,1}(y_2\pa_{y_2}Q_{\b_2}+(1-\b_2)\pa_{\b_2}Q_{\b_2})+i\frac{1-\mu}{\mu}\pa_\Gamma^2 T_{2,1}, i\rho_{\b_2}\Big)\notag\\
&=I+II+III+IV
\end{align}

\noindent
For IV, we have by Proposition \ref{propwtobublle} that
\be\label{IV}
|IV|\lesssim \frac{|1-\mu|}{R(1+(1-\b_1)R)}\ .
\ee

\noindent
For III,  we have by Proposition \ref{propwtobublle} that $|\pa_{\Gamma}B_{2,1}|\lesssim\frac{1}{R(1+(1-\b_1)R)}$. Then,
\begin{align*}
|III|&=\Big|\Big(i\pa_{\Gamma} B_{2,1}(y_2\pa_{y_2}Q_{\b_2}+(1-\b_2)\pa_{\b_2}Q_{\b_2}), i\rho_{\b_2}\Big)\Big|\\
&\lesssim \Big(\big|(iy_2\pa_{y_2}Q_{\b_2}, i\rho_{\b_2})\big|+(1-\b_2)\big|(i\pa_{\b_2}Q_{\b_2},i\rho_{\b_2})\big|\Big)\frac{1}{R(1+(1-\b_1)R)}
\end{align*}

\noindent
Using Proposition \ref{furtherQbdot} and \eqref{irho}, 
\begin{equation}\label{1}
(1-\b_2)\big|(i\pa_{\b_2}Q_{\b_2},i\rho_{\b_2})\big|\lesssim (1-\b_2)|\log(1-\b_2)|.
\end{equation}

\noindent
Then, by \eqref{irho}, \eqref{QbQ+} and the identity
$$y\pa _yQ^+=Q^++\frac i2\pa_yQ^+$$
we have
\begin{align}
(iy_2\pa_{y_2}Q_{\b_2}, i\rho_{\b_2})
&=(iy_2\pa_{y_2}Q_{\b_2}, Q_{\b_2}+\frac i2\pa _{y_2}Q_{\b _2})+ O((1-\b_2)^{\frac 12}\vert \log(1-\b _2)\vert ^{\frac 12})\notag\\
&=(iy\pa_y Q^+, Q^++\frac i2\pa_yQ^+)+O((1-\b_2)^{\frac 12}\vert \log(1-\b _2)\vert ^{\frac 12}) \notag\\
&=O((1-\b_2)^{\frac 12}\vert \log(1-\b _2)\vert ^{\frac 12})\ .\label{y_pa_y_Q_rho}
\end{align}

\noindent
Thus, we conclude that 
\begin{equation}\label{III}
|III|\lesssim \frac{(1-\b_2)^{1/2}|\log(1-\b _2)|^{1/2}}{R(1+(1-\b_1)R)}.
\end{equation}

\noindent
For II,  we have by Proposition \ref{propwtobublle} that $|\pa_{\Gamma}M_{2,1}|\lesssim\frac{1}{R(1+(1-\b_1)R)}$. 
Then, by \eqref{y_pa_y_Q_rho} and \eqref{irho} :
\begin{align}
(i\Lambda Q_{\b_2}, i\rho_{\b_2})
&=\frac 12(iQ_{\b_2}, i\rho_{\b_2})+(iy_2\pa_{y_2}Q_{\b_2}, i\rho_{\b_2})\label{Lambda_Q_rho}\\
&=\frac 12(iQ_{\b_2},Q_{\b_2})
+\frac 14 (iQ_{\b_2},i\pa_{y_2}Q_{\b_2})
+O((1-\b_2)^{\frac 12}\vert \log(1-\b _2)\vert ^{\frac 12})\\
&=O((1-\b_2)^{\frac 12}\vert \log(1-\b _2)\vert ^{\frac 12}).\notag
\end{align}

\noindent
Therefore,
\begin{equation}\label{II}
|II|\lesssim \frac{(1-\b_2)^{\frac 12}\vert \log(1-\b _2)\vert ^{\frac 12}}{R(1+(1-\b_1)R)}.
\end{equation}

\noindent
Finally, for I, we have that
\begin{align*}
I&=-(\pa_\Gamma \mathcal{E}_{2,0}, i\rho_{\b_2})\\
&=-\Big ((1-\chi_R)[-2i\sqrt{\mu}e^{-i\Gamma}Q_{\b_1}|Q_{\b_2}|^2+i\sqrt{\mu}e^{i\Gamma}\overline{Q_{\b_1}}Q_{\b_2}^2-2i\mu e^{-2i\Gamma}Q_{\beta_1}^2\overline{Q_{\beta_2}}], i\rho_{\b_2}\Big)\\
&=-\sqrt{\mu }\Re \Big(ie^{i\Gamma}\int (1-\chi_R)\overline{Q_{\b_1}(y_1)}[2|Q_{\b_2}|^2i\rho_{\b_2}+Q_{\b_2}^2\overline{i\rho_{\b_2}}]dy_2\Big)
+O\Big(\frac{1}{R^2(1+(1-\b_1)R)^2}\Big)\\
&=\sqrt{\mu }\Im\Big(e^{i\Gamma}\int_{|y_2|\leq \frac{R}{2b\mu}} \overline{Q_{\b_1}(y_1)}[2|Q_{\b_2}|^2i\rho_{\b_2}+Q_{\b_2}^2\overline{i\rho_{\b_2}}]dy_2\Big)
+O\Big(\frac{1}{R^2(1+(1-\b_1)R)^2}\Big).
\end{align*}

\noindent
Let $z_2:=\frac{b\mu y_2}{R}.$
We then Taylor expand for $|z_2|\leq \frac 12$, or equivalently $|y_2|\leq\frac{R}{2b\mu}$, and obtain by Proposition \ref{boundQb}:
\bee
Q_{\beta_1}(y_1)& =& Q_{\beta_1}\left(R(1+z_2)\right)=Q_{\beta_1}(R)-\int_0^1Rz_2\pa_{y_1}Q_{\beta_1}\left(R(1+tz_2)\right)dt\\
& = & Q_{\beta_1}(R)+O\left(\frac{R|z_2|}{R^2(1+(1-\beta_1)R)}\right)=  Q_{\beta_1}(R)+O\left(\frac{b |y_2|}{R^2(1+(1-\beta_1)R)}\right)
\eee
\noindent
Therefore,
\begin{align}\label{I}
I&=\sqrt{\mu }\Im\Big(e^{i\Gamma}\overline{Q_{\b_1}(R)}\int [2|Q_{\b_2}|^2i\rho_{\b_2}+Q_{\b_2}^2\overline{i\rho_{\b_2}}](y_2)dy_2\Big)+O\Big(\frac{1}{R^2(1+(1-\b_1)R)}\Big).
\end{align}

\noindent
Using \eqref{irho} and Lemma \ref{lemmaalgebraSzeg\H{o}}, we have that
\begin{align}\label{3}
\int [2|Q_{\b_2}|^2&i\rho_{\b_2}+Q_{\b_2}^2\overline{i\rho_{\b_2}}](y_2)dy_2
=3\int |Q_{\b_2}|^2Q_{\b_2}dy_2+i\int |Q_{\b_2}|^2\pa_{y_2}Q_{\b_2}dy_2\notag\\
&-\frac{i}{2}\int Q_{\b_2}^2\overline{\pa_{y_2}Q_{\b_2}}dy_2
+O((1-\b_2)^{\frac 12}\vert \log(1-\b _2)\vert ^{\frac 12})\notag\\
&=3\int |Q^+|^2Q^+dy+i\int |Q^+|^2\pa_{y}Q^+dy-\frac{i}{2}\int (Q^+)^2\overline{\pa_{y}Q^+}dy
+O((1-\b_2)^{\frac 12}\vert \log(1-\b _2)\vert ^{\frac 12})\notag\\&=-6\pi i+2\pi i+2\pi i + O((1-\b_2)^{\frac 12}\vert \log(1-\b _2)\vert ^{\frac 12})\notag\\
&=-2\pi i + O((1-\b_2)^{\frac 12}\vert \log(1-\b _2)\vert ^{\frac 12}).
\end{align}

\noindent
Then,
\begin{align*}
I&=-2\pi \Re\Big(e^{i\Gamma}\overline{Q_{\b_1}(R)}\Big)+O\Big(\frac{|1-\mu|+(1-\b_2)^{\frac 12}\vert \log(1-\b _2)\vert ^{\frac 12}}{R(1+(1-\b_1)R)}\Big)+O\Big(\frac{1}{R^2(1+(1-\b_1)R)}\Big).
\end{align*}

\noindent
Combining this with \eqref{IV}, \eqref{initial}, \eqref{III} and \eqref{II}, the conclusion of the 
lemma follows.

\end{proof}

\begin{lemma}\label{lemma:pa_R_T21}
We have 
\begin{align*}
\Big(i\frac{\pa T_{2,1}}{\pa R}, \pa_{y_2} Q_{\b_2}\Big)&=-2\pi\Im (e^{i\Gamma}\overline{\pa_{y_1}Q_{\b_1}(R)})+O\Big(\frac{|1-\mu|+
(1-\b_2)^{\frac 12}\vert \log(1-\b _2)\vert ^{\frac 12}}{R^2(1+(1-\b_1)R)}\Big)\\
&+O\Big(\frac{1}{R^3(1+(1-\b_1)R)}\Big).
\end{align*}

\end{lemma}

\begin{proof}
The proof follows the same lines as the above one.
With the same notation as above, we have
\begin{align}\label{initialbis}
&(i\pa_R T_{2,1}, \pa_{y_2}Q_{\b_2})\notag\\
&=-\Big(\pa_R\mathcal{E}_{2,0}-i\pa_R M_{2,1}\Lambda Q_{\b_2}+i\pa_R B_{2,1}(y_2\pa_{y_2}Q_{\b_2}+(1-\b_2)\pa_{\b_2}Q_{\b_2}), i\rho_{\b_2}\Big)\notag\\
&-\Big(i\frac{1-\mu}{\mu}\pa_\Gamma \pa_R T_{2,1}, i\rho_{\b_2}\Big)\notag\\
&=V+VI+VII+VIII.
\end{align}

\noindent
By Proposition \ref{propwtobublle}, we have that
\be\label{VIII}
|VIII|\lesssim \frac{|1-\mu|}{R^2(1+(1-\b_1)R)}
\ee 

\noindent
Using Proposition \ref{propwtobublle}, we have 
that $|\pa_R B_{2,1}|\lesssim \frac{1}{R^2(1+(1-\b_1)R)}$.
Then, it follows by \eqref{1} and \eqref{y_pa_y_Q_rho} that
\[|VII|\lesssim \frac{(1-\b_2)^{\frac 12}\vert \log(1-\b _2)\vert ^{\frac 12}}{R^2(1+(1-\b_1)R)}.\]

\noindent
Since $|\pa_R M_{2,1}|\lesssim \frac{1}{R^2(1+(1-\b_1)R)}$ by Proposition \ref{propwtobublle}, we have according to \eqref{Lambda_Q_rho} that
\be\label{VI}
|VI|\lesssim \frac{(1-\b_2)^{\frac 12}\vert \log(1-\b _2)\vert ^{\frac 12}}{R^2(1+(1-\b_1)R)}
\ee

\noindent
Lastly, by \eqref{3} we have that $\int (2|Q_{\b_2}|^2i\rho_{\b_2}+Q_{\b_2}^2\overline{i\rho_{\b_2}})dy_2=-2\pi i +O((1-\b_2)^{\frac 12}\vert \log(1-\b _2)\vert ^{\frac 12})$, and thus
\begin{align*}
V&=-(\pa_R \mathcal E_{2,0}, i\rho_{\b_2})=\sqrt{\mu}\Re e^{i\Gamma}\int (1-\chi_R) \overline{\pa_{y_1}Q_{\b_1}}(2|Q_{\b_2}|^2i\rho_{\b_2}+Q_{\b_2}^2\overline{i\rho_{\b_2}})dy_2\\
&+O\left (\int \chi'\left (1+\mu \frac {by_2}{R}\right )\frac{\mu b|y_2|}{R^2}(|Q_{\b_1}|^2+|Q_{\b_1}||Q_{\b_2}|)|Q_{\b_2}i\rho_{\b_2}|dy_2\right )\\
&+O\Big(\int (1-\chi_R)|Q_{\b_1}||\pa_{y_1}Q_{\b_1}||Q_{\b_2}i\rho_{\b_2}|dy_2\Big)\\
&=-2\pi\Im (e^{i\Gamma}\overline{\pa_{y_1}Q_{\b_1}(R)})
+O\left (\frac{(1-\b_2)^{\frac 12}\vert \log(1-\b _2)\vert ^{\frac 12}}{R^2(1+(1-\b_1)R)}\right )\\
&+O\left (\frac{1}{R^3(1+(1-\b_1)R)}\right ).
\end{align*}

\end{proof}


\begin{lemma}\label{lemma:pa_b_T21} We have
\be \label{deg2}
\Big|(1-\beta_2)\Big(i\frac{\pa T_{2,1}}{\pa \beta_2},\pa_{y_2}Q_{\beta_2}\Big)\Big| \notag \\
\lesssim  \frac{(1-\b_2)^\frac{1}{2}\vert \log (1-\b_2)\vert ^{\frac12}+|1-\mu|}{R(1+(1-\b_1)R)}+\frac{\sqrt b}{R^2(1+(1-\b_1)R)}\ .
\ee

\end{lemma}

\begin{proof}
Using the symmetry of $\mathcal{L}_{\b_2}$ with respect to the real scalar product, we write

\begin{align}\label{1S10}
(1-\b_2)\Big(i\frac{\pa T_{2,1}}{\pa \b_2}, &\pa_{y_2}Q_{\b_2}\Big)
=-(1-\b_2)\Big(\frac{\pa T_{2,1}}{\pa \b_2}, i\pa_{y_2}Q_{\b_2}\Big)\\
&=-\Big((1-\b_2)\frac{\pa T_{2,1}}{\pa \b_2}, \mathcal{L}_{\b_2}(i\rho_{\b_2})\Big)
=-\Big((1-\b_2)\mathcal{L}_{\b_2}\Big(\frac{\pa T_{2,1}}{\pa \b_2}\Big), i\rho_{\b_2}\Big)\notag\\
&=-(1-\b_2)\Big(\pa_{\b_2}(\mathcal{L}_{\b_2}T_{2,1})+2\pa_{\b_2}(|Q_{\b_2}|^2)T_{2,1}+\pa_{\b_2}(Q_{\b_2}^2)\overline{T_{2,1}}, i\rho_{\b_2}\Big)\notag\\
&-\Big(\frac{2DT_{2,1}^-}{1-\b_2}, i\rho_{\b_2}\Big).\notag
\end{align}

\noindent
We start by estimating the last term. Firstly,
\begin{align*}
\left|\Big(\frac{2DT_{2,1}^-}{1-\b_2}, i\rho_{\b_2}\Big)\right| 
\leq \left\|\frac{2DT_{2,1}^-}{1-\b_2}\right\|_{L^2}\|i\rho^-_{\b_2}\|_{L^2}.
\end{align*}
Projecting  the equation satisfied by $T_{2,1}$ onto
negative frequencies,
we obtain:
\begin{align*}
-\frac{1+\b_2}{1-\b_2}DT_{2,1}^- &+T_{2,1}^--2\Pi^- (|Q_{\b_2}|^2T_{2,1})
-\Pi^- (Q_{\b_2}^2\overline{T_{2,1}})\\
&=\Pi^-(\mathcal E_{2,0})-iM_{2,1}\Pi^- (\Lambda Q_{\b_2})
+i B_{2,1} \Pi^- (y_2\pa_{y_2}Q_{\b_2}+(1-\b_2)\pa_{\b_2}Q_{\b_2})\\
&+i(1-\mu)\pa_\Gamma T_{2,1}^-
\end{align*}
and therefore, using the $2$--admissibility of $R(1+(1-\b_1)R)T_{2,1}$
and $R(1+(1-\b_1)R)\mathcal E_{2,0}$, as well as the $L^\infty$-admissibility $R(1+(1-\b_1)R)B_{2,1}$ and $R(1+(1-\b_1)R)M_{2,1}$, we infer 
\[\left\|\frac{2DT_{2,1}^-}{1-\b_2}\right\|_{L^2}\lesssim \frac 1{R(1+(1-\b_1)R)}.\]
On the other hand, by \eqref{irho}, we have
\begin{align*}
\|i\rho_{\b_2}^-\|_{L^2}=\left\|Q_{\b_2}^-+\frac i2\pa_{y_2}Q_{\b_2}^-\right\|_{L^2}+O((1-\b_2)^{\frac 12}\vert \log(1-\b_2)\vert^{1/2})\lesssim (1-\b_2)^{\frac 12}\vert \log(1-\b_2)\vert^{1/2}.
\end{align*}
This shows that
\be\label{DT_21^-}
\left|\Big(\frac{2DT_{2,1}^-}{1-\b_2}, i\rho_{\b_2}\Big)\right| \lesssim \frac{(1-\b_2)^{\frac 12}\vert \log(1-\b_2)\vert^{1/2}}{R(1+(1-\b_1)R)}.
\ee
Then, by \eqref{irho}, we easily notice that, for every $p\in (2,\infty)$,
\begin{align}
\left|\Big(2\pa_{\b_2}(|Q_{\b_2}|^2)T_{2,1}+\pa_{\b_2}(Q_{\b_2}^2)\overline{T_{2,1}}, i\rho_{\b_2}\Big)\right|
&\lesssim \frac{\Vert \dot Q_{\b_2}^+\Vert_{L^2} +\Vert \dot Q_{\b_2}^-\Vert_{L^p}+\Vert \dot Q_{\b_2}^-\Vert_{L^2}(1-\b_2)^{1/2}\vert \log(1-\b_2)\vert^{1/2}}{R(1+(1-\b_1)R)}\notag \\
&\lesssim \frac{\vert \log(1-\b_2)\vert +p(1-\b_2)^{-1/p}+ \vert \log(1-\b_2)\vert^{1/2}}{R(1+(1-\b_1)R)}\ ,\notag
\end{align}
where we have used \eqref{Qbdotplus} and \eqref{Qbdotmoins} combined to the Hausdorff--Young inequality. Choosing $p=\vert \log(1-\b_2)\vert$, we conclude
\be \label{pa_bQT_21}
(1-\b_2)\left|\Big(2\pa_{\b_2}(|Q_{\b_2}|^2)T_{2,1}+\pa_{\b_2}(Q_{\b_2}^2)\overline{T_{2,1}}, i\rho_{\b_2}\Big)\right|\lesssim \frac{(1-\b_2)\vert \log(1-\b_2)\vert}{R(1+(1-\b_1)R)}\ .
\ee
Finally, we deal with the term
\[(1-\b_2)\Big(\pa_{\b_2}(\mathcal{L}_{\b_2}T_{2,1}), i\rho_{\b_2}\Big).\]
Recalling the equation of $T_{2,1}$, we have:
\begin{align}
\label{eqtbetzdeux}
\pa_{\b_2}\big(\mathcal L_{\beta_2}T_{2,1}\big)&=\pa_{\b_2}\mathcal E_{2,0}-i\pa_{\b_2}M_{2,1}\Lambda Q_{\beta_2}
-iM_{2,1}\Lambda \pa_{\b_2}Q_{\beta_2}\\
&+i\pa_{\b_2}B_{2,1}\left[y_2\pa_{y_2}Q_{\beta_2}+(1-\beta_2)\frac{\pa Q_{\beta_2}}{\pa \beta_2}\right]\notag\\
&+iB_{2,1}\left[y_2\pa_{y_2}\pa_{\b_2}Q_{\beta_2}-\pa_{\b_2}Q_{\b_2}+(1-\beta_2)\pa_{\b_2}^2Q_{\beta_2}\right]\notag\\
&+i(1-\mu )\pa_\Gamma \pa_{\b_2}T_{2,1}\notag
\end{align}
 with 
 $$
 \mathcal E_{2,0}=  (1-\chi_R)\left[2\sqrt{\mu}e^{-i\Gamma}Q_{\b_1}|Q_{\b_2}|^2+\mu e^{-2i\Gamma}Q_{\b_1}^2\overline{Q_{\b_2}}+2\mu Q_{\b_2}|Q_{\b_1}|^2+\sqrt{\mu}e^{i\Gamma}\overline{Q_{\b_1}}Q_{\b_2}^2\right].
 $$
Because of  Proposition \ref{propwtobublle}, we have the pointwise bound on the Fourier coefficients of $\mathcal E_{2,0}$:
 \be
 \label{poitboundone}
\sum _\pm\sum _{r=1}^2 |\mathcal E_{2,0,r}^\pm (y_2) |+\vert \mathcal E_{2,0,0}(y_2)\vert \lesssim \frac 1{R(1+(1-\b_1)R)\la y_2\ra(1+(1-\beta_2)
\vert y_2\vert)}.
 \ee
 Using the fact that $\frac{\pa y_1}{\pa\b_2}=-\frac{\mu y_2}{1-\b_1}$, we also have the pointwise bound
 \begin{align} 
& \sum _\pm\sum _{r=1}^2 |\pa _{\b_2}\mathcal E_{2,0,r}^pm (y_2) |+\vert \pa_{\b_2} \mathcal E_{2,0,0}(y_2)\vert \\
&\lesssim \frac{|y_2|}{(1-\b_1)R}\pmb{1}_{|y_1|\sim R, |y_2|\sim \frac Rb}\Big(|Q_{\b_1}|
|Q_{\b_2}|^2+|Q_{\b_1}|^2|Q_{\b_2}|\Big)\notag\\
&+\frac{|y_2|}{1-\b_1}\pmb{1}_{|y_1|\geq \frac R4}|\pa_{y_1}Q_{\b_1}|\Big(
|Q_{\b_1}||Q_{\b_2}|+|Q_{\b_2}|^2\Big)\notag\\
&+\pmb{1}_{|y_1|\geq \frac R4}|\pa_{\b_2}Q_{\b_2}|\Big(|Q_{\b_1}|^2+|Q_{\b_1}||Q_{\b_2}|\Big)\notag \\
&=IX+X+XI\ .\notag
 \end{align}
 Using the bounds \eqref{boundQb} on $Q_\b$ and \eqref{Qbdotplus}, \eqref{Qbdotmoins} combined with Hausdorff--Young yield
 \begin{align}  
 &\Vert IX\Vert _{L^2} +\Vert X\Vert _{L^2}\lesssim \frac{1}{(1-\b_1)\sqrt{b}R^2(1+(1-\b_1)R)}\ ,\label{poitboundonebis}\\
 &\Vert XI\Vert _{L^2+L^p}\lesssim  \frac{\vert \log(1-\b_2)\vert+p(1-\b_2)^{-1/p}}{R(1+(1-\b_1)R)}\ ,\ 2\le p<\infty \ .\label{poitboundoneter}
 \end{align}
 We are going to use this to estimate $\pa_{\b_2} B_{2,1}$
 and $\pa_{\b_2}M_{2,1}$.
Recall that
$$B_{2,1,r}^\pm = -\frac{(\matchal E_{2,0,r}^\pm +i(1-\mu)T_{2,1,r}^\pm ,iQ_{\beta_2})}{\big(iy_2\pa_{y_2}Q_{\beta_2}+i (1-\beta_2)\frac{\pa Q_{\beta_2}}{\pa \beta_2},iQ_{\beta_2}\big)},$$ 
and a similar identity for $B_{2,1,0}$. 
Taking the derivative with respect to $\b _2$ and using \eqref{poitboundone}, 
we have
\begin{align*}
|\pa_{\b_2} B_{2,1,r}^\pm |
&\lesssim \left|\left(\pa_{\b_2}\mathcal E_{2,0,r}^\pm +i(1-\mu )\pa_{\b_2}T_{2,1,r}^\pm , iQ_{\b_2}\right)\right|\\
&+\left|(\matchal E_{2,0,r}^\pm +i(1-\mu)T_{2,1,r}^\pm ,i\pa_{\b_2}Q_{\beta_2})\right|\\
&+\frac 1{R(1+(1-\b_1)R)} \left|\big(iy_2\pa_{y_2}\pa_{\b_2}Q_{\beta_2}-i\pa_{\b_2}Q_{\b_2}+i (1-\beta_2)\pa_{\b_2}^2Q_{\beta_2},iQ_{\beta_2}\big)\right|\\
&+\frac 1{R(1+(1-\b_1)R)}\left|\big(iy_2\pa_{y_2}Q_{\beta_2}+i (1-\beta_2)\pa_{\b_2}Q_{\beta_2},i\pa_{\b_2}Q_{\beta_2}\big)\right|\ .
\end{align*}
We estimate the inner products in the right hand side of the above inequality as follows. Notice that, from the admissibility properties and \eqref{poitboundone}, for every $q\in (1,2]$,  
$$\Vert Q_{\b _2}\Vert_{L^q}\lesssim \frac{1}{q-1}\ ,\ \Vert T_{2,1,r}^\pm \Vert _{L^q}+\Vert \mathcal E_{2,1,r}^\pm \Vert _{L^q}\lesssim \frac{1}{(q-1)R(1+(1-\b_1)R)}\ .$$
Given $p\in [2,\infty)$, using \eqref{poitboundonebis}, \eqref{poitboundoneter}, \eqref{Qbdotplus}, \eqref{Qbdotmoins},  H\"older's inequality  leads to
\bee
\vert (\pa _{\b _2}\matchal E_{2,0,r}^\pm ,iQ_{\b _2})\vert &\lesssim & \frac{1}{(1-\b_1)\sqrt{b}R^2(1+(1-\b_1)R)}+ \frac{p\vert \log (1-\b_2)\vert +p^2(1-\b_2)^{-1/p}}{R(1+(1-\b_1)R)}   \\
\vert (\matchal E_{2,0,r}^\pm ,i\pa_{\b _2}Q_{\b _2})\vert &\lesssim &  \frac{p\vert \log (1-\b_2)\vert+p^2(1-\b_2)^{-1/p}}{R(1+(1-\b_1)R)}   \\
\vert (T_{2,1,r}^\pm ,i\pa_{\b_2}Q_{\beta_2})\vert &\lesssim & \frac{p\vert \log (1-\b_2)\vert+ p^2(1-\b_2)^{-1/p}}{R(1+(1-\b_1)R)}\\
\vert (\pa _{\b _2}T_{2,1,r}^\pm ,iQ_{\b _2})\vert &\lesssim &\frac{1}{(1-\b_2)R(1+(1-\b_1)R)}\ .
\eee
The other inner products are estimated thanks to \eqref{Hbinners}.
Choosing $p=\vert \log(1-\b_2)\vert$ in the above inequalities, we infer
$$(1-\b_2)|\pa_{\b_2} B_{2,1,r}^\pm |\lesssim   \frac{\sqrt b}{R^2(1+(1-\b_1)R)}+\frac{(1-\b_2)(\log(1-\b_2))^2}{R(1+(1-\b_1)R)}+\frac{\vert 1-\mu\vert}{R(1+(1-\b_1)R)}\ . $$
We obtain the same estimate for $(1-\b_2)\pa _{\b_2}B_{2,1,0}$.\\ 
Arguing analogously, we obtain
$$(1-\b_2)|\pa_{\b_2} M_{2,1,r}^\pm |\lesssim   \frac{\sqrt b}{R^2(1+(1-\b_1)R)}+\frac{(1-\b_2)(\log(1-\b_2))^2}{R(1+(1-\b_1)R)}+\frac{\vert 1-\mu\vert}{R(1+(1-\b_1)R)}\ . $$
Putting together the  above estimates and using  the fact that $|B_{2,1}|+|M_{2,1}|\lesssim \frac 1{R(1+(1-\b_1)R)}$, we obtain from \eqref{eqtbetzdeux}:
\begin{align*}
\left|(1-\b_2)\Big(\pa_{\b_2}(\mathcal{L}_{\b_2}T_{2,1}), i\rho_{\b_2}\Big)\right|
&\lesssim  \frac{\sqrt b}{R^2(1+(1-\b_1)R)}+\frac{(1-\b_2)(\log(1-\b_2))^2+\vert 1-\mu\vert }{R(1+(1-\b_1)R)}\ .
\end{align*}
This together with \eqref{1S10}, \eqref{DT_21^-}, and \eqref{pa_bQT_21} show that
$$
\Big|(1-\beta_2)\Big(i\frac{\pa T_{2,1}}{\pa \beta_2},\pa_{y_2}Q_{\beta_2}\Big)\Big|\lesssim  \frac{(1-\b_2)^\frac{1}{2}\vert \log (1-\b_2)\vert ^{\frac12}+|1-\mu|}{R(1+(1-\b_1)R)}+\frac{\sqrt b}{R^2(1+(1-\b_1)R)}\ ,
$$
which proves \eqref{deg2}.

\end{proof}


\subsection{Sharp modulation equations}


We now compute explicitly the leading order modulation equations. We need to exhibit some fine cancellations which could be computed to the expense of lengthy computations\footnote{because we need the cancellation to the order 2 in the scaling law.} which can be avoided using the following nonlinear algebra. 

Before stating the result, let us define some more notation.
We set
\be \label{NbPb}
N_\b :=\frac{1}{2\pi}\Vert Q_\b\Vert _{L^2}^2\ ,\ P_\b :=\frac{1}{2\pi}(DQ_\b,Q_\b)\ .
\ee
and we recall that
$$c_\b :=\frac{i}{2\pi}\int _\R \vert Q_\b(y)\vert ^2Q_\b(y)\, dy\ .$$
and the asymptotics from Proposition \ref{furtherQbdot},
\bee
N_\b&=&1+O((1-\b)\log(1-\b))\ , \tL_{\b} N_\b =O((1-\b)\log (1-\b))\ ,\\
\ P_\b&=&1+O((1-\b)\log(1-\b))      \ ,\ c_\b= 1+O((1-\b)\log(1-\b))        \ .
\eee
\begin{proposition}[Sharp modulation equations]\label{sharpmod}
Let $B_j^{(N)}, M_j^{(N)}$ be defined by Proposition \ref{propwtobublle}. The following estimates hold for $\mathcal P\in \mathcal O$.
\bea
B_1^{(N)}&=&2\frac{\Re \Big(Q_{\b_2}\Big(-\frac{R}{b\mu}\Big)\, \overline {c_{\b_1}}\, e^{i\Gamma}\Big)}{N_{\b _1}-\tL_{\b _1}N_{\b _1}}
+O\left (\frac{ b(\vert 1-\mu\vert+R^{-1})}{ R(1+(1-\b_1)R)}\right )\ , \label{sharpmodB1}
\\
B_2^{(N)}&=&2\frac{ \Re \Big(Q_{\b_1}(R)\, \overline{c_{\b _2}}\, e^{-i\Gamma}\Big)}
{N_{\b _2}-\tL_{\b _2}N_{\b _2}}
+O\left (\frac{\vert 1-\mu\vert +R^{-1}}{ R(1+(1-\b_1)R)}\right )\ .\label{sharpmodB2}
\eea
\bea
&&M_1^{(N)}-\frac{\tL_{\b_1}P_{\b_1}}{P_{\b_1}}B_1^{(N)}
=O\left ( \frac{b(|1-\mu |+R^{-1})}{R(1+(1-\b_1)R)}\right )\ ,\label{sharpmodM1}
\\
&&M_2^{(N)}-\frac{\tL_{\b_2}P_{\b_2}}{P_{\b_2}}B_2^{(N)} +2(1-\mu)\Re(e^{i\Gamma }\overline{Q_{\b_1}(R)})
+2 \Im (e^{i\Gamma }\overline{\pa_{y_1}Q_{\b_1}(R)})\label{sharpmodM2} \\
&& =O\Big(\frac{(\vert 1-\mu\vert +R^{-1})(\vert 1-\mu\vert +b+(1-\b_2)^{1/2}\vert \log(1-\b_2)\vert ^{1/2})+R^{-2}}{R(1+(1-\b_1)R)}\Big)\ .\notag
\eea
\end{proposition}
\begin{proof}
We recall the system of {\it nonlinear elliptic equations} solved in Proposition \ref{propwtobublle}.
$$
\mathcal \Vert \mathcal E_{1,N}\Vert _{\b _1}=O(bR^{-N-1})\ ,\ 
 \Vert \mathcal E_{2,N}\Vert _{\b _2}=O(R^{-N-1})\ .
$$
To simplify the notation, we will use $v_j$ instead of $V_j^{(N)}$
all along this proof. We will also drop the indices $(N)$
from $B_j, M_j$ for $j=1,2$.
\vskip 0.25cm
Let us recall the expressions of $\mathcal E_1, \mathcal E_2$.
 \bea
  \label{expressioneonebis}
\nonumber \mathcal E_1& = & -\frac{(|D|-\beta_1 D)v_1}{1-\b_1}-v_1+ v_1|v_1|^2-iM_1\Lambda v_1+iB_1\left[y_1\pa_{y_1}v_1+(1-\beta_1)\frac{\pa v_1}{\pa \beta_1}\right]\\
& + & i\l_1M_1\frac{\pa v_1}{\pa \l_1}+i\l_1M_2\frac{\pa v_1}{\pa \l_2}+i\frac{(1-\beta_2)B_2}{\mu}\frac{\pa v_1}{\pa \beta_2}\\
\nonumber & + & i\frac{1-\mu}{\mu}\frac{\pa v_1}{\pa \Gamma}+i(1-b+(B_1-M_1)R)\frac{\pa v_1}{\pa R}\\
\nonumber & + &\chi_R\left[\frac{2}{\mu}v_1|v_2|^2+\frac{e^{-i\Gamma}}{\sqrt{\mu}}v^2_1 \overline{v_2}+2\frac{e^{i\Gamma}}{\sqrt{\mu}}|v_1|^2v_2+\frac{e^{2i\Gamma}}{\mu}\overline{v_1}v_2^2\right],
\eea
 \bea
 \label{expressionetwobis}
 \nonumber \mathcal E_2& = & -\frac{(|D|-\beta_2 D)v_2}{1-\b_2}-v_2+ v_2|v_2|^2-iM_2\Lambda v_2+iB_2\left[y_2\pa_{y_2}v_2+(1-\beta_2)\frac{\pa v_2}{\pa \beta_2}\right]\\
 &+& i\l_2M_2\frac{\pa v_2}{\pa \l_2}+i\l_2M_1\frac{\pa v_2}{\pa \l_1}+i\mu(1-\beta_1)B_1\frac{\pa v_2}{\pa \beta_1}\\
 \nonumber & + & i(1-\mu)\frac{\pa v_2}{\pa \Gamma}+i\mu(1-b+(B_1-M_1)R)\frac{\pa v_2}{\pa R}\\
\nonumber & + & (1-\chi_R)\left[2\sqrt{\mu}e^{-i\Gamma}v_1|v_2|^2+\mu e^{-2i \Gamma}v_1^2\overline{v_2}+2\mu v_2|v_1|^2+\sqrt{\mu}e^{i\Gamma}\bar{v_1}v_2^2\right].
\eea
Our strategy is to extract information on $B_j,M_j$ from \eqref{solutioneone}, \eqref{expressioneonebis}, \eqref{expressionetwobis} and the admissibility properties of $v_1,v_2$.

\medskip

\noindent{\bf Step 1:} Speed for the first bubble and estimate on $B_1$. We take the scalar product of \fref{expressioneonebis} with $i v_1$. We observe the cancellations
$$\Big( -\frac{(|D|-\beta_1 D)v_1}{1-\b_1}-v_1+ v_1|v_1|^2,iv_1\Big)=0\ ,\ (i\Lambda v_1,iv_1)=0Ê\ .$$
 Recall from Proposition \ref{propwtobublle}  that 
 $$\vert B_1\vert+\vert M_1\vert \lesssim \frac{b}{R(1+(1-\b_1)R)}\ ,$$
 and that $b^{-1}(1+(1-\b_1)R)R^jT_{1,j}$ is $1$--admissible.
We obtain
\bee &&B_1[(\Lambda _{y_1}Q_{\b _1}+\tL_{\b _1}Q_{\b_1},Q_{\b_1})+O(R^{-1})]=-\frac{1}{\sqrt \mu}\Im \left (e^{i\Gamma }\int _\R \chi_R\vert v_1\vert ^2\overline v_1v_2\, dy_1\right )\\
&&+O\left (\int _{\R }\chi_R\vert v_1\vert^2\vert v_2\vert ^2\, dy_1+\frac{ b}{ R(1+(1-\b_1)R)}\left [\vert 1-\mu\vert+\frac{1}{R}\right ]\right )\ .
\eee
From the $2$--admissibility of $v_2$, we have 
$$\chi_R(y_1)\left \vert v_2\left (\frac{y_1-R}{b\mu}\right )\right \vert ^2\le \frac{b^2}{R^2(1+(1-\b_1)R)^2}\ .$$
This allows to neglect the integral 
$$\int _{\R }\chi_R\vert v_1\vert^2\vert v_2\vert ^2\, dy_1\ .$$
On the other hand,
$$\left \vert \chi _R(y_1)\left (v_2\left (\frac{y_1-R}{b\mu}\right )- v_2\left (\frac{-R}{b\mu}\right )  \right )\right \vert  \lesssim \frac{b^2}{R^2(1+(1-\b_1)R)}\ ,$$
and more precisely, since  $R^j(1+(1-\b_1)R)T_{2,j}$ is $2$--admissible, 
$$\left \vert v_2\left (\frac{-R}{b\mu}\right )-Q_{\b_2}\left ( -\frac{R}{b\mu}  \right )\right \vert \lesssim \frac{b}{R^2(1+(1-\b_1)R)^2}\ .$$
Therefore we can replace $v_2$ by $Q_{\b_2}(-R/(b\mu))$ in the integral
$$\int _\R \chi_R\vert v_1\vert ^2\overline v_1v_2\, dy_1\ .$$
Similarly, because of the estimates on $T_{1,j}$, one can replace $v_1$ by $Q_{\b_1}$ in the above integral, and finally drop the factor $\chi_R$, since the tale of $\vert Q_{\b_1}\vert ^3$ at infinity is small enough.
Identifying the coefficient of $B_1$, we infer
\bee -\pi B_1(N_{\b _1}-\tL_{\b _1}N_{\b _1})&=&-\frac{1}{\sqrt \mu}\Im \left (Q_{\b _2}\left (\frac{-R}{b\mu}\right ) \int _\R \vert Q_{\b _1}\vert ^2\overline Q_{\b _1} \, e^{i\Gamma}\right )\\
&&+O\left (\frac{ b}{ R(1+(1-\b_1)R)}\left [\vert 1-\mu\vert+\frac{1}{R}\right ]\right )\ ,
\eee
which, using the notation for $c_\b$, provides \eqref{sharpmodB1}. Notice that  the factor $1/\sqrt\mu $ has been replaced by $1$ up to an error 
$$O\left (\frac{\vert 1-\mu\vert b}{ R(1+(1-\b_1)R)}\right )\ .$$

\medskip

\noindent{\bf Step 2:} Speed for the second bubble and estimate on $B_2$. We proceed for the second bubble exactly as in Step 1. This leads to  \eqref{sharpmodB2}, as can be checked easily by the reader. Notice that the absence of the factor $b$ in the remainder term is due to the slightly different estimate for $T_{2,1}$ in Proposition \ref{propwtobublle}.

\medskip

\noindent{\bf Step 3:} Scaling for the first bubble and estimate on $M_1$. 
We take the scalar product  of \fref{expressioneonebis} with $\pa_{y_1} v_1$. 
We observe the cancellation
$$\Big( -\frac{(|D|-\beta_1 D)v_1}{1-\b_1}-v_1+ v_1|v_1|^2,\pa_{y_1}v_1\Big)=0.$$ 
We now compute the leading order non linear term. First, by integration by parts,
\begin{align}\label{1S0}
&\left(\chi_R\left[\frac2\mu v_1|v_2|^2+\frac{e^{2i\Gamma}}{\mu}\overline{v_1}v_2^2\right],\pa_{y_1}v_1\right)\\
& = -\frac 1\mu \int |v_1|^2\pa_{y_1}(\chi_R|v_2|^2)dy_1-\frac 1{2\mu}Re\left(\int e^{2i\Gamma}\overline{v_1}^2\pa_{y_1}(\chi_Rv_2^2)dy_1\right).\notag
\end{align}
From Proposition  \ref{propwtobublle}, we have the rough bound
\begin{equation}\label{decay_v_j}
|v_j|+|y_j\pa_{y_j}v_j|\lesssim \frac{1}{\la y_j\ra (1+(1-\b_j)\vert y_j\vert )}.
\end{equation}
Combining this with the fact that on the support of $\chi_{R}$
we have $\frac{R}{2\mu b}\leq |y_2|\leq \frac{3R}{2\mu b}$, we estimate 
\begin{align*}
|\pa_{y_1}(\chi_R|v_2|^2)|&\lesssim \frac{\pmb{1}_{|y_2|\geq \frac{R}{2\mu b}}}{R}|v_2|^2+ \frac{\pmb{1}_{|y_2|\geq \frac{R}{2\mu b}}}{b\mu}\pa_{y_2}(|v_2|^2)\lesssim \frac{b^2}{R^3(1+(1-\b_1)R)^2}
\end{align*}
\noindent
Then, by \eqref{1S0} and \eqref{decay_v_j}, we have
\begin{align}\label{1S00}
\Big|\Big(\chi_R\Big[\frac2\mu v_1|v_2|^2&+\frac{e^{2i\Gamma}}{\mu}\overline{v_1}v_2^2\Big],\pa_{y_1}v_1\Big)\Big|\lesssim\frac{b^2}{R^3(1+(1-\b_1)R)^2}.
\end{align}
For the remaining nonlinear term, we integrate by parts and obtain
\begin{align}\label{1S1}
&\left(\chi_R\left[\frac{e^{-i\Gamma}}{\sqrt{\mu}}v^2_1 \overline{v_2}+2\frac{e^{i\Gamma}}{\sqrt{\mu}}|v_1|^2v_2\right],\pa_{y_1}v_1\right)\notag\\
& = \Re\left(\int \frac{\chi_R}{\sqrt{\mu}}\left[e^{-i\Gamma}v_1^2\overline{v_2}\overline{\pa_{y_1}v_1}+2e^{i\Gamma}v_1\overline{v_1}v_2\overline{\pa_{y_1}v_1}\right]dy_1\right)\notag\\
& = \Re\left(\int \frac{\chi_R}{\sqrt{\mu}}\left[e^{-i\Gamma}\overline{v_2}\left[\pa_{y_1}(v_1^2\overline{v_1})-2v_1\pa_{y_1}v_1\overline{v_1}\right]+2e^{i\Gamma}v_1\overline{v_1}v_2\overline{\pa_{y_1}v_1}\right]dy_1\right)\notag\\
& = -\Re\left(\int \frac{e^{-i\Gamma}}{\sqrt{\mu}}v_1|v_1|^2\pa_{y_1}\left[\chi_R\overline{v_2}\right]dy_1\right)
\end{align}
We extract the leading order term using the following pointwise bound
which is a consequence of the $1$--admissibility of $b^{-1}R(1+(1-\b_1)R)(v_1-Q_{\b_1})$, and of the $2$--admissibility of $R(1+(1-\b_1)R)(v_2-Q_{\b_2})$,
\begin{align*}
\Big|v_1|v_1|^2\pa_{y_1}\left[\chi_R\overline{v_2}\right]&- Q_{\beta_1}|Q_{\beta_1}|^2\pa_{y_1}\left[\chi_R\overline{Q_{\beta_2}}\right]\Big|\\
& \lesssim \frac{b}{R^3(1+(1-\b_1)R)^2\la y_1\ra ^3}
\end{align*}
and thus:
\begin{align}\label{1S2}
-\Re\left(\int \frac{e^{-i\Gamma}}{\sqrt{\mu}}v_1|v_1|^2\pa_{y_1}\left[\chi_R\overline{v_2}\right]dy_1\right)&=  -\Re\left(\int \frac{e^{-i\Gamma}}{\sqrt{\mu}}Q_{\beta_1}|Q_{\beta_1}|^2\pa_{y_1}\left[\chi_R\overline{Q_{\beta_2}}\right]dy_1\right)\notag\\
&+O\left( \frac{b}{R^3(1+(1-\b_1)R)^2}\right)
\end{align}
We now compute the leading order term. 
Let $$z_1=\frac{y_1}{R}$$
Then, using $|\pa_{y_2}^2Q_{\b_2}|\lesssim\frac{1}{\jb{y_2}^3}$, we have for $|z_1|\leq \frac 12$ that
\begin{align*}
\pa_{y_2}Q_{\beta_2}(y_2)&=\pa_{y_2}Q_{\beta_2}\left(\frac{-R}{b\mu}(1-z_1)\right)\\
&=\pa_{y_2}Q_{\beta_2}\left(\frac{-R}{b\mu}\right)+\int_0^1\frac{Rz_1}{b\mu}\pa^2_{y_2}Q_{\beta_2}\left(\frac{-R}{b\mu}(1-tz_1)\right)dt\\
& = \pa_{y_2}Q_{\beta_2}\left(\frac{-R}{b\mu}\right)+O\left(\frac{R|z_1|}{b}
\Big(\frac{b}{R}\Big)^3\right)= \pa_{y_2}Q_{\beta_2}\left(\frac{-R}{b\mu}\right)+O\left(\frac{b^2|y_1|}{R^3}\right).
\end{align*}
Thus,
\begin{align}\label{1S3}
& -\Re\left(\int \frac{e^{-i\Gamma}}{\sqrt{\mu}}Q_{\beta_1}|Q_{\beta_1}|^2\pa_{y_1}\left[\chi_R\overline{Q_{\beta_2}}\right]dy_1\right)\notag\\
& = -\Re\left(\int_{|y_1|\leq \frac R2} \frac{e^{-i\Gamma}}{b\mu\sqrt{\mu}}Q_{\beta_1}|Q_{\beta_1}|^2\chi_R\overline{\pa_{y_2}Q_{\beta_2}}dy_1\right)\notag\\
& - \Re\left(\int_{\frac{R}{4}\leq |y_1|\leq \frac{R}{2}} \frac{e^{-i\Gamma}}{\sqrt{\mu}}Q_{\beta_1}|Q_{\beta_1}|^2\left[\pa_{y_1}\chi_R\right]\overline{Q_{\beta_2}}dy_1\right)\notag\\
& = -\Re\left(\int _{|y_1|\leq \frac R2}\frac{e^{-i\Gamma}}{b\mu\sqrt{\mu}}Q_{\beta_1}|Q_{\beta_1}|^2\pa_{y_2}\overline{Q_{\beta_2}}dy_1\right)\notag\\
& -\Re\left(\int_{|y_1|\leq \frac R2} \frac{e^{-i\Gamma}}{b\mu\sqrt{\mu}}(\chi_R-1)Q_{\beta_1}|Q_{\beta_1}|^2\pa_{y_2}\overline{Q_{\beta_2}}dy_1\right)
+O\left(\frac{b}{R^4}\right)\notag\\
& = -\Re\left(\overline{\pa_{y_2}Q_{\beta_2}\left(\frac{-R}{b\mu}\right)}\int_{|y_1|\leq \frac{R}{2}} \frac{e^{-i\Gamma}}{b\mu\sqrt{\mu}}Q_{\beta_1}|Q_{\beta_1}|^2dy_1\right)+O\left(\frac{b}{R^3}\right)\notag\\
& = -\frac{1}{b\mu}\Re\left(\pa_{y_2}Q_{\beta_2}\left(\frac{-R}{b\mu}\right)\frac{e^{i\Gamma}}{\sqrt{\mu}}\int \overline Q_{\beta_1}|Q_{\beta_1}|^2dy_1\right)+O\left(\frac{b}{R^3}\right)\notag \\
&=\frac{2\pi}{b\mu\sqrt{\mu}}\Im\left (e^{i\Gamma}\overline c_{\b_1}\pa_{y_2}Q_{\beta_2}\left(\frac{-R}{b\mu}\right)\right )+O\left(\frac{b}{R^3}\right)\ .
\end{align}
 
\noindent 
Finally we use again the following bound,
$$\frac{1}{b\mu}\left|\pa_{y_2}Q_{\beta_2}\left(\frac{-R}{b\mu}\right)\right|\lesssim \frac{b}{R^2(1+(1-\b_1)R)}\ .
$$
This together with \eqref{1S1}, \eqref{1S2}, and \eqref{1S3}, yields
$$\Big|\left(\chi_R\left[\frac{e^{-i\Gamma}}{\sqrt{\mu}}v^2_1 \overline{v_2}+2\frac{e^{i\Gamma}}{\sqrt{\mu}}|v_1|^2v_2\right],\pa_{y_1}v_1\right)\Big|
 \lesssim  \frac{b}{R^2(1+(1-\b_1)R)}.
$$

\noindent
Combining this with \eqref{1S00}, we get that
the contribution of the nonlinearity is
\be\label{1S4}
\Big|\Big(\chi_R\Big[\frac2\mu v_1|v_2|^2+\frac{e^{2i\Gamma}}{\mu}\overline{v_1}v_2^2
+\frac{e^{-i\Gamma}}{\sqrt{\mu}}v^2_1 \overline{v_2}+2\frac{e^{i\Gamma}}{\sqrt{\mu}}|v_1|^2v_2\Big],\pa_{y_1}v_1\Big)\Big|
\lesssim \frac{b}{R^2(1+(1-\b_1)R)}.
\ee

In view of the expression \eqref{expressioneonebis} of $\mathcal E_1$, we infer ---assuming $N$ large enough---, 
\bee
&&M_1(-i\Lambda v_1,\pa _{y_1}v_1)+B_1\left(iy_1\pa_{y_1}v_1+i(1-\beta_1)\frac{\pa v_1}{\pa \beta_1},\pa_{y_1}v_1\right)\\
& + & \l_1M_1\left (i\frac{\pa v_1}{\pa \l_1},\pa_{y_1}v_1\right )+\l_1M_2\left (i\frac{\pa v_1}{\pa \l_2},\pa_{y_1}v_1\right )+\frac{(1-\beta_2)B_2}{\mu}\left (i\frac{\pa v_1}{\pa \beta_2},\pa_{y_1}v_1\right )\\
& + & \frac{1-\mu}{\mu}\left (i\frac{\pa v_1}{\pa \Gamma},\pa_{y_1}v_1\right )+(1-b+(B_1-M_1)R)\left (i\frac{\pa v_1}{\pa R},\pa_{y_1}v_1\right )=O (bR^{-2}(1+(1-\b_1)R)^{-1}).
\eee
We now compute the terms involving the modulation equations. 
First, by Proposition \ref{propwtobublle},
we have that 
\bea
\label{ekonvbvr}
\nonumber (-i\Lambda v_1,\pa_{y_1}v_1)& =& (-i\Lambda Q_{\beta_1},\pa_{y_1}Q_{\beta_1})+O\left(\frac{b}{R(1+(1-\b_1)R)}\right)\\
& = & -\pi P_{\b_1}+O\left(\frac{b}{R(1+(1-\b_1)R)}\right).
\eea
On the other hand,
\bee
\Big(iy_1\pa _{y_1}v_1+ i(1-\b_1)\frac{\pa v_1}{\pa \beta_1},\pa_{y_1}v_1\Big)
&=&\Big(i(1-\b_1)\frac{\pa Q_{\beta_1}}{\pa\beta_1},\pa_{y_1}Q_{\beta_1}\Big)
+O\left(\frac{b}{R(1+(1-\b_1)R)}\right),\\
&=&\pi \tL_{\b_1}P_{\b_1}+O\left(\frac{b}{R(1+(1-\b_1)R)}\right)\ .
\eee
Then, by Proposition \ref{propwtobublle},
\begin{align*}
\left|\Big(i\frac{\pa v_1}{\pa {\l_1}},\pa_{y_1}v_1\Big)\right|
&+\left|\Big(i\frac{\pa v_1}{\pa {\l_2}},\pa_{y_1}v_1\Big)\right|
 + \left|\Big(i(1-\beta_2)\frac{\pa v_1}{\pa_{\beta_2}},\pa_{y_1}v_1)\Big)\right|\\
&\lesssim \frac{b}{R(1+(1-\b_1)R},
\end{align*}
and 
\begin{align*}
&\left|\frac{1-\mu}{\mu}\Big(i\frac{\pa v_1}{\pa {\Gamma}},\pa_{y_1}v_1\Big)\right|
+\left|\Big(i\frac{\pa v_1}{\pa {R}},\pa_{y_1}v_1\Big)\right|\\
&\lesssim \frac{b (\vert 1-\mu\vert +R^{-1})}{R(1+(1-\b_1)R)}\ .
\end{align*}
The collection of above bounds yields the identity:
$$-\pi P_{\b _1}M_1+\tL_{\b_1}P_{\b_1}B_1
=   O\left(\frac{b(|1-\mu|+R^{-1})}{R(1+(1-\b_1)R)}\right )\ ,$$
which  leads to the bound
\be
\label{loimun}
\left|M_1-\frac{\tL_{\b_1}P_{\b_1}}{P_{\b_1}}B_1\right|\lesssim   \frac{b(|1-\mu|+R^{-1})}{R(1+(1-\b_1)R)}\ .
\ee


\medskip

\noindent{\bf Step 4:} Scaling for the second bubble and estimate on $M_2$. 
We take the scalar product of \fref{expressionetwo} with $\pa_{y_2} v_2$. 
We observe the cancellation
$$\Big( -\frac{(|D|-\beta_2 D)v_2}{1-\b_2}-v_2+ v_2|v_2|^2,\pa_{y_2}v_2\Big)=0.$$ 
We now compute the contribution of the non linear term. Firstly,  by integration by parts,
\begin{align}\label{2S1}
\Big((1-\chi_R)&\left[2\mu v_2|v_1|^2+\mu e^{-2i\Gamma}\overline{v_2}v_1^2\right],\pa_{y_2}v_2\Big)\\
= & - \mu\int |v_2|^2\pa_{y_2}((1-\chi_R)|v_1|^2)dy_2-\frac \mu{2}Re\left(\int e^{-2i\Gamma}\overline{v_2}^2\pa_{y_2}((1-\chi_R)v_1^2)dy_2\right).\notag
\end{align}
By the rough bound \eqref{decay_v_j}, we have
\bee
\big|\pa_{y_2}\big((1-\chi_R)|v_1|^2\big)\big|&\lesssim & \frac{ b\mu}{R}{\bf 1}_{\frac R4 \leq |y_1|\leq \frac R2}|v_1|^2+ b\mu{\bf 1}_{|y_1|\geq \frac R4}\pa_{y_1}(|v_1|^2)\\
& \lesssim & \frac{b{\bf 1}_{\frac R4 \leq |y_1|\leq \frac R2}}{R\la y_1\ra ^2}+ b\frac{{\bf 1}_{|y_1|\geq \frac R4}}{\la y_1\ra ^3}\lesssim \frac{b}{R^3(1+(1-\b_1)R)^2}.
\eee
Then, by \eqref{2S1}, we have
\be\label{1S5}
\Big|\Big((1-\chi_R)\left[2\mu v_2|v_1|^2+\mu e^{-2i\Gamma}\overline{v_2}v_1^2\right],\pa_{y_2}v_2\Big)\Big|
\lesssim \frac{b}{R^3(1+(1-\b_1)R)^2}.
\ee
For the remaining nonlinear term, we integrate by parts and obtain
\begin{align}\label{1S6}
&\Big((1-\chi_R)\left[\sqrt{\mu}e^{i\Gamma}v^2_2 \overline{v_1}+2\sqrt{\mu}e^{-i\Gamma}|v_2|^2v_1\right],\pa_{y_2}v_2\Big)\notag\\
& =  \Re\left(\int \sqrt{\mu}(1-\chi_R)\left[e^{i\Gamma}v_2^2\overline{v_1}\overline{\pa_{y_2}v_2}+2e^{-i\Gamma}v_2\overline{v_2}v_1\overline{\pa_{y_2}v_2}\right]dy_2\right)\notag\\
& = \Re\left(\int  \sqrt{\mu}(1-\chi_R)\left[e^{i\Gamma}\overline{v_1}\left[\pa_{y_2}(v_2^2\overline{v_2})-2v_2\pa_{y_2}v_2\overline{v_2}\right]+2e^{-i\Gamma}v_2\overline{v_2}v_1\overline{\pa_{y_2}v_2}\right]dy_2\right)\notag\\
& = -\Re\left(\int \sqrt{\mu}e^{i\Gamma}v_2|v_2|^2\pa_{y_2}\left[(1-\chi_R)\overline{v_1}\right]dy_2\right)
\end{align}
We extract the leading order term using the pointwise bound:
\bee
&&\Big| v_2|v_2|^2\pa_{y_2}\left[(1-\chi_R)\overline{v_1}\right]- Q_{\beta_2}|Q_{\beta_2}|^2\pa_{y_2}\left[(1-\chi_R)\overline{Q_{\beta_1}}\right]\Big|\\
& \lesssim &\frac{1}{R\la y_2\ra^3}\left[ \frac{b\mu{\bf 1}_{\frac R4 \leq |y_1|\leq \frac R2}}{R\la y_1\ra (1+(1-\b_1)\vert y_1\vert )}+\frac{b\mu{\bf 1}_{|y_1|\geq \frac R4}}{\la y_1\ra^2(1+(1-\b_1)\vert y_1\vert )}\right]\lesssim \frac{b}{R^3\la y_2\ra ^3(1+(1-\b_1)R )}.
\eee
Thus,
\begin{align}\label{1S7}
&-\Re\left(\int \sqrt{\mu}e^{i\Gamma}v_2|v_2|^2\pa_{y_2}\left[(1-\chi_R)\overline{v_1}\right]dy_2\right)\notag\\
& =  -\Re\left(\int \sqrt{\mu}e^{i\Gamma}Q_{\beta_2}|Q_{\beta_2}|^2\pa_{y_2}\left[(1-\chi_R)\overline{Q_{\beta_1}}\right]dy_2\right)+O\left( \frac{b}{R^3(1+(1-\b_1)R )}\right)
\end{align}
We now compute the leading order term. Let $z_2=\frac{b\mu y_2}{R},$ then for $|z_2|\leq \frac 12$:
\bee
\pa_{y_1}Q_{\beta_1}(y_1)&=&\pa_{y_1}Q_{\beta_1}\left(R(1+z_2)\right)=\pa_{y_1}Q_{\beta_1}(R)+\int_0^1Rz_2\pa^2_{y_1}Q_{\beta_1}\left(R(1+tz_2)\right)dt\\
& = & \pa_{y_1}Q_{\beta_1}(R)+O\left(\frac{R|z_2|}{R^3}\right)= \pa_{y_1}Q_{\beta_1}(R)+O\left(\frac{b\mu |y_2|}{R^3}\right)
\eee
and thus:
\bee
&& -\Re\left(\int \sqrt{\mu}e^{i\Gamma}Q_{\beta_2}|Q_{\beta_2}|^2\pa_{y_2}\left[(1-\chi_R)\overline{Q_{\beta_1}}\right]dy_2\right)\\
&=&  -\Re\left(\int_{|y_2|\geq \frac{R}{2b\mu}}b\mu \sqrt{\mu}e^{i\Gamma}Q_{\beta_2}|Q_{\beta_2}|^2\pa_{y_1}\left[(1-\chi_R)\overline{Q_{\beta_1}}\right]dy_2\right)\\
& -& \Re\left(\int_{|y_2|\leq \frac{R}{2b\mu}} b\mu\sqrt{\mu}e^{i\Gamma}Q_{\beta_2}|Q_{\beta_2}|^2\pa_{y_1}\left[(1-\chi_R)\overline{Q_{\beta_1}}\right]dy_2\right)\\
& = & -\Re\left(\int_{|y_2|\leq \frac{R}{2b\mu}} b\mu\sqrt{\mu}e^{i\Gamma}Q_{\beta_2}|Q_{\beta_2}|^2\pa_{y_1}\overline{Q_{\beta_1}}dy_2\right)+O\left(\frac{b^3}{R^4}\right)\\
& = &  -\Re\left( b\mu\sqrt{\mu}e^{i\Gamma}\overline{\pa_{y_1}Q_{\beta_1}(R)}\int_{|y_2|\leq \frac{R}{2b\mu}}Q_{\beta_2}|Q_{\beta_2}|^2dy_2\right)\\
&+&O\left(\frac{b^3}{R^4}+b\int_{|y_2|\leq \frac R{2b\mu}}\frac{b\mu}{R^3}\frac{|y_2|}{\la y_2\ra ^3}dy_2\right)\\
& = & -\Re\left(b\mu\sqrt{\mu}e^{i\Gamma}\overline{\pa_{y_1}Q_{\beta_1}(R)}\int Q_{\beta_2}|Q_{\beta_2}|^2dy_2\right)+O\left(\frac{b}{R^3}\right)\\
&=& O\left(\frac{b}{R^2(1+(1-\beta_1)R)}\right)
 \eee
where we used \eqref{DQbx} in the last step. 
Combining this with \eqref{1S6} and \eqref{1S7}, we obtain
that
$$\Big|\Big((1-\chi_R)\left[\sqrt{\mu}e^{i\Gamma}v^2_2 \overline{v_1}+2\sqrt{\mu}e^{-i\Gamma}|v_2|^2v_1\right],\pa_{y_2}v_2\Big)\Big|
\lesssim \frac{b}{R^2(1+(1-\beta_1)R)}\ .$$

\noindent
This, together with \eqref{1S5} yields
\begin{align*}
&\Big|\Big((1-\chi_R)
\left[2\mu v_2|v_1|^2+\mu e^{-2i\Gamma}\overline{v_2}v_1^2+\sqrt{\mu}e^{i\Gamma}v^2_2 \overline{v_1}+2\sqrt{\mu}e^{-i\Gamma}|v_2|^2v_1\right],\pa_{y_2}v_2\Big)\Big|\notag\\
&\lesssim \frac{b}{R^2(1+(1-\beta_1)R)}\ .\end{align*}
In view of the expression \eqref{expressionetwobis} of $\mathcal E_2$, we infer
\bee
&&M_2(-i\Lambda v_2,\pa _{y_2}v_2)+B_2\left(iy_2\pa_{y_2}v_2+i(1-\beta_2)\frac{\pa v_2}{\pa \beta_2},\pa_{y_2}v_2\right)\\
& + & \l_2M_2\left (i\frac{\pa v_2}{\pa \l_2},\pa_{y_2}v_2\right )+\l_2M_1\left (i\frac{\pa v_2}{\pa \l_1},\pa_{y_2}v_2\right )+\mu B_1\left (i(1-\b_1)\frac{\pa v_2}{\pa \beta_1},\pa_{y_2}v_2\right )\\
& + & (1-\mu)\left (i\frac{\pa v_2}{\pa \Gamma},\pa_{y_2}v_2\right )+(1-b+(B_1-M_1)R)\left (i\frac{\pa v_2}{\pa R},\pa_{y_2}v_2\right )\\
&=&
O \left (\frac{b}{R^2(1+(1-\b_1)R)}+\frac{1}{R^{N+1}}\right ).
\eee
Next we compute the terms involving the modulation equations. On the one hand,
\be
\label{cneneneonevno}
\nonumber (i\Lambda v_2,\pa_{y_2}v_2) = (i\Lambda Q_{\beta_2},\pa_{y_2}Q_{\beta_2})+O\left(\frac{1}{R(1+(1-\b_1)R )}\right)
=  \pi P_{\b_2}+O\left(\frac{1}{R(1+(1-\b_1)R )}\right).
\ee
On the other hand, taking into account Lemma \ref{lemma:pa_b_T21} and \eqref{Qbdotplus}, \eqref{Qbdotmoins}, 
\bee
\Big(iy_2\pa _{y_2}v_2+ i(1-\b_2)\frac{\pa v_2}{\pa \beta_2},\pa_{y_2}v_2\Big)
&=&\Big(i(1-\b_2)\frac{\pa Q_{\beta_2}}{\pa\beta_2},\pa_{y_2}Q_{\beta_2}\Big)\\
&+&O\left(\frac{\vert 1-\mu\vert +(1-\b_2)^{1/2}\vert\log(1-\b _2)\vert ^{1/2}+R^{-1}}{R(1+(1-\b_1)R )}\right),\\
&=&\pi \tL_{\b_2}P_{\b_2}+O\left(\frac{\vert 1-\mu\vert +(1-\b_2)^{1/2}\vert\log(1-\b _2)\vert ^{1/2}+R^{-1}}{R(1+(1-\b_1)R )}\right)\ .
\eee
Then, by construction,
$$
\left|\Big(i\frac{\pa v_2}{\pa {\l_1}},\pa_{y_2}v_2\Big)\right|+\left|\Big(i\frac{\pa v_2}{\pa{\l_2}},\pa_{y_2}v_2\Big)\right|
+\left|\Big(i(1-\beta_1)\frac{\pa v_2}{\pa{\beta_1}},\pa_{y_2}v_2\Big)\right|
 \lesssim \frac{1}{R(1+(1-\b_1)R )}.
$$
Moreover, by Lemmas \ref{lemma:pa_Gamma_T21} and \ref{lemma:pa_R_T21},
we have
\begin{align*}
&(1-\mu)\Big(i\frac{\pa v_2}{\pa {\Gamma}},\pa_{y_2}v_2\Big)
+\mu (1-b)\Big(i\frac{\pa v_2}{\pa {R}},\pa_{y_2}v_2\Big)\\
&=(1-\mu)\Big[-2\pi \Re(e^{i\Gamma}\overline{Q_{\b_1}(R)})\Big]
+\mu (1-b)\Big[-2\pi \Im (e^{i\Gamma}\overline{\pa_{y_1}Q_{\b_1}(R)})\Big]\\
&+O\Big(\frac{(1-\mu)^2+\vert 1-\mu\vert ((1-\b_2)^{1/2}\vert \log(1-\b_2)\vert ^{1/2}+R^{-1})+R^{-2}}{R(1+(1-\b_1)R)}\Big)\ .
\end{align*}
Notice that, in view of \eqref{DQbx}, the factor $\mu (1-b)$ in the above right hand side can be replaced by $1$ up to the expense of the additional error
$$O\left (\frac{b(\vert 1-\mu\vert +R^{-1})}{R(1+(1-\b_1)R}\right )\ .$$

Summing up,  we obtain
\begin{align*}
&M_2-\frac{\tL_{\b _2}P_{\b _2}}{P_{\b _2}}B_2+ 2(1-\mu)\Re(e^{i\Gamma}\overline{Q_{\b_1}(R)})
+2 \Im (e^{i\Gamma}\overline{\pa_{y_1}Q_{\b_1}(R)})=\notag\\
&O\Big(\frac{(\vert 1-\mu\vert +R^{-1})(\vert 1-\mu\vert +b+(1-\b_2)^{1/2}\vert \log(1-\b_2)\vert ^{1/2})+R^{-2}}{R(1+(1-\b_1)R)}\Big)\ .
\end{align*}
This completes the proof.
\end{proof}


\subsection{Solving the reduced dynamical system}
\label{sectionreduced}

Our aim in this section is to exhibit a suitable exact solution to the idealized dynamical system
\be
\label{dynamicalystem}
(S)^{\infty}\ \ \left\{ 
\begin{array}{lll}(x_j)_{t}=\beta_j, \ \ 
(\gamma_j)_{t}=\frac{1}{\l_j},\\
 (\l_j)_{t}=M_j(\mathcal P),\ \ \frac{(\beta_j)_{t}}{1-\beta_j}
=\frac{B_j(\mathcal P)}{\l_j},\\
\Gamma=\gamma_2-\gamma_1, \ \ R=\frac{x_2-x_1}{\l_1(1-\beta_1)}
  \end{array}
  \right. \ \ j=1,2,
  \ee  
  with $\mathcal P=(\l_1,\l_2,\beta_1,\beta_2,\Gamma, R)$,
  which will correspond to the leading order two-soliton motion, and where from now on and for the rest of this paper we omit the subscript $N$ for the sake of simplicity.\\
  
Let $0<\eta,\delta\ll 1$. Define the times
 \be
 \label{deftimes}
 T_{\rm in}=\frac{1}{\eta^{2\delta}}<T^-=\frac{\delta}{\eta} 
 \ee
 and consider explicitly the solution $$\tilde{\mathcal P}^\infty=(\l^\infty_1,\l^\infty_2,\beta^\infty_1,\beta^\infty_2,\gamma^\infty_1,\gamma^\infty_2,x^\infty_1,x^\infty_2)$$ to \fref{dynamicalystem} with data at $t=T^-$:
 \be
 \label{datateta}
 \left\{\begin{array}{llll} \l^\infty_1=1, \ \ \l^\infty_2=1 ,\\
\gamma^\infty_1=\gamma^\infty_2=0 ,\\
 1-\beta^\infty_1=\eta, \ \ b^\infty=\frac{1}{(T^-)^2}\ \ \mbox{ie}\ \ 1-\beta^\infty_2=\frac{\eta}{(T^-)^2} ,\\
 x^\infty_1=0, \ \ R^\infty=T^-\ \ \ \ \mbox{ie}\ \  x^\infty_2=T^-\eta=\delta.
 \end{array}\right.
 \ee
  The fact that the system \fref{dynamicalystem} with data \fref{datateta} admits a unique maximal solution   is a simple consequence of the Cauchy--Lipschitz theorem.\\
  
  We first claim the backwards control of this solution in the following perturbative form.
   
 \begin{lemma}[Control of the solution in the perturbative turbulent regime]
 \label{propturbulent}
 Let $\delta>0$ small enough and $0<\eta<\eta^*(\delta,N)$ small enough. Let $\tilde{\mathcal P}$ be the solution to the approximate system 
 \be
\label{dynamicalystemapproximate}
\left\{ 
\begin{array}{lll}(x_j)_{t}=\beta_j+O\left(\frac{1}{t^3}\right), \ \ 
(\gamma_j)_{t}=\frac{1}{\l_j}+O\left(\frac{1}{t^3}\right),\\
 (\l_j)_{t}=M_j(\mathcal P)+O\left(\frac{1}{t^3}\right),\ \ \frac{(\beta_j)_{t}}{1-\beta_j}
=\frac{B_j(\mathcal P)}{\l_j}+O\left(\frac{1}{t^3}\right),\\
\Gamma=\gamma_2-\gamma_1, \ \ R=\frac{x_2-x_1}{\l_1(1-\beta_1)}
  \end{array}
  \right. \ \ j=1,2,
  \ee  
  with initial data at $T^-$ satisfying :
  \be
  \label{datatetaapp}
  |\tilde{\mathcal P}(T^-)-\tilde{\mathcal P}^\infty (T^-)|\leq \eta^{10},
  \ee
  then the parameters satisfy in $t\in [T_{\rm in},T^-]$ the bounds:
\be
\label{esttrubulentregime}
\left\{\begin{array}{llll}
\l_1(t)=1+O\left(\frac{\eta^{\delta}}{t}\right), \ \ \l_2(t)=1+O\left(\frac{\eta^\delta+\eta t|\log \eta t|}{t}\right)\\
1-\beta_1(t)=\eta(1+O(\eta^\delta)), \ \ b(t)=\frac{1+O(\sqrt{\delta})}{t^2}\\
\Gamma(t)=O(\eta t|\log \eta t|)\\
R=t(1+O(\eta^\delta)).
\end{array}\right.
\ee
\end{lemma}
\begin{remark} Notice that the small quantity $\eta t\vert \log\eta t\vert $ grows on $[T_{\rm in},T^-]$ from $(1-\delta )\eta ^{1-\delta}\vert \log \eta\vert $ to $\delta \vert \log \delta \vert$. Therefore, if $\delta $ is small and if $\eta <\eta ^*(\delta )$, this quantity is first smaller that $\eta ^\delta $, then it becomes bigger than $\eta ^\delta $. This explains why we have to keep both quantities in the remainder terms.
\end{remark}
\begin{proof}[Proof of Lemma \ref{propturbulent}] From \fref{datateta} and \fref{datatetaapp}, we may assume the following bounds:
\be
\label{bootsboundturbulent}
\left\{\begin{array}{llll|}
|\l_1(t)-1|\leq \frac{\eta ^{\delta}}{t},\ \ j=1,2\\
|\l_2(t)-1|\leq K \frac{\eta^{\delta}+\eta t|\log(\eta t)|}{t}\\
|1-\beta_1(t)-\eta|\leq \eta^{1+\delta} , \\
\frac{\eta}{2t^2}\leq 1-\beta_2(t)\leq \frac{2\eta}{t^2}\\
\frac{\left|R(t)-t\right|}{t}\leq \eta ^{\delta }\\
|\Gamma(t)|\leq K( \eta^\delta+\eta t|\log (\eta t)|)
\end{array}\right.
\ee
and aim at improving them for some large enough universal constant $K$, and  for $0<\delta<\delta^*(K)$, $0<\eta<\eta^*(K,\delta )$, which proves \eqref{esttrubulentregime} through a standard continuity argument. The difficulty is that the growth of Sobolev norms in \fref{esttrubulentregime} relies on an uniform control of the phase which is not allowed to move, and this requires two integrations in time in the presence of $O(\frac{1}{t^2})$ decay only and hence some suitable cancellation in the modulation equations.\\

\noindent {\bf Step 1:} Leading order modulation equations. We extract the leading order modulation equations of Proposition \ref{sharpmod} in the regime \fref{esttrubulentregime} using the sharp description of the asymptotic structure of $Q_\beta$ given by Proposition \ref{prop:equivalent Q_b}. We estimate from \fref{bootsboundturbulent} 
$$R\sim t\leq\frac{\delta}{\eta}$$ and hence 
$$0<(1-\beta_1)R\lesssim \eta t\lesssim \delta\ll1.$$ 
Now we appeal to the precise description of $Q_\b$ given by \eqref{asymQ_b}:
\bea
\label{nekvnevneoneoneo}
\nonumber Q_{\beta_1}(R)&=&\frac{1+O((1-\beta_1)|\log(1-\beta_1)|)}{R}\left[1+O((1-\beta_1) R\log ((1-\beta_1) R))\right]+O\left(\frac 1{R^2}\right)\\
& = & \frac1t+O\left(\frac{\eta^{\delta}}{t}+\eta|\log \eta t|\right)
\eea
where we used the localization of $R$ given by \eqref{bootsboundturbulent} in the last step. 
Similarly, using \eqref{DQbx}, it follows that
\bea
\label{estderivative}
\nonumber &&\pa_{y_1}Q_{\beta_1}(R)\\
\nonumber &=&\frac{1+O(\eta|\log \eta|)}{R^2}\left\{-1+\frac{i}{2}(1-\beta_1)R\left[1+O((1-\beta_1)R|\log (1-\beta_1)R|)\right]\right\}+O\left(\frac{\eta^\delta}{R^2}\right)\\
& = & \frac{1}{t^2}\left[-1+\frac{i}{2}\eta t\right]+O\left (\frac{\eta^\delta}{t^2}+\eta^2|\log (\eta t)|\right ).
\eea
We also have
\begin{align*}
(1-\b_2)\frac{R}{b\mu}=\frac{(1-\b_1)R}{\mu}\lesssim (1-\b_1)R\lesssim \delta
\end{align*}
and thus,
\bea
\label{estimationbetadeux}
&&\nonumber Q_{\beta_2}\left(-\frac{R}{b\mu}\right)\\
\nonumber&=&-\frac{1+O((1-\beta_2)|\log(1-\beta_2)|)}{\frac{R}{b\mu}}\left[1+O((1-\beta_2) \frac{R}{b}\log ((1-\beta_2) \frac{R}{b}))\right]+O\left(\frac{b^2}{R^2}\right)\\
& = & O\left(\frac{b}{t}\right).
\eea
We now compute the leading order modulation equations of Proposition \ref{sharpmod}. We first have the rough bound
\be
\label{estvoenvoneo}
B_1=O\left(\frac bt\right)
\ee 
and the finer control from \eqref{nekvnevneoneoneo}:
\bea
\label{estvniovnonvoe}
\nonumber B_2&=&2\left[1+O((1-\beta_2)|\log (1-\beta_2)|)\right]\Re\left\{(\cos\Gamma-i\sin\Gamma)\left[\frac1t+O\left(\frac{\eta^{\delta}}{t}+\eta|\log \eta t|\right)\right]\right\}\\
\nonumber & + & O\left(\frac{1-\mu}{t}+\frac 1{t^2}\right)\\
& = & \frac{2\cos\Gamma}{t}+O\left(\frac{\eta^{\delta}}{t}+\eta|\log (\eta t)|+\frac{|1-\mu|}{t}\right)\\
& = &\frac{2}{t}+O\left(\frac{\eta^{\delta}}{t}+\eta|\log (\eta t)|\right)
\eea
where in the last step we used from \eqref{bootsboundturbulent}: 
\bea
\label{esgammam}
\nonumber \frac{\Gamma^2}{t}&\lesssim & \frac{K^2(\eta^{2\delta}+\eta^2t^2|\log(\eta t)|^2)}{t}\lesssim \frac{\eta^{\delta}}{t}+\eta |\log (\eta t)|K^2\eta t|\log (\eta t)|\\
&\lesssim & \frac{\eta^\delta}{t}+\eta |\log (\eta t)| K^2\delta|\log \delta| \lesssim \frac{\eta^\delta}{t}+   \eta |\log (\eta t)|
\eea
 for $\delta<\delta^*(K)$ small enough.
We similarly derive the rough bound
\be
\label{estpointwisemone}
|M_1|\lesssim |1-\beta_1||\log (1-\beta_1)||B_1|+\frac{|b(1-\mu)|}{t}+\frac{|b|}{t^2}\lesssim \frac{\eta^{2\delta}}{t^2}.
\ee
We now estimate $M_2$. First we compute from \eqref{nekvnevneoneoneo}, \eqref{esgammam}:
\bee
2(1-\mu)\Re(e^{i\Gamma}\overline{Q_{\beta_1}(R)})&=&2(1-\mu)\Re\left\{(\cos\Gamma+i\sin\Gamma)\left(\frac1t+O\left (\frac{\eta^{\delta}}{t}+\eta|\log \eta t|\right )\right)\right\}\\
& = & \frac{2(1-\mu)\cos\Gamma}{t}+|1-\mu|O\left(\frac{\eta^{\delta}}{t}+\eta|\log \eta t|\right)\\
& = & \frac{2(1-\mu)}{t}+O\left(\frac{\eta^\delta}{t^2}+K\eta^2|\log \eta t|^2\right)
\eee
where we used in the last step from \eqref{bootsboundturbulent}:
\be
\label{vneivneonvenve}
|1-\mu|\lesssim |\l_2-1|+|\l_1-1|\lesssim K\frac{\eta^\delta+\eta t|\log(\eta t)|}{t}
\ee
and hence
\bee
|1-\mu|\left(\frac{\eta^{\delta}}{t}+\eta|\log \eta t|\right)&\lesssim& K\frac{(\eta^\delta+\eta t|\log \eta t|)^2}{t^2}\lesssim \frac{\eta^\delta}{t^2}+K\eta^2|\log \eta t|^2
\eee
for $\eta <\eta^*(K,\delta)$ small enough.
similarly from \eqref{estderivative}:
\bee
2\Im\left\{e^{i\Gamma}\overline{\pa_{y_1}Q_{\beta_1}(R)}\right\}&=&2\Im\left\{(\cos\Gamma+i\sin\Gamma)\left[\frac{1}{t^2}\left(-1+\frac{i}{2}\eta t\right)+O\left (\frac{\eta^\delta}{t^2}+\eta^2|\log (\eta t)|\right )\right]\right\}\\
& =&-\frac{2\sin\Gamma}{t^2}+\frac{\eta}{t}\cos\Gamma+O\left (\frac{\eta^\delta}{t^2}+\eta^2 |\log (\eta t)|\right )\\
& = & -\frac{2\Gamma}{t^2}+\frac{\eta}{t}+O\left (\frac{\eta^\delta}{t^2}+\eta^2 |\log (\eta t)|\right ),
\eee
where we used in the last step the development of $\cos\Gamma,\sin\Gamma$ with the bounds:
\bee
\frac{|\Gamma|^3}{t^2}+\frac{\eta \Gamma^2}{t}&\lesssim& \frac{K^3(\eta^{3\delta}+\eta^3t^3|\log\eta t|^3)}{t^2}+\frac{\eta K^2(\eta^{2\delta}+\eta^2t^2|\log \eta t|^2)}{t}\\
&\lesssim& \frac{\eta^\delta}{t^2}+\eta^2|\log \eta t|\left[K^3\eta t|\log \eta t|^2+K^2\eta t|\log \eta t|\right]\\
& \lesssim & \frac{\eta^\delta}{t^2}+\eta^2|\log \eta t| K^3\delta |\log \delta|^2\leq \frac{\eta^{\delta}}{t^2}+\eta^2|\log \eta t|
\eee
for $\delta<\delta^*(K)$ small enough.
Using from \eqref{estvniovnonvoe} the rough bound $|B_2|\lesssim \frac 1t$ ensures the finer bound from \eqref{sharpmodM2}:
\bea
\label{sharpboundmtwo}
\nonumber &&\Big| M_2+\frac{2(1-\mu)}{t}-\frac{2\Gamma}{t^2}+\frac{\eta}{t}\Big|\lesssim\frac{\eta^{\delta}}{t^2}+K^2\eta^2|\log(\eta t)|^2+\frac{|1-\beta_2||\log (1-\beta_2)|}{t}\\
& \lesssim &\frac{\eta^{\delta}}{t^2}+K^2\eta^2|\log(\eta t)|^2
\eea
where we used \eqref{bootsboundturbulent} in the last step to estimate $1-\beta_2$.

\noindent{\bf Step 2:} Control of the speeds. We first integrate the law for $\beta_2$ from \eqref{estvniovnonvoe}:
$$
\frac{(\beta_2)_{t}}{1-\beta_2}=\frac{B_2}{\l_2}=\frac{1}{\l_2}\left[\frac{2}{t}+O\left(\frac{\eta^{\delta}}{t}+\eta|\log (\eta t)|\right)\right]= \frac{2}{t}+O\left(\frac{\eta^{\delta}}{t}+\eta|\log (\eta t)|\right).
$$
We integrate on $[t,T^-]$ and use 
\bee
&&\int_t^{T^-}\eta |\log (\eta \tau)|d\tau\lesssim \int_0^{\delta} |\log \sigma|d\sigma\leq \sqrt{\delta}\\
&& \int_{t}^{T^-}\frac{\eta^\delta}{\tau}d\tau=O\left(\eta^\delta|\log \eta|\right)\leq \sqrt{\delta},
\eee
for $\eta <\eta ^*(K,\delta )$,
 to estimate
\bee
-\log\left(\frac{1-\beta_2(T^-)}{1-\beta_2(t)}\right)=2\log\left(\frac{T^-}{t}\right)+O(\sqrt{\delta})
\eee
from which using the initialization \eqref{datateta}, \eqref{datatetaapp}:
\be
\label{loidebetadeux}
1-\beta_2(t)=\frac{(T^-)^2(1-\beta_2(T^-))}{t^2}e^{O(\sqrt{\delta})}=\left[1+O(\sqrt{\delta})\right]\frac{\eta}{t^2}.
\ee
We now compute for $\beta_1$ from \eqref{estvoenvoneo}: $$\Big|\frac{(\beta_1)_{t}}{1-\beta_1}\Big|=\Big|\frac{B_1}{\l_1}\Big|\lesssim \frac{b}{t}\lesssim \frac1{t^3}$$
which time integration using \eqref{datateta}, \eqref{datatetaapp} yields
\be
\label{loibetone}
1-\beta_1(t)=(1-\beta_1(T^-))e^{O(\frac 1{t^2})}=\eta\left (1+O\left (\frac 1{t^2}\right )\right ).
\ee 
Since $t\ge T_{{\rm in}}=\eta ^{-\delta }$, this improves the estimate on $1-\b_1-\eta $.
This yields with \eqref{loidebetadeux}: 
\be
\label{lvejoveveovneoe}
b(t)=\frac{1+O(\sqrt{\delta})}{t^2}.
\ee

\noindent{\bf Step 3:} Control of the scaling and the phase shift. We need to be extra careful to reintegrate the law for $\Gamma$ which requires two integrations in time in the presence of $\frac{1}{t^2}$ decay only, and hence the possibility of logarithmic losses which would be dramatic to control the smallness of the phase and hence the growth of the Sobolev norm. We first integrate $\l_1$ from \eqref{estpointwisemone}: $$|(\l_1)_t|\lesssim |M_1|\lesssim \frac{\eta^{2\delta}}{t^2}$$ and hence from \eqref{datateta}, \eqref{datatetaapp}: 
\be
\label{estloneboot}
\l_1(t)=1+O\left(\frac{\eta^{2\delta}}{t}\right).
\ee 
Now consider $$v=1-\mu$$. Using  \eqref{estloneboot}, we have
\bee
\Gamma_t&=&\frac{1}{\l_2}-\frac{1}{\l_1}=\frac{1-\mu}{\l_2}=\frac{1-\mu}{\l_1(1-(1-\mu))}=v\left[1+O\left (\frac{\eta^{2\delta}}{t}\right )\right][1+O(v)]\\
& = & v+O\left (\frac{\eta^\delta}{t^2}\right )+O(v^2)\ .
\eee
and we now estimate from \eqref{vneivneonvenve}:
$$
v^2\lesssim K^2\frac{(\eta^\delta+\eta t|\log \eta t|)^2}{t^2}\lesssim \frac{\eta^{\delta}}{t^2}+K^2\eta^2|\log \eta t|^2,$$ 
whence the first equation, 
$$\Gamma_t=v+O\left (\frac{\eta^\delta}{t^2}+K^2\eta^2|\log \eta t|^2\right )\ .$$ 

Hence from \eqref{estpointwisemone},  \eqref{sharpboundmtwo}:
\bee
v_t&=&-\mu_t=-\mu \left[\frac{(\l_2)_t}{\l_2}-\frac{(\l_1)_t}{\l_1}\right]=\mu\frac{M_1}{\l_1}-\frac{M_2}{\l_1}\\
&=&-M_2\left[1+O\left (\frac{\eta^{2\delta}}{t}\right )\right]+O\left (\frac{\eta ^{2\delta }}{t^2}\right )=-M_2+O\left (\frac{\eta ^{2\delta }}{t^2}\right )\ .
\eee
and hence from \eqref{sharpboundmtwo}:
\bee
v_t=\frac{2v}{t}-\frac{2\Gamma}{t^2}+\frac{\eta}{t}+O\left(\frac{\eta^{\delta}}{t^2}+K^2\eta^2|\log(\eta t)|^2\right).
\eee
We therefore obtain the following system,
\be
\label{estf}
\left\{\begin{array}{ll}\Gamma_t=v+R_{\Gamma}(t),\\ v_t=\frac{2v}{t}-\frac{2\Gamma}{t^2}+\frac{\eta}{t}+R_v(t)\end{array}\right. 
\ee
with $$|R_\Gamma (t)|+|R_v(t)|\lesssim \frac{\eta^{\delta}}{t^2}+K^2\eta^2|\log(\eta t)|^2,$$ and with the initial data 
$$\Gamma (T^-)=O(\eta ^{10})\ ,v(T^-)=O(\eta ^{10}).$$
A basis of solutions to the linear homogeneous system
\be
\left\{\begin{array}{ll}\Gamma_t=v\\ v_t=\frac{2v}{t}-\frac{2\Gamma}{t^2}\end{array}\right. 
\ee is given by $\{ (\Gamma_1(t),v_1(t))=(t,1), (\Gamma_2(t),v_2(t))=(t^2,2t)\} $,
with Wronskian $$W=v_2\Gamma_1-\Gamma_2v_1=t^2$$ and hence the explicit solution with data \eqref{datateta} is given by:
\bee
&&\Gamma(t)=\Gamma_0(t)-\Gamma_1(t)\int_t^{T^-}\frac{R_\Gamma v_2-R_v\Gamma_2}{W}d\tau-\Gamma_2(t)\int_t^{T^-}\frac{R_v\Gamma_1-R_\Gamma v_1}{W}d\tau,\\
&& v(t)=v_0(t)-v_1(t)\int_t^{T^-}\frac{R_\Gamma v_2-R_v\Gamma_2}{W}d\tau-v_2(t)\int_t^{T^-}\frac{R_v\Gamma_1-R_\Gamma v_1}{W}d\tau,
\eee
where $(\Gamma_0,v_0)$ is the explicit homogeneous solution given by 
\bee
\Gamma_0(t)=\Gamma_1(t)\left (O(\eta ^{10})+\int_t^{T^-}\frac{\eta}{\tau}\psi_2\frac{d\tau}W\right )-\Gamma_2(t)\left (O(\eta ^{10})+\int_t^{T^-}\frac{\eta}{\tau}\Gamma_1\frac{d\tau}{W}\right )= O\left(\eta t (|\log \eta t|)\right),
\eee
and
\bee
v_0(t)&=&v_1(t)\left (O(\eta ^{10})+\int_t^{T^-}\frac{\eta}{\tau}\psi_2\frac{d\tau}W\right )-v_2(t)\left (O(\eta ^{10})+\int_t^{T^-}\frac{\eta}{\tau}\Gamma_1\frac{d\tau}{W}\right )\\
& = & \int_t^{T^-}\frac{\eta}{\tau}d\tau-2t\int_t^{T-}\frac{\eta}{\tau^2}d\tau=\eta\log\left(\frac{T^-}{t}\right)-2\eta t\left(\frac 1t-\frac{1}{T^-}\right)\\
& = & O\left(\frac{\eta t (|\log \eta t|)}{t}\right).
\eee
We now estimate the error:
\bee
&&\left|v_1(t)\int_t^{T^-}\frac{R_\Gamma v_2-R_v\Gamma_2}{W}d\tau-v_2(t)\int_t^{T^-}\frac{R_v\Gamma_1-R_\Gamma v_1}{W}\right|\\
& \lesssim &  \int_t^{T^-}\left[\frac{\eta^\delta}{\tau^2}+K^2\eta^2|\log(\eta \tau)|^2\right]d\tau\lesssim  \frac{\eta^\delta}{t}+K^2\eta\int_0^{\delta}|\log \tau|^2d\tau\lesssim \frac{\eta^\delta}{t}+K^2\eta\delta|\log \delta|^2\\
& \lesssim & \frac{\eta^\delta}{t}+\frac{K^2\delta|\log \delta|^2 \eta t |\log \eta t|}{t|\log \eta t|}\lesssim \frac{\eta^\delta}{t}+\frac{K^2\delta|\log \delta|^2}{|\log \delta|}\frac{\eta t|\log \eta t|}{t}\lesssim  \frac{\eta^\delta+\eta t|\log \eta t|}{t},
\eee
for $\delta<\delta^*(K)$ small enough, and similarly:
$$\left|\Gamma_1(t)\int_t^{T^-}\frac{R_\Gamma v_2-R_v\Gamma_2}{W}d\tau-\Gamma_2(t)\int_t^{T^-}\frac{R_v\Gamma_1-R_\Gamma v_1}{W}\right|\lesssim  \eta^\delta+\eta t|\log \eta t|.
$$
The collection of above bounds using the modified initial data easily ensures 
$$|v(t)|\lesssim \frac{\eta^\delta+\eta t|\log \eta t|}{t}, \ \ |\Gamma(t)|\lesssim \eta^\delta+\eta t|\log \eta t|$$ which closes the bootstrap \eqref{esttrubulentregime} for $\l_2,\Gamma$ on $[T_{\rm in},T^-]$ for $K$ universal large enough.\\

\noindent{\bf Step 4:} Control of the centers and the relative distance.

We compute from \eqref{loibetone}, \eqref{lvejoveveovneoe}:
\bee
(x_2)_t-(x_1)_t&=&\beta_2-\beta_1=1-\beta_1-(1-\beta_2)=(1-\beta_1)(1-b(t))\\
&=& \eta\left (1+O\left (\frac 1{t^2}\right )\right )\left[1-\frac{1+O(\sqrt{\delta})}{t^2}\right]= \eta \left (1+O\left (\frac 1{t^2}\right )\right ).
\eee
Hence using $(x_2-x_1)(T^-)=\eta T^-+O(\eta^9)$ from \eqref{datateta}, we obtain by integration in time:
$$(x_2-x_1)(t)= (x_2-x_1)(T^-)+\eta(t-T^-)+O\left (\frac{\eta }{ t}\right )=\eta t+O\left (\frac{\eta }{ t}\right )\ ,$$
and hence, using \eqref{loibetone}, \eqref{estloneboot}:
\bee
\frac{R(t)-t}{t}&=&\frac{x_2-x_1}{t\l_1(1-\beta_1)}-1=\frac{x_2-x_1}{\eta t}(1+O(\eta^{2\delta}))-1\\
&=&O(\eta^{2\delta})\leq \frac 12\eta^\delta
\eee
which closes the $R$ bound in \eqref{esttrubulentregime}.

\end{proof}

We now come back the exact solution $\tilde{\mathcal P}^{\infty}$ of \eqref{dynamicalystem} with data \eqref{datateta} and claim that the corresponding dynamics is frozen for $t\ge T^-$.

 \begin{lemma}[Post interaction dynamics]
\label{lemmapostinteraction}
For $\delta $ sufficiently small and $\eta <\eta ^*(\delta )$, there holds on $[T^-,+\infty)$:
\be
\label{estsortieinteraction}
\left\{\begin{array}{llll}
\l^\infty_1(t)=1+O(\eta), \ \ \l^\infty_2(t)=1+O(\eta)\\
1-\beta^\infty_1(t)=\eta(1+O(\eta^\delta)), \ \ 1-\beta^\infty_2(t)=\eta^3e^{O(\frac 1{\delta})},\\
\Gamma^\infty(t)=O(t)\\
R^\infty=t(1+O(\eta^\delta)).
\end{array}\right.
\ee
\end{lemma}

\begin{proof} We bootstrap the following bounds on $[T^-,+\infty)$,
\be
\label{estroutside}
\left|\begin{array}{lll}
|1-\l_1(t)|+|1-\l_2|\leq K\eta,\\
|1-\beta_1-\eta|\leq K\eta^\delta, \ \ |1-\beta_2|\leq \eta^2\\
R(t)\geq \frac t2\\
\end{array}\right.
\ee
for some large enough universal constant $K=K(\delta)$, and where we omit the $\infty$ subscript for the sake of clarity. 
Notice that the notation $A\lesssim B$ in this context means $A\le C\, B$ with a constant $C$ independent of $\delta $,
assuming $\eta <\eta ^*(\delta )$. 
  
By \eqref{estroutside} we have
\be
\label{esvnroknrkbnrobn} |b|\lesssim \eta
\ee
and using \eqref{boundQ_b} and \eqref{DQbx}, it follows for $R(1-\beta_1)\gtrsim \delta $ that
\bee
&&|Q_{\beta_1}(R)|
\lesssim \frac{1}{\eta t^2},\\
&& |Q'_{\beta_1}(R)|\lesssim \frac{1}{t^2},\\
&& \left|Q_{\beta_2}\left(-\frac{R}{b\mu}\right)\right|\lesssim \frac{b}{(1-\b_1) t^2}\lesssim \frac{1}{t^2}.
\eee
We may therefore estimate in brute force the parameters using Proposition \ref{sharpmod}:
\bee
&&|B_1|\lesssim \frac{\eta}{t^2}+\frac{K\eta^2}{\eta t^2}\lesssim \frac 1{t^2}\\
&&|B_2|\lesssim \frac{1}{\eta t^2}+\frac{K\eta}{\eta t^2}\lesssim \frac{1}{\eta t^2}\\
&&|M_1| \lesssim \frac{(1-\beta_1)|\log (1-\beta_1)|}{t^2}+\frac{K\eta^2}{\eta t^2}+\frac{K\eta}{t^2}\lesssim \frac{\eta |\log \eta|}{t^2}\\
&&|M_2|\lesssim \frac 1{t^2}+\frac{\vert 1-\mu\vert}{\eta t^2}\ .
\eee
We therefore control the speeds on $[T^-,+\infty)$ using \eqref{esttrubulentregime}:
\bee
&&\left|\frac{(\beta_1)_t}{1-\beta_1}\right|\lesssim |B_1|\lesssim \frac 1{t^2},\ \ \mbox{i.e.}\ \ 1-\beta_1(t)=\eta e^{O\left(\frac{1}{T^-}\right)}=\eta(1+O(\eta))\\
&&\left|\frac{(\beta_2)_t}{1-\beta_2}\right|\lesssim |B_2|\lesssim \frac 1{\eta t^2}, \ \ \mbox{i.e.}\ \ 1-\beta_2(t)=\frac{\eta}{(T^-)^2}e^{O\left(\frac{1}{\eta T^-}\right)}=\eta^3e^{O(\frac1{\delta})}
\eee
and similarly for the first size,
$$\left|\frac{(\l_1)_t}{\l_1}\right|\lesssim |M_1|\lesssim \frac{\sqrt{\eta}}{t^2},\ \ \mbox{i.e.} \ \ \l_1(t)=1+O(\eta).$$ Hence:
$$\left|\mu_t\right|\lesssim |M_2|+|M_1|\lesssim  \frac{|1-\mu|}{\eta t^2}+\frac 1{t^2}$$ 
from  which we infer, using $\mu(T^-)=1$,
$$|1-\mu(t)|\lesssim \frac{1}{T^-}+\int_{T^-}^t \frac{|1-\mu(\tau)|}{\eta \tau ^2}d\tau.$$ 
By Gronwall's lemma, we conclude $$|1-\mu(t)|\lesssim \frac{1}{T^-}e^{O\left(\frac{1}{\eta T^-}\right)}=\eta \, e^{O(\frac 1\delta )}\ .$$
 Hence the control of scalings and speeds is closed for $K=K(\delta )$  large enough in \eqref{estroutside}. We now integrate the position. $$(x_2)_t-(x_1)_t=\beta_2-\beta_1=1-\beta_1-(1-\beta_2)=\eta(1+O(\eta))$$ from which we get $$x_2(t)-x_1(t)=\eta (1+O(\eta))(t-T^-)+\eta T^-=\eta t+O(\eta^2t)$$ and $$R(t)=\frac{x_2-x_1}{t\l_1(1-\beta_1)}\geq \frac 23,$$ which concludes the proof of Lemma \ref{lemmapostinteraction}.
\end{proof}


\section{Energy estimates}
\label{sectionenergy}


This section is devoted to the construction of an exact solution to \fref{halfwave} with two-soliton asymptotic behavior and transient turbulent regime. The strategy is based as in \cite{KMRhartree, MRSring} on an energy method near the explicit approximate solution which can be closed thanks to the arbitrary high order expansion of the approximate solution, and the $R(t)\sim t$ distance between the two waves.


\subsection{Backwards integration and parametrization of the flow}


Given parameters $$\mathcal P=(\l_1,\l_2,\beta_1,\beta_2,\Gamma,R),\ \ \widetilde{\matchal P}=(\mathcal P, x_1,x_2,\gamma_1,\gamma_2),$$ we let $$\Phi^{(N)}_{\widetilde{\mathcal P}}(x)=\Phi^{(N,1)}_{\pt}(x)+\Phi^{(N,2)}_{\pt}(x)$$ 
with 
$$\Phi^{(N,j)}_{\widetilde{\mathcal P}}(x)=\frac{1}{\l_j^{\frac 12}}V^{(N)}_{j}(y_j,\mathcal P)e^{i\gamma_j},\ \ y_j=\frac{x-x_j}{\l_j(1-\beta_j)}, \ \ j=1,2, $$ constructed in Proposition \ref{propwtobublle}. We now fix one and for all a large enough number $N\gg 1$, and for the rest of the paper, we omit the subscript $N$ in order to ease notations. We then pick a small enough universal constant $\delta >0$ and, for $0<\eta<\eta ^*(\delta )$, we  consider
$$\pt^\infty=(\l_1^{\infty}, \l_2^{\infty},\gamma_1^{\infty}, \gamma_2^{\infty}, x_1^{\infty},x_2^{\infty})$$ 
to be the exact solution to \eqref{dynamicalystem} with data \eqref{datateta} which is well defined on $[T^-,+\infty)$ from Lemma  \ref{lemmapostinteraction}.\\
We now build an exact solution to the full system \eqref{halfwave} by integrating backwards in time from $+\infty$: we let a sequence $T_n\to +\infty$ and consider $u_n(t)$ the solution to
\be
\label{defunt}
\left\{\begin{array}{ll} i \partial_t u_n = |D| u_n - |u_n|^2 u_n,\\
 u_n(T_n)=\Phi_{\pt^{\infty}(T_n)}(x).\end{array}\right.
 \ee
 We will very precisely study the properties of $u_n(t)$. Here and in the sequel, we omit as much as possible the subscript $n$ to ease notations.\\
From standard modulation argument, as the solution remains close in $H^{\frac 12}$ to a modulated tube around the decoupled two solitary waves , we may consider a decomposition of the flow
\be
\label{decopmun}
u(t,x)=\Phi_{\pt(t)}(x)+\e(t,x)
\ee 
where the parameters 
$$\pt(t)=(\l_1(t),\l_2(t),\beta_1(t),\beta_2(t),x_1(t),x_2(t),\Gamma(t),R(t))\ ,$$ with the explicit dependence 
\be
\label{dependancegammar}
 \Gamma=\gamma_2-\gamma_1, \ \ R=\frac{x_2-x_1}{\l_1(1-\beta_1)}\ ,
 \ee
are chosen for each fixed $t$  in order to manufacture suitable orthogonality conditions on the remainders
\be
\label{defej}
\e_j(t,y_j):=\l_j^{\frac 12}(t)\e\big(t,\l_j(t)(1-\b_j(t))y_j+x_j(t)\big)e^{-i\gamma_j(t)}, \ \ j=1,2.
\ee
Observe that 
\be
\label{relationltwo}
\|\e\|_{L^2}^2=(1-\beta_j)\|\e_j\|_{L^2}^2, \ \ j=1,2.
\ee
Let $\omega$ be the symplectic form
\[\omega (f,g)= Im\int f\bar{g}dx=(f,ig),\]
and consider the generalized null space of the operator $i\mathcal L_\beta$
formed of functions $f\in H^{1/2}$
such that $(i\mathcal L_\b)^2f=0$.
This generalized null subspace
consists of $iQ_{\b}$, $\pa_y Q_\b$,
$\Lambda Q_\b$, and $i\rho_\b$,
where $\rho_\b$ is the unique $H^{\frac 12}$ solution to the problem \eqref{nceionceoneo}.
Indeed,
one can directly check that $i\mathcal L_\b (iQ_\b)
=i\mathcal L_\b (\pa_y Q_\b)=0$
and 
$$(i\mathcal L_\b)^2 (\Lambda Q_\b)
=(i\mathcal L_\b)^2 (i\rho_\b)=0.$$ 
We then impose the set of symplectic orthogonality conditions:
\bee
\omega(\e_j,iQ_{\beta_j})=\omega(\e_j,\pa_{y_j}Q_{\beta_j})= 
\omega(\e_j,\Lambda Q_{\beta_j})=\omega (\e_j,i\rho_j)=0, \ \ j=1,2,
\eee
or equivalently,
\be
\label{orthoe}
(\e_j,Q_{\beta_j})=(\e_j,i\pa_{y_j}Q_{\beta_j})= (\e_j,i\Lambda Q_{\beta_j})=(\e_j,\rho_j)=0, \ \ j=1,2.
\ee

Let $\sigma_j:=(\lambda_j, x_j,\gamma_j,\b_j)$, $j=1,2$ and $\Sigma$ be a compact subset of 
\[\big(\R_+^\ast\times \R\times\R\times (1-\b_\ast,1)\big)^2.\]
For $(\sigma_1,\sigma_2)\in\Sigma$ and $f\in H^{1/2}$, we define
\[\mathcal S_{\sigma_j}f (x)=\frac{1}{\l_j^{1/2}}f\Big(\frac{x-x_j}{\l_j(1-\b_j)}\Big)e^{i\gamma_j}.\]
The existence and uniqueness for each $t$ of $\pt(t)$ ensuring the decomposition \fref{decopmun}, \fref{orthoe} is now a standard consequence of the implicit function theorem 
applied to the function
$G:H^{1/2}\times \Sigma\to \R^8$, $G(\psi,\sigma)=0$, where $G$ is defined by
\begin{align*}
G(\psi,\sigma)
=\left(\begin{array}{llll||||}
(\psi-\mathcal S_{\sigma_1}V_1(\mathcal P)-\mathcal S_{\sigma_2}V_2(\mathcal P), \mathcal S_{\sigma_1}Q_{\b_1})\\
(\psi-\mathcal S_{\sigma_1}V_1(\mathcal P)-\mathcal S_{\sigma_2}V_2(\mathcal P), \mathcal S_{\sigma_1}i\pa_xQ_{\b_1})\\
(\psi-\mathcal S_{\sigma_1}V_1(\mathcal P)-\mathcal S_{\sigma_2}V_2(\mathcal P), \mathcal S_{\sigma_1}i\Lambda Q_{\b_1})\\
(\psi-\mathcal S_{\sigma_1}V_1(\mathcal P)-\mathcal S_{\sigma_2}V_2(\mathcal P), \mathcal S_{\sigma_1}\rho_{\b_1})\\
(\psi-\mathcal S_{\sigma_1}V_1(\mathcal P)-\mathcal S_{\sigma_2}V_2(\mathcal P), \mathcal S_{\sigma_2}Q_{\b_2})\\
(\psi-\mathcal S_{\sigma_1}V_1(\mathcal P)-\mathcal S_{\sigma_2}V_2(\mathcal P), \mathcal S_{\sigma_2}i\pa_xQ_{\b_2})\\
(\psi-\mathcal S_{\sigma_1}V_1(\mathcal P)-\mathcal S_{\sigma_2}V_2(\mathcal P), \mathcal S_{\sigma_2}i\Lambda Q_{\b_2})\\
(\psi-\mathcal S_{\sigma_1}V_1(\mathcal P)-\mathcal S_{\sigma_2}V_2(\mathcal P), \mathcal S_{\sigma_2}\rho_{\b_2})
\end{array}\right),
\end{align*}
where $\sigma=(\sigma_1,\sigma_2)$ and $\mathcal{P}=(\l_1,\l_2,\b_1,\b_2,\Gamma, R)$.
The key ingredient here is that, for any $(\sigma_1^0,\sigma_2^0)\in\Sigma$,
the Jacobian matrix
$$\pa_\sigma G\big(\mathcal S_{\sigma_1^0}V_1^{(N)}+\mathcal S_{\sigma^0_2}V_2^{(N)},\sigma\big)\Big|_{\sigma=(\sigma_1^0,\sigma_2^0)}$$ 
is invertible,
which follows from the fact that the matrix
$$A_j=\left(\begin{array}{llll}
(\Lambda Q_{\beta_j},Q_{\beta_j}) & (\Lambda Q_{\beta_j},i\pa_{y_j}Q_{\beta_j}) &  (\Lambda Q_{\beta_j},i\Lambda Q_{\beta_j}) & (\Lambda Q_{\beta_j},\rho_j)\\
(i Q_{\beta_j},Q_{\beta_j}) & (i Q_{\beta_j},i\pa_{y_j}Q_{\beta_j}) &  (i Q_{\beta_j},i\Lambda Q_{\beta_j}) & (i Q_{\beta_j},\rho_j)\\
(\pa_{y_j}Q_{\beta_j}, Q_{\beta_j}) & (\pa_{y_j}Q_{\beta_j},i\pa_{y_j}Q_{\beta_j}) &  (\pa_{y_j}Q_{\beta_j},i\Lambda Q_{\beta_j}) & (\pa_{y_j}Q_{\beta_j},\rho_j)\\
(\Sigma_j,Q_{\beta_j}) & (\Sigma_j,i\pa_{y_j}Q_{\beta_j}) &  (\Sigma_j,i\Lambda Q_{\beta_j}) & (\Sigma_j,\rho_j)
\end{array}\right)
$$ 
with 
\be
\label{defsigmaj}
\Sigma_j:=y\pa_yQ_{\beta_j}+(1-\beta_j)\pa_{\beta_j}Q_{\beta_j}
\ee
 is non degenerate 
\be
\label{hessianinvert}
\lim_{\beta_j\to 1}|\det A_j|\neq 0, \ \ j=1,2,
\ee
see Appendix \ref{invertaj}. 


\subsection{Localized $H^{\frac 12}$-energy}


The heart of our analysis is the derivation of a suitable monotonicity formula for a {\it suitable localized $H^{\frac 12}$ energy identity}. The localization procedure is mandatory in order to dynamically adapt the functional to the dramatically changing size of the bubble, but this will lead to serious difficulties due to nonlocal
nature of the problem and the slow decay of the solitary wave. The limiting Szeg\H{o} problem will arise in the form of various different estimates for $\Pi^{\pm}\e$ which will be essential to close the estimates.\\

Let us start by introducing suitable cut-off functions which adapt the energy functional to the dramatic change of size of the second solitary wave.\\

\noindent{\it Space localization}. We pick explicitly a sufficiently smooth non increasing function
\be
\label{defPhione}
\Psi_1(z_1)=\left|\begin{array}{lll} 1\ \ \mbox{for}\ \ z_1\leq \frac 14\\ (1-z_1)^{10}\ \ \mbox{for}\ \ \frac 12\leq z_1\leq 1\\ 0\ \ \mbox{for}\ \ z_1\geq 1.\end{array}\right..
\ee
and let
\be
\label{defPhionebis}
\Phi_1(t,z_1)=\Psi_1+b(t)(1-\Psi_1)=\left|\begin{array}{ll} 1\ \ \mbox{for}\ \ z_1\leq \frac 14\\ b(t)\ \ \mbox{for}\ \ z_1\geq 1.\end{array}\right.
\ee
From this function of $(t,z_1)$ we deduce a function of $(t,y_1)$ and $(t,x)$ via the following change of variables,
$$\phi(t,x)=\phi_1(t,y_1)=\Phi_1(t,z_1), \ \ z_1=\frac{y_1}{R(t)(1-b(t))}.$$
We then define the localization associated to kinetic momentum
\be
\label{defzetabeta}
\zeta(t,x)=\beta_1(t)+(1-\beta_1(t))(1-\phi (t,x)),
\ee
so that
\be
\label{defzeta}
\zeta (t,x)=\zeta_1(t,y_1)=\left\{\begin{array}{ll} \beta_1(t)\ \ \mbox{for}\ \ y_1\leq \frac{(1-b(t))R(t)}4\\ \beta_2 (t)\ \ \mbox{for}\ \ y_1\geq (1-b(t))R(t).
\end{array}\right..
\ee
similarly, let
\be
\label{defPhionebisbis}
\widetilde{\Phi_1}(t,z_1)=\mu (t)\Psi_1(z_1)+(1-\Psi_1(z_1))=\left|\begin{array}{ll} \mu(t)\ \ \mbox{for}\ \ z_1\leq \frac 14\\ 1\ \ \mbox{for}\ \ z_1\geq 1,\end{array}\right.
\ee
with the same change of variables as before,
$$\tilde{\phi}(t,x)=\tilde{\phi}_1(t,y_1)=\widetilde{\Phi_1}(t,z_1), \ \ z_1=\frac{y_1}{R(t)(1-b(t))}
\ .$$
We define the localization attached to the localization of mass,
\be
\label{deftheta}
\theta(t,x)=\frac{1}{\l_2(t)}\tilde{\phi}(t,x)=\theta_1(t,y_1), 
\ee
so that
$$\theta(t,y_1)=\left\{\begin{array}{ll}\frac{1}{\l_1(t)}\ \ \mbox{for}\ \ y_1\leq \frac{(1-b(t))R(t)}4\\ \frac{1}{\l_2(t)} \ \ \mbox{for}\ \ y_1\geq (1-b(t))R(t)
\end{array}\right..
$$
Explicit estimates used throughout the proof involving functions $\zeta,\theta$ are stated in Appendix \ref{zppappcomm}.\\

\noindent{\it Localized energy}. We now introduce the localized energy functional:
\bea
\label{defgt}
\nonumber \mathcal G(\eps):&=&\frac 12\left[(|D|\e-\zeta D\e,\e)+(\theta\e,\e)\right]\\
 &-& \frac 14\left[ \int_\R(|\e+\Phi|^4-|\Phi|^4)dx -4(\e,\Phi|\Phi|^2)\right]
\eea
Notice that the inner products 
are taken  in the $x$ variable, and that $\Phi $ denotes the approximate solution $\Phi _{\pt (t)}$. This functional will be used as our main energy functional.  We indeed first claim that $\mathcal G$ is a coercive functional.

\begin{proposition}[Coercivity of the localized energy]
\label{coerclinearizedenergy}
There holds\footnote{for some universal coercivity constant which is related to the coercivity of the limiting Szeg\H{o} functional \eqref{ceroclimiting}.} :
\be
\label{coerc}
\mathcal G(\eps)\gtrsim(1-\beta_1)\left[\int |\e_1|^2dy_1+\int\phi_1|\Dhalf \e_1^+|^2dy_1\right]+ \int |\Dhalf \e_1^-|^2dy_1
\ee
where $\e_1$ was defined in \eqref{defej}.
\end{proposition}

The proof adapts the argument in \cite{MMannals2} and relies on a careful localization of the kinetic energy and the coercivity of the limiting Szeg\H{o} quadratic form. A key fact is that the relative distance $R$ between the solitary waves is always large. The presence of the localization $\phi_1$ in \eqref{coerc} is an essential difficulty of the analysis and shows that one looses control of $\|D^{\frac 12}\e^+\|_{L^2}$ as $\beta_1\to 1$ (through the factor $1-\beta_1$), which reflects the singular nature of the bifurcation $Q^+\to Q_\beta$. This will be a fundamental issue for the forthcoming analysis. The proof of Proposition \ref{coerclinearizedenergy} is detailed in  Appendix \ref{appendixcoerc}.


\subsection{Bootstrap argument}


Since $\e(T_n)=0$ and $\mathcal P(T_n)=\matchal P^\infty(T_n)$, we run a bootstrap argument in the following form.  Let 
\be
\label{defbetatilde}\tilde{\beta}_j:=\log(1-\beta_j)
\ee
and 
\bea
\label{defdeltas}
&&|\Delta \l_j|(t):=\sup_{\tau\in [t,T_n]}|\l_j-\l_j^{\infty}|(\tau), \ \ |\Delta \tilde{\beta}_j|(t):=\sup_{\tau\in [t,T_n]}|\tilde{\beta}_j-\tilde \beta_j^{\infty}|(\tau),\\
\label{defdeltasbis}
&& \ \ |\Delta R|(t):=\sup_{\tau\in [t,T_n]}|R-R^{\infty}|(\tau), \ \ |\Delta \Gamma|(t):=\sup_{\tau\in [t,T_n]}|\Gamma-\Gamma^\infty|(\tau),
\eea
we assume on some interval $[T,T_n]$, with $T_{\rm in}\le T\le T_n$, the $H^1$-bounds:
\be
\label{bootdynamicalproved}
\forall t\in [T_{\rm in},T_n], \ \  
\left\{\begin{array}{ll}
\mathcal G(\eps(t))\leq \frac{1}{t^{\frac N2}}\\
\|\e(t)\|^2_{H^1}\leq \frac{1}{t^{\frac N4}}
\end{array}\right.
\ee
and the bounds on the parameters:\\
\noindent{1. post interaction estimates}: for $t\in [T^-,T_n]\cap [T,T_n]$,
\be
\label{boootbound}
\left\{\begin{array}{lll}
|\Delta R|\leq \frac{1}{t^{\frac N8-1}}\\
|\Delta \tilde{\beta}_j|+|\Delta \Gamma|\leq\frac{1}{t^{\frac{N}{8}}},\\
\sum_{j=1,2}|\Delta \l_j|\leq \frac{1}{t^{\frac N8+1}};
\end{array}\right.
\ee
\noindent{2. rough turbulent bounds}: for $t\in [T_{\rm in},T^-]\cap [T,T_n]$,
\be
\label{esttrubulentregimelossy}
\left\{\begin{array}{llll}
|\l_1-1|+|\l_2-1|\leq \frac 1t\\
\frac{\eta}2\leq 1-\beta_1(t)\leq 2\eta, \ \ \frac12\leq t^2b(t)\leq 2\\
|\Gamma(t)|\leq \sqrt{\delta}\\
\frac t2\leq R\leq 2t.
\end{array}\right.
\ee
The heart of our analysis is that all these bounds can be improved.

\begin{proposition}[Bootstrap]
\label{propboot}
For $N\geq N^*$ large enough and $0<\eta<\eta^*(N)$ small enough, the following holds:
\be
\label{bootdynamicalprovedfinal}
\forall t\in [T,T_n], \ \  
\left\{\begin{array}{ll}
\mathcal G(\eps(t))\lesssim \frac{1}{Nt^{\frac N2}}\\
\|\e(t)\|^2_{H^1}\lesssim \frac{1}{Nt^{\frac N4}}
\end{array}\right.
\ee
and the bounds on the parameters:\\
\noindent{1. post interaction estimates}: for $t\in [T^-,T_n]\cap [T,T_n]$,
\be
\label{boootboundfinal}
\left\{\begin{array}{lll}
|\Delta R|\lesssim \frac{1}{Nt^{\frac N8-1}}\\
|\Delta \tilde{\beta}_j|+|\Delta \Gamma|\lesssim\frac{1}{Nt^{\frac{N}{8}}},\\
\sum_{j=1,2}|\Delta \l_j|\lesssim \frac{1}{Nt^{\frac N8+1}};
\end{array}\right.
\ee \noindent{2. rough turbulent bounds}: for $t\in [T_{\rm in},T^-]\cap [T,T_n]$, $\mathcal P$ satisfies \eqref{esttrubulentregime}.
\end{proposition}

Of course, the bounds \eqref{bootdynamicalprovedfinal}, \eqref{boootboundfinal}, \eqref{esttrubulentregime} improve on \eqref{bootdynamicalproved}, \eqref{boootbound}, \eqref{esttrubulentregimelossy} for $N$ universal large enough, so that we can finally set $T=T_{\rm in}$. Proposition \ref{propboot} is the heart of the analysis and implies Theorem \ref{thmmain} through a now classical argument which we detail in Subsection \ref{proofthm} for the convenience of the reader.\\

From now until Subsection \ref{proofthm}, we assume the bounds \eqref{bootdynamicalproved}, \eqref{boootbound}, \eqref{esttrubulentregimelossy} and aim at improving them. Since $t\geq T_{in}=\frac{1}{\eta^{2\delta}}$, we will systematically use the bound $$\frac{1}{\eta^C t^{\sqrt{N}}}\leq 1\ \ \mbox{for}\ \ N\geq N(\delta),  \  \eta<\eta^*(N).$$

Let us also observe from \eqref{boootbound}, \eqref{esttrubulentregimelossy}, \eqref{estsortieinteraction} injected into Proposition \ref{sharpmod} the bounds: $\forall t\in[T_{\rm in},T_n]$,
\be
\label{boundessential}
|B_1|+|M_1|\lesssim \frac bt, \ \ |B_2|\lesssim \frac 1t, \ \ |M_2|\lesssim \frac 1{t^2}.
\ee


\subsection{Equation for $\e$}


Let us start by writing the equation for $\e$. Using $\frac{d s_j}{dt}=\frac{1}{\l_j}$,
we compute from \fref{dependancegammar} the generalized modulation equations:
\be
\label{eqgammeexacte}
\Gamma_{s_1}=\frac{(\gamma_2)_{s_2}}{\mu}-(\gamma_1)_{s_1}=\frac{1}{\mu}-1+\frac{(\gamma_2)_{s_2}-1}{\mu}-((\gamma_1)_{s_1}-1), \ \  \Gamma_{s_2}=\mu \Gamma_{s_1}
\ee
and 
\bea
\label{eqgammeexactebis}
\nonumber R_{s_1}&=&1-b+(B_1-M_1)R+\frac{1}{1-\beta_1}\left(\frac{(x_2)_{s_2}}{\l_2}-\beta_2\right)-\frac{1}{1-\beta_1}\left(\frac{(x_1)_{s_1}}{\l_1}-\beta_1\right)\\
& - &R\left(\frac{(\l_1)_{s_1}}{\l_1}-M_1\right)+R\left(\frac{(\beta_1)_{s_1}}{1-\beta_1}-B_1\right)
\eea
We compute by construction:
$$i\pa_t \Phi_{\tilde{P}}-|D|\Phi_{\tilde{P}}+\Phi_{\tilde{P}}|\Phi_{\tilde{P}}|^2=\Psi+\sum_{j=1}^2\frac{1}{\l_j^{\frac 32}}S_j\left(t,y_j\right)e^{i\gamma_j}, \ \ j=1,2,$$ 
where 
\bea
\label{defsj}
\nonumber &&S_j(t,y_j) :=  -i\left[\frac{(\l_j)_{s_j}}{\l_j}-M_j\right]\Lambda V_j-\frac{1}{1-\b_j}\left[\frac{(x_j)_{s_j}}{\l_j}-\beta_j\right]i\pa_{y_j}V_j\\
\nonumber& + & \left[\frac{(\b_j)_{s_j}}{1-\b_j}-B_j\right]i[y_j\pa_{y_j}V_j+(1-\beta_j)\pa_{\beta_j}V_j]-[(\gamma_j)_{s_1}-1]V_j\\
& + & \tilde{S}_j
\eea
encodes the deviation of modulation equations from the idealized dynamical system \eqref{dynamicalystem} with the lower order error computed from \fref{eqgammeexacte}:
\bea
\label{tildestwobis}
&&\tilde{S}_1:= i\left[\frac{\gamma_{s_2}-1}{\mu}-(\gamma_{s_1}-1)\right]\frac{\pa V_1}{\pa \Gamma}\\
\nonumber & + &i \Bigg\{\frac{1}{1-\beta_1}\left(\frac{x_{s_2}}{\l_2}-\beta_2\right)-\frac{1}{1-\beta_1}\left(\frac{x_{s_1}}{\l_1}-\beta_1\right)- R\left(\frac{(\l_1)_{s_1}}{\l_1}-M_1\right)\\
\nonumber & + & R\left(\frac{(\beta_1)_{s_1}}{1-\beta_1}-B_1\right)\Bigg\}\frac{\pa V_1}{\pa R}\\
\nonumber & + &i\l_1\left[\frac{(\lambda_1)_{s_1}}{\lambda_1}-M_1\right]\frac{\pa V_1}{\pa \l_1} +i\l_1\left[\frac{(\lambda_2)_{s_2}}{\lambda_2}-M_2\right]\frac{\pa V_1}{\pa \l_2}\\
\nonumber &+&i\frac{(1-\beta_2)}{\mu}\left[\frac{(\beta_2)_{s_2}}{1-\beta_2}-B_2\right]\frac{\pa V_1}{\pa \beta_2}
\eea
\bea
\label{tildestwo}
&&\tilde{S}_2:= i\left[\gamma_{s_2}-1-\mu(\gamma_{s_1}-1)\right]\frac{\pa V_2}{\pa \Gamma}\\
\nonumber & + &i\mu \Bigg\{\frac{1}{1-\beta_1}\left(\frac{x_{s_2}}{\l_2}-\beta_2\right)-\frac{1}{1-\beta_1}\left(\frac{x_{s_1}}{\l_1}-\beta_1\right)- R\left(\frac{(\l_1)_{s_1}}{\l_1}-M_1\right)\\
\nonumber & + & R\left(\frac{(\beta_1)_{s_1}}{1-\beta_1}-B_1\right)\Bigg\}\frac{\pa V_2}{\pa R}\\
\nonumber & + & i\l_2\left[\frac{(\lambda_1)_{s_1}}{\lambda_1}-M_1\right]\frac{\pa V_2}{\pa \l_1}+ i\l_2\left[\frac{(\lambda_2)_{s_2}}{\lambda_2}-M_2\right]\frac{\pa V_2}{\pa \l_2}\\
\nonumber &+& i\mu(1-\beta_1)\left[\frac{(\beta_1)_{s_1}}{1-\beta_1}-B_1\right]\frac{\pa V_2}{\pa \beta_1}.
\eea
The error term 
\be
\label{defpsi}
\Psi(t,x)=\sum_{j=1}^2\frac{1}{\l_j^{\frac 32}}\mathcal E_{j,N}\left(y_j,\mathcal P(t)\right)e^{i\gamma_j}
\ee 
encodes the error in the construction of $V_j$ and satisfies by construction
\be
\label{errorpsin}
\|\Psi\|_{H^2}\leq\frac{C_N}{\eta^{C_0}R^{N+1}}\leq \frac{1}{\eta^Ct^{N+1}},
\ee
where we recall that $N$ will be fixed later and $\eta<\eta^*(N)$. 
We write the equation for $\e$,
\be
\label{eqe}
i\pa_t\e-|D|\e+2|\Phi_{\tilde{P}}|^2\e+(\Phi_{\tilde{P}})^2\overline{\e}=-N(\e)-\Psi-\sum_{j=1}^2\frac{1}{\l_j^{\frac 32}}S_j\left(t,y_j\right)e^{i\gamma_j},
\ee
where $$N(\e)=(\Phi_{\tilde{P}}+\e)|\Phi_{\tilde{P}}+\e|^2-\Phi_{\tilde{P}}|\Phi_{\tilde{P}}|^2-2|\Phi_{\tilde{P}}|^2\e-(\Phi_{\tilde{P}})^2\overline{\e}.$$ 
In the sequel, we use the notation $$j+1=1\ \ \mbox{for}\ \ j=2.$$ 


\subsection{Modulation equations}


At this stage we can evaluate the right hand side of the modulation system applied to the parameters $\mathcal P(t)$ given by the modulation argument.

\begin{lemma}[Modulation equations]
\label{lemmmod}
Let 
$${\rm Mod}_j(t):= \left|\frac{(\l_j)_{s_j}}{\l_j}-M_j\right|+\frac{1}{1-\b_j}\left|\frac{(x_j)_{s_j}}{\l_j}-\beta_j\right|+\left|\frac{(\b_j)_{s_j}}{1-\b_j}-B_j\right|+\left|(\gamma_j)_{s_j}-1\right|,$$
then 
\be
\label{estmod}
\Mod_j (t)\lesssim  \frac{1}{\eta^Ct^{N+1}}+\frac{\|\e_j\|_{L^2}}{t}.
\ee
\end{lemma}

\begin{proof}[Proof of Lemma \ref{lemmmod}] Let $j=1$ or $j=2$ and consider a generic multiplier 
\be
\Theta(t,x)=\frac{1}{\l_j^{\frac 12}}\Theta_j(y_j,\beta_j)e^{i\gamma_j}, 
\ee
with $\Theta_j$ strongly $j$-admissible. We compute from \fref{eqe}:
\begin{align}
\frac{d}{dt}(\e,\Theta)&=(\e,\pa_t\Theta)+(i\pa_t\e,i\Theta)= 
(\e,-i\pa_t(i\Theta)+|D|(i\Theta)-2|\Phi_{\tilde{\mathcal P}}|^2(i\Theta)-(\Phi_{\tilde{\mathcal P}})^2\overline{i\Theta}))\notag\\
&- \left(N(\e)+\Psi+\Sigma_{k=1}^2\frac{1}{\l_k^{\frac 32}}S_k\left(y_k\right)e^{i\gamma_k},i\Theta\right)\label{pa_t_eps_Theta}
\end{align}
and estimate all terms in this identity.\\

\noindent{\em \underline{The linear terms}}. Using the fact that $M_j,B_j$ are $L^\infty$-admissible, we estimate:
\begin{align*}
i\pa_t\Theta -|D|\Theta& = \frac{1}{\l_j^{\frac 32}}\Big[ -\frac{(|D|-\beta_j D)\Theta_j}{1-\b_j}-\Theta_j-i\frac{(\lambda_j)_{s_j}}{\l_j}\Lambda \Theta_j-\frac{i}{1-\b_j}\left(\frac{(x_j)_{s_j}}{\l_j}-\beta_j\right)\pa_{y_j}\Theta_j\\
& + i\frac{(\beta_j)_{s_j}}{1-\beta_j}[y_j\pa_{y_j}\Theta_j+(1-\beta_j)\pa_{\b_j}\Theta_j]-((\gamma_j)_{s_j}-1)\Theta_j\Big]e^{i\gamma_j}(y_j)\\
& = -\frac{1}{\l_j^{\frac 32}}\left[\frac{(|D|-\beta_j D)\Theta_j}{1-\b_j}+\Theta_j \right]e^{i\gamma_j}(y_j)\\
&+O\left(\Mod_j(t)\big(|\Lambda \Theta_j(y_j)|+|\pa_{y_j}\Theta (y_j)|
+|y_j\pa_{y_j}\Theta_j|+|(1-\b_j)\pa_{\b_j}\Theta_j|+|\Theta_j(y_j)|\big)\right)\\
&+O\Big(|M_j||\Lambda\Theta_j|+|B_j||y_j\pa_{y_j}\Theta_j+(1-\b_j)\pa_{\b_j}\Theta_j|\\
& = -\frac{1}{\l_j^{\frac 32}}\left[\frac{(|D|-\beta_j D)\Theta_j}{1-\b_j}+\Theta_j \right]e^{i\gamma_j}(y_j)\\
&+\Big(\Mod_j(t)+\frac 1t\Big)O\Big(|\Theta_j|+|\pa_{y_j}\Theta_j|+|\Lambda \Theta_j|+|(1-\b_j)\pa_{\b_j}\Theta_j|\Big).
\end{align*}
Then, changing to the $y_j$ variable, using the definition of $\e_j$ in \eqref{defej},
and Cauchy--Schwarz,
we have:
\begin{align*}
(\eps,&-i\pa_t(i\Theta)+|D|i\Theta-2|\Phi_{\tilde{\mathcal P}}|^2(i\Theta)-(\Phi_{\tilde{\mathcal P}})^2\overline{(i\Theta)})
 = \Big(\eps,\frac{1}{\l_j^{\frac 32}}\left[\matchal L_{\beta_j}(i\Theta_j)\right]e^{i\gamma_j}(y_j)\Big)\\
&  +(1-\b_j)\Big(\Mod_j(t)\|\e_j\|_{L^2}+\frac{\|\e_j\|_{L^2}}{t}\Big)\\
&\times O\Big(\|\Theta_j\|_{L^2}+\|\pa_{y_j}\Theta_j\|_{L^2}+\|\Lambda \Theta_j\|_{L^2}+\|(1-\b_j)\pa_{\b_j}\Theta_j\|_{L^2}\Big)\\
&+(1-\b_j)O\Big(\|(|V_j|^2-|Q_{\b_j}|^2)\Theta_j\|_{L^2}\|\e_j\|_{L^2}\Big)\\
&+(1-\b_j)O\Big(\||V_{j+1}|^2\Theta_j\|_{L^2}\|\e_j\|_{L^2}\Big)\\
&+(1-\b_j)O\Big(\|V_jV_{j+1}\Theta_j\|_{L^2}\|\e_j\|_{L^2}\Big)
 \end{align*}
with the convention $y_{j+1}=y_1$ for $j=2$. 
To estimate the remainder, we estimate using that $R(V_j-Q_{\beta_j})$ is $j$-admissible:
\begin{align*}
\|(|V_j|^2-|Q_{\b_j}|^2)\Theta_j\|_{L^2}&\lesssim \|\frac{1}{R\la y_j\ra}\|_{L^2}\lesssim \frac 1t.
\end{align*}

We now use 
\be
\label{lienparametres}
y_1=R+b\mu y_2
\ee 
so that $|y_1|\leq \frac R2$ implies $|y_2|\geq \frac{R}{2\mu b}$ and hence the bounds
\bee
\int \frac{dy_1}{\la y_1\ra^2\la y_2\ra^4}&= &  \int_{|y_1|\leq\frac R2} \frac{d y_1}{\la y_1\ra^2\la y_{2}\ra^4}+\int_{|y_1|\geq \frac R2} \frac{d y_1}{\la y_1\ra^2\la y_{2}\ra^4} \\
& \lesssim & \frac{b^4}{R^4}\int_{|y_1|\leq \frac R2} \frac{d y_1}{\la y_1\ra^2}+\frac{1}{R^2}\int \frac{bd y_2}{\la y_{2}\ra^4}\lesssim \frac{b}{R^2}\lesssim \frac{b}{t^2},
\eee
\bee
\int \frac{dy_1}{\la y_2\ra^2\la y_1\ra^4}&= &  \int_{|y_1|\leq\frac R2} \frac{d y_1}{\la y_2\ra^2\la y_{1}\ra^4}+\int_{|y_1|\geq \frac{R}{2}} \frac{d y_1}{\la y_2\ra^2\la y_{1}\ra^4} \\
& \lesssim &  \frac{b^2}{R^2}\int_{|y_1|\leq\frac R2} \frac{d y_1}{\la y_{1}\ra^4}+\frac{1}{R^4}\int\frac{b\, dy_2}{\la y_2\ra^2}\lesssim \frac{1}{t^2},
\eee
which implies
\[\||V_{j+1}|^2\Theta_j\|_{L^2}+\|V_jV_{j+1}\Theta_j\|_{L^2}\lesssim \frac 1t.\]

\noindent
The above collection of bounds yields
\begin{align}\label{Mod1}
(\eps, -i\pa_t(i\Theta)+|D|\Theta&-2|\Phi_{\tilde{P}}|^2(i\Theta)-(\Phi_{\tilde{P}})^2\overline{(i\Theta)})
 = \frac{1-\b_j}{\l_j}\big(\e_j,\matchal L_{\beta_j}(i\Theta_j)\big)\notag\\
& +(1-\b_j)O\Big[ \frac{\|\e_j\|_{L^2}}{t}+\Mod_j(t)\|\e_j\|_{L^2}\Big]
 \end{align}
 \noindent{\em  \underline{The nonlinear term}}. We estimate using \eqref{bootdynamicalproved}:
\begin{align}
|(N(\e),i\Theta_j)|&\lesssim (1-\beta_j)\int \frac{|\e_j|^2|\Phi_{\tilde{P}}|+|\e_j|^3}{\la y_j\ra}dy_j
\lesssim (1-\beta_j)\big(\|\e_j\|_{L^2}^2+\|\e_j\|_{L^2}^2\|\e_j\|_{H^1}\big)\notag\\
&\lesssim (1-\b_j) \|\e_j\|_{L^2}^2\leq (1-\beta_j)\frac{\|\e_j\|_{L^2}}{t}\label{Mod2}
\end{align}
\noindent{The $\Psi$ term}. From \eqref{errorpsin},
\be
\label{Mod3}(\Psi,i\Theta)\lesssim \frac{1}{\eta^Ct^{N+1}}.
\ee
\noindent{\em  \underline{The $S$-terms and conclusion}}. We now pick $$\Theta_j\in A_j:=\{Q_{\beta_j},i\pa_{y_j}Q_{\beta_j},\Lambda Q_{\beta_j},\rho_j\}$$ which are strongly $j$-admissible, and estimate all terms in \eqref{pa_t_eps_Theta} using \eqref{Mod1}, \eqref{Mod2}, \eqref{Mod3}. The derivative in time of $(\eps , \Theta )$ drops using the orthogonality conditions \fref{orthoe}. 
Moreover, the same orthogonality conditions  \eqref{orthoe} imply that $(\eps_j, \mathcal{L}_{\b_j}(i\Theta_j))=0$.
We now use Appendix \ref{invertaj} to compute all the scalar products and conclude:
$$\Big|\Big((S_j-\tilde S_j)e^{i\gamma_j},i\Theta\Big)\Big|\sim (1-\beta_j) \Mod_j.$$  Thus, in order to estimate $\Mod_j$, we are left with computing the crossed terms and the error $\tilde{S}_j$ terms given by \eqref{tildestwo}, \eqref{tildestwobis}. The detailed estimates are given below.\\

\smallskip
\noindent{\em  \underline{Case $j=1$}}. We rescale to the $y_1$ variable and use the $1$--admissibility of $R(V_1-Q_{\b_1})$ to estimate:
$$|(\tilde{S}_1e^{i\gamma_1},i\Theta)|\lesssim (1-\beta_1)\frac{|\Mod_1|+|\Mod_2|}{t}.$$ 
We now recall \eqref{lienparametres} to estimate:
\bee
 &&\int \frac{d y_1}{\la y_1\ra(1+(1-\beta_1)\la y_1\ra)\la y_{2}\ra(1+(1-\beta_2)\la y_2\ra)}\\
& \lesssim & \int_{ |y_1|\leq \frac R2} \frac{d y_1}{\la y_1\ra(1+(1-\beta_1)\la y_1\ra)\la y_{2}\ra}\\
&+& \int_{\substack{|y_1|\geq \frac R2 }}\frac{d y_1}{\la y_1\ra(1+(1-\beta_1)\la y_1\ra)\la y_{2}\ra(1+(1-\beta_2)\la y_2\ra)} 
\\
&\lesssim &\frac{b}{R}\int \frac{d y_1}{\la y_1\ra(1+(1-\beta_1)\la y_1\ra)}
+\frac 1{R(1+\eta R)} \int \frac{b d y_2}{\la y_{2}\ra(1+(1-\beta_2)\la y_2\ra)} \\
&\lesssim & \frac{b}{t}\frac{\vert \log \eta\vert +\log t}{1+\eta t} \lesssim \frac{b\vert \log \eta \vert}{t},
\eee 
and hence the estimate of the crossed term:
$$
|(S_2e^{i\gamma_2},\Theta_1)|\lesssim (1-\beta_1)\left[\frac{b\vert \log \eta\vert }{t}(\Mod_2+\Mod_1)\right].
$$
This yields the first bound,
\bea
\label{fisboeieoen}
\Mod_1\lesssim \frac{\Mod_1+\Mod_2}{t}+\frac{\|\e_1\|_{L^2}}{t}+\frac{1}{\eta^Ct^{N+1}},
\eea

\smallskip
\noindent{\em  \underline{Case $j=2$}}. We estimate similarly 
$$|(\tilde{S}_2e^{i\gamma_2},\Theta)|\lesssim (1-\beta_2)\frac{\Mod_1+\Mod_2}{t}$$ and 
\bee
&&\int \frac{d y_2}{\la y_1\ra(1+(1-\beta_1)\la y_1\ra)(\la y_{2}\ra(1+(1-\beta_2)\la y_2\ra)}\\
& \lesssim &\frac{1}{b}\int \frac{d y_1}{\la y_1\ra(1+(1-\beta_1)\la y_1\ra)(\la y_{2}\ra(1+(1-\beta_2)\la y_2\ra)}\lesssim \frac{|\log \eta|}{t}
\eee
from which
\bee
\Mod_2&\lesssim& \frac{\Mod_1+\Mod_2}{t}+\frac{|\log \eta|}{t}(\Mod_1+\Mod_2)+\frac{\|\e_2\|_{L^2}}{t}\\
& \lesssim & \frac{|\log \eta|}{t}(\Mod_1+\Mod_2)+\frac{\|\e_2\|_{L^2}}{t}+\frac{1}{\eta^Ct^{N+1}}.
\eee

\smallskip
\noindent{\it  \underline{Conclusion}}. Combined with \eqref{fisboeieoen}, since $t\gg \vert \log \eta \vert $, this yields $$\Mod_1+\Mod_2\lesssim \frac{\|\e_1\|_{L^2}+\|\e_2\|_{L^2}}{t}+\frac{1}{\eta^Ct^{N+1}}$$ and hence using $\|\e_1\|_{L^2}=\sqrt{b}\|\e_2\|_{L^2}$:
\bee
\Mod_2&\lesssim &\Mod_1+\Mod_2\lesssim \frac{\|\e_1\|_{L^2}+\|\e_2\|_{L^2}}{t}+\frac{1}{\eta^Ct^{N+1}}\lesssim  \frac{1}{\eta^Ct^{N+1}}+\frac{\|\e_2\|_{L^2}(1+\sqrt{b})}{t}\\
& \lesssim & \frac{1}{\eta^Ct^{N+1}}+\frac{\|\e_2\|_{L^2}}{t}
\eee 
and from \eqref{fisboeieoen}:
\bee
\Mod_1&\lesssim& \frac{1}{\eta^Ct^{N+1}}+\frac{\|\e_1\|_{L^2}}{t}+\frac{\|\e_2\|_{L^2}}{t^2}\lesssim  \frac{1}{\eta^Ct^{N+1}}+\frac{\|\e_1\|_{L^2}}{t}\left[1+\frac{1}{t\sqrt{b}}\right]\\
& \lesssim &  \frac{1}{\eta^Ct^{N+1}}+\frac{\|\e_1\|_{L^2}}{t}
\eee
where we used 
\be
\label{estbtfond}
t\sqrt{b}\gtrsim 1
\ee 
\end{proof}


\subsection{Energy estimate}


We are now in position to derive the key monotonicity formula for the linearized energy $\mathcal G$ which is the second crucial element of our analysis.

\begin{proposition}[Energy estimate for $\mathcal G$]
\label{propenergy}
There holds the improved pointwise bound on $[T_{\rm in}, T_n]$:
\be
\label{energyestimate}
\mathcal G(\eps(t))\leq \frac{C}{Nt^{\frac N2}}
\ee
for some universal constant $C$ independent of $N,\eta,t$.
\end{proposition}

\begin{proof}[Proof of Proposition \ref{propenergy}] The proof relies on the careful treatment of all terms induced by the localization of mass and energy when computing the time variation of the energy $\matchal G$. The main difficulty is the loss of control of the kinetic energy and mass as $\beta\to1$ for $\e_1^+$ as reflected by \eqref{coerc}, which forces different set of estimates for $\e^\pm$.\\
We rewrite \fref{eqe} as:
\be
\label{eqebis}
i\pa_t\e-|D|\e+(\Phi+\e)|\Phi+\e|^2-\Phi|\Phi|^2=F,
\ee
$$F:=-\Psi-S, \ \ S=\sum_{j=1}^2\frac{1}{\l_j^{\frac 32}}S_j\left(y_j\right)e^{i\gamma_j},$$
or equivalently
\be
\label{defequagt}
\left|\begin{array}{lll} i\pa_t\e-|D|\e+2|\Phi|^2\e+\Phi^2\overline{\e}=G\\ 
N(\e):=(\Phi+\e)|\Phi+\e|^2-\Phi|\Phi|^2-2|\Phi|^2\e-\Phi^2\overline{\e}\\
G:=F-N(\e)=-\Psi-S-N(\e).
 \end{array}\right.
 \ee

\noindent{\bf Step 1:} Localization of mass. We compute the localized mass conservation law and claim
\bea
\label{firstenergyalgebrabism}
\nonumber \frac{d}{dt}\frac 12(\theta\e,\e)&=&\frac 12((\pa_t \theta) \e,\e)+(-i|D|\e,\theta \e)+(i\Phi^2,\theta\e^2)+(iS,\theta\e)\\
&+& O\left(\frac{\matchal G}{t}+\frac{1}{t^{N+1}}\right).
\eea
Indeed, from \eqref{defequagt}:
\bea
\label{deriv_mass}
 \frac{d}{dt}\frac 12(\theta\e,\e)&=&(\theta\pa_t \e,\e)+\frac 12((\pa_t \theta) \e,\e)\\
\nonumber &=& (-i|D|\e +i (2|\Phi|^2\e+\Phi^2\bar{\e})-i G, \theta \e)+\frac 12((\pa_t \theta) \e,\e)\\
&= & (i\Phi^2,\theta\e^2)-(iG,\theta\e)+\frac 12((\pa_t \theta) \e,\e) \notag\\
\nonumber&=& \frac 12((\pa_t \theta) \e,\e)+(-i|D|\e,\theta \e)+(i\Phi^2,\theta\e^2)+ (iN(\e),\theta \e)+(i\Psi, \theta \e)+(iS, \theta \e).
\eea
We estimate from \eqref{errorpsin}, \fref{bootdynamicalproved}:
$$|(\Psi,\theta\e)|\lesssim \frac{\|\e\|_{L^2}}{\eta^Ct^{N+1}}\lesssim \frac{1}{t^{N+1}}.$$ 
For the nonlinear term, we estimate from \eqref{bootdynamicalproved} and \eqref{coerc},
$$|(N(\e),\theta \e)|\lesssim \int (|\e|^4+|\e|^3)\lesssim \|\e\|_{L^{\infty}}\|\e\|_{L^2}^2\lesssim \frac{\mathcal G}{t}$$
and \eqref{firstenergyalgebrabism} is proved.\\

\noindent{\bf Step 2:} Localization of kinetic momentum. We compute the localized kinetic momentum conservation law and claim
\bea
\label{qbuibeibbei}
\frac 12\frac{d}{dt}(\zeta D\e,\e)&=& (2\Phi|\e|^2+\overline{\Phi}\e^2,\zeta\pa_x\Phi)+O\left(\frac{\matchal G}{t}+\frac{1}{t^{N+1}}\right)\\
\nonumber &+&  \frac 12 (\pa_t\zeta D\e,\e)+(-i|D|\e,\zeta D\e+\frac 12\e D\zeta)+(iS,\zeta D\e+\frac 12\e D\zeta).
\eea
Indeed, we compute from \eqref{eqebis}:
\bee
\frac 12\frac{d}{dt}(\zeta D\e,\e)&=&\frac 12 (\pa_t\zeta D\e,\e)+\frac 12(\zeta D\pa_t\e,\e)+\frac 12(\zeta D\e,\pa_t\e)\\
& = & \frac 12 (\pa_t\zeta D\e,\e)+(\pa_t\e,\zeta D\e+\frac 12\e D\zeta)\\
& = & \frac 12 (\pa_t\zeta D\e,\e)+(-i|D|\e +i (2|\Phi|^2\e+\Phi^2\bar{\e})-i G,\zeta D\e+\frac 12\e D\zeta)\\
& = & \frac 12 (\pa_t\zeta D\e,\e)+(-i|D|\e,\zeta D\e+\frac 12\e D\zeta)\\
& + & (i(2|\Phi|^2\e+\Phi^2\bar{\e}),\zeta D\e+\frac 12\e D\zeta)+(-iG,\zeta D\e+\frac 12\e D\zeta).
\eee
We integrate by parts the quadratic term using the pointwise bound \eqref{estderivativezeta}:
\bee
(i(2|\Phi|^2\e+\Phi^2\bar{\e}),\zeta D\e)=(2\Phi|\e|^2+\overline{\Phi}\e^2,\zeta\pa_x\Phi)+O\left(\frac{\|\e\|_{L^2}^2}{t}\right).
\eee
We estimate from \eqref{errorpsin} after integrating by parts:
$$|(i\Psi,\zeta D\e+\frac 12\e D\zeta)|\lesssim \|\Psi\|_{H^1}\|\e\|_{L^2}\lesssim \frac{1}{t^{N+1}}.$$
For the nonlinear term:
$$|(iN(\e),\zeta D\e+\frac 12\e D\zeta)|\lesssim \|\e\|_{H^1}\|\e\|_{L^2}^2\lesssim \frac{\mathcal G}{t}$$
and \eqref{qbuibeibbei} is proved.\\

\noindent{\bf Step 3:} Localized energy identity. We now compute the variation of the linearized energy:
\begin{align}
\label{expersionnenergyeacrte}
\frac{d}{dt}&\left\{\frac 12(|D|\e,\e)-\frac14 \left[\int (|\e+\Phi|^4-|\Phi|^4)-(4\e,\Phi|\Phi|^2)\right]\right\}\\
&=(\pa_t\e,|D|\e)
-\left((\e+\Phi)|\e+\Phi|^2, \pa_t \e+\pa_t \Phi\right)
+(\Phi |\Phi|^2, \pa_t \Phi)\notag\\
&+(\pa_t\e, \Phi |\Phi|^2)
+(\e, \pa_t (\Phi|\Phi|^2))\notag\\
&=(\pa_t \e, |D|\e-(\e+\Phi)|\e+\Phi|^2+\Phi |\Phi|^2)
-(\pa_t\Phi, N(\e))\notag\\
&=(i\Psi +iS, |D|\e-(\e+\Phi)|\e+\Phi|^2+\Phi |\Phi|^2)-(\pa_t\Phi, N(\e))\notag
\end{align}
We estimate all terms in \fref{expersionnenergyeacrte} and in particular first extract the quadratic terms. 
From \eqref{errorpsin}, Sobolev, $\|\Phi\|_{L^{\infty}}\lesssim 1$  and \eqref{bootdynamicalproved}:
\be
\label{Psi}
\Big | (i\Psi,|D|\e-(\Phi+\e)|\Phi+\e|^2+\Phi|\Phi|^2)\Big|\lesssim \|\Psi \|_{H^1}\|\e\|_{L^2}\lesssim \frac{1}{t^{N+1}}.
\ee
Let us estimate the term $(\pa_t\Phi, N(\e))$. Since $V_j$, $R(V_j-Q_{\b_j})$ are $j$--admissible, and $RM_j,RB_j$ are $L^\infty$-admissible, we compute
$$
\pa_{s_j}V_j=\sum_{k=1}^2 \Big[\frac{\pa V_j}{\pa \l_k}(\l_k)_{s_j}+(1-\b_k)\frac{\pa V_j}{\pa \b_k}\cdot\frac{(\b_k)_{s_j}}{1-\b_k}\Big]
+\frac{\pa V_j}{\pa \Gamma} \Gamma_{s_j} + \frac{\pa V_j}{\pa R} R_{s_j}\\
$$
and hence, using \eqref{estmod} and the bootstrap assumption, we infer
\be
\label{beubebei}
\vert \pa_sV_j\vert \ \lesssim \frac{1}{t\la y_j\ra}\ .
\ee
Consequently, the  admissibility of $V_j$,  \eqref{estmod}, 
and the bounds
$1-\b_1\sim\eta$ and $1-\b_2\gtrsim \eta^{3}$ ensure
\begin{align}
\label{expressionpatphi}
\pa_t\Phi = \sum_{j=1}^2\frac{1}{\l_j^{\frac 32}}\Big[&\pa_{s_j}V_j-\frac{(\l_j)_{s_j}}{\l_j}\Lambda V_j-\frac{1}{1-\beta_j}\left[\frac{(x_j)_{s_j}}{\l_j}-\b_j\right]\pa_{y_j}V_j
-\frac{\b_j}{1-\b_j}\pa_{y_j}V_j\notag\\
&+\frac{(\beta_j)_{s_j}}{1-\beta_j}y_j\pa_{y_j}V_j+i(\gamma_j)_{s_j}V_j\Big]e^{i\gamma_j}(y_j)
=  O\left(\sum_{j=1}^2\frac{1}{\eta^{C}\la y_j\ra}\right)
\end{align}
We use this with \eqref{bootdynamicalproved} to estimate:
\bee
&&-(\pa_t \Phi, N(\e))=
-(\pa_t\Phi,(\Phi+\e)|\Phi+\e|^2-\Phi|\Phi|^2-2|\Phi|^2\e-\Phi^2\overline{\e})\\
& = & - \left(\pa_t\Phi,2\Phi|\e|^2+\overline{\Phi}\e^2\right) 
+  O\left(\frac{\|\e\|_{H^{1}}\|\e \|_{L^2}^2}{\eta^{C}}\right)=  - \left(\pa_t\Phi,2\Phi|\e|^2+\overline{\Phi}\e^2\right)  +O\left(\frac{\mathcal G}{t}\right).
\eee
similarly, using \fref{estmod} and \eqref{bootdynamicalproved}:
\bee
&& (iS,|D|\e-(\Phi+\e)|\Phi+\e|^2+\Phi|\Phi|^2)\\
&=& (iS,|D|\e-2|\Phi|^2\e-\Phi^2\overline{\e})+O\left( \frac{\Mod_1+\Mod_2}{\eta^C}\|\e\|_{L^2}^2\right)\\
& =& (iS,|D|\e-2|\Phi|^2\e-\Phi^2\overline{\e})+O\left(\frac{\mathcal G}{t}\right).
\eee
The collection of above bounds yields
\bea
\label{firstenergyalgebrabis}
&&\frac{d}{dt}\left\{\frac 12(|D|\e,\e)-\frac14 \left[\int (|\e+\Phi|^4-|\Phi|^4)-(4\e,\Phi|\Phi|^2)\right]\right\}\\
\nonumber & = & (\e,|D|(iS)-2|\Phi|^2(iS)-\Phi^2\overline{iS}) - \left(\pa_t\Phi,2\Phi|\e|^2+\overline{\Phi}\e^2\right)\\ 
\nonumber &+& O\left( \frac{\mathcal G}{t}+\frac{1}{t^{N+1}}\right).
\eea
We now treat the remaining quadratic terms more carefully and combine them with the leading order quadratic terms in \eqref{firstenergyalgebrabism}, \eqref{qbuibeibbei}. Indeed, we rewrite \fref{expressionpatphi} using \fref{estmod}, \eqref{beubebei}, \eqref{boundessential} and the $j$-admissibility of $V_j$:
\bee
&&\pa_t\Phi\\
&=& \sum_{j=1}^2\frac{1}{\l_j^{\frac 32}}\left[\pa_{s_j}V_j-\frac{(\l_j)_{s_j}}{\l_j}\Lambda V_j-\frac{1}{1-\beta_j}\left[\frac{(x_j)_{s_j}}{\l_j}\right]\pa_{y_j}V_j+\frac{(\beta_j)_{s_j}}{1-\beta_j}y_j\pa_{y_j}V_j+i(\gamma_j)_{s_j}V_j\right]e^{i\gamma_j}(y_j)\\
& = &\sum_{j=1}^2  \frac{i}{\l_j}\Phi^{(j)}-\beta_j\pa_x\Phi^{(j)}+O\left(\sum_{j=1}^2 \frac 1t \frac{1}{\la y_j\ra}\right)
\eee
where we have set
$$\Phi^{(j)}(t,x):=\frac{1}{\l_j^{\frac 12}} V_j\left (\frac{x-x_j}{\l_j(1-\b_j)}\right )e^{i\gamma_j}\ .$$
We infer the bound
\bee
 &&-  \left(\pa_t\Phi,2\Phi|\e|^2+\overline{\Phi}\e^2\right)+(\e^2,i\theta\Phi^2)-(2\Phi|\e|^2+\overline{\Phi}\e^2,\zeta\pa_x\Phi)\\
\nonumber & =&\left(\beta_1\pa_x\Phi^{(1)}
+\beta_2\pa_x\Phi^{(2)}+O\left(\Sigma_{j=1}^2 \frac 1t \frac{1}{\la y_j\ra}\right),2\Phi|\e|^2+\overline{\Phi}\e^2\right)-(2\Phi|\e|^2+\overline{\Phi}\e^2,\zeta\pa_x\Phi)\\
\nonumber & + & (\e^2,i\Phi\left[\theta\Phi-\Sigma_{j=1}^2\frac{1}{\l_j}\Phi^{(j)}\right])-2(i\left(\frac{\Phi^{(1)}}{\l_1}+\frac{\Phi^{(2)}}{\l_2}\right),(\Phi^{(1)}+\Phi^{(2)})|\e|^2)\\
\nonumber & = & -(2\Phi|\e|^2+\overline{\Phi}\e^2,(\zeta-\beta_1)\pa_x\Phi^{(1)}+(\zeta-\beta_2)\pa_x\Phi^{(2)}) + (\e^2,i\Phi\Sigma_{j=1}^2\left(\theta-\frac{1}{\l_j}\right)\Phi^{(j)})\\
& -&\frac{2}{\l_1}(|\e|^2,i\Phi^{(1)}\overline{\Phi^{(2)}})-\frac{2}{\l_2}(|\e|^2,i\Phi^{(2)}\overline{\Phi^{(1)}}) +O\left(\frac{\|\e\|_{L^2}^2}{t}\right)\\
\eee
 We recall \eqref{lienparametres}, and hence $|y_1|\leq \frac R2$ implies $|y_2|\geq \frac{R}{2b\mu}$ from which $$\||\Phi^{(1)}\Phi^{(2)}\|_{L^{\infty}}\lesssim \|\frac{1}{\la y_1\ra \la y_2\ra}\|_{L^{\infty}}\lesssim \frac 1t$$ and hence  
 $$|(|\e|^2,i\Phi^{(1)}\overline{\Phi^{(2)}})|+|(|\e|^2,i\Phi^{(2)}\overline{\Phi^{(1)}})|\lesssim \frac{\|\e\|_{L^2}^2}{t}\lesssim \frac{\mathcal G}{t}.$$ We then 
 use $\la y_1\ra\gtrsim R$ on ${\rm Supp}(1-\phi_1)$ and ${\rm Supp}(\frac 1{\l_1}-\theta)$ and the explicit formula \eqref{defzetabeta} to estimate:
 \bee
&& \left|(\zeta-\beta_1)\pa_x\Phi^{(1)}\right|\lesssim \frac{|1-\phi_1|}{\la y_1\ra ^2}\lesssim \frac 1t\\
 &&\left|\left(\theta-\frac{1}{\l_1}\right)\Phi^{(1)}\right|\lesssim \frac{1}{t}\ .
 \eee
 Similarly, we use  $\la y_2\ra\gtrsim R$ on ${\rm Supp}(b-\phi_1)$ and ${\rm Supp}(\frac 1{\l_2}-\theta)$, and the relation $\beta_2-\zeta=(1-\beta_1)(\phi_1-b)$ to get
 \bee
 &&\left |(\zeta-\beta_2)\pa_x\Phi^{(2)}\right|\lesssim \frac{|b-\phi_1|}{b\la y_2\ra ^2}\lesssim \frac 1t\\
  &&\left|\left(\theta-\frac{1}{\l_2}\right)\Phi^{(2)}\right|\lesssim \frac 1t.
 \eee
 The second estimate above is straightforward. Let us explain how to obtain the first estimate. Recall that $b-\phi_1=(b-1)\Psi _1$, and $0\le \Psi (z_1) \le1$, with $\Psi _1(z_1)=1$ for $z_1\le 1/4$, $\Psi _1(z_1)=(1-z_1)^{10}$ for $1/2\le z_1\le 1$, and $\Psi _1(z_1)=0$ for $z_1\ge 1$, so we may assume $z_1\ge 1$. Moreover, recall that
 $$1-z_1=1-\frac{y_1}{R(1-b)}=1-\frac{R+\mu by_2}{R(1-b)}=\frac{-b}{1-b}\left (1+\frac{\mu y_2}{R}\right )\ge 0\ .$$
 If 
 $$-1\ge \frac{\mu y_2}{R}\ge -\frac{1}{\sqrt b}, $$
  then $\vert 1-z_1\vert \lesssim \sqrt b$, and 
  $$\frac{\Psi _1}{b\la y_2\ra ^2}\lesssim \frac{b^4}{\la y_2\ra ^2}\le \frac{b^4}{R^2}\lesssim \frac{1}{t}\ .$$
  On the other hand, if 
  $$\frac{\mu y_2}{R}\le -\frac{1}{\sqrt b}, $$
  then $\la y_2\ra \gtrsim R/\sqrt b$, and
   $$\frac{\Psi _1}{b\la y_2\ra ^2}\le \frac{1}{b\la y_2\ra ^2}\lesssim \frac{1}{R^2}\lesssim \frac{1}{t}\ .$$
   We conclude using $\|\Phi\|_{L^{\infty}}\lesssim 1$: $$-  \left(\pa_t\Phi,2\Phi|\e|^2+\overline{\Phi}\e^2\right)+(\e^2,i\theta\Phi^2)-(2\Phi|\e|^2+\overline{\Phi}\e^2,\zeta\pa_x\Phi)=O\left(\frac{\mathcal G}{t}\right).$$ 
 Injecting this estimate into \fref{firstenergyalgebrabism}, \eqref{qbuibeibbei} and \fref{firstenergyalgebrabis} yields the full localized energy identity:
\bea
\label{firstenergyalgebra}
\nonumber &&\frac{d}{dt}\left\{\frac 12(|D|\e+\theta\e,\e)-\frac 12(\zeta D\e,\e)-\frac14 \left[\int (|\e+\Phi|^4-|\Phi|^4)-(4\e,\Phi|\Phi|^2)\right]\right\}\\
\nonumber & = & \frac 12((\pa_t \theta) \e,\e)+(-i|D|\e,\theta \e)+  (\e,|D|(iS)+i\theta S-2|\Phi|^2(iS)-\Phi^2\overline{iS})\\
\nonumber & - &  \frac 12 (\pa_t\zeta D\e,\e)+(i|D|\e,\zeta D\e+\frac 12\e D\zeta)-(iS,\zeta D\e+\frac 12\e D\zeta)\\
\nonumber &+& O\left(\frac{\mathcal G}{t}+\frac{1}{t^{N+1}}\right)\\
\nonumber & =& \frac 12((\pa_t \theta) \e,\e)+(-i|D|\e,\theta \e) -  \frac 12 (\pa_t\zeta D\e,\e)+(i|D|\e,\zeta D\e+\frac 12\e D\zeta)\\
\nonumber & +&(\e,(|D|-\zeta D)(iS)+i\theta S-2|\Phi|^2(iS)-\Phi^2\overline{iS})+\frac 12(\e,iSD\zeta)\\
&+& O\left(\frac{\mathcal G}{t}+\frac{1}{t^{N+1}}\right)
\eea
where we integrated by parts the term $(iS,\zeta D\e+\frac 12\e D\zeta)$ in the last step. We now estimate all remaining terms in \eqref{firstenergyalgebra}. The linear terms in \eqref{firstenergyalgebra} induced by the localization of the mass and kinetic momentum\footnote{which is necessary due to the dramatic change of size of each bubble.} are particularly critical for our analysis.\\

\noindent{\bf Step 4:} Modulation equations terms. 
We estimate the remaining modulation equations terms in \eqref{firstenergyalgebra} and claim
\be
\label{estsremainingone}
\left|(\e,(|D|-\zeta D)(iS)+i\theta S-2|\Phi|^2(iS)-\Phi^2\overline{iS})\right|+|(\e,iSD\zeta)|\lesssim \frac{\mathcal G}{t}+\frac{1}{t^{N+1}}.
\ee
Indeed, we first estimate the $S$ terms in the $y_1$ variable. From \eqref{defsj}, \eqref{tildestwobis} and \eqref{estmod} with $\|\e_2\|_{L^2}=\frac{\|\e_1\|_{L^2}}{\sqrt{b}}$:
\bea
\label{estsonehone}
\nonumber \|S_1\|_{H^1_{y_1}}&\lesssim& \Mod_1+\frac{\Mod_2}{t}\lesssim \frac{1}{\eta^Ct^{N+1}}+\frac{\|\e_1\|_{L^2}}{t}\left[1+\frac{1}{\sqrt{b}t}\right]\\
&\lesssim & \frac{1}{\eta^Ct^{N+1}}+\frac{\|\e_1\|_{L^2}}{t}
\eea
where we used \eqref{estbtfond} in the last step, and
\be
\label{estsonehonebis}
\|S_2\|_{H^1_{y_2}}\lesssim \Mod_2+\frac{\Mod_1}{t}\lesssim \frac{1}{\eta^Ct^{N+1}}+\frac{\|\e_2\|_{L^2}}{t}.
\ee
We also have similarly the pointwise bound using the admissibility of $V_j$: 
\be
\label{pointwiseboundstwo}
|\pa^k_{y_2}S_2|\lesssim \frac{1}{\la y_2\ra^k}\left[\frac{1}{\eta^Ct^{N+1}}+\frac{\|\e_2\|_{L^2}}{t}\right].
\ee
In particular,
\bea
\label{seestimatesltwo}
\nonumber \|S\|_{L^2_{y_1}}&\lesssim& \|S_1\|_{L^2_{y_1}}+\|S_2\|_{L^2_{y_1}}\lesssim \|S_1\|_{L^2_{y_1}}+\sqrt{b}\|S_2\|_{L^2_{y_2}}\\
&\lesssim& \frac{1}{\eta^Ct^{N+1}}+\frac{\|\e_1\|_{L^2}}{t}.
\eea
We therefore renormalize to the $y_1$ variable and estimate from \eqref{seestimatesltwo}, \eqref{defzetabeta}:
$$|(\e,iSD\zeta)|\lesssim (1-\beta_1) |(\e_1,S\pa_{y_1}\phi_1)|\lesssim \frac{1-\beta_1}t\|\e_1\|_{L^2}^2+\frac{1}{t^{N+1}},$$ and similarly using $\|\Phi\|_{L^{\infty}}\lesssim 1$:
$$|\left|(\e,i\theta S-2|\Phi|^2(iS)-\Phi^2\overline{iS})\right|\lesssim (1-\beta_1)\|\e_1\|_{L^2}\|S\|_{L^2_{y_1}}\lesssim \frac{1-\beta_1}t\|\e_1\|_{L^2}^2+\frac{1}{t^{N+1}}.$$ We now use $\zeta_1=\beta_1+(1-\beta_1)(1-\phi_1)=1-(1-\beta_1)\phi_1$ to compute:
\bee
|(\e,(|D|-\zeta D)(iS))|&\lesssim& |((|D|-\zeta_1D)\e_1,iS)|\\
&\lesssim &(1-\beta_1)|(\phi_1D\e_1,iS)|+ |(i\e_1^-,D\Pi^-S)|:=I+II.
\eee
We claim:
\be
\label{njeneononeone}
I+II\lesssim \frac{\mathcal G}{t}+\frac{1}{t^{N+1}}
\ee
which concludes the proof of \eqref{estsremainingone}.\\

\smallskip
\noindent{\em  \underline{Control of $I$}}. We split $S=S_1+S_2$ and first estimate after an integration by parts and using \eqref{estsonehone}:
$$|(1-\beta_1)|(\phi_1D\e_1,iS_1)|\lesssim (1-\beta_1)\|\e_1\|_{L^2}\|S_1\|_{H^1_{y_1}}\lesssim \frac{1}{t^{N+1}}+\frac{1-\beta_1}{t}\|\e_1\|_{L^2}^2\lesssim \frac{1}{t^{N+1}}+\frac{\mathcal G}{t}.$$ Next,
\bee
&&|(1-\beta_1)|(\phi_1D\e_1,iS_2)\lesssim (1-\beta_1)|(\phi_2D\e_2,iS_2)|\\
&\lesssim & (1-\beta_2)|(\e_2,iDS_2)|+(1-\beta_1)|(\e_2,D((\phi_2-b)iS_2)|.
\eee
The first term is estimated from \eqref{estsonehonebis}:
\bee
(1-\beta_2)|(\e_2,iDS_2)| \lesssim  (1-\beta_2)\|\e_2\|_{L^2}\left[\frac{1}{\eta^Ct^{N+1}}+\frac{\|\e_2\|_{L^2}}{t}\right]\lesssim \frac{1}{t^{N+1}}+\frac{\matchal G}{t}.
\eee
The second term is estimated using \eqref{pointwiseboundstwo}, \eqref{lienparametres}, $\la y_2\ra\gtrsim \frac{b}{R}$ on $\mbox{Supp}(b-\phi_2)$ and $\|\pa_{y_2}\phi_2\|_{L^{\infty}}\lesssim b\|\pa_{y_1}\phi_1\|_{L^{\infty}}\lesssim \frac bR$ so that:
\bee
(1-\beta_1)|(\e_2,D((\phi_2-b)iS_2)|&\lesssim&  (1-\beta_1)\frac{b}{R}\left[\frac{1}{\eta^Ct^{N+1}}+\frac{\|\e_2\|_{L^2}}{t}\right]\int \frac{|\e_2|}{\la y_2\ra}dy_2\\
& \lesssim & \frac{1-\beta_2}{t}\|\e_2\|_{L^2}^2+\frac{1}{t^{N+1}}\lesssim \frac{\mathcal G}{t}+\frac{1}{t^{N+1}}
\eee
which concludes the proof of \eqref{njeneononeone} for $I$.\\

\smallskip
\noindent{\em  \underline{Control of $II$}}. Consider $S_j-\tilde{S_j}$. Then by commuting the null space relations $$\mathcal L_\beta(\Lambda Q_\beta)=-Q_\beta, \ \ \mathcal L_\beta(iQ_\beta)=0, \ \ \mathcal L_\beta(\pa_yQ_\beta)=0$$ and \eqref{eqQbdot} with $\Pi^-$, we estimate:
$$\|D\Pi^-(\Lambda Q_\beta)\|_{L^2}+\|D\Pi^-Q_\beta\|_{L^2}+\|D\Pi^-\tilde{\Lambda}Q_\beta\|_{L^2}+\|D\Pi^-\pa_yQ_\beta\|_{L^2}\lesssim 1-\beta.$$ 
Hence from \eqref{defsj}:
$$\|D\Pi^-(S_j-\tilde{S}_j)\|_{L^2_{y_j}}\lesssim (1-\beta_j)\Mod_j\lesssim (1-\beta_j)\left[\frac{1}{\eta^Ct^{N+1}}+\frac{\|\e_j\|_{L^2}}{t}\right]$$ from which:
$$|(i\e_1,D\Pi^-(S_1-\tilde{S}_1))|\lesssim  (1-\beta_1)\|\e_1\|_{L^2}\left[\frac{1}{\eta^Ct^{N+1}}+\frac{\|\e_1\|_{L^2}}{t}\right]\lesssim 
\frac{\mathcal G}{t}+\frac{1}{t^{N+1}}$$
and renormalizing to the $y_2$ variable:
\bee
|(i\e_1,D\Pi^-(S_2-\tilde{S}_2))|&=&|(i\e_2,D\Pi^-(S_2-\tilde{S}_2))|\lesssim (1-\beta_2)\|\e_2\|_{L^2}\left[\frac{1}{\eta^Ct^{N+1}}+\frac{\|\e_2\|_{L^2}}{t}\right]\\
&\lesssim& \frac{\mathcal G}{t}+\frac{1}{t^{N+1}}.
\eee
We now argue similarly for the $\tilde{S}_j$ terms. Indeed, from Corollary \ref{DPiminusDprimeV}, we have 
$$\|D\Pi^-\pa_\Gamma V_j\|_{L^2}+\|D\Pi^-\Lambda_RV_j\|_{L^2}+\|D\Pi^-\pa_{\l_{j+1}} V_j\|_{L^2}+\|D\Pi^-(1-\beta_{j+1})\pa_{\beta_{j+1}} V_j\|_{L^2}\lesssim \frac{1-\beta_j}{R}.$$
Hence, arguing like for \eqref{estsonehone}:
$$\|\|D\Pi^-\tilde{S}_1\|_{L^2_{y_1}}\lesssim \frac{1-\beta_1}{t}\left[\Mod_1+\Mod_2\right]\lesssim(1-\beta_1)\left[\frac{1}{\eta^Ct^{N+1}}+\frac{\|\e_1\|_{L^2}}{t}\right]$$ which implies
$$|(i\e_1,D\Pi^-\tilde{S}_1)|\lesssim (1-\beta_1)\|\e_1\|_{L^2}\left[\frac{1}{\eta^Ct^{N+1}}+\frac{\|\e_1\|_{L^2}}{t}\right]\lesssim 
\frac{\mathcal G}{t}+\frac{1}{t^{N+1}}.$$ similarly:
$$\|\|D\Pi^-\tilde{S}_2\|_{L^2_{y_2}}\lesssim \frac{1-\beta_2}{t}\left[\Mod_1+\Mod_2\right]\lesssim (1-\beta_2)\left[\frac{1}{\eta^Ct^{N+1}}+\frac{\|\e_2\|_{L^2}}{t}\right]$$ and
$$|(i\e_2,D\Pi^-\tilde{S}_2)|\lesssim (1-\beta_2)\|\e_2\|_{L^2}\left[\frac{1}{\eta^Ct^{N+1}}+\frac{\|\e_2\|_{L^2}}{t}\right]\lesssim 
\frac{\mathcal G}{t}+\frac{1}{t^{N+1}}.$$
This concludes the proof of \eqref{njeneononeone}.\\

\noindent{\bf Step 5:} Linear momentum terms. Let 
\be
\label{defetildeoneminus}
\widetilde{\e_1}=\frac{\e_1^-}{\la z_1\ra^{\frac{1+\alpha}{2}}}, \ \ z_1=\frac{y_1}{R},
\ee
we claim: 
\bea
\label{estmomentm}
&&-\frac 12 (\pa_t\zeta D\e,\e)+(i|D|\e,\zeta D\e+\frac 12\e D\zeta)\\
\nonumber & = & \frac{d}{dt}\left\{o_{\eta \to 0}(1)\mathcal G\right\}+O\left(\frac{1}{t^{N+1}}+\frac{1}{t}\left[\mathcal G(t)+\frac{\|\et_1\|^2_{L^2}}{t}\right]\right).
\eea

We first compute:
\bee
\nonumber &&(\pa_t\zeta D\e,\e)+(-i|D|\e,\e D\zeta)=(\pa_t\zeta(D\e^++D\e^-),\e^++\e^-)+(D\e^+-D\e^-,(\e^++\e^-)\pa_x\zeta)\\
\nonumber& =& ((\pa_t\zeta+\pa_x\zeta)D\e^+,\e^+)+((\pa_t\zeta-\pa_x\zeta)D\e^-,\e^-)+((\pa_t\zeta-\pa_x\zeta)D\e^-,\e^+)+((\pa_t\zeta+\pa_x\zeta)D\e^+,\e^-)\\
\nonumber& = & ((\pa_t\zeta+\pa_x\zeta)D\e^+,\e^+)+((\pa_t\zeta+\pa_x\zeta)D\e^-,\e^-)+((\pa_t\zeta+\pa_x\zeta)D\e^-,\e^+)+((\pa_t\zeta+\pa_x\zeta)D\e^+,\e^-)\\
&  - & 2(\pa_x\zeta D\e^-,\e^-)-2(\pa_x\zeta D\e^-,\e^+)
\eee
and now estimate the various contributions.\\

\noindent\underline{\em Term $|(\pa_x \zeta D\e^-,\e^-)|$}. We claim:
\be
\label{enjivbibeibeb}
|(\pa_x \zeta D\e^-,\e^-)|\lesssim \frac{1}{t}\left[\mathcal G(t)+\frac{\|\et_1\|^2_{L^2}}{t}\right].
\ee
Indeed, recall \eqref{defzetabeta} and renormalize to the $y_1$ variable to compute:
$$|(\pa_x \zeta D\e^-,\e^-)|\lesssim |(D\e_1^-,\pa_{y_1}\phi_1\e_1^-)|.$$
We then commute: 
\bee
&&|(D\e^-_1,\pa_{y_1}\phi_1\e^-_1)|=\frac{1}{R(1-b)}|(\frac{1}{\la z_1\ra^{\frac{1+\alpha}{2}}}\chi_RD\e^-_1,\e^-_1)|\\
& =& \frac{1}{R(1-b)} \left|(\frac{1}{\la z_1\ra^{\frac{1+\alpha}{2}}}[-|D|^{\frac 12}\chi_R|D|^{\frac 12}\e_1^-+[|D|^{\frac 12},\chi_R]|D|^{\frac 12}\e_1^-],\e_1^-)\right|\\
& \lesssim &\frac 1R\left\{ |(\chi_R|D|^{\frac 12}\e_1^-,|D|^{\frac 12}\widetilde{\e_1})|+|(|D|^{\frac 12}\e_1^-,[|D|^{\frac 12},\chi_R]\widetilde{\e_1})|\right\}\\
& \lesssim &\frac 1R\||D|^{\frac 12}\e_1^-\|_{L^2}\left[ \|\chi_R\la z_1\ra^{\frac{1+\alpha}{2}}\|_{L^{\infty}}\|\frac{1}{\la z_1\ra^{\frac{1+\alpha}{2}}}|D|^{\frac 12}\tilde{\e_1}\|_{L^2}+\|[|D|^{\frac 12},\chi_R]\widetilde{\e_1}\|_{L^2}\right].
\eee
We estimate from \eqref{commestimateone}:
$$\|[|D|^{\frac 12},\chi_R]\widetilde{\e_1}\|_{L^2}\lesssim \frac{1}{\sqrt R}\|\widetilde{\e_1}\|_{L^2}$$ and from \eqref{commestimateoneweighted} applied to $\frac{1}{\la z_1\ra^{\frac{1+\alpha}{2}}}$:
\bea
\label{estdonehalftilde}
\nonumber \|\frac{1}{\la z_1\ra^{\frac{1+\alpha}{2}}}|D|^{\frac 12}\tilde{\e_1}\|_{L^2}&\lesssim & \||D|^{\frac12}\e_1^-\|_{L^2}+\| \frac{1}{\la z_1\ra^{\frac{1+\alpha}{2}}}[|D|^{\frac 12},\frac 1{\la z_1\ra^{\frac{1+\alpha}{2}}}]\e_1^-\|_{L^2}\\
& \lesssim &\||D|^{\frac12}\e_1^-\|_{L^2}+\frac{1}{\sqrt{R}}\|\widetilde{\e}_1\|_{L^2}
\eea
and hence the bound:
$$|(D\e^-_1,\pa_{y_1}\phi_1\e^-_1)|\lesssim \frac{1}{R}\left[\||D|^{\frac12}\e_1^-\|_{L^2}^2+\frac{\|\et_1\|_{L^2}^2}{R}\right]\lesssim \frac{1}{t}\left[\mathcal G+\frac{\|\et_1\|_{L^2}^2}{t}\right],$$ this is \eqref{enjivbibeibeb}.\\

\noindent\underline{{\it Term $(\pa_x\zeta D\e^-,\e^+)$}}. This term cannot be treated directly due to the $\eta$ loss in $\|\e_1^{\pm}\|_{L^2}\lesssim \frac{\mathcal G}{\eta}$. We claim that
\be
\label{nvjenoenovnoe}
(\pa_x\zeta D\e^-,\e^+)=\frac{d}{dt}\left\{o_{\eta \to 0}(\mathcal G)\right\}+O\left(\frac{1}{t^{N+1}}+\frac{\mathcal G(t)}{t}\right).
\ee
Indeed, first we renormalize to the $y_1$ variable,
$$(\pa_x\zeta D\e^-,\e^+)=\frac{1}{\l_1^2}(D\e_1^-,\pa_{y_1}\phi_1\e_1^+)$$ and now we need to use the equation. We rewrite \eqref{eqebis} as 
\bee
&&i\pa_t\e-|D|\e=\tilde{F}\\
&&\tilde{F}(t,x)=-\Psi-S-((\Phi+\e)|\Phi+\e|^2-\Phi|\Phi|^2), \ \ \tilde{F}(t,x)=\tilde{F}_1(s_1,y_1)
\eee 
and renormalize to the $y_1=\frac{x-x_1}{\l_1(1-\beta_1)}$ variable so that
\bee
&&i\pa_{s_1}\e_1-\frac{|D|-\beta_1D}{1-\beta_1}\e_1\\
&=&\l_1^{1+\frac 12}\tilde{F}_1+ i\frac{(\l_1)_{s_1}}{\l_1}(\frac{\e_1}{2}+y_1\pa_{y_1}\e_1)-i\frac{(\beta_1)_{s_1}}{1-\beta_1}y_1\pa_{y_1}\e_1+i\frac{(x_1)_t-\beta_1}{1-\beta_1}\pa_{y_1}\e_1+\gamma_{s_1}\e_1
\eee
and thus after projecting with $\Pi^-$ and using $[\Pi^{\pm},\pa_y]=[\Pi^{\pm},y\pa_y]=0$: 
\bea
\label{estoneminus}
&&i\pa_{s_1}\e_1^-+\frac{1+\beta_1}{1-\beta_1}D\e_1^-\\
\nonumber &=&\l_1^{1+\frac 12}\Pi^-\tilde{F}_1+ i\frac{(\l_1)_{s_1}}{\l_1}(\frac{\e^-_1}{2}+y_1\pa_{y_1}\e^-_1)-i\frac{(\beta_1)_{s_1}}{1-\beta_1}y_1\pa_{y_1}\e^-_1+i\frac{(x_1)_t-\beta_1}{1-\beta_1}\pa_{y_1}\e^-_1+(\gamma_1)_{s_1}\e^-_1,
\eea
and
\bea
\label{estoneplus}
&&i\pa_{s_1}\e_1^+-D\e_1^+\\
\nonumber &=&\l_1^{1+\frac 12}\Pi^+\tilde{F}_1+ i\frac{(\l_1)_{s_1}}{\l_1}(\frac{\e^+_1}{2}+y_1\pa_{y_1}\e^+_1)-i\frac{(\beta_1)_{s_1}}{1-\beta_1}y_1\pa_{y_1}\e^+_1+i\frac{(x_1)_t-\beta_1}{1-\beta_1}\pa_{y_1}\e^+_1+(\gamma_1)_{s_1}\e^+_1.
\eea
Using \eqref{estoneminus}, we have
\bee
&&\frac{1}{\l^2_1}(D\e_1^-,\pa_{y_1}\phi_1\e_1^+)=  \frac{1-\beta_1}{\l_1(1+\beta_1)}(-i\pa_t\e_1^-,\pa_{y_1}\phi_1\e_1^+)\\
& + & \frac{1-\beta_1}{\l^2_1(1+\beta_1)}\left(\l_1^{1+\frac 12}\Pi^-\tilde{F}_1+i\frac{(\l_1)_{s_1}}{\l_1}(\frac{\e^-_1}{2}+y_1\pa_{y_1}\e^-_1)-i\frac{(\beta_1)_{s_1}}{1-\beta_1}y_1\pa_{y_1}\e^-_1\right.\\
&+& \left.i\frac{(x_1)_t-\beta_1}{1-\beta_1}\pa_{y_1}\e^-_1+(\gamma_1)_{s_1}\e^-_1,\pa_{y_1}\phi_1\e_1^+\right).
\eee
We use  $\mbox{Supp}(\pa_{y_1}\phi_1)\subset \{\frac {t}{4}\leq y_1\leq t\}$, $\|\pa_{y_1}\phi_1\|_{L^{\infty}}\lesssim \frac 1t$ and the rough bound 
\be
\label{veryroughbnoud}
\left|\frac{(\l_1)_{s_1}}{\l_1}\right|+\left|\frac{(\beta_1)_{s_1}}{1-\beta_1}\right|+\left|\frac{(x_1)_t-\beta_1}{1-\beta_1}\right|\lesssim \frac 1t, \ \ |(\gamma_1)_{s_1}|\lesssim 1
\ee 
to estimate
\bee
(1-\beta_1)\left|\left(i\frac{(\l_1)_{s_1}}{\l_1}\frac{\e^-_1}{2}+(\gamma_1)_{s_1}\e^-_1,\pa_{y_1}\phi_1\e_1^+\right)\right| \lesssim  \frac{1-\beta_1}t\|\e_1\|_{L^2}^2\lesssim \frac{\matchal G}{t},
\eee
and
\bee
&&(1-\beta_1)\left|\left(i\frac{(\l_1)_{s_1}}{\l_1}y_1\pa_{y_1}\e^-_1-i\frac{(\beta_1)_{s_1}}{1-\beta_1}y_1\pa_{y_1}\e^-_1+i\frac{(x_1)_t-\beta_1}{1-\beta_1}\pa_{y_1}\e^-_1,\pa_{y_1}\phi_1\e_1^+\right)\right|\\
& \lesssim & \frac{1-\beta_1}t\|\e_1\|_{L^2}^2\lesssim \frac{\matchal G}{t}.
\eee
Indeed, in order to absorb the derivative in the second estimate, we make use of the commutator estimate \eqref{estcommlonebis}. For instance,
$$ \vert (y_1\pa_{y_1}\e^-_1,\pa_{y_1}\phi_1\e_1^+)\vert =\vert (\pa_{y_1}[\Pi ^+,y_1\pa_{y_1}\phi_1]\e^-_1,\e^+_1)\vert +O(\|\e_1\|_{L^2}^2)
\lesssim \|\e_1\|_{L^2}^2\ ,$$
ad the two other terms are treated similarly. 

The rough $L^{\infty}$-bound $\|\e_1\|_{L^{\infty}}\le 1$, \eqref{errorpsin} and \eqref{seestimatesltwo} ensure
\be
\label{estfyilde}
\|\tilde{F}_1\|_{L^2}\lesssim \|\e_1\|_{L^2}+\frac{1}{\eta^Ct^{N+1}} 
\ee
and hence:
$$(1-\beta_1)|(\Pi^{-}\tilde{F}_1,\pa_{y_1}\phi_1\e_1^+)|\lesssim \frac{1-\beta_1}{t}\|\e_1\|_{L^2}^2+\frac{1}{t^{N+1}} \lesssim \frac{\matchal G}{t}+\frac{1}{t^{N+1}}.$$ We now integrate by parts in time:
\bee
&&\frac{1}{\l^2_1}(D\e_1^-,\pa_{y_1}\phi_1\e_1^+)=-\frac{1-\beta_1}{\l_1(1+\beta_1)}(i\pa_t\e_1^-,\pa_{y_1}\phi_1\e_1^+)+O\left(\frac{\matchal G}{t}+\frac{1}{t^{N+1}}\right)\\
& = & -\frac{d}{dt}\left\{\frac{1-\beta_1}{\l_1(1+\beta_1)}(i\e_1^-,\pa_{y_1}\phi_1\e_1^+)\right\}-\frac{1-\beta_1}{\l_1(1+\beta_1)} \left(\e_1^-,\pa_{y_1}\phi_1i\pa_t\e_1^+\right)+O\left(\frac{\matchal G}{t}+\frac{1}{t^{N+1}}\right)
\eee
where we used \eqref{veryroughbnoud} and the rough bound $$|\pa_t\pa_{y_1}\phi_1|\lesssim \frac 1t$$ in the last step. We now inject \eqref{estoneplus} and conclude using a similar chain of estimates as above:
\bee
\nonumber \frac{1}{\l^2_1}(D\e_1^-,\pa_{y_1}\phi_1\e_1^+)& = & -\frac{d}{dt}\left\{\frac{1-\beta_1}{\l_1(1+\beta_1)}(i\e_1^-,\pa_{y_1}\phi_1\e_1^+)\right\}+O\left(\frac{\matchal G}{t}+\frac{1}{t^{N+1}}\right)\\
& - &  \frac{1-\beta_1}{\l_1(1+\beta_1)}(\e_1^-,\pa_{y_1}\phi_1D\e_1^+).
\eee
The last term is handled using again the commutator estimate \eqref{estcommlonebis}:
$$|(\e_1^-,\pa_{y_1}\phi_1D\e_1^+)|=|(\e_1^-,[\pa_{y_1}\phi_1,\Pi^+]D\e_1^+)|\lesssim \|\e_1^+\|_{L^2}\|D[\pa_{y_1}\phi_1,\Pi^+]\e_1^-\|_{L^2}\lesssim \frac{\|\e_1\|_{L^2}^2}{t^2}$$
and the boundary term in time is estimated using $\|\pa_{y_1}\phi_1\|_{L^{\infty}}\lesssim \frac 1t$:
$$\left|\frac{1-\beta_1}{\l_1(1+\beta_1)}(i\e_1^-,\pa_{y_1}\phi_1\e_1^+)\right|\lesssim \frac{1-\beta_1}{t}\|\e_1\|_{L^2}^2\lesssim \frac{\mathcal G}{t}.$$
The collection of above bounds yields \eqref{nvjenoenovnoe}.\\

\noindent\underline{{\it Term $(-i|D|\e,\zeta D\e)$}}. We claim similarly
\be
\label{kjvneoneoneonv}
(-i|D|\e,\zeta D\e)=\frac{d}{dt}\left\{o_{\eta \to 0}(\mathcal G)\right\}+O\left(\frac{1}{t^{N+1}}+\frac{\mathcal G(t)}{t}\right).
\ee
Indeed, we compute:
\bee
(-i|D|\e,\zeta D\e)&=&\frac{1}{\l_1^2(1-\beta_1)}(-i(D\e_1^+-D\e_1^-),\zeta_1(D\e_1^++D\e_1^-))\\
&=& \frac{1}{\l_1^2(1-\beta_1)}\left[(-iD\e_1^+,\zeta_1 D\e_1^-)+(iD\e_1^-,\zeta_1 D\e_1^+)\right]\\
& = & \frac{2}{\l_1^2(1-\beta_1)}(iD\e_1^-,\zeta_1 D\e_1^+)=-\frac{2}{\l_1^2}(iD\e_1^-,\phi_1 D\e_1^+).
\eee
We compute from \eqref{estoneminus}:
\bee
&&\frac{1}{\l_1^2}(iD\e_1^-,\phi_1 D\e_1^+)=\frac{1-\beta_1}{\l_1(1+\beta_1)}(i\pa_t\e_1^-,i\phi_1 D\e_1^+)\\
& - & \frac{1-\beta_1}{\l^2_1(1+\beta_1)}\left(\l_1^{1+\frac 12}\Pi^-\tilde{F}_1+ i\frac{(\l_1)_{s_1}}{\l_1}(\frac{\e^-_1}{2}+y_1\pa_{y_1}\e^-_1)-i\frac{(\beta_1)_{s_1}}{1-\beta_1}y_1\pa_{y_1}\e^-_1\right.\\
&+& \left.i\frac{(x_1)_t-\beta_1}{1-\beta_1}\pa_{y_1}\e^-_1+\gamma_{s_1}\e^-_1,i\phi_1 D\e_1^+\right).
\eee
We estimate from \eqref{estfyilde}, \eqref{estcommlonebis}:
\bee
&&(1-\beta_1)|(\Pi^-\tilde{F}_1,i\phi_1 D\e_1^+)|=(1-\beta_1)|(\Pi^-\tilde{F}_1,[\Pi^+,\phi_1] D\e_1^+)|\\
&\lesssim& (1-\beta_1)\|D[\Pi^+,\phi_1]\Pi^-\tilde{F}_1\|_{L^2}\|\e_1^+\|_{L^2}\lesssim  \frac{1-\beta_1}{t}\|\tilde{F}_1\|_{L^2}\|\e_1\|_{L^2}\lesssim \frac{1-\beta_1}t\|\e_1\|^2_{L^2}+\frac{1}{t^{N+1}}\\
&\lesssim &\frac{\mathcal G}{t}+\frac{1}{t^{N+1}}.
\eee
We integrate by parts,
\bee
|(iy_1\pa_{y_1}\e_1^-,i\phi_1 D\e_1^+)|&=&|(i\e_1^-,\pa_{y_1}(y_1\phi_1\pa_{y_1}\e_1^+)|\\
& \lesssim & |(i\e_1^-,(\phi_1+y_1\pa_{y_1}\phi_1)\pa_{y_1}\e_1^+)|+(i\e_1^-,y_1\phi_1\pa^2_{y_1}\e_1^+)|\ .
\eee
For the first term, we estimate from \eqref{estcommlonebis}:
$$|(i\e_1^-,(\phi_1+y_1\pa_{y_1}\phi_1)\pa_{y_1}\e_1^+)|\lesssim \|\e_1^+\|_{L^2}\|\pa_{y_1}[\Pi^+,\phi_1+y_1\pa_{y_1}\phi_1]\e_1^-\|_{L^2}\lesssim \frac{1}{t}\|\e_1\|^2_{L^2}.
$$
For the second term, we use $[\Pi^+,y_1]\pa_{y_1}\e_1=0$ 
and \eqref{estcommlonebisbis} to estimate
\bee
|(i\e_1^-,y_1\phi_1\pa^2_{y_1}\e_1^+)|&=&|(i\e_1^-,\phi_1\Pi^+(y_1\pa^2_{y_1}\e_1^+))|=(i\e_1^-,[\phi_1,\Pi^+](y_1\pa^2_{y_1}\e_1^+))|\\
& \lesssim &\|\e_1^+\|_{L^2}\|\la y_1\ra\pa_{y_1}^2[\Pi^+,\phi_1]\e_1^-\|_{L^2}\lesssim \frac{\|\e_1\|_{L^2}^2}{t}.
\eee
Similarly, $$|(i\pa_{y_1}\e_1^-,\phi_1 D\e_1^+)|+|(\e_1^-,\phi_1 D\e_1^+)|\lesssim \frac{1}{t}\|\e_1\|^2_{L^2}.$$
We therefore integrate by parts and using \eqref{veryroughbnoud}
\bee
&&\frac{1}{\l_1^2}(iD\e_1^-,\phi_1 D\e_1^+)=\frac{1-\beta_1}{\l_1(1+\beta_1)}(i\pa_t\e_1^-,i\phi_1 D\e_1^+)+O\left(\frac{\matchal G}{t}+\frac{1}{t^{N+1}}\right)\\
& = & \frac{d}{dt}\left\{\frac{1-\beta_1}{\l_1(1+\beta_1)}(\e_1^-,\phi_1 D\e_1^+)\right\}-\frac{1-\beta_1}{\l^2_1(1+\beta_1)}(i\phi_1\e_1^-,Di\pa_{s_1}\e_1^+)+O\left(\frac{\matchal G}{t}+\frac{1}{t^{N+1}}\right).
\eee
We now reinject \eqref{estoneplus} and estimate all terms similarly as above using \eqref{estcommlonebis}, \eqref{estcommlonebisbis}, and \eqref{kjvneoneoneonv} follows through a completely similar chain of estimates.
\vskip0.25cm
\noindent\underline{{\em $(\pa_t+\pa_x)\zeta$ terms}}. These terms gain an extra $1-\beta_1$ which is essential to treat the degeneracy of the kinetic energy and the $L^2$ mass for $\e_1^+$ in the lower bound \eqref{coerc}, and we claim:
\bea
\label{neoneovnovnene}
&&\nonumber |((\pa_t\zeta+\pa_x\zeta)D\e^+,\e^+)|+|((\pa_t\zeta+\pa_x\zeta)D\e^-,\e^-)|+|((\pa_t\zeta+\pa_x\zeta)D\e^-,\e^+)|\\
&+& |((\pa_t\zeta+\pa_x\zeta)D\e^+,\e^-)|\lesssim \frac{\mathcal G}{t}.
\eea
 Indeed, let
  \be
  \label{defpsi1}
  \psi(t,x)=\frac{\pa_t\zeta+\pa_x\zeta}{\sqrt{\phi}}, \ \ \psi(x)=\psi_1(y_1).
  \ee
  We estimate, after renormalization to the $y_1 $ variable, using \fref{defzetabeta}, \fref{potzniephione}, \fref{coomestimate}, \fref{poitpsi}, \fref{estconveonovd},
  \bee
 && |(D\e^+,(\pa_t\zeta +\pa_x\zeta)\e^{\pm})|\lesssim \left|(\sqrt{\phi_1}D\e_1^+,\psi_1\e_1^{\pm})\right|\\
  & \lesssim & \left|(D(\sqrt{\phi_1}\e_1^+),\psi_1\e_1^{\pm})\right|+\int \frac{|\pa_{y_1}\phi_1|}{\sqrt{\phi_1}}|\psi_1||\e_1|^2dy_1\\
  & \lesssim & \|\Dhalf(\sqrt{\phi_1}\e_1^+)\|_{L^2}\left[\|[\Dhalf,\psi_1]\e_1^{\pm}\|_{L^2}+\|\psi_1\Dhalf\e_1^{\pm}\|_{L^2}\right]+\frac{1-\beta_1}{R^2}\|\e_1\|_{L^2}^2\\
  & \lesssim & \left[\|\sqrt{\phi_1}\Dhalf \e_1^+\|_{L^2}+\frac{1}{\sqrt{R}}\|\e_1\|_{L^2}\right]\times \left[\frac{1-\beta_1}{t^{\frac 32}}\|\e_1\|_{L^2}+\frac{1-\beta_1}{t}\|\sqrt{\phi_1}\Dhalf \e_1^{\pm}\|_{L^2}\right]\\
  & + & \frac{1-\beta_1}{R^2}\int |\e_1|^2dy_1\lesssim  \frac{1}{t}\matchal G(t)\ .
  \eee
Finally, we infer
\bee
&&|((\pa_t\zeta+\pa_x\zeta)D\e^-,\e^{\pm})|=|(\psi\sqrt{\phi}D\e^-,\e^{\pm})|=|(\sqrt{\phi_1}D\e^-_1,\psi_1\e_1^{\pm})|\\
& = & \left|([\sqrt{\phi_1},|D|^{\frac 12}]\e_1^-+|D|^{\frac 12}\sqrt{\phi_1}|D|^{\frac 12}\e_1^-,\psi_1\e_1^{\pm})\right|\\
& \lesssim & \|[\sqrt{\phi_1},|D|^{\frac 12}]\e_1^-\|_{L^2}\|\psi_1\e_1^{\pm}\|_{L^2}+\|\sqrt{\phi_1}|D|^{\frac 12}\e_1^-\|_{L^2}\left(\|\psi_1|D|^{\frac 12}\e^{\pm}_1\|_{L^2}+\|[|D|^{\frac 12},\psi_1]\e_1\|_{L^2}\right)\ ,
\eee
and hence, using \eqref{poitpsi}, \eqref{estconveonovd}, \eqref{coomestimate},
\bee
&&|((\pa_t\zeta+\pa_x\zeta)D\e^-,\e^{\pm})|\lesssim \frac{1}{\sqrt{t}}\frac{1-\beta_1}{t}\|\e_1\|^2_{L^2}\\
& + & \||D|^{\frac 12}\e_1^-\|_{L^2}\left(\frac{1-\beta_1}{t}\|\sqrt{\phi_1}|D|^{\frac 12}\e_1^{\pm}\|_{L^2}+\frac{1-\beta_1}{t^{\frac 32}}\|\e_1\|_{L^2}\right)\lesssim  \frac{\matchal G}{t},
\eee
and \eqref{neoneovnovnene} is proved.\\

\noindent{\bf Step 6:} Control of mass terms. We claim: 
\bea
\label{estlinearmass}
&&\frac 12((\pa_t \theta) \e,\e)+(-i|D|\e,\theta \e)=\frac{d}{dt}\left\{o_{\eta\to 0}(\mathcal G)\right\}\\
\nonumber &+&  O\left(\frac{1}{t^{N+1}}+\frac{1}{t}\left[\mathcal G(t)+\frac{\|\et_1\|^2_{L^2}}{t}\right]\right).
\eea
Indeed, we split $\e=\e^++\e^-$ and compute:
\bee
&&\frac 12((\pa_t \theta) \e,\e)+(-i|D|\e,\theta \e)=\frac 12((\pa_t\theta)(\e^++\e^-),\e^++\e^-)+(-i(D\e^+-D\e^-),\theta(\e^++\e^-))\\
& = & \frac 12((\pa_t\theta)(\e^++\e^-),\e^++\e^-)-(\pa_x\e^+-\pa_x\e^-,\theta(\e^++\e^-))\\
& = & \frac 12((\pa_t\theta+\pa_x\theta) \e^+,\e^+)+\frac 12((\pa_t\theta-\pa_x\theta) \e^-,\e^-)+(\pa_t\theta \e^+,\e^-)+(\pa_x\e^-,\theta \e^+)-(\pa_x\e^+,\theta \e^-)\\
& = & \frac 12((\pa_t\theta+\pa_x\theta) \e^+,\e^+)+\frac 12((\pa_t\theta-\pa_x\theta) \e^-,\e^-)+((\pa_t\theta+\pa_x\theta) \e^+,\e^-)\\
& - & (\pa_x\theta \e^+,\e^-)+(\pa_x\e^-,\theta \e^+)+(\e^+,\pa_x\theta \e^-+\theta\pa_x\e^-)\\
& = &  \frac 12((\pa_t\theta+\pa_x\theta) \e^+,\e^+)+\frac 12((\pa_t\theta-\pa_x\theta) \e^-,\e^-)+((\pa_t\theta+\pa_x\theta) \e^+,\e^-)\\
& + & 2(\theta\e^+,\pa_x\e^-)\ ,
\eee 
and estimate all terms.\\

\smallskip
\noindent\underline{\em $(\pa_t+\pa_x)\theta$ terms}. We estimate from \eqref{estpatpxtheta},
$$|((\pa_x\theta+\pa_t\theta) \e^{\pm},\e^{\pm})|\lesssim \frac{\|\e\|_{L^2}^2}{t}\lesssim \frac{\mathcal G}{t}.$$

\noindent\underline{\em Term $((\pa_t-\pa_x)\theta \e^-,\e^-)$}. For $t\geq T^-$, we use $$|\l_2-\l_1|\lesssim \eta$$ and \eqref{estlinfitythetea}:
$$|((\pa_x\theta) \e^{-},\e^{-})| \lesssim\frac{|\l_2-\l_1|}{(1-\beta_1)R}\|\e\|_{L^2}^2\lesssim  \frac{\mathcal G(t)}{t}.$$ For $t\leq T^-$, we use the bound $$|\l_2-\l_1|\lesssim \frac 1t$$ and the space localization of $\pa_{y_1}\theta_1$ to estimate from \eqref{estnvienvoene}:
$$|((\pa_x\theta) \e^{-},\e^{-})|\lesssim |((\la z_1\ra^{1+\alpha}\pa_{y_1}\theta_1)\et_1,\et_1)|\lesssim \frac{|\l_2-\l_1|}{t}\|\et_1\|_{L^2}^2\lesssim \frac{\|\et_1\|_{L^2}^2}{t^2}.$$

\noindent\underline{\em Term $(\theta \e^+,\pa_x\e^-)$}. For the last term, we renormalize to the $y_1$ variable 
$$
(\pa_x\e^-,\theta \e^+)=\frac{1}{\l_1}(D\e^-_1,-i\theta_1\e^+_1)
$$
and hence, using \eqref{estoneminus},
\bee
(\pa_x\e^-,\theta \e^+)&=& \frac{1-\beta_1}{1+\beta_1}(i\pa_t\e_1^-,i\theta_1\e^+_1)\\
& + & \frac{1-\beta_1}{\l_1(1+\beta_1)}\left(\l_1^{1+\frac 12}\Pi^-\tilde{F}_1+ i\frac{(\l_1)_{s_1}}{\l_1}(\frac{\e^-_1}{2}+y_1\pa_{y_1}\e^-_1)\right.\\
&-& \left .i\frac{(\beta_1)_{s_1}}{1-\beta_1}y_1\pa_{y_1}\e^-_1+i\frac{(x_1)_t-\beta_1}{1-\beta_1}\pa_{y_1}\e^-_1+(\gamma_1)_{s_1}\e^-_1,\theta_1\e_1^+\right)\\
& = &  \frac{1-\beta_1}{1+\beta_1}(i\pa_t\e_1^-,i\theta_1\e^+_1)+O\left(\frac{1-\beta_1}{t}\|\e_1\|_{L^2}^2\right)
\eee
and hence, integrating by parts in time and using \eqref{estoneplus}, \eqref{veryroughbnoud}, \eqref{estfyilde},
\bee
(\pa_x\e^-,\theta \e^+)&=&\frac{d}{dt}\left\{\frac{1-\beta_1}{1+\beta_1}(\e_1^-,\theta_1\e_1^+)\right\}\\
& - &\frac{1-\beta_1}{\l_1(1+\beta_1)} (\e_1^-,\theta_1D\e_1^+)+O\left(\frac{1-\beta_1}{t}\|\e_1\|_{L^2}^2+\frac{1}{t^{N+1}}.\right)\ .
\eee
We estimate from \eqref{estcommlonebis},
\bee
|(\e_1^-,\theta_1D\e_1^+)|=|(D[\Pi^+,\theta_1]\e_1^-,\e_1^+)|\lesssim \frac{1}{t}\|\e_1\|_{L^2}^2
\eee
and, for the boundary term in time, we use $$\theta_1=\frac{1}{\l_2}\left[\mu\Psi_1+1-\Psi_1\right]\ ,$$ to compute
$$\frac{1-\beta_1}{1+\beta_1}(\e_1^-,\theta_1\e_1^+)=\frac{1-\beta_1}{\l_2(1+\beta_1)}(\e_1^-,(\mu \Psi_1+1-\Psi_1)\e_1^+)=\frac{1-\beta_1}{\l_2(1+\beta_1)}(\mu-1)(\e_1^-,\Psi_1\e_1^+)\ .$$ 
Hence
$$|\frac{1-\beta_1}{1+\beta_1}(\e_1^-,\theta_1\e_1^+)|\lesssim |\l_2-\l_1|(1-\beta_1)\|\e_1\|_{L^2}^2\lesssim \eta^\delta \mathcal G$$
which concludes the proof of \eqref{estlinearmass}.\\

\noindent{\bf Step 7:} Small time improved bound for $\|\et_1\|_{L^2}$.
The collection of above estimates yields the differential control:
\be
\label{bvubivbbrirb}
\left|\frac{d}{dt}\left\{\mathcal G(t)(1+o_{\eta \to 0}(1))\right\}\right|\lesssim\frac{1}{t}\left[\mathcal G(t)+\frac{\|\et_1\|_{L^2}^2}{t}\right]+\frac{1}{t^{N+1}}.
\ee
We now estimate the $\et_1$ term first through the following space time bound: 
\be
\label{essentialbooot}
\int_t^{T_n}\frac{\|\widetilde{\e_1}(\tau)\|^2_{L^2}}{\tau}d\tau\lesssim \int_t^{T_n}\left[\mathcal G(\tau)+\frac{1}{\tau^{N+1}}\right]d\tau.
\ee
which improves on the trivial bound $\|\tilde{\e_1}(t)\|_{L^2}^2\lesssim \|\e_1^-\|_{L^2}^2\lesssim \frac{\mathcal G(t)}{\eta}$ for $t\leq T^-$. Indeed, let $$h(s_1,y_1)=H\left(\frac{y_1}{s_1}\right), \ \ H(z_1)=\int_{z_1}^{+\infty}\frac{dz}{1+\la z\ra^{1+\alpha}}.$$ We estimate from \eqref{estoneminus}: 
\bee
&&\frac12 \frac{d}{ds_1}\int h|\e_1^-|^2=\frac12\int\left(\pa_{s_1}-\frac{1+\beta_1}{1-\beta_1}\pa_{y_1}\right)h|\e_1^-|^2\\
& + & \left(ih\e^-_1,\l_1^{1+\frac 12}\Pi^-\tilde{F}+ i\frac{(\l_1)_{s_1}}{\l_1}(\frac{\e^-_1}{2}+y_1\pa_{y_1}\e^-_1)-i\frac{(\beta_1)_{s_1}}{1-\beta_1}y_1\pa_{y_1}\e^-_1+i\frac{(x_1)_t-\beta_1}{1-\beta_1}\pa_{y_1}\e^-_1+(\gamma_1)_{s_1}\e^-_1\right)\\
& = & \frac12\int\left(\pa_{s_1}-\frac{1+\beta_1}{1-\beta_1}\pa_{y_1}\right)h|\e_1^-|^2+O\left(\|\e_1\|_{L^2}^2+\frac{1}{t^{N+1}}\right)
\eee
where we integrated by parts and use \eqref{veryroughbnoud}, \eqref{estfyilde} in the last step. Moreover,

\bee
\left(\pa_{s_1}-\frac{1+\beta_1}{1-\beta_1}\pa_{y_1}\right)h=\frac{1}{s_1}\left(-z_1-\frac{1+\beta_1}{1-\beta_1}\right)\pa_{z_1}H=\frac{1}{s_1}\left(z_1+\frac{1+\beta_1}{1-\beta_1}\right)\frac{1}{1+\la z_1\ra^{1+\alpha}}
\eee
and hence the bound using $\frac{|z_1|}{\la z_1\ra^{1+\alpha}}\leq 1$:
\bee
\frac{1}{s_1}\int \left(\frac{1+\beta_1}{1-\beta_1}\right)\frac{|\e_1^-|^2}{1+\la z_1\ra^{1+\alpha}}\leq  C\frac{\mathcal G}{1-\beta_1}+\frac1{t^{N+1}}+\frac 12\frac{d}{ds_1}\int h|\e_1^-|^2.
\eee
We integrate this on $[s_1(t),s_1(T_n)]$ with $\e_1(s_1(T_n))=0$ and \eqref{essentialbooot} follows from $s_1\sim t\sim R$, $\frac{\eta}{2}\leq 1-\beta_1\leq 2\eta$.
\vskip 0.25cm
\noindent{\bf Step 8:} Conclusion. We integrate \eqref{bvubivbbrirb} in time on $[t,T_n]$ using $\e(T_n)=0$ so that
$$\matchal G(t)\lesssim \int_t^{T_n}\frac{\matchal G(\tau)}{\tau}d\tau+\int_t^{T_n}\frac{\|\et_1\|_{L^2}^2}{\tau^2}d\tau+\frac{1}{t^{N}}.$$ The first term is estimated using the bootstrap bound \eqref{bootdynamicalproved}:
$$\int_t^{T_n}\frac{\matchal G(\tau)}{\tau}d\tau\lesssim \int_t^{T_n}\frac{1}{\tau^{1+\frac N2}}d\tau\lesssim \frac{1}{N}\frac{1}{t^{\frac N2}}.$$ For the second term, we estimate from \eqref{essentialbooot}:
\bee
\int_t^{T_n}\frac{\|\et_1\|_{L^2}^2}{\tau^2}d\tau&\lesssim &\frac{1}{t}\int_t^{T_n}\frac{\|\et_1\|_{L^2}^2}{\tau}d\tau\lesssim \frac 1t\int_t^{T_n}\left[\mathcal G(\tau)+\frac{1}{\tau^{N+1}}\right]d\tau\\
& \lesssim & \frac{1}{t^{N+1}}+ \frac 1t\int_t^{T_n}\frac{d\tau}{\tau^{\frac N2}}\lesssim  \frac{1}{N}\frac{1}{t^{\frac N2}}
\eee
which concludes the proof of \eqref{energyestimate}.
\end{proof}


\subsection{Proof of the bootstrap Proposition \ref{propboot}}


We are now in position to conclude the control of the geometrical parameters and the $H^1$ bound.

\begin{proof}[Proof of Proposition \ref{propboot}] First observe that \eqref{energyestimate} yields the improved $H^{\frac 12}$ bound in \eqref{bootdynamicalprovedfinal}. Moreover, the bounds \eqref{boootbound} at $T^-$ and  \eqref{estmod}, \eqref{coerc} \eqref{energyestimate} allow us to apply the perturbative Lemma \ref{propturbulent} and conclude that $\matchal P$ satisfies \eqref{esttrubulentregime}.
 We therefore need to prove \eqref{boootboundfinal} and the improved $H^1$ bound in \eqref{bootdynamicalprovedfinal}.\\

\noindent{\bf Step 1:} Proof of  \eqref{boootboundfinal}. Recall \eqref{defbetatilde} so that $$(1-\beta_j)\pa_{\beta_j}=\pa_{\tilde{\beta}_j}.$$ Since $RM_j,RB_j$ are $L^{\infty}$-admissible, we have 
\bea
\label{vneneenvo}
&&\sum_{j,k=1}^2\left|\pa_{\l_j}M_k\right|+|\pa_{\tilde{\beta_k}}M_k|+\sum_{j=1}^2 |R\pa_RM_j|+|\pa_{\Gamma}M_j|\lesssim \frac 1t\\
\label{vneneenvobis}
&&\sum_{j,k=1}^2\left|\pa_{\l_j}B_k\right|+|\pa_{\tilde{\beta_k}}B_k|+\sum_{j=1}^2 |R\pa_RB_j|+|\pa_{\Gamma}B_j|\lesssim \frac 1t.
\eea
For $t\ge T^-$, the same chain of estimates like for the proof of Proposition \ref{sharpmod} using $$|1-\mu|\lesssim \eta\ \ \mbox{for}\ \ t\ge T^-$$ ensures the more precise control:
\be
\label{betterphase}
\sum_{j=1}^2|\pa_{\Gamma}M_j|+|R\pa_RM_j|+\sum_{j,k=1}^2|\pa_{\tilde{\beta_k}}M_j|\lesssim \frac1{t^2}.
\ee 
Indeed, if $j=1$, we know that $b^{-1}R(1+(1-\b_1)R)M_1$ is $L^\infty $-admissible, so that 
$$|\pa_{\Gamma}M_1|+|R\pa_RM_1|+\sum_{k=1}^2|\pa_{\tilde{\beta_k}}M_1|\lesssim \frac b{R(1+(1-\b_1)R)}\lesssim \frac 1{t^2}$$
since, for $t\ge T^-$, $b\simeq \eta ^2$, $1-\b_1\sim \eta $ and $R\sim t$.
If $j=2$, Corollary \ref{5.77} leads to 
$$\vert \pa _\Gamma M_2\vert +\vert R\pa_RM_2\vert +\sum_{k=1}^2(1-\b_k)\vert \pa_{\b_k}M_2\vert \lesssim \frac{\vert 1-\mu\vert +(1-\b_2)\vert \log(1-\b_2)\vert +R^{-1}}{R(1+(1-\b_1)R)}\ .$$
Since, for $t\ge T^-$, $\vert 1-\mu\vert \lesssim \eta , 1-\b_2\simeq \eta ^3, 1-\b_1\sim \eta ,R\sim t$, we infer \eqref{betterphase}.\\
Recalling \eqref{defdeltas}, \eqref{defdeltasbis}, then \eqref{vneneenvo}, \eqref{vneneenvobis}, \eqref{betterphase} ensure:
\bee
&&|B_j-B_j^{\infty}|\lesssim \frac{1}{t}\left[\sum_{j=1}^2(|\Delta \l_j|+|\Delta \tilde{\beta}_j|)+|\Delta \Gamma|\right]+\frac{1}{t^2}|\Delta R|\lesssim \frac{1}{t^{\frac N8+1}}\\
&&|M_j-M_j^{\infty}|\lesssim \frac{1}{t}\sum_{j=1}^2|\Delta \l_j|+\frac{1}{t^2}\left(\sum_{j=1,2}|\Delta \tilde{\beta}_j|+|\Delta \Gamma|\right)+\frac{|\Delta R|}{t^3}\lesssim \frac{1}{t^{\frac N8+2}}.
\eee
Moreover, from \eqref{dynamicalystem}, \eqref{estmod}, \eqref{energyestimate}:
\bee
|(\l_j)_t-(\l_j^{\infty})_t|=\left|\frac{(\l_j)_{s_j}}{\l_j}-\frac{(\l_j^{\infty})_{s_j^{\infty}}}{\l_j^{\infty}}\right|\lesssim |M_j-M_j^{\infty}|+\Mod_j\lesssim  \frac{1}{t^{\frac N8+2}}
\eee
which time integration using \eqref{estbeogeo} ensures:
$$|\Delta \l_j|\lesssim \frac{1}{Nt^{\frac N8+1}}.$$
We now compute similarly:
$$|(\tilde{\beta}_j)_t-(\tilde{\beta}^{\infty}_j)_t|=\left|\frac{1}{\l_j}\left[B_j+O(\Mod_j)\right]-\frac{1}{\l_j^{\infty}}B_j^{\infty}\right|\lesssim \frac{1}{t^{\frac N4}}+|B_j-B_j^{\infty}|+|B_j^{\infty}||\l_j-\l_j^{\infty}|\lesssim \frac{1}{t^{\frac N8+1}}$$
and hence by integration in time:
$$|\Delta \tilde{\beta}_j|\lesssim \frac{1}{Nt^{\frac N8}}.$$
We now compute the phase shift:
$$
|\Gamma_t-\Gamma_t^{\infty}|=\left|\frac{1}{\l_2}-\frac{1}{\l^{\infty}_2}-\left(\frac{1}{\l_2}-\frac{1}{\l^{\infty}_2}\right)+O(\Mod_1+\Mod_2)\right|\lesssim \frac{1}{Nt^{\frac N8+1}}
$$
and hence $$|\Delta \Gamma|\lesssim \frac{1}{N^2t^{\frac N8}}.$$ We now estimate from \eqref{eqgammeexactebis}:
$$
|R_t-R^{\infty}_t|\lesssim \sum_{j=1,2}|\Delta \tilde{\beta}_j|+|\Delta \l_j|+\Mod_j+|\Delta \Gamma|+\frac{1}{t}|\Delta R|\lesssim \frac{1}{t^{\frac N8}}
$$
which time integration concludes the proof of \eqref{boootboundfinal}.\\
  
 \noindent{\bf Step 2:} Proof of the $H^1$ bound in \eqref{bootdynamicalprovedfinal}. Since we have closed the $H^{\frac 12}$ bound at the linear level, closing the $H^1$ bound or any higher Sobolev norm is now elementary. Recall \eqref{defequagt} $$
i\pa_t\e-|D|\e+2|\Phi|^2\e+\Phi^2\overline{\e}=G.$$ Let $$z=\Dhalf \e,$$ then:
\be
\label{eqzetaee}
\left\{\begin{array}{ll} i\pa_tz-|D|z+2|\Phi|^2z+\Phi^2\overline{z}=\tilde{G}\\
 \tilde{G}=\Dhalf G-2[\Dhalf, |\Phi|^2]\e-[\Dhalf, \Phi^2]\overline{\e}
 \end{array}\right..
 \ee
 We now run an energy identity on \fref{eqzetaee}.
 We consider
 $$
\mathcal G_0(z):=\frac{1}{2}\Big[(|D|z,z)+(z,z)-(2|\Phi|^2z+\Phi^2\overline{z},z)\Big]$$
 then from \eqref{bootdynamicalproved}: 
 \be
 \label{estionononvv}
 \|z\|_{H^{\frac 12}}^2\lesssim \mathcal G_0(z)+\|z\|_{L^2}^2\lesssim \mathcal G_0(z)+\frac{1}{t^{\frac N2}}. 
 \ee
 We compute the associated energy identity:
 \begin{align}
 \label{esitmatevnieo}
\frac{d}{dt}\mathcal G_0 & = (\pa_t z,|D|z+z-2|\Phi|^2z-\Phi^2\overline{z})-(\pa_t(|\Phi|^2)z+\frac 12\pa_t\Phi^2\overline{z},z) \notag\\
 & = (-i\tilde{G},|D|z+z-2|\Phi|^2z-\Phi^2\overline{z})+(z^2,i\Phi^2)-(\pa_t\Phi, 2\Phi |z|^2+\bar{\Phi }z^2)\notag\\
 &= (z^2,i\Phi^2)+(i\Dhalf (\Psi+S),|D|z+z-2|\Phi|^2z-\Phi^2\overline{z})\notag\\
& +\left(i\left[2[\Dhalf, |\Phi|^2]\e-[\Dhalf, \Phi^2]\overline{\e}\right],|D|z+z-2|\Phi|^2z-\Phi^2\overline{z}\right)\notag\\
& +\left(i\Dhalf N(\e),|D|z+z-2|\Phi|^2z-\Phi^2\overline{z}\right)\notag\\
& -(\pa_t\Phi, 2\Phi |z|^2+\bar{\Phi }z^2)\notag\\
&=I+II+III+IV+V.
 \end{align}
 We now estimate all terms in \fref{esitmatevnieo}. 
 From \fref{bootdynamicalproved} and $\|\Phi\|_{L^{\infty}}\lesssim 1$:
 $$|I|=|(z^2,i\Phi^2)|\lesssim \|\e\|_{H^{\frac 12}}^2\leq \frac 1{t^{\frac N4}}.$$ 
 For II, we use \eqref{errorpsin} and an integration by parts and \eqref{bootdynamicalproved} to estimate:
$$
|(i\Dhalf \Psi,|D|z+z-2|\Phi|^2z-\Phi^2\overline{z})|\lesssim \|\Psi\|_{H^{\frac 32}}\|z\|_{L^2}\lesssim \frac{1}{\eta^Ct^{N+1}}\frac{1}{t^{\frac N4}}\leq \frac{1}{t^{\frac N4}}.
$$
 For the modulation equation term, we estimate in brute force using the admissibility of $V_j$, \eqref{estmod} and \fref{bootdynamicalproved}: $$\|S\|_{H^{\frac 32}}\lesssim\frac{1}{\eta^C}(\Mod_1+\Mod_2)\lesssim \frac{1}{\eta^C} \frac{1}{t^{\frac N4}}$$ and hence:
 $$|(i\Dhalf S,|D|z+z-2|\Phi|^2z-\Phi^2\overline{z})|\lesssim \|S\|_{H^{\frac 32}}\|z\|_{L^2}\lesssim \frac{1}{\eta^Ct^{\frac N4}}\frac{1}{t^{\frac N4}}\leq \frac{1}{t^{\frac N4}}.$$ For III, we use that for any function $\chi$: $$\Dhalf[\Dhalf,\chi]=[|D|,\chi]-[\Dhalf,\chi]\Dhalf$$ 
and hence using \fref{bootdynamicalproved}, \eqref{commestimateone} with $R=1$ and the admissibility of $V_j$:
\bee
&&\left|\left(2i[\Dhalf, |\Phi|^2]\e,|D|z\right)\right|\lesssim \|\Dhalf [\Dhalf, |\Phi|^2]\e\|_{L^2}\|z\|_{H^{\frac 12}}\\
&\lesssim &\Big(\|[|D|, |\Phi|^2]\e\|_{L^2}+\|[\Dhalf, |\Phi|^2]\Dhalf \e\|_{L^2}\Big)\|z\|_{H^{\frac 12}} \lesssim  \frac{1}{\eta^C}\|\e\|_{H^{\frac 12}}\|z\|_{H^{\frac 12}}\\
&\lesssim& \frac{1}{\eta^Ct^{\frac N4}}\|z\|_{H^{\frac 12}}\lesssim \frac{1}{t^{\frac N4+1}}.
\eee
The term $(2i[\Dhalf, |\Phi|^2]\e, z-2|\Phi|^2z-\Phi^2\overline{z})$
being easier to handle and proceeding analogously for the terms containing $i[\Dhalf, \Phi^2]\bar{\eps}$, we conclude that
\[|III| \lesssim \frac{1}{t^{\frac N4+1}}\]
For IV, we develop the cubic non linear term. The most dangerous nonlinear term is the following which we estimate in brute force by Sobolev and \fref{bootdynamicalproved}:
\begin{align*}
\left|\left(i\Dhalf (\e|\e|^2),|D|z\right)\right|&=\left|\left(i|D|(|\e|^2\e), \Dhalf z\right)\right|
\lesssim \|D(|\e|^2\e)\|_{L^2}\|z\|_{H^{\frac 12}}\\
&\lesssim \|D\e\|_{L^2}\|\e\|_{L^{\infty}}^2\|z\|_{H^{\frac 12}}\lesssim (\|\e\|_{L^2}^2+\|D\e\|_{L^2}^2)\|z\|_{H^{\frac 12}}^2\\
&\lesssim \frac{1}{t^{\frac N4+1}}.
\end{align*}
Then, by the fractional Leibniz rule and \fref{bootdynamicalproved}, we also have
\begin{align*}
\left|\left(i\Dhalf (\e|\e|^2),z-2|\Phi|^2z-\Phi^2\overline{z}\right)\right|
&\lesssim \|\Dhalf\e\|_{L^4}\|\e^2\|_{L^4}\|z\|_{L^2}
\lesssim \|z\|_{H^{\frac 12}}\|\e\|_{H^{\frac 12}}^2\|z\|_{L^2}\\
&\lesssim  \frac{1}{t^{\frac N4+1}}.
\end{align*}
We argue similarly for the quadratic terms and obtain:
\begin{align*}
\left|\left(i\Dhalf(2|\e|^2\Phi+\e^2\overline{\Phi}), |D|z+ z-2|\Phi|^2z-\Phi^2\overline{z}\right)\right|
&\lesssim \frac{1}{\eta^{\frac{1+4\delta}{2}}}(\|\e\|_{L^2}^2+\|\e\|_{\dot{H}^1}^2)\|z\|_{H^{\frac 12}}\\
&\lesssim \frac{1}{t^{\frac N4+1}}
\end{align*}
Finally, to estimate V, we use from \eqref{expressionpatphi} the rough bound $\|\pa_t\Phi\|_{L^{\infty}}\lesssim \frac{1}{\eta^{C}}$ to estimate:
$$
|V|=|(\pa_t\Phi, 2\Phi |z|^2+\bar{\Phi }z^2)|\lesssim \|\pa_t\Phi\|_{L^{\infty}}\|\Phi\|_{L^{\infty}}\|z\|_{L^2}^2\lesssim \frac{1}{\eta^Ct^{\frac N2}}\lesssim \frac{1}{t^{\frac N4+1}}.$$
The collection of above bounds yields 
$$\left|\frac{d}{dt}\mathcal G_0\right|\lesssim \frac{1}{t^{\frac N4+1}}$$ which time integration using 
$\e(T_n)=z(T_n)=0$ with \eqref{estionononvv} yields $$\|z\|_{H^{\frac 12}}^2\lesssim \frac{1}{Nt^{\frac N4}}.$$ This concludes the proof of \eqref{bootdynamicalprovedfinal} and of Proposition \ref{propboot}.
\end{proof}


\subsection{Proof of Theorem \ref{thmmain}}
\label{proofthm} 

We are now in position to conclude the proof of Theorem \ref{thmmain} as a simple consequence of Proposition \ref{propboot}. The argument is now classical \cite{Martelmulti}, we recall it for the convenience of the reader.

\begin{proof}[Proof of Theorem \ref{thmmain}] First observe that Proposition \ref{propboot} implies that $u_n(t)$ solution to \eqref{defunt} satisfies:
\be
\label{estbeogeo}
\forall n\geq 1, \ \ \forall t\in [T_{in},T_n], \ \ \|u_n(t)-\Phi_{\pt^{\infty}}(t)\|_{H^1}\leq \frac{1}{t^{\frac N{10}}}.
\ee
We now let $n\to +\infty$ and extract a non trivial limit to produce the dynamics described by Theorem \ref{thmmain}.\\

\noindent {\bf Step 1:} $H^{\frac 12}$-compactness. We claim that that the sequence $u_n(T_{in})$ is up to a subsequence $H^{\frac 12}$ compact. Indeed it is $H^1$ bounded from \fref{estbeogeo}. We now claim that it is $H^{\frac 12}$ tight: $\forall \e_0>0,$ $\exists R(\e_0)$ such that:
\be
\label{ltwotightneness}
 \int_{|x|\geq R(\e_0)}|u_n(T_{in})|^2+ \int_{|x|\geq R(\e_0)}|\Dhalf u_n(T_{in})|^2< \e_0.
\ee
Indeed, pick $\e_0>0$, then from \fref{estbeogeo}, we may find a time $T(\e_0)$ such that 
$$\|u_n(T(\e_0))-\Phi_{\pt^{\infty}}(T(\e_0))\|_{H^1}<\e_0$$ 
and then by construction of 
$\Phi_{\widetilde{\mathcal P}^{\infty}}$, we may find $R=R(\e_0)$ 
such that 
$$\forall R\geq R(\e_0), \ \ \int(1-\chi_R)|\Phi_{\pt^{\infty}}(T(\e_0))|^2+\int(1-\chi_R)||\Dhalf \Phi_{\pt^{\infty}}(T(\e_0))|^2<\e_0$$ from which 
$$\int(1-\chi_R)|u_n(T(\e_0))|^2+\int(1-\chi_R)|\Dhalf u_n(T(\e_0))|^2\lesssim \e_0.$$ 
We now propagate this information backwards at $T_{in}$ by localizing the mass and energy conservation laws. Indeed, a brute force computation and \eqref{commd} ensure
$$\left|\frac{d}{dt}\int(1-\chi_R)|u_n|^2\right|\lesssim \frac{\|u_n\|_{L^2}^2}{R}\lesssim \frac{1}{R}$$ and hence $$\int(1-\chi_R)|u_n(T_{in})|^2\lesssim \e_0+\frac{T(\e_0)-T_{in}}{R(\e_0)}\lesssim \e_0$$ by possibly raising the value of $R(\e_0)$. We similarly localize the conservation of energy with $\zeta_R=1-\chi_R$ and estimate using \fref{commestimate}:
\bee
&&\left|\frac{d}{dt}\left\{\frac12\int \zeta_R|\Dhalf u_n|^2+\frac 14\int \zeta_R|u_n|^4\right\}\right|=\left|(\pa_tu_n,[\Dhalf ,\zeta_R]\Dhalf u_n)\right|\\
&\lesssim &\frac{\|u_n\|_{H^1}^2+\|u_n\|^4_{H^1}}{\sqrt{R}}\lesssim \frac{1}{\eta^C\sqrt{R}}
\eee
from which $$\int(1-\chi_R)|\Dhalf u_n(T_{in})|^2\lesssim \e_0+\frac{T(\e_0)-T_{in}}{\eta^CR(\e_0)}\lesssim \e_0$$ by possibly raising the value of $R(\e_0)$, and \fref{ltwotightneness} is proved.\\

\noindent {\bf Step 2:} Conclusion. The $H^1$ global bound and the tightness \fref{ltwotightneness} ensure using the compactness of the Sobolev embedding $H^1\hookrightarrow H^{\frac 12}_{\rm loc}$ the strong convergence up to a subsequence $$u_n(T_{in})\to u(T_{in})\ \ \mbox{in}\ \ H^{\frac 12}\ \ \mbox{as}\ \ n\to +\infty.$$ 
Let $u$ be the solution to \fref{halfwave} with data $u(T_{in})$ , then the continuity of the flow in $H^{\frac 12}$ now yield the convergence of the whole sequence 
$$\forall t\ge T_{in}, \ \ u_n(t)\to u(t)\ \ \mbox{in}\ \ H^{\frac 12}\ \ \mbox{as}\ \ n\to +\infty
$$ and hence from \fref{estbeogeo} and lower semi continuity of the norm: 
$$\forall t\ge T_{in}, \ \ \|u(t)-\Phi_{\pt^{\infty}}(t)\|_{H^1}\leq \frac{1}{t^{\frac N{10}}}.$$ Moreover, since the modulation equation are computed from local in space scalar products, we have\footnote{see for example \cite{MMannals2} for a detailed proof in a similar functional setting.} $$\forall t\ge T_{in}, \ \ \tilde{\matchal P}_{u_n}(t)\to \tilde{\mathcal P}_u(t)\ \ \mbox{as}\ \ n\to+\infty,$$ and hence passing to the limit in the estimates \eqref{boootboundfinal}, \eqref{esttrubulentregime} ensures that $u$ satisfies the expected dynamics of Theorem \ref{thmmain}.
\end{proof}


\begin{appendix}


\section{Algebra for the Szeg\H{o} profile}


\begin{lemma}[Algebraic relations]\label{Alg_rel_Q+}
\label{lemmaalgebraSzeg\H{o}}
There holds:
\be
\label{massmomentum}
\int |Q^+|^2=2\pi, \ \ \int\pa_y Q^+\overline{Q^+}=2i\pi,
\ee
\be
\label{momentun}
\int |Q^+|^2\overline{\pa_yQ^+}=2\pi,\ \ \int (Q^+)^2\overline{\pa_yQ^+}=-4\pi,
\ee
\be
\label{momentdeux}
\int |Q^+|^2\overline{Q^+}=2i\pi,\ \  \int (Q^+)^2\overline{Q^+}=-2i\pi.
\ee
\begin{align}
(y\pa_y Q^+, iQ^+)&=0\label{A1}\\
(y\pa_y Q^+, \pa_y Q^+)&=0.\label{A2}
\end{align}

\end{lemma}

\begin{proof}
Since 
$$Q^+(y)=\frac{1}{y+\frac i2}\ ,$$
these formulas are for instance easy consequences of the residue theorem.
\end{proof}


\section{The resonant two-soliton Szeg\H{o} dynamics}
\label{systemszego}


This appendix revisits the result of Pocovnicu \cite{Po2011} about two-soliton solutions for the cubic Szeg\H{o} equation on the line, by putting emphasis on the ODE system on modulation parameters. 
For ease of notation, in this appendix we set
$$Q(x):=Q^+(x)=\frac{1}{x+\frac i2}\ ,$$
and we look for a solution $u=u(t,x)$ of the cubic Szeg\H{o} equation on the line
$$i\pa _tu-Du+\Pi (u^2\overline u)=0$$
of the form
$$u(t,x)=\alpha _1(t)Q\left (\frac{x-x_1(t)}{\kappa _1(t)}\right )+\alpha _2(t)Q\left (\frac{x-x_2(t)}{\kappa _2(t)}\right )\ =:\alpha _1Q_1+\alpha _2Q_2\ .$$

\subsection{Derivation of the system}

Notice that 
$$Q'=-Q^2\ ,\ xQ'(x)=-Q(x)+\frac i2Q(x)^2\ ,$$
so that
\bee
Du-i\pa_tu&=&i\frac{\alpha _1}{\kappa _1}Q_1^2-\left (i\dot \alpha _1+i\alpha _1\frac{\dot \kappa _1}{\kappa _1}\right )Q_1-\alpha _1\left (i\frac{\dot x_1}{\kappa_1}+\frac 12\frac{\dot \kappa _1}{\kappa _1}\right )Q_1^2+\\
&&i\frac{\alpha _2}{\ka _2}Q_2^2- \left (i\dot \alpha _2+i\alpha _2\frac{\dot \kappa _2}{\kappa _2}\right )Q_2-\alpha _2\left (i\frac{\dot x_2}{\kappa_2}+\frac 12\frac{\dot \kappa _2}{\kappa _2}\right )Q_2^2
\eee
On the other hand, using partial fraction decompositions, it is easy to check the following identities,  for $j,k=1,2$,
\bee
\Pi (Q_j^2\overline Q_k)&=&-\frac{\kappa _j\kappa _k}{\left (x_j-x_k-i\frac{\kappa _j+\kappa _k}{2}\right )^2}\, Q_j+\frac{\ka _k}{x_j-x_k-i\frac{\kappa _j+\ka _k}2}\, Q_j^2\ ,\\
\Pi (Q_1Q_2\overline Q_j)&=&\frac{\kappa _2\ka _j\, Q_1}{ \left (x_1-x_2+i\frac{\kappa _2-\ka _1}2\right )\left (x_1-x_j-i\frac{\ka _1+\ka _j}2\right )}+\frac{\ka _1\ka _j\, Q_2}{ \left (x_2-x_1+i\frac{\ka _1-\ka _2}2\right )\left (x_2-x_j-i\frac{\ka _2+\ka _j}2\right )}\ .
\eee
This leads to
\bee
\Pi (u^2\overline u)&=&\Pi (Q_1^2\overline Q_1)+\Pi (Q_1^2\overline Q_2)+2\Pi (Q_1Q_2\overline Q_1)+2\Pi (Q_1Q_2\overline Q_2)
+\Pi (Q_2^2\overline Q_1)+\Pi (Q_2^2\overline Q_2)\\
&=&\beta _1Q_1^2+\gamma _1Q_1+\beta _2Q_2^2+\gamma _2Q_2\ ,
\eee
with
\bee 
\beta _1&=&i\alpha _1^2\overline \alpha _1+\frac{\ka _2}{\left (x_1-x_2-i\frac{\ka _1+\ka _2}2\right )}\alpha _1^2\overline \alpha _2\\
\gamma _1&=&\alpha _1^2\overline \alpha _1-\frac{\ka _1\ka _2\alpha _1^2\overline \alpha _2}{\left (x_1-x_2-i\frac{\ka _1+\ka _2}2\right )^2}+\frac{2i\ka _2\alpha _1\alpha _2\overline \alpha _1}{x_1-x_2+i\frac{\ka _2-\ka _1}2}+\frac{2\ka _2^2\alpha _1\alpha _2\overline \alpha _2}{\left (x_1-x_2+i\frac{\ka _2-\ka _1}2\right )\left( x_1-x_2-i\frac{\ka _1+\ka _2}2\right )}\\
\beta _2&=&i\alpha _2^2\overline \alpha _2+\frac{\ka _1}{\left (x_2-x_1-i\frac{\ka _2+\ka _1}2\right )}\alpha _2^2\overline \alpha _1\\
\gamma _2&=&\alpha _2^2\overline \alpha _2-\frac{\ka _2\ka _1\alpha _2^2\overline \alpha _1}{\left (x_2-x_1-i\frac{\ka _2+\ka _1}2\right )^2}+\frac{2i\ka _1\alpha _2\alpha _1\overline \alpha _2}{x_2-x_1+i\frac{\ka _1-\ka _2}2}+\frac{2\ka _1^2\alpha _2\alpha _1\overline \alpha _1}{\left (x_2-x_1+i\frac{\ka _1-\ka _2}2\right )\left( x_2-x_1-i\frac{\ka _2+\ka _1}2\right )}
\eee
Identifying $i\pa _tu$ and $\Pi (u^2\overline u)$, we obtain the following system,
\bee
i\frac{1-\dot x_1}{\ka _1}-\frac 12\frac{\dot \ka _1}{\ka _1}&=&i\vert \alpha _1\vert ^2+\frac{\ka _2}{\left (x_1-x_2-i\frac{\ka _1+\ka _2}2\right )}\alpha _1\overline \alpha _2\\
-i\left ( \frac{\dot \alpha _1}{\alpha _1}+\frac{\dot \ka _1}{\ka _1}\right )&=&\vert \alpha _1\vert ^2-\frac{\ka _1\ka _2\alpha _1\overline \alpha _2}{\left (x_1-x_2-i\frac{\ka _1+\ka _2}2\right )^2}+\frac{2i\ka _2\alpha _2\overline \alpha _1}{x_1-x_2+i\frac{\ka _2-\ka _1}2}+\frac{2\ka _2^2\vert \alpha _2\vert ^2
\left( x_1-x_2-i\frac{\ka _1+\ka _2}2\right )^{-1}}{\left (x_1-x_2+i\frac{\ka _2-\ka _1}2\right )}\\
i\frac{1-\dot x_2}{\ka _2}-\frac 12\frac{\dot \ka _2}{\ka _2}&=&i\vert \alpha _2\vert ^2+\frac{\ka _1}{\left (x_2-x_1-i\frac{\ka _2+\ka _1}2\right )}\alpha _2\overline \alpha _1\\
-i\left ( \frac{\dot \alpha _2}{\alpha _2}+\frac{\dot \ka _2}{\ka _2}\right )&=&\vert \alpha _2\vert ^2-\frac{\ka _2\ka _1\alpha _2\overline \alpha _1}{\left (x_2-x_1-i\frac{\ka _2+\ka _1}2\right )^2}+\frac{2i\ka _1\alpha _1\overline \alpha _2}{x_2-x_1+i\frac{\ka _1-\ka _2}2}+\frac{2\ka _1^2\vert \alpha _1\vert ^2\left( x_2-x_1-i\frac{\ka _2+\ka _1}2\right )^{-1}}{\left (x_2-x_1+i\frac{\ka _1-\ka _2}2\right )}
\eee
\subsection{Conservation laws}
Taking the real part of the combination of the first and of the third equation with coefficients $\ka _1$ and $\ka _2$, we derive the first conservation law,
\begin{equation}\label{C1}
\frac{\ka _1+\ka _2}{2}=K\ .
\end{equation}
The other conservation laws are not so easy to figure out. The first one corresponds to the mass conservation,
$$\Vert u\Vert _{L^2}^2=\vert \alpha _1\vert ^2\Vert Q_1\Vert _{L^2}^2+\vert \alpha _2\vert ^2\Vert Q_2\Vert _{L^2}^2+2{\rm Re}[\alpha _1\overline \alpha _2(Q_1\vert Q_2)]\ .$$
An elementary computation leads to
\begin{equation}\label{C2}
(2\pi )^{-1}\Vert u\Vert _{L^2}^2=\vert \alpha _1\vert ^2\ka _1+\vert \alpha _2\vert ^2\ka _2+2\ka _1\ka _2{\rm Im}\left (  \frac{\alpha _1\overline \alpha _2}{x_1-x_2-i\frac{\ka _1+\ka _2}2}\right )=:C\ .
\end{equation}
For the other conservation laws, we use the Lax pair property for the Hankel operators $H_u$, ensuring that the eigenvalues of $H_u^2$ are conservation laws. Recalling that $H_u(h):=\Pi (u\overline h)$, the matrix of $H_u$ in the basis $(Q_1,Q_2)$ is 
\bee
\mathscr M=\left ( \begin{array} {cc}   i\alpha _1 & \frac{\alpha _1\ka _2}{x_1-x_2-i\frac{\ka _1+\ka _2}2}\\
\frac{\alpha _2\ka _1}{x_2-x_1-i\frac{\ka _1+\ka _2}2}  & i\alpha _2\end{array}\right )\ .
\eee
Since $H_u$ is antilinear the trace of $H_u^2$ is 
\begin{equation}\label{C3}
{\rm tr}\left (\mathscr M\overline{\mathscr M}\right )=\vert  \alpha _1\vert ^2+\vert \alpha _2\vert ^2-2\ka _1\ka _2{\rm Re}\left ( \frac{\alpha _1\overline \alpha _2}{\left (x_1-x_2-i\frac{\kappa _1+\kappa _2}2  \right )^2} \right )=M\ ,
\end{equation}
which is also the momentum of $u$, divided by $2\pi $. The determinant of $H_u^2$ is 
\begin{equation}\label{C4}
\left \vert {\rm det}\mathscr M\right \vert ^2=\vert \alpha _1\vert ^2\vert \alpha _2\vert ^2\left (1-\frac{\ka _1\ka _2}{(x_1-x_2)^2+\left (\frac{\ka _1+\ka _2}2\right )^2 }  \right )^2=D\ .
\end{equation}
Let us specify the link of $D$ with the conservation laws $K, M$ and 
$$H:=\frac{1}{2\pi }\Vert u\Vert _{L^4}^4.$$
We claim that
$$4KD=2MC-H\ .$$
\begin{proof} Let us check this identity by calculating $H$. We set $X:=x_1-x_2$.
\bee
H&=&\vert \alpha _1\vert ^4\frac{\Vert Q_1\Vert _{L^2}^4}{2\pi }+ \vert \alpha _2\vert ^4\frac{\Vert Q_2\Vert _{L^2}^4}{2\pi }+  4\vert \alpha _1\vert ^2\vert \alpha _2\vert ^2\frac{\Vert Q_1Q_2\Vert _{L^2}^2}{2\pi }\\
&+&2{\rm Re}\left (\alpha _1^2\overline \alpha _2^2\frac{(Q_1^2\vert Q_2^2)}{2\pi}\right )+4{\rm Re}\left (\alpha _1^2\overline \alpha _1\overline \alpha _2\frac{(Q_1^2\vert Q_1Q_2)}{2\pi}\right )+
4{\rm Re}\left (\alpha _1\alpha _2\overline \alpha _2^2\frac{(Q_1Q_2\vert Q_2^2)}{2\pi}\right )\ .
\eee
Using 
\bee
\frac{\Vert Q_1\Vert _{L^2}^4}{2\pi }&=&2\ka _1\ ,\ \frac{\Vert Q_2\Vert _{L^2}^4}{2\pi }=2\ka _2\ ,\ \frac{\Vert Q_1Q_2\Vert _{L^2}^2}{2\pi }=\frac{2K\ka _1\ka _2}{X^2+K^2}\ ,\ \frac{(Q_1^2\vert Q_2^2)}{2\pi}=\frac{2i\ka _1^2\ka _2^2}{(X-iK)^3}\ ,\\
\frac{(Q_1^2\vert Q_1Q_2)}{2\pi}&=&\frac{-i\ka _1\ka _2}{X-iK}-\frac{\ka _1^2\ka _2}{(X-iK)^2}\ ,\ \frac{(Q_1Q_2\vert Q_2^2)}{2\pi}=\frac{-i\ka _1\ka _2}{X-iK}-\frac{\ka _1\ka _2^2}{(X-iK)^2}\ ,
\eee
we infer
\bee
H&=&2\ka _1\vert \alpha _1\vert ^4+2\ka _2\vert \alpha _2\vert ^4+\vert \alpha _1\vert ^2\vert \alpha _2\vert ^2\frac{8K\ka _1\ka _2}{X^2+K^2}+
4{\rm Re}\left (\alpha _1^2\overline \alpha _2^2\frac{i\ka _1^2\ka _2^2}{(X-iK)^3}\right )\\
&+& 4{\rm Re}\left (\alpha _1^2\overline \alpha _1\overline \alpha _2\left ( \frac{-i\ka _1\ka _2}{X-iK}-\frac{\ka _1^2\ka _2}{(X-iK)^2}    \right )\right )+4{\rm Re}\left (\alpha _1\alpha _2\overline \alpha _2^2\left ( \frac{-i\ka _1\ka _2}{X-iK}-\frac{\ka _1\ka _2^2}{(X-iK)^2}   \right )\right )\ \\
&=&2\ka _1\vert \alpha _1\vert ^4+2\ka _2\vert \alpha _2\vert ^4+\vert \alpha _1\vert ^2\vert \alpha _2\vert ^2\frac{8K\ka _1\ka _2}{X^2+K^2}+
4{\rm Re}\left (\alpha _1^2\overline \alpha _2^2\frac{i\ka _1^2\ka _2^2}{(X-iK)^3}\right )\\
&+&4\ka _1\ka _2(\vert \alpha _1\vert ^2+\vert \alpha _2\vert ^2){\rm Im}\left (\frac{\alpha _1\overline \alpha _2}{X-iK} \right )-4(\ka _1\vert \alpha _1\vert ^2+\ka _2\vert \alpha _2\vert ^2)\ka _1\ka _2K{\rm Re}\left (\frac{\alpha _1\overline \alpha _2}{(X-iK)^2}   \right )\ .
\eee
On the other hand,
\bee
M&=&\vert  \alpha _1\vert ^2+\vert \alpha _2\vert ^2-2\ka _1\ka _2{\rm Re}\left ( \frac{\alpha _1\overline \alpha _2}{(X-iK)^2} \right ) \ ,   \\
C&=& \ka _1\vert \alpha _1\vert ^2+\ka _2\vert \alpha _2\vert ^2+2\ka _1\ka _2{\rm Im}\left (  \frac{\alpha _1\overline \alpha _2}{X-iK}\right )\ ,
\eee
hence 
\bee 2MC&=&2\ka _1\vert \alpha _1\vert ^4+2\ka _2\vert \alpha _2\vert ^4+4K\vert \alpha _1\vert ^2\vert \alpha _2\vert ^2-4(\ka _1\vert \alpha _1\vert ^2+\ka _2\vert \alpha _2\vert ^2)\ka _1\ka _2{\rm Re}\left ( \frac{\alpha _1\overline \alpha _2}{(X-iK)^2} \right )   \\
&+& 4(\vert  \alpha _1\vert ^2+\vert \alpha _2\vert ^2) \ka _1\ka _2{\rm Im}\left (  \frac{\alpha _1\overline \alpha _2}{X-iK}\right )-8\ka _1^2\ka _2^2
{\rm Re}\left ( \frac{\alpha _1\overline \alpha _2}{(X-iK)^2} \right ){\rm Im}\left (  \frac{\alpha _1\overline \alpha _2}{X-iK}\right )\ ,
\eee
and 
\bee 2MC-H&=&4K\vert \alpha _1\vert ^2\vert \alpha _2\vert ^2\left (1-\frac{2\ka _1\ka _2}{X^2+K^2}\right ) \\
&-&8\ka _1^2\ka _2^2
{\rm Re}\left ( \frac{\alpha _1\overline \alpha _2}{(X-iK)^2} \right ){\rm Im}\left (  \frac{\alpha _1\overline \alpha _2}{X-iK}\right )-4{\rm Re}\left (\alpha _1^2\overline \alpha _2^2\frac{i\ka _1^2\ka _2^2}{(X-iK)^3}\right )\ .
\eee
Now just observe that, for every complex numbers $a,b$,
$$-8{\rm Re}(a){\rm Im}(b) +4{\rm Im}(ab)=4{\rm Im}(a){\rm Re}(b)-4{\rm Im}(b){\rm Re}(a)=4{\rm Im}(a\overline b)\ .$$
Applying this identity to
$$a=\frac{\alpha _1\overline \alpha _2}{(X-iK)^2}\ ,\ b= \frac{\alpha _1\overline \alpha _2}{X-iK}\ ,$$
we infer
$$-8
{\rm Re}\left ( \frac{\alpha _1\overline \alpha _2}{(X-iK)^2} \right ){\rm Im}\left (  \frac{\alpha _1\overline \alpha _2}{X-iK}\right )-4{\rm Re}\left (\frac{i\alpha _1^2\overline \alpha _2^2}{(X-iK)^3}\right )=\frac{4K}{(X^2+K^2)^2}$$
and finally
$$2MC-H=4K\vert \alpha _1\vert ^2\vert \alpha _2\vert ^2\left (1-\frac{2\ka _1\ka _2}{X^2+K^2}+\frac{\ka _1^2\ka _2^2}{(X^2+K^2)^2}\right )=4KD\ .$$
\end{proof}
\subsection{The reduced variables}
Notice that
\begin{equation}\label{C5}
\dot x_1+\dot x_2=2-C\ .
\end{equation} 
Therefore, it is natural to introduce 
$$X:=x_1-x_2\ ,\ \nu :=\frac{\ka _1-\ka _2}2\ .$$
Setting
$$\zeta _1:=\frac{\alpha _1}{X-iK}\ ,\ \zeta _2:=\frac{\alpha _2}{X+iK}\ ,$$
the system reads
\bee
\dot X&=&(X^2+K^2)[(K-\nu )\vert \zeta _2\vert ^2-(K+\nu)\vert \zeta _1\vert ^2]\ ,\\
\dot \nu &=&-2(K^2-\nu ^2){\rm Re}[\zeta _1\overline \zeta _2(X-iK)]\ .
\eee
Furthermore, the last three conservation laws read
\bee
C&=&(X^2+K^2)[(K+\nu )\vert \zeta _1\vert ^2+(K-\nu)\vert \zeta _2\vert ^2]+2(K^2-\nu ^2){\rm Im}[\zeta _1\overline \zeta _2(X-iK)]\\
M&=&(X^2+K^2)(\vert \zeta _1\vert ^2+\vert \zeta _2\vert ^2)-2(K^2-\nu ^2){\rm Re}(\zeta _1\overline \zeta _2)\\
&=&(K^2-\nu ^2)\vert \zeta _1-\zeta _2\vert ^2+(X^2+\nu ^2)(\vert \zeta _1\vert ^2+\vert \zeta _2\vert ^2) \\
D&=&\vert \zeta _1\vert ^2\vert \zeta _2\vert ^2(X^2+\nu ^2 )^2
\eee
\subsection{The resonance condition}
Notice that 
\bee
M^2-4D&=& (K^2-\nu ^2)^2\vert \zeta _1-\zeta _2\vert ^4+2(K^2-\nu ^2)(X^2+\nu ^2)\vert \zeta _1-\zeta _2\vert ^2(\vert \zeta _1\vert ^2+\vert \zeta _2\vert ^2)\\
&+& (X^2+\nu ^2)^2(\vert \zeta _1\vert ^2-\vert \zeta _2\vert ^2)^2\ .
\eee
Therefore, this conservation law cancels if and only if
$$\zeta _1=\zeta _2=:\zeta \ .$$
In this case, the above three conservation laws degenerate as 
$$\vert \zeta \vert ^2 (X^2+\nu ^2)=\frac M2=\sqrt D\ ,\ C=KM\ .$$
Using the laws $M,C,H,K$ and the identity
$$2MC-H=4KD\ ,$$
 we observe that the condition $M^2=4D$ is therefore equivalent to the set of two conditions,
$$MC=H \quad{\rm and}\quad C=KM\ .$$
Indeed, on the one hand, $M^2=4D$ implies $C=KM$ as we have already observed, and therefore, 
$$4KD=2M^2K-H=8KD-H$$
so that $H=4KD=2MC-H$, hence $H=MC$. On the other hand, if $MC=H$ and $C=KM$, then
$MC=4KD$ and $C=KM$, hence $KM^2=4KD$, so $M^2=4D$. 
\vskip 0.5cm
Under the resonance condition, the system in the reduced variables can be written
\bee
\dot X&=&-M\, \nu \, \frac{X^2+K^2}{X^2+\nu ^2}\ ,\\
\dot \nu &=&-M\, X\,  \frac{K^2-\nu ^2}{X^2+\nu ^2}\ .
\eee
In particular,
$$\frac{d}{dt}(X\nu)=-M\, K^2\ .$$
This means that $X\nu $ cancels exactly once, so either $X$ cancels and $\nu $ keeps the same sign, or $\nu $ cancels and $X$ keeps the same sign. In both cases, $\vert X(t)\vert $ tends to infinity like $KM\vert t\vert $, and $\vert \nu \vert $ tends to $K$.
Furthermore, in this case, we have 
$$\vert \alpha _1\vert =\vert \alpha _2\vert ,$$
and the phase shift is given by
$$\frac{\alpha _1}{\alpha _2}={\rm e}^{i\Gamma}=\frac{X-iK}{X+iK}\ ,$$
so the phase shift cancels at infinity. More precisely,
$$i\dot \Gamma =\frac{\dot X}{X-iK}-\frac{\dot X}{X+iK}=-M\, \frac{\nu }{X^2+\nu ^2}\, [X+iK-(X-iK)]=\frac{-2iKM\, \nu}{X^2+\nu ^2}\ .$$
Since $\vert X(t)\vert $ tends to infinity like $KM\vert t\vert $, we conclude that $\vert \dot \Gamma (t)\vert $ cancels as fast as $t^{-2}$.


\section{Proof of the non degeneracy \fref{hessianinvert}}
\label{invertaj}


The non degeneracy \fref{hessianinvert} follows from an explicit computation on the limiting Szeg\H{o} profile $Q^+$. However, before proceeding with the limiting process,
we need more precise information on $i\rho_\b$ and $(1-\b)\pa_\b Q_\b$. 

By \eqref{irho} and Lemma \ref{lemmacalculrho}, we have
\[ \rho =-i Q_\b+\frac 12\pa_y Q_\b+o_{\b\to 1}(1), \quad \quad \Sigma=y\pa_y Q_\b +o_{\b\to 1}(1).\]
which together with Lemma \ref{lemmaalgebraSzeg\H{o}} ensures:
\bee
\det A_j & = & \det \left(\begin{array}{llll}
(\Lambda Q_{\beta_j},Q_{\beta_j}) & (\Lambda Q_{\beta_j},i\pa_{y_j}Q_{\beta_j}) &  (\Lambda Q_{\beta_j},i\Lambda Q_{\beta_j}) & (\Lambda Q_{\beta_j},\rho_j)\\
(i Q_{\beta_j},Q_{\beta_j}) & (i Q_{\beta_j},i\pa_{y_j}Q_{\beta_j}) &  (i Q_{\beta_j},i\Lambda Q_{\beta_j}) & (i Q_{\beta_j},\rho_j)\\
(\pa_{y_j}Q_{\beta_j}, Q_{\beta_j}) & (\pa_{y_j}Q_{\beta_j},i\pa_{y_j}Q_{\beta_j}) &  (\pa_{y_j}Q_{\beta_j},i\Lambda Q_{\beta_j}) & (\pa_{y_j}Q_{\beta_j},\rho_j)\\
(\Sigma_j,Q_{\beta_j}) & (\Sigma_j,i\pa_{y_j}Q_{\beta_j}) &  (\Sigma_j,i\Lambda Q_{\beta_j}) & (\Sigma_j,\rho_j)
\end{array}\right)\\
& \to &\left(\begin{array}{llll}
0 &-\pi &0 & 0\\
0 &0 &  0& -\pi\\
0 & 0&  \pi& 0\\
-\pi &0 &  0 & 0
\end{array}\right)=-\pi^4\ \ \mbox{as}\  \ \beta_j\uparrow 1
\eee
and \fref{hessianinvert} is proved.


\section{Commutator estimates}


This Appendix is devoted to the derivation of commutator estimates used all along Section \ref{sectionenergy}. All proofs are more or less standard but the involved norms and associated decay are critical for the proof of Proposition \ref{propenergy}, so we display all estimates in detail.\\
We let in this section $\chi$ denote a  bounded Lipschitz continuous function and let $$ \chi_R(x)=\chi\left(\frac {x}R\right), \ \ R\geq 1.$$ 

\begin{lemma}[$|D|^{\frac 12}$ commutator]
There holds the global bound
\be
\label{commestimateone}
\|[|D|^{\frac 12},\chi _R]g\|_{L^2}\lesssim \frac{\|\chi\|_{W^{1,\infty}}}{\sqrt{R}}\|g\|_{L^2},
\ee
and the weighted bound for $0<\alpha<1$:
\be
\label{commestimateoneweighted}
\|\frac{1}{\la z\ra^{\frac{1+\alpha}{2}}}[|D|^{\frac 12},\chi _R]g\|_{L^2}\lesssim  \frac{\|\chi\|_{W^{1,\infty}}}{\sqrt{R}}\|\frac{g}{\la z\ra ^{\frac{1+\alpha}{2}}}\|_{L^2}\ \ \mbox{with}\ \ z=\frac{x}{R}.
\ee
\end{lemma}

\begin{proof}

\noindent{\bf Step 1:} Kernel representation. First we provide a description of the operator $\vert D\vert ^{\frac 12}$ in the space variables. This operator is the convolution operator with the tempered distribution
$$k:=\mathcal F^{-1}(\vert \xi \vert ^{\frac 12})\ .$$
From the properties of the Fourier transform we know that $k$ is homogeneous of degree $-3/2$, and is even. As a consequence, it is characterized up to a multiplicative constant. For every function $\varphi $ in the Schwartz space, we therefore have
$$\langle k,\varphi \rangle =c\int _\R \frac{\varphi (x)-\varphi (0)}{\vert x\vert ^{\frac 32}}\, dx\ ,\ c:=(2\pi)^{-\frac 12}\frac{\int _\R \vert \xi \vert ^{\frac 12}e^{-\frac {\xi ^2}2}\, d\xi}{\int _\R \frac{e^{-\frac {x^2}2}-1}{\vert x\vert^{\frac 3 2}}\, dx }\ ,$$
and
$$k*\varphi (x)=c\int _\R \frac{\varphi (y)-\varphi (x)}{\vert x-y\vert ^{\frac 32}}\, dy\ .$$
Consequently, we can write
$$ [\vert D\vert ^{\frac 12}, \chi _R]\, g(x)=c\int _\R \frac{\chi _R(y)-\chi _R(x)}{\vert x-y\vert ^{\frac 32}}\, g(y)\, dy\ .$$

\noindent{\bf Step 2:} Proof of \eqref{commestimateone}. We split the kernel in two parts,
\bee
[\vert D\vert ^{\frac 12}, \chi _R]\, g(x)&=&c(T^{med}g(x)+T^{off}g(x))\ ,\\
T^{med}g(x)&:=&\int _{ \vert x-y\vert \le 5R} \frac{\chi _R(y)-\chi _R(x)}{\vert x-y\vert ^{\frac 32}}\, g(y)\, dy\ ,\\
T^{off}g(x)&:=&\int _{\vert x-y\vert >  5R } \frac{\chi _R(y)-\chi _R(x)}{\vert x-y\vert ^{\frac 32}}\, g(y)\, dy\ .
\eee
We have 
$$\frac{|\chi_R(x)-\chi_R(y)|}{|x-y|^{\frac 32}}\lesssim \frac{\|\chi'_R\|_{L^{\infty}}}{|x-y|^{\frac 12}}\lesssim\frac{1}{R|x-y|^{\frac 12}}$$ and hence,
 by Young's inequality,
$$\|T^{med}\|_{L^2}\lesssim \frac{\|\chi\|_{W^{1,\infty}}}R\|\frac{{\bf 1}_{|x|\leq 5 R}}{|x|^{\frac 12}}\star g\|_{L^2}\lesssim  \frac{\|\chi\|_{W^{1,\infty}}}R\|\frac{{\bf 1}_{|x|\leq 5R}}{|x|^{\frac12}}\|_{L^1}\|g\|_{L^2} \lesssim \frac{\|\chi\|_{W^{1,\infty}}}{\sqrt R}\|g\|_{L^2}.
$$
Similarly,
\bee 
\|T^{off}\|_{L^2}\lesssim \|\chi\|_{L^{\infty}}\|\frac{{\bf 1}_{|x|\geq 5R}}{|x|^{\frac 32}}\star g\|_{L^2}\lesssim \|\chi\|_{W^{1,\infty}}\|\frac{{\bf 1}_{|x|\geq 5R}}{|x|^{\frac 32}}\|_{L^1}\|g\|_{L^2}\lesssim \frac{\|\chi\|_{W^{1,\infty}}}{\sqrt R}\|g\|_{L^2}
\eee
and \eqref{commestimateone} is proved.\\

\noindent{\bf Step 3:} Proof of \eqref{commestimateoneweighted}. For $|x-y|\leq 5R$, we have $\la \frac yR\ra\lesssim \la \frac x R\ra$ and we infer
$$|\frac{1}{\la \frac xR\ra^{\frac{1+\alpha}{2}}}T^{med} g|\lesssim  \frac{\|\chi\|_{W^{1,\infty}}}R\frac{{\bf 1}_{|x|\leq 5R}}{|x|^{\frac 12}}\star \frac{|g|}{\la \frac xR \ra^{\frac{1+\alpha}{2}}}$$ from which, as above, from Young's inequality,
$$\|\frac{1}{\la \frac xR\ra^{\frac{1+\alpha}{2}}}T^{med} g\|_{L^2}\lesssim \frac{\|\chi\|_{W^{1,\infty}}}{\sqrt R}\|\frac{|g|}{\la \frac xR \ra^{\frac{1+\alpha}{2}}}\|_{L^2}.$$ For $|x-y|\geq 5R$, we distinguish
\bee
&&T^{off}_1=\int _{\vert x-y\vert > 5R, \ \ |y|\leq 2|x| } \frac{\chi _R(y)-\chi _R(x)}{\vert x-y\vert ^{\frac 32}}\, g(y)\, dy\\
&&T^{off}_2=\int _{\vert x-y\vert > 5R, \ \ |y|\geq 2|x| } \frac{\chi _R(y)-\chi _R(x)}{\vert x-y\vert ^{\frac 32}}\, g(y)\, dy\\
\eee 
For the first kernel, $\la \frac yR\ra\lesssim \la \frac xR\ra$ and thus 
$$|\frac{1}{\la \frac xR\ra^{\frac{1+\alpha}{2}}}T^{off}_1g|\lesssim  \|\chi\|_{L^{\infty}}\frac{{\bf 1}_{|x|\geq 5R}}{|x|^{\frac 32}}\star \frac{|g|}{\la \frac xR \ra^{\frac{1+\alpha}{2}}}$$ from which, as above, 
$$\|\frac{1}{\la \frac xR\ra^{\frac{1+\alpha}{2}}}T^{off}_1g\|_{L^2}\lesssim \|\chi\|_{L^{\infty}}\|\frac{{\bf 1}_{|x|\geq 5R}}{|x|^{\frac 32}}\|_{L^1}\|\frac{g}{\la \frac xR\ra^{\frac{1+\alpha}{2}}}\|_{L^2}\lesssim \frac{\|\chi\|_{W^{1,\infty}}}{\sqrt R}\|\frac{g}{\la \frac xR\ra^{\frac{1+\alpha}{2}}}\|_{L^2}.$$ 
For the second kernel, $|y|\geq 2|x|$ and $|x-y|\geq 5R$, we have $|y|\gtrsim R$ and $\vert x-y\vert \gtrsim \vert y\vert$. Therefore, from Cauchy--Schwarz' inequality,
\bee
|T^{off}_2g|&\lesssim& \|\chi\|_{L^{\infty}}\int_{|y|\gtrsim R}\frac{|g(y)|}{ \la \frac yR\ra^{\frac{1+\alpha}{2}}}\left(\frac{|y|}{R}\right)^{\frac{1+\alpha}{2}}\frac{dy}{|y|^{\frac 32}}\\
& \lesssim & \frac{\|\chi\|_{L^{\infty}}}{R^{\frac{1+\alpha}{2}}}\|\frac{g}{\la \frac xR\ra^{\frac{1+\alpha}{2}}}\|_{L^2}\left(\int_{|y|\gtrsim R}\frac{dy}{|y|^{2-\alpha}}\right)^{\frac 12}\\
& \lesssim & \frac{\|\chi\|_{L^{\infty}}}{R}\|\frac{g}{\la \frac xR\ra^{\frac{1+\alpha}{2}}}\|_{L^2}
\eee
where we used $\alpha<1$, from which 
$$\|\frac{T^{off}g}{\la \frac xR\ra^{\frac{1+\alpha}{2}}}\|_{L^2}\lesssim \frac{\|\chi\|_{L^{\infty}}}{R}\|\frac{g}{\la \frac xR\ra^{\frac{1+\alpha}{2}}}\|_{L^2}\|\frac{1}{\la \frac xR\ra^{\frac{1+\alpha}{2}}}\|_{L^2}\lesssim \frac{\|\chi\|_{W^{1,\infty}}}{\sqrt R}\|\frac{g}{\la \frac xR\ra^{\frac{1+\alpha}{2}}}\|_{L^2}$$
where we simply changed variables and used $\alpha >0$ in the last step. This concludes the proof of \eqref{commestimateoneweighted}.
\end{proof}

We shall also use the following slightly different version.

\begin{lemma}[Commutator estimate in $L^2$]
\label{commest}
For a general function $\chi$ such that $\pa _x\chi \in L^1$, there holds the following bounds.
\be
\label{commestimate}
\|[|D|^{\frac 12},\chi]g\|_{L^2}\lesssim \||\xi|^{\frac 12}\widehat{\chi}\|_{L^1}\| g\|_{L^2}\lesssim \left(\|\pa_x\chi\|_{L^1}\|\pa_{xx}\chi\|_{L^1}\right)^{\frac 12}\|g\|_{L^2},
\ee
\be
\label{commd}
\|[|D|,\chi]g\|_{L^2}+\|[\Pi^+|D|,\chi]\|_{L^2}\lesssim\left(\|\pa_x\chi\|_{L^1}\|\pa^3_x\chi\|_{L^1}\right)^{\frac 12} \|g\|_{L^2},
\ee
\bea
\label{commdbis}
&&\||D|^{\frac 12}[|D|^{\frac 12},\chi]g\|_{L^2}\\
\nonumber &\lesssim&\left(\|\pa_x\chi\|_{L^1}\|\pa_{xx}\chi\|_{L^1}\right)^{\frac 12}\||D|^{\frac 12}g\|_{L^2}+\left(\|\pa_x\chi\|_{L^1}\|\pa^3_x\chi\|_{L^1}\right)^{\frac 12} \|g\|_{L^2},
\eea
\end{lemma}

\begin{proof}{\bf Step 1:} Proof of \fref{commestimate}. Since $\pa _x\chi \in L^1$, $\hat \chi (\xi )$ is discontinuous only at $\xi =0$, with a mild singularity  justifying the  calculations below for every $g$ in the Schwartz space. We have
$$ \widehat{[|D|^{\frac 12},\chi] g}= \widehat{|D|^{\frac 12}(\chi g)}(\xi) -  \widehat{\chi(|D|^{\frac 12}g)}(\xi) = \int_{\R}(|\xi|^{\frac 12} - |\eta|^{\frac 12})\hat{\chi}(\xi-\eta)\hat{g}(\eta)\,d\eta.
$$
We use 
\be
\label{ineqxihalf}
\big| |\xi|^{\frac 12} - |\eta|^\frac12\big|\leq |\xi - \eta|^\frac 12 .
\ee
to estimate pointwise
\[
 \left| \widehat{[|D|^{\frac 12},\chi] g}(\xi)\right|\lesssim \int_{\R}|\xi - \eta|^\frac 12 |\hat{\chi}|(\xi-\eta)|\hat{g}|(\eta)\,d\eta=|\xi|^{\frac 12}|\hat{\chi}|\star |\hat{g}|.
\]
We conclude, from Young's inequality and the Plancherel formula,
$$
\|[|D|^{\frac 12},\chi]g\|_{L^2}\lesssim \|(|\xi|^{\frac 12}|\hat{\chi}|)\star |\hat{g}|\|_{L^2}\lesssim\||\xi|^{\frac 12}\hat{\chi}\|_{L^1}\|g\|_{L^2}.
$$
Finally, we estimate
\bea
\label{cnekonekoneoneo}
\nonumber \int |\xi|^{\frac 12}|\hat{\chi}|d\xi&\leq& \int_{|\xi|\leq A} \frac{\| \widehat{\pa_x\chi}\|_{L^\infty}}{|\xi|^{\frac 12}}d\xi+ \int_{|\xi|\geq A} \frac{ \|\widehat{\pa_{xx}\chi}\|_{L^\infty}}{|\xi|^{\frac 32}}d\xi\lesssim \sqrt{A}\|\pa_x\chi\|_{L^1}+\frac{\|\pa_{xx}\chi\|_{L^1}}{\sqrt{A}}\\
& \lesssim & (\|\pa_x\chi\|_{L^1}\|\pa_{xx}\chi\|_{L^1})^{\frac 12}.
\eea   
 by optimizing in $A$.\\
  
 \noindent{\bf Step 2:} Proof of \fref{commd}. We compute
 \bee
 | \widehat{[|D|,\chi] g}(\xi )|&=&\left| \widehat{|D|(\chi g)}(\xi) -  \widehat{\chi(|D|g)}(\xi)\right| = \left|\int_{\R}(|\xi| - |\eta|)\hat{\chi}(\xi-\eta)\hat{g}(\eta)\,d\eta\right|\\
 &\lesssim & \int_{\R}|\xi - \eta| |\hat{\chi}|(\xi-\eta)|\hat{g}|(\eta)\,d\eta=(|\xi||\hat{\chi}|)\star |\hat{g}|
\eee
and hence
\bee
\|[|D|,\chi]g\|_{L^2}\lesssim \|(|\xi||\hat{\chi}|)\star |\hat{g}|\|_{L^2}\lesssim \||\xi|\hat{\chi}\|_{L^1}\|g\|_{L^2}.
\eee
We now estimate 
\be
\label{cneocneononor}\int |\xi||\hat{\chi}|\lesssim \int_{|\xi|\leq A}\|\pa_x\chi\|_{L^{1}}d\xi+\int_{|\xi|\geq A}\frac{\|\pa^3_x\chi\|_{L^1}}{|\xi|^2}d\xi\lesssim \left(\|\pa_x\chi\|_{L^1}\|\pa^3_x\chi\|_{L^1}\right)^{\frac 12}
\ee
 and the first commutator estimate in \fref{commd} is proved. Similarly,
  \bee
 | \widehat{[\Pi^+|D|,\chi] g}|&=&\left|\int_{\R}(|\eta|{\bf 1}_{\eta>0}-|\xi|{\bf 1}_{\xi>0})\hat{\chi}(\xi-\eta)\hat{g}(\eta)\,d\eta\right|\\
 &\lesssim & \int_{\R}|\xi - \eta| |\hat{\chi}|(\xi-\eta)|\hat{g}|(\eta)\,d\eta=(|\xi||\hat{\chi}|)\star |\hat{g}|
\eee
and the conclusion follows as above.\\

\noindent{\bf Step 3:} Proof of \eqref{commdbis}. We compute, using \eqref{ineqxihalf},

\bee
||D|^{\frac 12}\widehat{[|D|^{\frac 12},\chi] g}(\xi)|&\lesssim & \int_{\R}|\xi|^{\frac 12}\left||\xi|^{\frac 12} - |\eta|^{\frac 12}\right|\left|\hat{\chi}(\xi-\eta)\hat{g}(\eta)\right|d\eta\\
& \lesssim &  \int_{\R}|\xi-\eta||\hat{\chi}(\xi-\eta)||\hat{g}(\eta)|d\eta+ \int_{\R}|\xi - \eta|^\frac 12 |\hat{\chi}|(\xi-\eta)|\eta|^{\frac 12}|\hat{g}|(\eta)\,d\eta\\
& \lesssim & |\xi\chi|\star |\hat{g}|+|\xi|^{\frac 12}|\hat{\chi}|\star |\eta^{\frac 12}\hat{g}|
\eee
and the conclusion follows as in the previous two steps.
\end{proof}

We similarly estimate $\Pi^{\pm}$ commutators.

\begin{lemma}[$\Pi^{\pm}$ commutator] Assume that the derivative $\chi'$ is supported in $[1,2]$. Then there holds
\be
\label{estcommlonebis}
\left\|D^k[\Pi^{\pm},\chi_R]g\right\|_{L^2}\lesssim \frac{\|\chi\|_{W^{k+1,\infty}}}{R^k}\|g\|_{L^2}, \ \ k=1,2,
\ee
and
\be
\label{estcommlonebisbis}
\left\|\la x\ra D^2[\Pi^{\pm},\chi_R]g\right\|_{L^2}\lesssim \frac{\|\chi\|_{W^{3,\infty}}}{R}\|g\|_{L^2}.
\ee
\end{lemma}

\begin{proof} We recall the standard representation formula 
$$[\Pi^{+},\chi_R]g(x)=c\int\frac{\chi_R(x)-\chi_R(y)}{x-y}g(y)dy.$$

\noindent{\bf Step 1:} Case $k=1$. We take a derivative, $$\pa_x[\Pi^{+},\chi_R]g(x)=-c\int\frac{\chi_R(x)-\chi_R(y)-(x-y)\chi_R'(x)}{(x-y)^2}g(y)dy.$$ We now split the kernel as
\bee
\pa_x[\Pi^+,\chi_R]g(x)&=& -c(T^{med}_Rg(x)+T^{off}g(x))\ ,\\
T^{med}_Rg(x)&:=&\int _{\vert x-y\vert <R} \frac{\chi_R(x)-\chi_R(y)-(x-y)\chi_R'(x)}{(x-y)^2} g(y)\, dy\ ,\\
T^{off}_Rg(x)&:=&\int _{\vert x-y\vert >R }\frac{\chi_R(x)-\chi_R(y)-(x-y)\chi_R'(x)}{(x-y)^2} g(y)\, dy\ .
\eee
We estimate
\be
\label{euisfbeibei}
\left|\frac{\chi_R(x)-\chi_R(y)-(x-y)\chi_R'(x)}{(x-y)^2}\right|\lesssim \|\chi_R''\|_{L^{\infty}}\lesssim \frac{\|\chi\|_{W^{2,\infty}}}{R^2}.
\ee 
Hence,  by \eqref{euisfbeibei} and Young's inequality,
$$\|T^{med}_Rg\|_{L^2}\lesssim \frac{\|\chi\|_{W^{2,\infty}}}{R^2}\|{\bf 1}_{|x-y|<R}\star g\|_{L^2}\lesssim \frac 1{R^2}\|{\bf 1}_{|x-y|<R}\|_{L^1}\|g\|_{L^2}\lesssim \frac{\|\chi\|_{W^{2,\infty}}}{R}\|g\|_{L^2}.$$
Off the diagonal, we use the special structure of $\chi_R$. Firstly, we have
\bee
|T^{off}g|&\lesssim &\|\chi\|_{L^{\infty}}\int_{|x-y|>R}\frac{|g(y)|}{|x-y|^2}dy+\int_{|x-y|>R}\frac{|\chi_R'(x)-\chi_R'(y)\vert }{|x-y|}|g(y)|dy\\
&+& \int_{|x-y|>R}\frac{1}{|x-y|}|\chi_R'(y) g(y)|dy:=  I+II+III.
\eee
The first term is estimated by Young's inequality,
\bee
 \|I\|_{L^2}&\lesssim&\|\chi\|_{L^{\infty}}\|\frac{1}{|x|^2}{\bf 1}_{|x|>R}\star |g|\|_{L^2}\lesssim\|\chi\|_{L^{\infty}} \|\frac{1}{|x|^2}{\bf 1}_{|x|>R}\|_{L^1}\|g\|_{L^2}\\
 &\lesssim & \frac{\|\chi\|_{W^{2,\infty}}}{R}\|g\|_{L^2}.
 \eee 
 For the second term, we use Young's inequality and the fact that $\chi_R'(x)$ is supported in $R\leq |x|\leq 2 R$. We obtain
\bee
\|II\|_{L^2}&\lesssim&\|\frac{\chi_R'}{|x|}{\bf 1}_{|x|>R}\star g\|_{L^2}\lesssim \|\frac{\chi_R'}{|x|}{\bf 1}_{|x|>R}\|_{L^1}\|g\|_{L^2}\lesssim \frac{\|\chi'\|_{L^{\infty}}}{R}\|g\|_{L^2}\int_{\frac R\leq |x|\leq 2R}\frac{dx}{\la x\ra}\\
&\lesssim & \frac{\|\chi\|_{W^{2,\infty}}}{R}\|g\|_{L^2}.
\eee
The last term is treated with Young's and Cauchy Schwarz's inequalities,
\bee
\|III\|_{L^2}&\lesssim& \|\frac{1}{|x|}{\bf 1}_{|x|>R}\star (\chi_R'g)\|_{L^2}|\lesssim  \|\frac{1}{|x|}{\bf 1}_{|x|>R}\|_{L^2}\|\chi_R'g\|_{L^1}\lesssim \frac{\|\chi'\|_{L^{\infty}}}{\sqrt{R}}\frac{1}{R}\|g\|_{L^1(R\leq|x|\leq  2R)}\\
&\lesssim & \frac{\|\chi\|_{W^{2,\infty}}}{R}\|g\|_{L^2}.
\eee
The collection of above bounds yields \eqref{estcommlonebis}for $k=1$.\\

\noindent{\bf Step 2:} Case $k=2$. The proof is similar. We take two derivatives,
\bee
\pa^2_x[\Pi^{+},\chi_R]g(x)&=&2c\int\frac{\chi_R(x)-\chi_R(y)-(x-y)\chi_R'(x)+\frac12(x-y)^2\chi_R''(x)}{(x-y)^3}g(y)dy\\
& = & c(T^{med}_Rg(x)+T^{off}g(x)).
\eee
We estimate 
$$\left|\frac{\chi_R(x)-\chi_R(y)-(x-y)\chi_R'(x)+\frac12(x-y)^2\chi_R''(x)}{(x-y)^3}\right|\lesssim \|\chi_R'''\|_{L^{\infty}}\lesssim \frac{\|\chi\|_{W^{3,\infty}}}{R^3}$$
from which
$$\|T^{med}_Rg\|_{L^2}\lesssim \frac{\|\chi\|_{W^{3,\infty}}}{R^3}\|{\bf 1}_{|x-y|<R}\star g\|_{L^2}\lesssim \frac 1{R^3}\|{\bf 1}_{|x-y|<R}\|_{L^1}\|g\|_{L^2}\lesssim \frac{\|\chi\|_{W^{3,\infty}}}{R^2}\|g\|_{L^2}.$$
Off the diagonal, we split
\bee
|T^{off}g|&\lesssim &\|\chi\|_{L^{\infty}}\int_{|x-y|>R}\frac{|g(y)|}{|x-y|^3}dy\\
&+& \int_{|x-y|>R}\frac{|\chi_R'(x)-\chi_R'(y)|}{|x-y|^2}|g(y)|dy+ \int_{|x-y|>R}\frac{|\chi'_R(y)|}{|x-y|^2}|g(y)|dy\\
&+& \int_{|x-y|>R}\frac{|\chi''_R(x)-\chi''_R(y)|}{|x-y|}|g(y)|dy+\int_{|x-y|>R}\frac{1}{|x-y|}|\chi_R''(y) g(y)|dy\\
&:=&   I+II+III.
\eee
The first term is estimated by Young's inequality,
\bee
 \|I\|_{L^2}&\lesssim&\|\chi\|_{L^{\infty}}\|\frac{1}{|x|^3}{\bf 1}_{|x|>R}\star |g|\|_{L^2}\lesssim\|\chi\|_{L^{\infty}} \|\frac{1}{|x|^3}{\bf 1}_{|x|>R}\|_{L^1}\|g\|_{L^2}\\
 &\lesssim & \frac{\|\chi\|_{W^{3,\infty}}}{R^2}\|g\|_{L^2}.
 \eee 
 For the second term, we use Young's inequality and the fact that $\chi_R'(x)$ is supported in $R\leq |x|\leq 2 R$. We obtain
\bee
\|\frac{\chi_R'}{|x|^2}{\bf 1}_{|x|>R}\star g\|_{L^2}&\lesssim& \|\frac{\chi_R'}{|x|^2}{\bf 1}_{|x|>R}\|_{L^1}\|g\|_{L^2}\lesssim \frac{\|\chi'\|_{L^{\infty}}}{R}\|g\|_{L^2}\int_{R\leq |x|\leq 2R}\frac{dx}{\la x\ra^2}\\
&\lesssim & \frac{\|\chi\|_{W^{3,\infty}}}{R^2}\|g\|_{L^2}
\eee
and 
\bee
\|\frac{1}{|x|^2}{\bf 1}_{|x|>R}\star (\chi_R'g)\|_{L^2}|&\lesssim & \|\frac{1}{|x|^2}{\bf 1}_{|x|>R}\|_{L^2}\|\chi_R'g\|_{L^1}\lesssim \frac{\|\chi\|_{W^{3,\infty}}}{R^{\frac 32}}\frac{1}{R}\|g\|_{L^1(R\leq|x|\leq  2R)}\\
&\lesssim & \frac{\|\chi\|_{W^{3,\infty}}}{R^2}\|g\|_{L^2}
\eee
and hence $$II\lesssim \frac{\|\chi\|_{W^{3,\infty}}}{R^2}\|g\|_{L^2}.$$
For the last term, we have
\bee
\|\frac{\chi''_R}{|x|}{\bf 1}_{|x|>R}\star g\|_{L^2}&\lesssim& \|\frac{\chi''_R}{|x|}{\bf 1}_{|x|>R}\|_{L^1}\|g\|_{L^2}\lesssim \frac{\|\chi\|_{W^{3,\infty}}}{R^2}\|g\|_{L^2}\int_{R\leq |x|\leq 2R}\frac{dx}{\la x\ra}\\
&\lesssim & \frac{\|\chi\|_{W^{3,\infty}}}{R^2}\|g\|_{L^2}
\eee
and 
\bee
\|\frac{1}{|x|}{\bf 1}_{|x|>R}\star (\chi''_Rg)\|_{L^2}|&\lesssim & \|\frac{1}{|x|}{\bf 1}_{|x|>R}\|_{L^2}\|\chi''_Rg\|_{L^1}\lesssim \frac{\|\chi\|_{W^{3,\infty}}}{\sqrt{R}}\frac{1}{R^2}\|g\|_{L^1( R\leq|x|\leq  2R)}\\
&\lesssim & \frac{\|\chi\|_{W^{3,\infty}}}{R^2}\|g\|_{L^2}.
\eee
The collection of above bounds yields \eqref{estcommlonebis}  for $k=2$.\\

\noindent{\bf Step 3:} Proof of \eqref{estcommlonebisbis}. We revisit the estimates of step 2 in the presence of the additional $\la x\ra $ weight. For $|x|\leq 10 R$, we estimate directly from \eqref{estcommlonebis},
$$\left\|\la x\ra D^k[\Pi^{\pm},\chi_R]g\right\|_{L^2(|x|\leq 10R)}\lesssim \frac{\|\chi\|_{W^{2,\infty}}}{R}\|g\|_{L^2}.$$ We therefore assume $|x|\geq 10R$. Since $\chi '=0$  outside $[1,2]$, $|x-y|<R$ implies $\chi_R(x)-\chi_R(y)=0$ and $\chi'_R(x)=0$. For $|x-y|>R$, we have $|x-y|>|x|$ if $x,y$ do not have the same sign, and if $x,y$ have the same sign, necessarily $|y|\leq R$, for otherwise $\chi_R(x)-\chi_R(y)=0$ again. In both cases, $|x-y|\gtrsim |x|$, and hence
$$
\|T^{off}g\|_{L^{\infty}(|x|\geq 10R)}\lesssim \int_{|x-y|\gtrsim |x|}\frac{1}{|x-y|^3}|g(y)|dy\lesssim \|g\|_{L^2}\|\frac{{\bf 1}_{|z|\gtrsim |x|}}{|z|^3}\|_{L^2}\lesssim \frac{1}{|x|^{\frac 52}}\|g\|_{L^2},
$$
therefore $$\|\la x\ra T^{off}g\|_{L^2(|x|\geq 10R)}\lesssim \|g\|_{L^2}\|\frac{1}{\la x\ra^{\frac 32}}\|_{L^2(|x|\geq 10R)}\lesssim \frac{\|g\|_{L^2}}{R},$$ and \eqref{estcommlonebisbis} is proved.
\end{proof}

We will need a standard localization formula for the kinetic energy.

\begin{lemma}[Localization of the kinetic energy]
\label{kinenergy}
There holds for given functions $Z,f$,
\bea
\label{estgenerale}
\int \vert Z\vert^2||D|^{\frac 12}f|^2&=&\int   ||D|^{\frac 12}( Zf)|^2\\
\nonumber &+& O\left(\|[\Dhalf,Z]f\|_{L^2}\left[\|\Dhalf(Zf)\|_{L^2}+\|Z\Dhalf f\|_{L^2}\right]\right).
\eea
In particular,
for $\chi_R(y):=\chi (\frac{y}{R})$ with $\chi $ a smooth
function satisfying 
\begin{align*}
\chi(y)=
\begin{cases}
1, \text{ if } |y|<\frac 14\\
0, \text{ if } |y|>\frac 12,
\end{cases}
\end{align*}
we have
\bea
\label{commestimatebis}
\int \chi_R^2||D|^{\frac 12}f|^2=\int   ||D|^{\frac 12}(\chi_Rf)|^2+O\left(\frac{\|f\|_{L^2}^2+\|\Dhalf( \chi_Rf)\|_{L^2}^2}{\sqrt{R}}\right).
\eea
\end{lemma}

\begin{proof}  We expand and estimate
\bee
\int |\Dhalf (Zf)|^2&=&(\Dhalf(Zf),\Dhalf(Zf))=([\Dhalf,Z]f+Z\Dhalf f,\Dhalf(Zf))\\
& = & O\left(\|[\Dhalf,Z]f\|_{L^2}\|\Dhalf(Zf)\|_{L^2}\right)+(Z\Dhalf f,[\Dhalf, Z]f+Z\Dhalf f)\\
& = & \int Z^2|\Dhalf f|^2+O\left(\|[\Dhalf,Z]f\|_{L^2}\left[\|\Dhalf(Zf)\|_{L^2}+\|Z\Dhalf f\|_{L^2}\right]\right)
\eee
and \fref{estgenerale} follows. We then estimate from \fref{commestimateone},
$$\|[\Dhalf, \chi_R]\|_{L^2\to L^2}\lesssim \frac 1{\sqrt R}$$
and \fref{commestimatebis} follows.
\end{proof}

Finally, for establishing the coercivity of our energy functional, we need the following --- non sharp --- estimate.
\begin{lemma}
Let $\chi $ be a smooth
function satisfying 
\begin{align*}
\chi(y)=
\begin{cases}
1, \text{ if } |y|<\frac 14\\
0, \text{ if } |y|>\frac 12,
\end{cases}
\end{align*}
There holds:
\be
\label{estcomm}
\left\|\frac{(\chi_R u^+)^-}{\la y\ra}\right\|_{L^2}\lesssim \frac{1}{R^{\frac 13}}\|u^+\|_{L^2}.
\ee
\end{lemma}
\begin{proof}
Using a standard duality argument, it suffices to show that
\be
\label{dual estcomm}
\left|((\chi_Ru^+)^-,v)\right|\lesssim \frac{1}{R^{\frac 13}}\|u^+\|_{L^2}\|\la y\ra v\|_{L^2}
\ee
for any $v\in L^2(\R)$ such that $\jb{y} v\in L^2(\R)$. Let $0<\eta<1$ and consider a cut off function 
$$\zeta_\eta(\xi)=\zeta\left(\frac{\xi}{\eta}\right), \ \ \zeta (\xi)=\left\{\begin{array}{ll} 1 \ \ \mbox{for}\ \ |\xi|\leq 1,\\ 0\ \ \mbox{for}\ \  |\xi|\geq 2\end{array}\right.$$ and let
 $$\hat{v}(\xi)=\zeta_\eta\hat{v}(\xi)+(1-\zeta_\eta)\hat{v}(\xi)=:\hat{v_1}(\xi)+\hat{v_2}(\xi).$$
For the high frequency part, we compute, using Plancherel's identity,
and the fact that $|y|\geq\frac{R}{4}$ on the support of $1-\chi_R$:
$$
|((\chi_R u^+)^-,v_2)|=|(\chi_R u^+ ,v_2^-)|=|((1-\chi_R)u^+,v_2^-)|\lesssim\frac{1}{R} \|u^+\|_{L^2}\|\la y\ra v_2^-\|_{L^2}
$$
and, by construction and Plancherel's identity,
 $$\|\la y\ra v_2^-\|_{L^2}^2\lesssim \int |\hat{v}^-_2|^2+|\pa_{\xi}\widehat{v_2^-}|^2\lesssim \frac{1}{\eta^2}\left[\int |\hat{v}|^2+|\pa_{\xi}\widehat{v}|^2\right]\lesssim \frac{\|\la y\ra v\|_{L^2}^2}{\eta^2}.$$ We estimate, for the low frequency part,
$$
|((\chi_R u^+)^-,v_1)|=\|u^+\|_{L^2}\|v_1\|_{L^2}\lesssim \|u^+\|_{L^2}\|\hat{v}_1\|_{L^2}
$$
and $$\|\hat{v_1}\|_{L^2}^2\lesssim \int_{|\xi|\leq 2\eta}|\hat{v}|^2\lesssim \eta\|\hat{v}\|_{L^{\infty}}^2\lesssim \eta\|v\|_{L^1}^2\lesssim \eta\|\la y\ra v\|^2_{L^2}.$$ 
The collection of above bounds and the choice $\eta=\frac{1}{R^{\frac 23}}$ yield
$$\left|((\chi_Ru^+)^-,v)\right|\lesssim \left[\frac{1}{\eta R}+\sqrt{\eta}\right]\|u^+\|_{L^2}\|\la y\ra v\|_{L^2}\lesssim \frac{1}{R^{\frac 13}}\|u^+\|_{L^2}\|\la y\ra v\|_{L^2},$$
which proves \fref{dual estcomm}.
\end{proof}


\section{Estimates on the cut-off functions}
\label{zppappcomm}

This Appendix is devoted to the derivation of various estimates related to the localization of mass and kinetic energy which are used throughout Section \ref{sectionenergy}. Recall \eqref{defPhione}, \eqref{defPhionebis}.\\

\noindent\underline{{\it $\zeta$ estimates}}. We recall the definition of the cut-off functions, see \eqref{defPhionebis}, \eqref{defPhionebisbis}. The function $\Psi_1$ is smooth enough, non increasing, with
$$\Psi_1(z_1)=\left|\begin{array}{lll} 1\ \ \mbox{for}\ \ z_1\leq \frac 14\\ (1-z_1)^{10}\ \ \mbox{for}\ \ \frac 12\leq z_1\leq 1\\ 0\ \ \mbox{for}\ \ z_1\geq 1.\end{array}\right..$$
Furthermore, $\Phi_1=\Psi_1+b(1-\Psi_1)$, and $$\phi_1(y_1)=\Phi_1\left (\frac{y_1}{R(1-b)}\right )\ ,\ \phi (x)=\phi _1\left (\frac{x-x_1}{\l_1(1-\b_1)}\right )\ .$$
Then, by construction, $b\pa_b\Phi_1=\Phi_1-\Psi_1\le \Phi_1$, and there holds the global control 
\be
\label{globalcancellation}
|(1-z_1)\pa_{z_1}\Phi_1|\lesssim \Phi_1.
\ee
Then, since, by \eqref{defzetabeta}, $\zeta=\beta_1+(1-\beta_1)(1-\phi )$, we have
\be
\label{estderivativezeta}
|\pa_x\zeta|\lesssim (1-\beta_1)|\pa_x\phi|\lesssim \frac 1{R}, \ \ |b\pa_{z_1}\Phi_1|\lesssim b\lesssim \Phi_1.
\ee
We estimate
\be
\label{potzniephione}
|\pa_{y_1}\phi_1|\lesssim \frac{1}{R},
\ee
and
$$\|\pa^2_{y_1}\phi_1\|_{L^1}\lesssim \frac 1R, \ \ \|\pa^3_{y_1}\phi_1\|_{L^1}\lesssim \frac{1}{R^2}$$ from which, using \eqref{commestimate},
\be
\label{ncieoneonveo}
\|[|D^{\frac 12},\pa_{y_1}\phi_1]\|_{L^2\rightarrow L^2}\lesssim \frac{1}{R^{\frac 32}}.
\ee
More generally,
\be
\label{estnorm[hiie}
\|\Phi_1\|_{W_{z_1}^{k,\infty}}\lesssim 1, \ \ k=2,3
\ee
and hence, from $\phi_1=\Phi_1\left(\frac{y_1}{R(1-b)}\right)$ and \eqref{estcommlonebis},
\bea
\label{commuttuatteressnitbis}
&&\|D^k[\Pi^{\pm},\phi_1]g\|_{L^2}\lesssim \frac{\|g\|_{L^2}}{R^k}, \ \ k=1,2\\
\label{commuttuatteressnitbisbis}
&&\|D[\Pi^{\pm},\pa_{y_1}\phi_1]g\|_{L^2}\lesssim \frac{1}{R^2}\|g\|_{L^2}
\eea
Next we compute
\be
\label{formulajoienonev}
 \pa_t\zeta+\pa_x\zeta= (\beta_1)_t\phi_1-(1-\beta_1)(\pa_t\phi_1+\pa_x\phi_1)=(1-\b_1)W\left (t,\frac{y_1}{R(1-b)}\right )\ ,\ 
 \ee
 with
 \bea
\label{formulajoienonevbis} 
&& W(t,z_1)=\frac{(\beta_1)_t}{1-\b_1}\Phi_1 - \frac{b_t}{b}(\Phi_1-\Psi_1)-\frac{1-(x_1)_t}{\l_1(1-\beta_1)(1-b)R}\pa_{z_1}\Phi_1\nonumber \\
&&-\left(-\frac{(\l_1)_t}{\l_1}+\frac{(\beta_1)_t}{1-\beta_1}+\frac{b_t}{1-b}-\frac{R_t}{R}\right)z_1\pa_{z_1}\Phi_1 \nonumber\\
&& =\frac{(\beta_1)_t}{1-\b_1}\Phi_1 - \frac{b_t}{b}(\Phi_1-\Psi_1)+\frac{\l_1R_tz_1-1}{\l_1R}\pa_{z_1}\Phi_1+\frac{b}{\l_1(1-b)R}\pa_{z_1}\Phi_1\nonumber \\
&&+\frac{(x_1)_t-\b_1}{\l_1(1-\beta_1)(1-b)R}\pa_{z_1}\Phi_1-\left(-\frac{(\l_1)_t}{\l_1}+\frac{(\beta_1)_t}{1-\beta_1}+\frac{b_t}{1-b}\right)z_1\pa_{z_1}\Phi_1\ .
\eea
We now use the bounds \eqref{boundessential}, \eqref{estmod}, \eqref{bootdynamicalproved} and $b\lesssim \phi_1$ to derive
$$\left|\frac{(\beta_1)_t}{1-\beta_1}\right|+\left|\frac{(\l_1)_t}{\l_1}\right|+\left|\frac{(x_1)_t-\beta_1}{1-\beta_1}\right|+|b_t|\lesssim \frac b{t}\lesssim \frac{\phi_1}{t}
$$ and hence, we obtain
\be 
\label{cbejbebeibie}
 |\pa_t\zeta+\pa_x\zeta|\lesssim   \frac{1-\beta_1}{t}\phi_1+(1-\beta_1)\left|1-\l_1R_tz_1\right|\frac{|\pa_{z_1}\Phi_1|}{R}\ . 
 \ee
Then we compute
\bee
R_t&=&\frac{(x_2)_t-(x_1)_t}{\l_1(1-\beta_1)}+R\left[-\frac{(\l_1)_t}{\l_1}+\frac{(\beta_1)_t}{1-\beta_1}\right]\\
& = & \frac{\beta_2-\beta_1}{\l_1(1-\beta_1)}+O(b)
\eee
and hence 
$$
1-\l_1R_tz_1= 1-\left[\frac{\beta_2-\beta_1}{1-\beta_1}+O(b)\right]z_1=1-z_1+O(bz_1).
$$
Injecting this into \eqref{cbejbebeibie} with \eqref{globalcancellation} and $b\lesssim \phi_1$, $R\sim t$,  finally yields the fundamental estimate,
\be
\label{relationzetat}
 \left|\pa_t\zeta+\pa_x\zeta\right| \lesssim \frac{1-\beta_1}{t}\phi_1\ .
\ee
Next we estimate the first three derivatives of $\sqrt{\phi_1}$ with respect to $y_1$. Since $\Phi_1=b+(1-b)\Psi_1$, with $\Psi_1$ non increasing, we have 
$$\Phi_1(z_1)\ge \frac{1}{2^{11}}\ ,\ z_1\le \frac12\ ,$$
hence $\pa_{z_1}^k\sqrt{\Phi_1}(z_1)$ are bounded for $k=1,2,3$ and $z_1\le \frac 12$. As for $\frac 12\le z_1\le 1$, 
$$\sqrt{\Phi_1}(z_1)=\left (b+(1-b)(1-z_1)^{10}\right )^{\frac 12}\ ,$$
hence again $\pa_{z_1}^k\sqrt{\Phi_1}(z_1)$ are bounded for $k=1,2,3$. Consequently,
\bee
\label{neionveobis}
\|\pa^k_{y_1}\sqrt{\phi_1}\|_{L^1}&\lesssim \int_{\frac{(1-b)R}{4}\leq |y_1|\leq (1-b)R}\frac{1}{R^k}dy_1\lesssim \frac{1}{R^{k-1}},\ \ k=1,2,3 \ ,
\eee
and thus, from \fref{commestimate}, \eqref{commdbis},
\bea
\label{coomestimate}
&&\|[|D|^{\frac 12},\sqrt{\phi_1}]f\|_{L^2}\lesssim \frac{1}{\sqrt{R}}\|f\|_{L^2}, \ \ \|[|D|,\sqrt{\phi_1}]f\|_{L^2}\lesssim \frac{1}{R}\|f\|_{L^2}\\
\label{coomestimatebis}&& \||D|^{\frac 12}[|D|^{\frac 12},\sqrt{\phi_1}]f\|_{L^2}\lesssim  \frac{1}{\sqrt{R}}\||D|^{\frac 12}f\|_{L^2}+\frac{1}{R}\|f\|_{L^2}.
\eea
According to \fref{defpsi1}, consider now $$\psi=\frac{\pa_t\zeta+\pa_x\zeta}{\sqrt{\phi}}, \ \ \psi(x)=\psi_1(y_1)$$ then from \eqref{relationzetat}:
\be
\label{poitpsi}
|\psi|\lesssim  \frac{(1-\beta_1)}{t}\sqrt{\phi}.
\ee
In order to estimate the first two derivatives of $\psi_1$ with respect to $y_1$, we use  \eqref{formulajoienonev}, \eqref{formulajoienonevbis}. We already noticed that the first three derivatives of $\sqrt{\Phi_1}(z_1)$ are bounded. By a similar argument, the first two derivatives of $\Psi_1/\sqrt{\Phi_1}$ are bounded. 
Consequently, using again $R\sim t$,
 $$\pa_{y_1}\psi_1=O\left(\frac{1-\beta_1}{t^2}{\bf 1}_{\frac{R(1-b)}{4}\leq y_1\leq R(1-b)}\right)\ ,\ \pa^2_{y_1}\psi_1=O\left(\frac{1-\beta_1}{t^3}{\bf 1}_{\frac{R(1-b)}{4}\leq y_1\leq R(1-b)}\right).$$
   Hence
$$\|\pa_{y_1}\psi_1\|_{L^1}\lesssim \frac{1-\beta_1}t, \ \ \|\pa^2_{y_1}\psi_1\|_{L^1}\lesssim \frac{1-\b_1}{t^2}.$$
We conclude, from \fref{commestimate}, that
\be
\label{estconveonovd}
\|[\Dhalf,\psi_1]\|_{L^2\to L^2}\lesssim \frac{1-\beta_1}{t^{\frac 32}}.
\ee
\noindent\underline{{\it $\theta$ estimates}}. Recall from \eqref{defPhionebisbis}, \eqref{deftheta}: $$\theta(t,x)=\theta(t,y_1)=\frac{1}{\l_1}\Psi_1(z_1)+\frac{1}{\l_2}(1-\Psi_1)(z_1).$$ Hence \be
\label{estnvienvoene}
|\pa_{y_1}\theta_1|\lesssim \frac{|\l_2-\l_1|}{R}\pmb{1}_{\frac{(1-b)R}{2}\leq y_1\leq (1-b)R}\ ,
\ee
and therefore
\be
\label{estlinfitythetea}
|\pa_x\theta|\lesssim \frac{|\l_1-\l_2|}{(1-\beta_1)R}\ .
\ee
Next $$[\Pi^{\pm},\theta_1]=[\Pi^{\pm},\frac{1}{\l_1}\Psi_1+\frac{1}{\l_2}(1-\Psi_1)]=\frac{\l_2-\l_1}{\l_1\l_2}[[\Pi^{\pm},\Psi_1]$$ and hence from \eqref{estcommlonebis}:
\be
\label{estcommlonebisdeclined}
\left\|\pa_{y_1}[\Pi^{\pm},\theta_1]g\right\|_{L^2}\lesssim \frac{\vert \l_2-\l_1\vert }{R}\|g\|_{L^2}\ .
\ee
We now estimate more carefully:
\bee
&&(\pa_t+\pa_x)\theta=-\frac{(\l_1)_t}{\l_1^2}\Psi_1-\frac{(\l_2)_t}{\l_2^2}(1-\Psi_1)\\
& + &\left (\frac{1}{\l_1}-\frac{1}{\l_2}\right )\left[ \frac{\beta_1-(x_1)_t}{(1-\beta_1)\l_1R(1-b)}+\frac{1}{\l_1R(1-b)}\right]\pa_{z_1}\Psi_1\\
&+&\left (\frac{1}{\l_1}-\frac{1}{\l_2}\right )\left[-\frac{(\l_1)_t}{\l_1} +\frac{(\beta_1)_t}{1-\beta_1}-\frac{R_t}{R}+\frac{b_t}{1-b}\right]z_1\pa_{z_1}\Psi_1
\eee
and hence 
\be
\label{estpatpxtheta}
(\pa_t+\pa_x)\theta=  O\left(\frac{1}{t}\right).
\ee


\section{Proof of Proposition \ref{coerclinearizedenergy}}
\label{appendixcoerc}

This Appendix is devoted to the proof of Proposition \ref{coerclinearizedenergy}. We recall the coercivity of the linearized Szeg\H{o} operator which we will use in the following form: there exists a universal constant $0<c_0<1$ such that for $u\in H^{\frac 12}_+$, 
\be
\label{coercSzeg\H{o}}
(\mathcal L_+u,u)\geq c_0\|u\|^2_{H^{\frac 12}_+}-\frac{1}{c_0}\left[(u,\pa_yQ^+)^2+(u,iQ^+)^2\right].
\ee

\begin{proof}[Proof of Proposition \ref{coerclinearizedenergy}]
We define the following functionals:
\begin{align*}
\mathcal G_1(\eps)&= \int ||D|^{\frac 12}\e_1^+|^2dy_1+\frac{1}{1-\beta_1}\int ||D|^{\frac 12}\e_1^-|^2dy_1+\l_1\|\e_1\|_{L^2}^2-(2|\Phi^{(1)}|^2\e_1+(\Phi^{(1)})^2\overline{\e_1},\e_1)\\
\matchal G_0(\e)&= \beta_1\int ||D|^{\frac 12}\e_1^+|^2dy_1+(1-\beta_1)\matchal G_1(\eps)\\
\end{align*}
where 
$$\Phi^{(1)}(y_1)=V_{1}(\matchal P, y_1)+\frac{1}{\mu^{\frac 12}}V_2 (\mathcal P,y_2)e^{i\Gamma}.$$ Then the full functional $\matchal G$ is exactly given by:
\bea
\label{formulafullg}
\mathcal G(\eps)&=&\frac{1}{2}\left[\frac{1}{\l_1}\mathcal G_0(\e,\e)-(\zeta D\e,\e)+((\theta-1)\e,\e)\right]\\
\nonumber & - & \frac 14\left[ \int(|\e+\Phi|^4-|\Phi|^4) -4(\e,\Phi|\Phi|^2)-2(2|\Phi|^2\e+\Phi^2\overline{\e},\e)\right]
\eea
The heart of the proof is the derivation of a suitable coercivity for $\matchal G_0$.\\

\noindent{\bf Step 1:} Splitting and coercivity for the first bubble. Let 
$\chi_\ell (y_1)=\chi^{(0)}(\frac{y_1}{R})$, where $\chi^{(0)}$ is a smooth cut off function
satisfying:
$$\chi^{(0)}(y_1)=\left\{\begin{array}{ll} 1\ \ \mbox{for}\ \ y_1\leq \frac{1}{10}\\ 0\ \ \mbox{for}\ \ y_1\geq \frac{1}{5}\end{array}\right..$$ 
We now split  the $L^2$ norm:
\bee
\int |\e_1|^2dy_1&=&\int |\e_1^+|dy_1+\int |\e_1^-|^2dy_1=\int |\chi_l\e_1^+|^2dy_1+\int(1-\chi^2_l)|\e_1^+|^2dy_1+\int |\e_1^-|^2dy_1\\
& = & \int |(\chi_l\e_1^+)^+|^2dy_1+\int |(\chi_l\e_1^+)^-|^2dy_1+\int(1-\chi^2_l)|\e_1^+|^2dy_1+\int |\e_1^-|^2dy_1.
\eee
We now split the kinetic energy according to \fref{commestimatebis}:
\bee
\int ||D|^{\frac 12}\e_1^+|^2dy_1&=& \int \chi^2_l ||D|^{\frac 12}\e_1^+|^2dy_1+\int (1-\chi^2_l) ||D|^{\frac 12}\e_1^+|^2dy_1\\
& = & \int \left|\left[|D|^{\frac 12}(\chi_l\e_1^+)\right]^+\right|^2dy_1+ \int \left|\left[|D|^{\frac 12}(\chi_l\e_1^+)\right]^-\right|^2dy_1\\
& + & \int (1-\chi^2_l) ||D|^{\frac 12}\e_1^+|^2dy_1\\
& + & O\left(\frac{\|\e_1\|^2_{L^2}+\||D|^{\frac 12}(\chi_l\e_1^+)\|_{L^2}^2}{\sqrt{R}}\right)
\eee
We now decompose the potential energy. We first estimate:
$$(2|\Phi^{(1)}|^2\e_1+(\Phi^{(1)})^2\overline{\e_1},\e_1)=(2|\Phi^{(1)}|^2\e^+_1+(\Phi^{(1)})^2\overline{\e^+_1},\e^+_1)+O\left(\|(\Phi^{(1)})^2\e_1^-\|_{L^2}\|\e_1\|_{L^2}\right).$$ 
We now estimate from $|V_j|\lesssim \frac1{\la y_j\ra}$ and Sobolev:
\bee
&& \int |V_1|^{4}|\e_1^-|^2dy_1\lesssim \int \frac{|\e_1^-|^2}{\la y_1\ra ^2}dy_1\lesssim \|\e_1^-\|_{L^4}^2\lesssim \|\e_1^-\|_{L^2}\|\e_1^-\|_{\dot{H}^{\frac 12}}\lesssim \|\e_1\|_{L^2}\|\e_1^-\|_{\dot{H}^{\frac 12}}\\
&& \int |V_2|^{4}|\e_1^-|^2dy_1\lesssim \int \frac{|\e_1^-|^2}{\jb{y_2}^4}dy_1\lesssim  \sqrt{b} \|\e_1\|_{L^2}\|\e_1^-\|_{\dot{H}^{\frac 12}}
\eee

We now develop the potential term:
\bee
\int |\Phi^{(1)}|^2|\e_1^+|^2dy_1&=&\int |\Phi^{(1)}|^2\left[\chi^2_l|\e_1^+|^2+(1-\chi^2_l)|\e_1^+|^2\right]dy_1\\
& = &  \int |Q_{\beta_1}|^2|\chi_l\e_1^+|^2dy_1+O\left(\frac{|\log\eta|^4\|\e_1\|_{L^2}^2}{R}\right)\\
& + &   \int |\Phi^{(1)}|^2(1-\chi^2_l)|\e_1^+|^2dy_1\\
& = &  \int |Q^+|^2|\chi_l\e_1^+|^2dy_1+O\left(\left[\frac{|\log\eta|^4}{R}
+(1-\beta_1)^{\frac 12}|\log(1-\b_1)|^{\frac 12}\right]\|\e_1\|_{L^2}^2\right)\\
& + &  \int |\Phi^{(1)}|^2(1-\chi^2_l)|\e_1^+|^2dy_1
\eee
by construction of $V_1$, the support properties of $\chi_l$ and the rough bound
 $$\|Q_{\beta_1}-Q^+\|_{L^{\infty}}\lesssim \|Q_{\b_1}-Q^+\|_{H^1}\lesssim (1-\beta_1)^{\frac 12}|\log(1-\b_1)|^{\frac 12}.$$ 
We now use \fref{estcomm} and $|Q^+|\lesssim \frac{1}{\la y_1\ra}$ which ensure
\be
\label{cnkneoneo}
\int \frac{|(\chi_l\e_1^+)^-|^2}{\la y_1\ra^2}dy_1\lesssim \frac{1}{R^{\frac 23}}\|\e_1\|_{L^2}^2
\ee 
to conclude:
\begin{align*}
\int |\Phi^{(1)}|^2|\e_1^+|^2dy_1& =  \int |Q^+|^2\left|\left[(\chi_l\e_1^+)\right]^+\right|^2dy_1
+O\left(\left[\frac{1}{R^{\frac 23}}+(1-\beta_1)^{\frac 12}|\log(1-\b_1)|^{\frac 12}\right]\|\e_1\|_{L^2}^2\right)\\
&+\int |\Phi^{(1)}|^2(1-\chi^2_l)|\e_1^+|^2dy_1.
\end{align*}
We argue similarly for the second potential term and obtain the first decomposition:
\bea
\label{fristformulagnot}
&&\mathcal G_1(\eps)=(\matchal L_+(\chi_l\e_1^+)^+,(\chi_l\e_1^+)^+)+(\l_1-1)\int |(\chi_\ell \e_1^+)^+|^2dy_1\\
\nonumber & + &  \frac{1}{1-\beta_1}\int |\Dhalf \e_1^-|^2dy_1+ \int ||D|^{\frac 12}(\chi_l\e_1^+)^-|^2dy_1+\l_1\int |(\chi_l\e^+_1)^-|^2dy_1+\l_1\int |\e_1^-|^2dy_1\\
\nonumber & + &  \int (1-\chi^2_l)|\Dhalf \e_1^+|^2dy_1+\l_1\int (1-\chi^2_l)|\e_1^+|^2dy_1\\
\nonumber & - & 2\int|\Phi^{(1)}|^2(1-\chi^2_l)|\e^+_1|^2-\Re\int(\Phi^{(1)})^2(1-\chi^2_l)\overline{(\e^+_1)^2}\\
\nonumber  & + & O\left(\left[\frac{1}{\sqrt{R}}+(1-\beta_1)^{\frac 12}|\log(1-\b_1)|^{\frac 12}\right]\|\e_1\|^2_{L^2}+\|\e_1\|^{\frac 32}_{L^2}\||D|^{\frac 12}\e^-_1\|^{\frac 12}_{L^2}+\frac{1}{\sqrt{R}}\|\Dhalf(\chi_l\e^+_1)\|^2_{L^2}\right).
\eea
From the choice of orthogonality conditions \fref{orthoe} we have:
\bee
0& = & (\e_1,Q_{\beta_1})^2=(\e_1^+,Q^+)^2+O((1-\beta_1)^{\frac 12}|\log(1-\b_1)|^{\frac 12}\|\e_1\|_{L^2}^2)\\
&=& (\chi_l\e^+_1,Q^+)^2+O\left(\left[(1-\beta_1)^{\frac 12}|\log(1-\b_1)|^{\frac 12}+\frac{1}{R}\right]\|\e_1\|_{L^2}^2\right) \\
&=&  ((\chi_l\e^+_1)^+,Q^+)^2+O\left(\left[(1-\beta_1)^{\frac 12}|\log(1-\b_1)|^{\frac 12}+\frac{1}{R}\right]\|\e_1\|_{L^2}^2\right),
\eee and similarly:
$$0=(\e_1,i\pa_{y_1}Q_{\beta_1})^2= ((\chi_l\e^+_1)^+,i\pa_{y_1}Q^+)^2+O\left(\left[
(1-\beta_1)^{\frac 12}|\log(1-\b_1)|^{\frac 12}+\frac{1}{R}\right]\|\e_1\|_{L^2}^2\right).$$ 

We now apply the coercivity estimate \fref{coercSzeg\H{o}} to $(\chi_l\e^+_1)^+$ and obtain from
 \fref{fristformulagnot} the control:
\bea
\label{controlgone}
&&\mathcal G_1(\eps)\geq c_0\left[\|\Dhalf (\chi_l\e^+_1)\|_{L^2}^2+\|\chi_l\e_1^+\|_{L^2}^2\right]+\frac{1}{(1-\beta_1)}\int |\Dhalf \e_1^-|^2\\
\nonumber & + &(\l_1-1)\int |(\chi_\ell\e_1^+)^+|^2dy_1 +\l_1\int |\e_1^-|^2+O\left(\left[
(1-\beta_1)^{\frac 12}|\log(1-\b_1)|^{\frac 12}+\frac{1}{\sqrt{R}}\right]\|\e_1\|_{L^2}^2\right)\\
\nonumber & + &  \int (1-\chi^2_l)|\Dhalf \e_1^+|^2+\l_1\int (1-\chi^2_l)|\e_1^+|^2\\
\nonumber & - & 2\int|\Phi^{(1)}|^2(1-\chi^2_l)|\e^+_1|^2-\Re\int(\Phi^{(1)})^2(1-\chi^2_l)\overline{(\e^+_1)^2}\\
\nonumber & + &O\Big(\|\e_1\|_{L^2}^{\frac 32}\||D|^{\frac 12}\eps_1^-\|_{L^2}^{\frac 12}
+\frac{1}{\sqrt{R}}\||D|^{1/2}(\chi_\ell\e_1^+)\|_{L^2}^2\Big)
\eea

\medskip

\noindent{\bf Step 2:} Coercivity for the second bubble. We now consider 
$\chi_R(y_2)=\chi^{(1)}(\frac{y_2}{R})$, where
$\chi^{(1)}$ is a smooth cut off function satisfying
$$\chi^{(1)}(y_2)=\left\{\begin{array}{ll}0 \ \ \mbox{for}\ \ y_2\leq -3\\
1\ \ \mbox{for}\ \ y_2\geq -2,
\end{array}\right.$$
and let 
\begin{align*}
\mathcal G_2(\eps):&=  b\int\chi^2_r |\Dhalf \e_1^+|^2dy_1+\l_1\int |\chi_r\e_1^+|^2- 2\int|\Phi^{(1)}|^2\chi_r^2|\e^+_1|^2dy_1\\
&-\Re\int(\Phi^{(1)})^2\chi_r^2\overline{(\e^+_1)^2}dy_1.
\end{align*}
$\mathcal G_2$ will be useful in finding a lower bound for $\mathcal G_1$.
We observe from the support property of $\chi_l,\chi_r$ and by construction of $V_j$ the bounds:
$$\left|2\int|\Phi^{(1)}|^2(1-\chi^2_l-\chi_r^2)|\e^+_1|^2dy_1-\Re\int(\Phi^{(1)})^2(1-\chi^2_l-\chi_r^2)\overline{(\e^+_1)^2}dy_1\right|\lesssim \frac{1}{R^2}\int |\e_1|^2$$
and therefore rewrite \fref{controlgone}:
\bea
\label{controlgonebispouet}
&&\mathcal G_1(\eps)\geq \mathcal G_2(\eps) + c_0\left[\|\Dhalf (\chi_l\e^+_1)\|_{L^2}^2+\|\chi_l\e_1^+\|_{L^2}^2\right]+\frac{1}{(1-\beta_1)}\int |\Dhalf \e_1^-|^2dy_1\\
\nonumber & + & (\l_1-1)\int |(\chi_\ell\e_1^+)^+|^2dy_1+
\l_1\int |\e_1^-|^2dy_1+  \int (1-\chi^2_l-b\chi_r^2)|\Dhalf \e_1^+|^2dy_1\\
\nonumber &+& \mathcal \l_1\int (1-\chi^2_l-\chi_r^2)|\e_1^+|^2dy_1\\
\nonumber &+&O\left(\left[(1-\beta_1)^{\frac 12}|\log(1-\b_1)|^{\frac 12}+\frac{1}{\sqrt{R}}\right]\|\e_1\|_{L^2}^2
+\|\e_1\|_{L^2}^{\frac 32}\||D|^{\frac 12}\e_1^-\|_{L^2}^2
+\frac{\||D|^{\frac 12}(\chi_\ell \e_1^+)\|_{L^2}^2}{\sqrt{R}}\right).
\eea
We renormalize to the $y_2$ variable using the formula 
$$\e_1(y_1)=\frac{e^{i\Gamma}}{\sqrt{\mu}}\e_2\left(\frac{y_1-R}{b\mu}\right)=\frac{e^{i\Gamma}}{\sqrt{\mu}}\e_2 (y_2)$$ 
and compute:
\begin{align*}
\mu\mathcal G_2(\eps)&= b\Big[\int\chi^2_r |\Dhalf \e_2^+|^2dy_2+\l_2\int |\chi_r\e_2^+|^2dy_2
- 2 \int|\Phi^{(2)}|^2\chi_r^2|\e^+_2|^2dy_2\\
&-\Re\int(\Phi^{(2)})^2\chi_r^2\overline{(\e^+_2)^2}dy_2\Big].
\end{align*}
where 
$$\Phi^{(2)}(y_2)=\mu^{\frac12}V_{1}(\matchal P, y_1)e^{-i\Gamma}+V_2 (\mathcal P,y_2).$$
We estimate using \fref{commestimatebis}:
\begin{align*}
\int(1-\chi^2_l) |\Dhalf \e_2^+|^2dy_2&=\int(1-\chi^2_l-\chi^2_r) |\Dhalf \e_2^+|^2dy_2+\int|\Dhalf(\chi_r \e_2^+)|^2dy_2\\
&+O\left(\frac{\|\e^+_2\|_{L^2}^2+\|\Dhalf(\chi_r \e_2^+)\|_{L^2}^2}{\sqrt{R}}\right)
\end{align*}
and estimate as for the first bubble the potential energy to obtain:
\bee
\mu\matchal G_2(\eps)&=&b\left[(\mathcal L_+(\chi_r\e_2^+)^+,(\chi_r\e_2^+)^+)+\int |\Dhalf(\chi_r\e_2^+)^-|^2dy_2+\l_2\int|(\chi_r\e_2^+)^-|^2dy_2\right]\\
& + & b(\l_2-1)\int |(\chi_r\e_2^+)^+|^2dy_2\\
& + & bO\left(\left[\frac{1}{\sqrt{R}}+(1-\beta_2)^{\frac 12}|\log(1-\b_2)|^{\frac 12}\right]\|\e_2\|_{L^2}^2
+\frac{\|\Dhalf(\chi_r \e_2^+)\|_{L^2}^2}{\sqrt{R}}\right).
\eee
We estimate using the orthogonality conditions \fref{orthoe}:
$$((\chi_r\e_2^+)^+,Q^+)^2+((\chi_r\e_2^+)^+,i\pa_yQ^+)^2\lesssim \left[(1-\beta_2)^{\frac 12}|\log(1-\b_2)|^{\frac 12}+\frac{1}{R}\right]\|\e_2\|_{L^2}^2$$
and hence conclude using the coercivity \fref{coercSzeg\H{o}}:
\bea
\label{cenionvenveo}
\nonumber \matchal G_2(\eps)&\geq &\frac{bc_0}{\mu}\left[\|\Dhalf (\chi_r\e^+_2)\|_{L^2}^2+\int\chi_r^2|\e^+_2|^2dy_2\right]+\frac{b(\l_2-1)}{\mu}\int |(\chi_r\e_2^+)^+|^2dy_2\\
& + & \frac{b}{\mu} O\left(\left[\frac{1}{\sqrt{R}}+(1-\beta_2)^{\frac 12}|\log(1-\b_2)|^{\frac 12}\right]\|\e_2\|_{L^2}^2+\frac{\|\Dhalf(\chi_r \e_2^+)\|_{L^2}^2}{\sqrt{R}}\right).
\eea

\noindent{\bf Step 3:} Coercivity of $\matchal G_0$. We sum \fref{controlgonebispouet} and \fref{cenionvenveo} and conclude:
\bee
\mathcal G_1(\eps)& \geq & c_0\left[\|\Dhalf (\chi_l\e^+_1)\|_{L^2}^2+\|\chi_l\e_1\|_{L^2}^2\right]+\frac{1}{(1-\beta_1)}\int |\Dhalf \e_1^-|^2dy_1\\
\nonumber & + & (\l_1-1)\int |(\chi_\ell\e_1^+)^+|^2dy_1+\l_1\int |\e_1^-|^2dy_1\\
\nonumber & + & \int (1-\chi^2_l-b\chi_r^2)|\Dhalf \e_1^+|^2dy_1+\int (1-\chi^2_l-\chi_r^2)|\e_1^+|^2dy_1\\
&+& O\left(\left[(1-\beta_1)^{\frac 12}|\log(1-\b_1)|^{\frac 12}+\frac{1}{\sqrt{R}}\right]\|\e_1\|_{L^2}^2
+\|\e_1\|_{L^2}^{\frac 32}\||D|^{\frac 12}\eps_1^-\|_{L^2}^{\frac 12}+\frac{\||D|^{\frac 12}(\chi_\ell\e_1^+)\|_{L^2}^2}{\sqrt{R}}\right)\\
& +& bc_0\left[\|\Dhalf (\chi_r\e^+_2)\|_{L^2}^2+\int\chi_r^2|\e^+_2|^2dy_2\right]
+b(\l_2-1)\int |(\chi_r\e_2^+)^+|^2dy_2\\
\nonumber & + &b O\left(\left[\frac{1}{\sqrt{R}}+(1-\beta_2)^{\frac 12}|\log(1-\b_2)|^{\frac 12}\right]\|\e_2\|_{L^2}^2+\frac{\|\Dhalf(\chi_r \e_2^+)\|_{L^2}^2}{\sqrt{R}}\right).
\eee
which after renormalization to the $y_1$ variable implies:
\bea
\label{ceocfinal}
\mathcal G_1(\eps)& \geq &  c_0\left[\int(\chi_l^2+b\chi_r^2)|\Dhalf \e^+_1|^2dy_1+\|\e_1\|_{L^2}^2\right]+\frac{1}{1-\beta_1}\int |\Dhalf \e_1^-|^2dy_1\\
\nonumber & + &   \int (1-\chi^2_l-b\chi_r^2)|\Dhalf \e_1^+|^2dy_1+{ \mathcal Err}(\e),
\eea
where
\begin{align}\label{error_H}
{\mathcal Err}(\e)
&=c_0\Big(\frac{1}{\mu}-1\Big)\|\chi_r\e_1^+\|_{L^2}^2
+c_0(\l_1-1)\|(\chi_\ell\e_1^+)^+\|_{L^2}^2
+\frac{\l_2-1}{\mu}\|(\chi_r\e_1^+)^+\|_{L^2}^2\notag\\
&+(\l_1-1)\|\e_1^-\|_{L^2}^2
+O\Big((1-\beta_1)^{\frac 12}|\log(1-\b_1)|^{\frac 12}+(1-\beta_2)^{\frac 12}|\log(1-\b_2)|^{\frac 12}+\frac{1}{\sqrt{R}})\|\e_1\|_{L^2}^2\notag\\
&+O\Bigg(\frac{\||D|^{\frac 12}(\chi_\ell\e_1^+)\|_{L^2}^2+b\||D|^{\frac 12}(\chi_r\e_1^+)\|_{L^2}^2}{\sqrt{R}}\Bigg)\notag\\
&+O\big(\|\e_1\|_{L^2}^{\frac 32}\|\Dhalf \e_1^-\|_{L^2}^{\frac 12}\big).
\end{align}

\noindent
Equivalently, this yields the lower bound:
\bee
\matchal G_0(\e)& = & \beta_1\int ||D|^{\frac 12}\e_1^+|^2dy_1+(1-\beta_1)\matchal G_1\\
& \geq & c_0(1-\beta_1)\|\e_1\|_{L^2}^2
+ \int \big[\beta_1+(1-\beta_1)(1-\phi_0+c_0\phi_0)\big]|\Dhalf \e_1^+|^2dy_1\\
\nonumber & + &\int |\Dhalf \e_1^-|^2dy_1+ (1-\b_1){\mathcal Err}(\e),
\eee
with 
\be
\label{defphizero}
\phi_0=\chi_l^2+b\chi_r^2.
\ee
We now observe from the support property of $\chi_r,\chi_l,\phi_1$ that $\phi_1\geq \phi_0$ and
since $c_0<1$ and $1-\b_1>0$, we have 
$$ \beta_1+(1-\beta_1)(1-\phi_0+c_0\phi_0)\geq \beta_1+(1-\beta_1)(1-\phi_1+c_0\phi_1).$$ 
We therefore have obtained the coercivity:
\bea
\label{coercvompleteone}
\matchal G_0(\e)&\geq & c_0(1-\beta_1)\int|\e_1|^2\\
\nonumber & + & \int \big[\beta_1+(1-\beta_1)(1-\phi_1+c_0\phi_1)\big]|\Dhalf \e_1^+|^2\\
\nonumber & + & \int |\Dhalf \e_1^-|^2 +(1-\b_1){\mathcal Err}(\e).
\eea

\noindent{\bf Step 4:} Control of the kinetic momentum and coercivity of $\matchal G$. We now consider the full functional given by \eqref{formulafullg}:
\bee
&&\mathcal G(\eps)=\frac{1}{2}\left[\frac{1}{\l_1}\mathcal G_0(\e,\e)-(\zeta D\e,\e)+((\theta-1)\e,\e)\right]+\matchal N(\e)\\
&&\mathcal N(\e)= \frac 14\left[ \int(|\e+\Phi|^4-|\Phi|^4) -4(\e,\Phi|\Phi|^2)-2(2|\Phi|^2\e+\Phi^2\overline{\e},\e)\right].
\eee
The cubic and higher order terms are easily estimated using the rough bound 
 $\|\e\|_{H^1}\ll 1$:
\bee
\mathcal N(\e)\lesssim \int |\e|^4 +C|\e^3||\Phi|dx &\lesssim & \|\e\|_{L^\infty} (\|\e\|_{L^\infty}+\|\Phi\|_{L^\infty})\|\e\|_{L^2}^2
\lesssim \|\e\|_{H^1}(\|\e\|_{H^1}+1)\|\e\|_{L^2}^2\\
&\leq & \frac{c_0}{10}\|\e\|_{L^2}^2=\frac{c_0}{10}(1-\b_1)\|\e_1\|_{L^2}^2.
\eee
The $L^2$ error is estimated from $|\mu|\ll 1$: $$|((\theta-1)\e,\e)|\lesssim |\mu|\|\e\|_{L^2}^2\leq\frac{c_0}{10}(1-\beta_1)\|\e_1\|_{L^2}^2.$$
We therefore conclude from \eqref{coercvompleteone}:
\bee
2\matchal G(\eps)&\geq&  \frac{c_0(1-\beta_1)}{\l_1}\left[\int |\e_1|^2dy_1+\int\phi_1|\Dhalf \e_1^+|^2dy_1\right]+ \frac{1-\b_1}{\l_1}{\mathcal Err}(\e)
\\
\nonumber & + & \frac{1}{\l_1}\left[\int \zeta_1|\Dhalf \e_1^+|^2+  \frac{1}{\l_1}\int |\Dhalf \e_1^-|^2dy_1-(\zeta_1D\e_1,\e_1)\right].
\eee
We now estimate the kinetic momentum term. We first compute from \fref{defzeta}:
\bee
&&\int \zeta_1|\Dhalf \e_1^+|^2-(\zeta_1D\e^+_1,\e^+_1)\\
& = & \int (1-(1-\beta_1)\phi_1)|\Dhalf \e_1^+|^2-((1-(1-\beta_1)\phi_1)D\e^+_1,\e^+_1)\\
& = & -(1-\beta_1)\left[\int \phi_1|\Dhalf \e_1^+|^2-(\phi_1D\e^+_1,\e^+_1)\right]
\eee
We then estimate using \fref{coomestimate} and \fref{estgenerale}:
\bee
&&(\phi_1D\e^+_1,\e^+_1) =  (\phi_1|D|\e^+_1,\e^+_1)=(\sqrt{\phi_1}|D|\e_1^+,\sqrt{\phi_1}\e_1^+)\\
&=& ([\sqrt{\phi_1},|D|]\e_1^++|D|(\sqrt{\phi_1}\e_1^+),\sqrt{\phi_1}\e_1^+)=  \int |\Dhalf \sqrt{\phi_1}\e_1^+|^2+O\left(\frac{1}{R}\|\e_1\|^2_{L^2}\right)\\
& = & \int \phi_1|\Dhalf \e_1^+|^2+O\left(\frac{1}{R}\|\e_1\|^2_{L^2}+\frac{1}{\sqrt{R}}\left[\|\e_1\|^2_{L^2}+\|\sqrt{\phi_1}\Dhalf \e_1^+\|^2_{L^2}\right]\right)
\eee
which yield thanks to the smallness of $\frac{1}{\sqrt{R}}$:
$$
\int \zeta_1|\Dhalf \e_1^+|^2dy_1-(\zeta_1D\e^+_1,\e^+_1)\geq-\frac{c_0(1-\beta_1)}{10}\Big[\int \phi_1|\Dhalf \e_1^+|^2dy_1+\int |\e_1|^2dy_1\Big].
$$
similarly using \eqref{commestimateone}:
\bee
-(\zeta_1D\e_1^-,\e_1^-)&=&(\beta_1+(1-\beta_1)\phi_1|D|\e_1^-,\e_1^-)\\
&=& \beta_1\||D|^{\frac 12}\e_1^-\|_{L^2}^2+(1-\beta_1)O\left(\||D|^{\frac 12}\e_1^-\|_{L^2}^2+\frac{\|\e_1\|_{L^2}^2}{R}\right)\\
& \geq & \frac 12\||D|^{\frac 12}\e_1^-\|_{L^2}^2-\frac{c_0}{10}(1-\beta_1)\|\e_1\|_{L^2}^2.
\eee
For the crossed terms, we estimate from \eqref{estcommlonebis}:
\bee
&&|(\zeta_1 D\e_1^-,\e_1^+)|+|(\zeta_1,D\e_1^+,\e_1^-)|=(1-\beta_1)|(\phi_1 D\e_1^-,\e_1^+)|+|(\phi_1,D\e_1^+,\e_1^-)|\\
& \lesssim & (1-\beta_1)\left[|(\e_1^-,D[\Pi^+,\phi_1]\e_1^+)|+|(\e_1^+,D[\Pi^-,\phi_1]\e_1^-)|\right]\lesssim \frac{1-\beta_1}{R^2}\|\e_1\|_{L^2}^2
\eee
The collection of above estimates yields the lower bound:
\be
\label{H_last_error}
\mathcal G(\e) \geq \frac{c_0(1-\beta_1)}{2\l_1}\left[\int |\e_1|^2+\int\phi_1|\Dhalf \e_1^+|^2\right]+\int |\Dhalf \e_1^-|^2+\frac{1-\b_1}{\l_1}{\mathcal Err}(\eps).
\ee
Finally, we need to treat the error ${\mathcal Err}(\e)$ defined in \eqref{error_H}.
Most of the terms can be bounded using the hypothesis
\[|\l_1-1|+|\l_2-1|+|\mu-1|+|1-\b_1|+|1-\b_2|+\frac{1}{R}\ll 1.\]
We turn to the last term in \eqref{error_H}
and by Young's inequality obtain that
\begin{align*}
C\|\e_1\|_{L^2}^{\frac 32}\||D|^{\frac 12}\e_1^-\|_{L^2}^{\frac 12}
&=C\Big(\sqrt{\frac{c_0}{3\l_1C}}\|\e_1\|_{L^2}\Big)^{\frac 32}\Big(\sqrt{\frac{3\l_1 C}{c_0}}\| |D|^{\frac 12}\e_1^-\|_{L^2}\Big)^{\frac 12}\\
&\leq \frac {c_0}{4\l_1}\|\e_1\|_{L^2}^2+\frac{C(3\l_1 C)^3}{4c_0^3}\||D|^{\frac 12}\e_1^-\|_{L^2}^2.
\end{align*}
Thus, the last term in $\frac{1-\b_1}{\l_1}{\mathcal Err}(\eps)$ has a lower bound:
\[-\frac{c_0(1-\b_1)}{4\l_1} \|\e_1\|_{L^2}^2-\frac{C(3\l_1 C)^3(1-\b_1)}{4c_0^3}\||D|^{\frac 12}\e_1^-\|_{L^2}^2,\]
Since $0<1-\b_1\ll 1$,
it can be absorbed by the main terms in \eqref{H_last_error} to obtain:
\bee
\mathcal G(\e) &\geq& \frac{c_0(1-\beta_1)}{5\l_1}\left[\int |\e_1|^2+\int\phi_1|\Dhalf \e_1^+|^2\right]+\frac{1}{\l_1}\int |\Dhalf \e_1^-|^2
\eee
which concludes the proof of Proposition \ref{coerclinearizedenergy}.
\end{proof}

\end{appendix}

\end{document}